\documentclass[11pt,twoside]{amsart}

\title{Unoriented 2-dimensional TQFTs and the category $\Rep(S_t\wr \mathbb Z_2)$}
\author{Agustina Czenky}
\usepackage{rotating}
\usepackage{ bbold }
\usepackage{graphicx}
\usepackage{amsmath,amsthm, amscd, amssymb, amsfonts}
\usepackage{tikz}
\usetikzlibrary{positioning}
\usepackage{color}
\usepackage[margin=1.25in]{geometry}
\usepackage{nicefrac}
\numberwithin{equation}{section}\theoremstyle{plain}
\usetikzlibrary{tqft}
\usetikzlibrary{scopes,intersections}
\usepackage{relsize}%
\newcommand{\Coordinate}[2]%
{ \coordinate (#1) at (#2);
}
\usepackage{array}
\usepackage{tabularx}
\usepackage{ latexsym }
\usepackage{tikz-cd}
\usepackage{float}
\usepackage{xfrac}
\usepackage{multicol}
\newtheoremstyle{defstyle}
{0.5cm}                   
{0.5cm}                   
{\normalfont}           
{}     
{\normalfont\bfseries}  
{:}                     
{0.3cm}              
{\thmname{#1}\thmnumber{ #2}\thmnote{ (#3)}}

\numberwithin{equation}{section}

\usepackage{fancyhdr}
\fancypagestyle{mypagestyle}{%
	\fancyhf{}
	\fancyhead[OC]{Unoriented 2-dimensional TQFTs and the category $\RepSt$}
	\fancyhead[EC]{A. Czenky}
	\fancyhead[R]{\thepage}%
}
\newtheorem*{rep@theorem}{\rep@title}
\newcommand{\newreptheorem}[2]{%
	\newenvironment{rep#1}[1]{%
		\def\rep@title{#2 \ref{##1}}%
		\begin{rep@theorem}}%
		{\end{rep@theorem}}}
\makeatother

\newtheorem*{rep@corollary}{\rep@title}
\newcommand{\newrepcorollary}[2]{%
	\newenvironment{rep#1}[1]{%
		\def\rep@title{#2 \ref{##1}}%
		\begin{rep@corollary}}%
		{\end{rep@corollary}}}

\newtheorem{theorem}{Theorem}[section]

\newtheorem{proposition}[theorem]{Proposition}
\newreptheorem{proposition}{Proposition}
\newtheorem{corollary}[theorem]{Corollary}
\newreptheorem{corollary}{Corollary}
\newtheorem{lemma}[theorem]{Lemma}

\newtheorem{theorem*}{Theorem}
\newreptheorem{theorem}{Theorem}

\theoremstyle{definition}
\newtheorem{definition}[theorem]{Definition}

\newtheorem{example}[theorem]{Example}
\newtheorem{remark}[theorem]{Remark}

\newtheorem{conjecture}[theorem]{Conjecture}

\newcommand\pf{\begin{proof}}
	\newcommand\epf{\end{proof}}

\newcommand\Hom{\operatorname{Hom}}
 
\newcommand\End{\operatorname{End}}

\newcommand\Rep{\operatorname{Rep}}

\newcommand\tr{\operatorname{tr}}
\newcommand\Id{\operatorname{Id}}

\newcommand\UCob{\operatorname{\text{UCob}_2}}
\newcommand\Cob{\operatorname{\text{Cob}_2}}

\newcommand\StC{\operatorname{S_t(\mathcal C)}}
\newcommand\RepSt{\operatorname{\text{Rep}(S_t \wr \mathbb Z_2)}}

\newcommand\VUCob{\operatorname{\text{VUCob}_{\alpha,\beta,\gamma}}}
\newcommand\SUCob{\operatorname{\text{SUCob}_{\alpha,\beta,\gamma}}}
\newcommand\PsUCob{\operatorname{\text{UCob}_{\alpha,\beta,\gamma}}}
\newcommand\OCob{\operatorname{\text{OCob}_{\alpha,\beta,\gamma}}}

\newcommand\PsOCob{\operatorname{\text{OCob}_{\alpha}}}
\newcommand\SOCob{\operatorname{\text{SOCob}_{\alpha}}}

\newtheorem{maintheorem}{Theorem}

\newcommand\figureud{
	\begin{tikzpicture}
		\draw (0,0) to (0,0.3);
		\draw (-0.1,0.1) to (0,0);
		\draw (0.1,0.1) to (0,0);
		\draw (-0.1,0.2) to (0,0.3);
		\draw (0.1,0.2) to (0,0.3);
	\end{tikzpicture}
}

\newcommand\figureXv{
	\begin{tikzpicture}
		\coordinate (O) at (0,0.1);
		\coordinate (A) at (1/9,1/2);
		\coordinate (B) at (0,0.4);
		\coordinate (C) at (1/9,0);
		
		\draw[style=thick] (O)--(A);
		\draw[style=thick,color=black] (B) to  (C);
	\end{tikzpicture}
}

\newcommand\figureXvl{
	\begin{tikzpicture}
		\coordinate (O) at (-0.1,0.1);
		\coordinate (A) at (1/9,0.6);
		\coordinate (B) at (-0.1,0.55);
		\coordinate (C) at (1/9,0);
		
		\draw[style=thick] (O)--(A);
		\draw[style=thick,color=black] (B) to  (C);
	\end{tikzpicture}
}

\newcommand\smallfigureXv{
	\resizebox{2pt}{!}{%
		\begin{tikzpicture}
			\coordinate (O) at (0,0.1);
			\coordinate (A) at (1/9,1/2);
			\coordinate (B) at (0,0.4);
			\coordinate (C) at (1/9,0);

			\draw[style=thick] (O)--(A);
			\draw[style=thick,color=black] (B) to  (C);
	\end{tikzpicture}}
}

\newcommand\figureA{
	\begin{tikzpicture}
		\coordinate (O) at (0,0);
		\coordinate (A) at (1/9,0.45);
		\coordinate (B) at (0,0.4);
		\coordinate (C) at (1/9,0);
		
		\draw[] (O)--(A);
		\draw[color=black] (B) to  (C);
	\end{tikzpicture}
}
\newcommand\figureS{
	\reflectbox{
		\begin{tikzpicture}
			\coordinate (O) at (0,0.1);
			\coordinate (A) at (1/9,1/2);
			\coordinate (B) at (0,0.4);
			\coordinate (C) at (1/9,0);
			
			\draw[style=thick] (O)--(A);
			\draw[style=thick,color=black] (B) to  (C);
	\end{tikzpicture}}
}

\newcommand{\Arrow}{\mathord{\begin{tikzpicture}[baseline=0ex, line width=1, scale=0.13]
			\draw (0,1) -- (2,1);
			\draw (1,2) -- (2,1);
			\draw (1,0)--(2,1);
\end{tikzpicture}}}

\makeatletter
\def\l@subsection{\@tocline{2}{0pt}{2.5pc}{5pc}{}}
\def\l@subsubsection{\@tocline{2}{0pt}{2.5pc}{5pc}{}}
\makeatother

\usepackage{tocvsec2}
\begin{document}

	\begin{abstract}
		We construct a family of unoriented 2-dimensional cobordism theories parametrized by certain triples of sequences. We also prove that some specializations of these sequences yield equivalences with an exterior product of Deligne categories. In particular, one of these specializations recovers the generalized Deligne category $\Rep\left(S_{t} \wr \mathbb Z_2\right)$.
	\end{abstract}
	
	\maketitle
	
	\tableofcontents
	
	\section{Introduction}
	
	Let $\mathbf{k}$ be  an algebraically closed field. Symmetric monoidal functors from the category $\text{Cob}_n$ of  oriented $n$-dimensional cobordisms into a symmetric monoidal category $\mathcal C$ are known as $n$-dimensional $\mathcal C$-valued topological quantum field theories (TQFTs) \cite{A}. In the $2$-dimensional oriented case, the theory of tensor categories has been employed to understand and build examples of TQFTs, see e.g. \cite{Kh, KKO, SP} and references therein. A special feature of $\Cob$ is that it admits a description by generators and relations, which provides an algebraic understanding of topological quantum field theories in the 2-dimensional case. Explicitly, it is well known (see e.g. \cite[Theorem 0.1]{SP}) that 2-dimensional TQFTs $\text{Cob}_2 \to \mathcal C$ are in bijection with commutative Frobenius algebras in $\mathcal C$.
	
	This paper is devoted to the study of \emph{unoriented} $2$-dimensional TQFTs, that is, symmetric monoidal functors from the category $\UCob$ of unoriented $2$-dimensional cobordisms into a symmetric monoidal category. As in the oriented case, $\UCob$ admits a description by generators and relations. We will use the one given in \cite{TT}, see also \cite{AN} for a related construction. Explicitly, the generators are given by the cup, cap, pair of pants, reverse pair of pants, and twist cobordisms, same as in $\Cob$, plus two extra ones, representing the orientation reversing diffeomorphism of the circle, and the M\"oebius band or \emph{crosscap}. The presence of these two extra generators results in a connection with \emph{extended} Frobenius algebras, which are Frobenius algebras with additional structure \cite{TT}.

	A family of \textbf{k}-linear $2$-dimensional TQFTs, one for each rational sequence $\alpha$ in \textbf{k},  was  introduced by Khovanov and Sazdanovic in \cite{KS}. One can   ``linearize" the category $\Cob$ using the sequence $\alpha$ by allowing \textbf{k}-linear combinations of cobordisms and evaluating  closed connected oriented surfaces of genus $g$ to $\alpha_g$. This results in the \textbf{k}-linear tensor category $\text{VCob}_{\alpha}$ of \emph{viewable cobordisms} \cite{KS, KKO}. It can then be quotiented by a tensor ideal defined using the sequence $\alpha$ to produce the category $\text{SCob}_{\alpha}$ of \emph{skein} cobordisms, with objects non-negative integers, and hom spaces $\Hom_{\text{SCob}_{\alpha}}(n,m)$ given by linear combinations of oriented cobordisms from $n$ to $m$ circles. See also \cite{CMS} for a related construction concerning $r$-spin TQFTs, which generalizes some of the results in \cite{KKO}. 
	
	When \textbf{k} has characteristic zero, 2-dimensional cobordisms can be understood in purely algebraic terms using the Deligne category $\Rep(S_t)$. Introduced in \cite{D}, $\Rep(S_t)$  interpolates the categories $\Rep(S_n)$ of finite-dimensional \textbf{k}-representations of the symmetric group $S_n$ from $n\in \mathbb N$ to any parameter $t\in \textbf{k}$. It was observed by Comes in \cite[Section 2.2]{C} that, quotienting $\Cob$ by the tensor ideal arising from the relation that a handle is the identity and evaluating 2-spheres to $t$,  induces an equivalence with $\Rep(S_t)$. In terms of the Khovanov-Sazdanovic construction, if we specialize
	\begin{align*}
		\alpha=(t, t, \dots),
	\end{align*}
	we get an equivalence
	\begin{align*}
	\text{Cob}_{\alpha}\cong \Rep(S_t),
	\end{align*}
	where $\text{Cob}_{\alpha}$ denotes the pseudo-abelian envelope of $\text{SCob}_{\alpha}$. That is, this construction recovers the Deligne category $\Rep(S_t)$ from $\Cob$, and so the categories $\text{SCob}_{\alpha}$ give generalizations of the Deligne categories. 
	
	\medspace
	
	The aim of this paper is to show an analogous story in the unoriented case. It turns out that, in characteristic zero, unoriented 2-dimensional cobordisms can also be understood in purely algebraic terms. We show a connection to  the category $\Rep(S_t \wr \mathbb Z_2)$, which interpolates the category $\Rep(S_n\wr \mathbb Z_2)$ of finite dimensional representations of the wreath product $S_n \wr \mathbb Z_2$, see \cite{Kn, M}. Our two main results,  Theorems \ref{maintheorem1} and \ref{maintheorem2}, found in Sections \ref{section:theorem I} and \ref{section:theorem II}, respectively, describe this connection explicitly.
	
 Following the construction in the oriented case, see \cite{Kh, KKO},  we define a family of $2$-dimensional unoriented TQFTs parametrized by sequences $\alpha$, $\beta$ and $\gamma$. We still evaluate orientable closed surfaces of genus $g$ to $\alpha_g$. However, the presence of the two extra sequences $\beta$ and $\gamma$ is due to the existence of unorientable surfaces: unorientable surfaces with one or two crosscaps and genus $g$ are evaluated to $\beta_g$ and $\gamma_g$, respectively. We thus start by constructing the category $\VUCob$ of viewable unoriented cobordisms, obtained by linearizing $\UCob$ and evaluating unoriented closed surfaces via $\alpha, \beta$ and $\gamma$. When these sequences satisfy certain conditions, we can quotient $\VUCob$ by the tensor ideal generated by the \emph{handle relation} associated to these sequences, see Section \ref{SUCob}. The resulting category  $\SUCob$ has objects non-negative integers and morphisms given by linear combinations of unoriented cobordisms from $n$ to $m$ circles,  with up to a certain number of handles. Hence the morphism spaces $\Hom_{\SUCob}(n,m)$ are finite dimensional. We denote the pseudo-abelian closure of $\SUCob$ by $\PsUCob$, see Table \ref{table:1}.

	Assume from now on that $\mathbf k$ has characteristic zero. For a spherical tensor category $\mathcal C$, we denote by $\underline{\mathcal C}$ its quotient by negligible morphisms. Our first main result, Theorem \ref{maintheorem1}, considers theories that evaluate unorientable closed surfaces to zero and orientable closed surfaces via a geometric progressions.
	\begin{reptheorem}{maintheorem1}
		Let $\alpha=(\alpha_0, \lambda \alpha_0, \lambda^2 \alpha_0, \dots)$ and $\beta=(0,0,  \dots)=\gamma$ be sequences in $\mathbf k$, for $\alpha_0, \lambda \in \mathbf k^{\times}$
	. We have an equivalence of symmetric tensor categories
	\begin{align*}
		\underline{\operatorname{OCob}_{\alpha}}\simeq  \underline{\Rep(\text{S}_t\wr \mathbb Z_2 )},
	\end{align*}
	where $t=\frac{\lambda \alpha_0}{2}$, and $\operatorname{OCob}_{\alpha}$ is the quotient of $\operatorname{SUCob}_{\alpha, \beta, \gamma}$ by the relation $\theta=0$, where $\theta$ denotes the crosscap cobordism. 
	\end{reptheorem}
	
	\begin{repcorollary}{Corollary I}
	Let $\alpha, \beta$ and $\gamma$ be as above. If $\lambda\alpha_0$ is not a non-negative even integer, then 
	\begin{align*}
		\operatorname{OCob}_{\alpha}\simeq \Rep(S_t\wr \mathbb Z_2).
	\end{align*}
	In particular, $\operatorname{OCob}_{\alpha}$ is semisimple.
\end{repcorollary}

	The proofs of Theorem \ref{maintheorem1} and Corollary \ref{Corollary I} can be found in Section \ref{section:theorem I}.
	
	We state now our second main result, where we consider theories that evaluate all closed surfaces via geometric progressions, and show their connection with a product of Deligne categories. 
	
\begin{reptheorem}{maintheorem2}
		Let $\alpha_0, \beta_0, \gamma_0, \lambda\in \mathbf{k}^{\times}$, and consider  the sequences  $$\alpha=(\alpha_0, \lambda \alpha_0, \lambda^2 \alpha_0, \dots), \ \ \beta=(\beta_0, \lambda \beta_0, \lambda^2 \beta_0, \dots) \ \text{and} \  \gamma=(\gamma_0, \lambda \gamma_0, \lambda^2 \gamma_0, \dots).$$  
		We have an equivalence
		\begin{align*}
		\underline{\operatorname{UCob}_{\alpha,\beta, \gamma}}\simeq  \underline{\Rep(\text{S}_t\wr \mathbb Z_2 )}\boxtimes \underline{\Rep(S_{t_+})}\boxtimes \underline{\Rep(S_{t_-})},
	\end{align*}
		where $t= \frac{\lambda}{2}\alpha_0 - \frac{1}{2}\gamma_0$, $t_+=\frac{\sqrt{\lambda}}{2}\beta_0 + \frac{1}{2}\gamma_0$ and $t_-=-\frac{\sqrt{\lambda}}{2}\beta_0 + \frac{1}{2}\gamma_0$. 
	\end{reptheorem}


	See Section \ref{section:theorem II} for the proof of Theorem \ref{maintheorem2}.
	
	\medspace
	This paper is organized as follows. A brief introduction to unoriented TQFTs, extended Frobenius algebras, and representation categories of wreath products in non-integral rank  is given in Section \ref{section:preliminaries}. A detailed description of the category of unoriented cobordisms in dimension 2  is given in Section \ref{section: unoriented 2dim}. In Section 4 we define and provide a full description of the categories $\VUCob$ and $\SUCob$; the table below contains a brief summary of the notation introduced. A decomposition  of $\PsUCob$ into an exterior product of categories is  obtained in Section \ref{section: direct sums}. Lastly, Sections \ref{section:SOCob} and \ref{section:theorem II} contain the precise statements and proofs of Theorem I and Theorem II, respectively.

	\begin{table}[H]\label{table}
		\centering
		\begin{tabular}{|m{2.5cm}|m {9cm}|}
			\hline
			Notation&  Category  \\
			\hline
			$\UCob$  & Unoriented 2-dimensional cobordisms.  \\
			\hline
			$\VUCob$  & Viewable unoriented cobordisms. Closed components are evaluated via $\alpha, \beta, \gamma.$ \\
			\hline
			$\SUCob $ & Quotient of $\VUCob$
			by the handle relation.    \\
			\hline
			$\PsUCob$ & Pseudo-abelian envelope of  $\SUCob.$\\
			\hline
			$	\SOCob  $& For $\beta=\gamma=(0,\dots)$, quotient of $\SUCob$
			by $\theta=0$.    \\	
			\hline
			$	\PsOCob$  &Pseudo-abelian envelope of  $\SOCob.$   \\
			\hline
		\end{tabular}
		\caption{}
		\label{table:1}
	\end{table}

	\settocdepth{part}
	\section*{Acknowledgements}
	\settocdepth{subsection}
I would like to warmly thank my advisor Victor Ostrik for posing
the problem, for providing insightful advice, and for his numerous comments that made the exposition of this work much clearer.  I also
thank the referee for carefully reading this work and for their many helpful suggestions.

	\section{Preliminaries}\label{section:preliminaries}
	We work over an algebraically closed field \textbf{k}.  We denote by $\text{Vec}$ the symmetric tensor category of finite dimensional vector spaces over \textbf{k}, and by $\Rep(G)$ the category of finite dimensional representations of a finite group $G$ over $\mathbf k$. 
	\subsection{Monoidal categories}
	We start by recalling some basic definitions.
	
	\begin{definition}
		$\diamondsuit$	We call a category $\mathcal C$ \emph{$\mathbf k$-linear} if for each pair of objects $X,Y\in \mathcal C$ the set $\Hom_{\mathcal C}(X,Y)$ has a \textbf{k}-module structure such that composition of morphisms is \textbf{k}-bilinear.
		
		$\diamondsuit$ We say $\mathcal C$ is \emph{additive} if it is \textbf{k}-linear and for every finite sequence of objects $X_1,\dots, X_n $ in $\mathcal C$, there exists their direct sum $X_1 \oplus \dots \oplus X_n$ in $\mathcal C$.
		
		$\diamondsuit$  We say $\mathcal C$ is \emph{Karoubian} if it is \textbf{k}-linear and every idempotent $e=e^2:X \to  X$ in $\mathcal C$ has a kernel, and hence also a cokernel.
		
		$\diamondsuit$  We say $\mathcal C$ is \emph{pseudo-abelian} if it is additive and Karoubian.
		
	\end{definition}
	
	If $\mathcal C$ is Karoubian then for any idempotent $e\in \End_{\mathcal C}(X)$ there exists its image $eX\in \mathcal C$. That is, we have a direct sum decomposition $X\simeq eX \oplus (1-e) X.$
	
	Starting with a \textbf{k}-linear category we can construct new categories by formally adding images of idempotents or direct sum of objects.
	
	\begin{definition}\label{karoubian}
		Let $\mathcal C$ be a \textbf{k}-linear category. 
		
		$\diamondsuit$ We define the \emph{additive envelope} $\mathcal{A(C)}$ as the category with:
		\begin{itemize}
			\item Objects:  Finite formal sums $X_1 \oplus \dots \oplus X_m$ of objects in $\mathcal C$.
			\item Morphisms: For every $X_1\oplus \dots X_m, Y_1\oplus \dots Y_n$ in $\mathcal{A(C)}$, let 
			$$\Hom_{\mathcal{A(C)}}(X_1\oplus \dots X_m, Y_1\oplus \dots \oplus Y_n) := \bigoplus\limits_{i,j} \Hom_{\mathcal C}(X_i, Y_j).$$ 
			Composition of morphisms is given by matrix multiplication.
		\end{itemize}
		
		$\diamondsuit$ We define the \emph{Karoubian envelope} $\mathcal K(\mathcal C)$ of $\mathcal C$ as the category with:
		\begin{itemize}
			\item Objects: Pairs $(X,e),$ where $X$ is an object in $\mathcal C$ and $e$ is an idempotent in $\End_{\mathcal C}(X).$
			\item Morphisms: For every $(X,e), (Y,f)\in \mathcal{K(C)}$, let $$\Hom_{\mathcal{K(C)}}((X,e),(Y,f))= f  \circ \Hom_{\mathcal C}(X,Y) \circ e.$$
			Composition of morphisms is as expected.
		\end{itemize}
		$\diamondsuit$ We define the \emph{pseudo-abelian envelope} of $\mathcal C$ as the category $\mathcal{P(C)}:= \mathcal{K(A(C))}$.
	\end{definition}
	
	We note that the operations above do not commute in general. These constructions take a \textbf{k}-linear category  and return an additive, Karoubian or pseudo-abelian category, respectively.

	\begin{definition}
		$\diamondsuit$ A \emph{monoidal category} is a collection $(\mathcal{C},\otimes,a, \mathbb{1},l,\emph{r})$, where $\mathcal{C}$ is a category, $\otimes : \mathcal{C} \times \mathcal{C} \to \mathcal{C}$ is a  bifunctor called \emph{tensor product}, $\mathbb 1 \in \mathcal{C}$ is called the  \emph{unit object} and  $a_{X,Y,Z}:(X\otimes Y) \otimes Z \to X \otimes (Y \otimes Z),\ r_X:X \otimes\mathbb{1} \to X, \ \text{and} \ l_X:\mathbb{1} \otimes X \to X$ are natural isomorphisms for $X,Y,Z \in \mathcal{C}$, called associativity and unit constraints, respectively,  subject to the so-called pentagon and triangle axioms, see \cite[Section 2.1]{EGNO}.
		
		$\diamondsuit$	A category $\mathcal C$ is a \emph{tensor category} if it is both \textbf{k}-linear and monoidal, and the tensor bifunctor is \textbf{k}-linear. 
	\end{definition}
	
	We remark that our definition of tensor category is not the standard one, as it is usually required for a tensor category to be abelian, see \cite{EGNO}.
	
	\begin{definition}
		$\diamondsuit$	A \emph{braided monoidal category} is a monoidal category $\mathcal C$ endowed with a natural isomorphism $c_{X,Y} : X \otimes  Y \to  Y\otimes X$, $X, Y \in \mathcal C$, called \emph{braiding}, subject to the so-called hexagon axioms, see \cite[Definition 8.1.1]{EGNO}.
		
		$\diamondsuit$	We say a braided monoidal category $\mathcal C$ is \emph{symmetric} if $c_{Y,X}c_{X,Y} = \operatorname{id}_{X\otimes Y},$ for all $X,Y \in \mathcal C$.
	\end{definition}
	
	In this paper, we will always assume that the unit object $\mathbb 1$ satisfies $\End_{\mathcal C}({\mathbb 1})\simeq \textbf{k}$.

	\subsection{Unoriented cobordisms and unoriented TQFTs}
	
	We denote by $\operatorname{Cob}_n$ the category of  $n$-dimensional oriented cobordisms. This is a symmetric monoidal category, with monoidal structure given by disjoint union, see for example \cite{Ko}.

	\begin{definition}\label{def of UCobn}
		Let $\text{UCob}_n$ be  the category of $\emph{n-dimensional unoriented cobordisms}$ defined as follows:
		
		$\diamondsuit$ Objects: closed $(n-1)$-dimensional smooth manifolds.
		
		$\diamondsuit$ Morphisms: given two objects $X, Y$, we define $\Hom_{\text{UCob}_n}(X,Y)$ to be the space consisting of $n$-dimensional smooth compact manifolds with boundary $M$, together with a diffeomorphism of their boundary $\partial M\simeq  X \sqcup Y$. Such two manifolds are considered to be the same if they differ by a diffeomorphism which induces  a diffeomorphism between their boundaries that is isotopic to the identity.
		
		$\diamondsuit$ Composition: let $M: X\to Y$ and $N:Y \to Z$ be morphisms in $\text{UCob}_n$. Since $\partial M\simeq X \sqcup Y$ and $\partial N\simeq  Y \sqcup Z$, we have  induced diffeomorphisms $M_0 \simeq Y$ and $N_0 \simeq Y$ for some $M_0\subseteq \partial M$ and $N_0 \subseteq \partial N$. The composition of $M$ and $N$ is  given by glueing $M$ with $N$ via these diffeomorphisms, see \cite{MSS}.
	\end{definition}

	The category $\text{UCob}_n$ is a rigid symmetric monoidal category, with monoidal structure induced by the disjoint union.
	
	\begin{example}\label{example}
		Let $X$ be a closed $(n-1)$-manifold. Then $\Hom_{\text{UCob}_n}(X,X)$ contains cylinder $n$-manifolds $M=X\times [0,1]$.
		In the special case where $n=2$ and $X=S^1$ we have exactly two classes of diffeomorphisms up to isotopy: one which is orientation preserving and another which is orientation reversing. Thus we have exactly two diffeomorphism classes in $\Hom_{\UCob}(S^1, S^1)$,  pictured below:
		
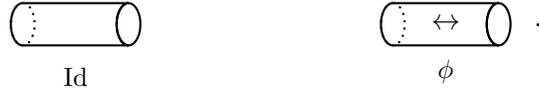
\begin{align*}
			\begin{tikzpicture}[tqft/cobordism/.style={draw,thick},
				tqft/view from=outgoing, tqft/boundary separation=30pt,
				tqft/cobordism height=40pt, tqft/circle x radius=8pt,
				tqft/circle y radius=4.5pt, tqft/every boundary component/.style={rotate=90}
				]
				\pic[tqft/cylinder,rotate=90,name=a,anchor={(1,0)}, every incoming
				boundary component/.style={draw,dotted,thick},every outgoing
				boundary component/.style={draw,thick}
				];
				\node at ([yshift=-20pt,xshift=-20pt]a-outgoing boundary 1){\small{Id}};
				\pic[tqft/cylinder,rotate=90,name=k,anchor={(1,-3.5)},every incoming
				boundary component/.style={draw,dotted,thick},every outgoing
				boundary component/.style={draw,thick}];
				\node at ([yshift=-18pt,xshift=-20pt]k-outgoing boundary 1){\small{$\phi$}};
				\node at ([xshift=-20pt]k-outgoing boundary 1){$\leftrightarrow$};
				\node at ([xshift=15pt]k-outgoing boundary 1){.};
			\end{tikzpicture} 
		\end{align*}
		From a slightly different point of view, we can consider the cylinders above as morphisms from $0\to 2$ and $2 \to 0$,  instead of $1\to 1$. This yields two different compositions of cylinders, which result in the Klein bottle and the Torus, respectively:
		
		
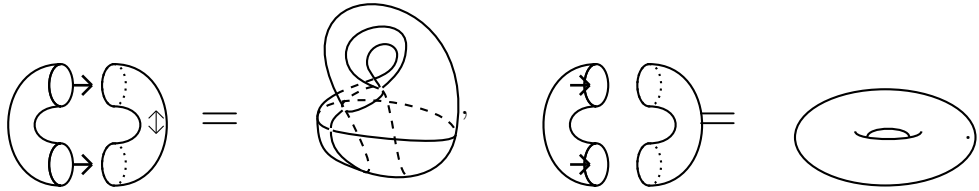
\begin{tikzpicture}[tqft/cobordism/.style={draw,thick},
			tqft/view from=outgoing, tqft/boundary separation=30pt,
			tqft/cobordism height=40pt, tqft/circle x radius=8pt,
			tqft/circle y radius=4.5pt, tqft/every boundary component/.style={draw,rotate=90},tqft/every incoming
			boundary component/.style={draw,dotted,thick},tqft/every outgoing
			boundary component/.style={draw,thick},scale=0.8]
			\pic[tqft,
			incoming boundary components=0,
			outgoing boundary components=2,
			rotate=90,name=j,anchor={(1,0)}];
			\node at ([xshift=10pt]j-outgoing boundary 1) {$\Arrow$};
			\node at ([xshift=10pt]j-outgoing boundary 2) {$\Arrow$};
			\pic[tqft,
			incoming boundary components=2,
			outgoing boundary components=0,
			rotate=90,name=m,anchor={(1,-0.5)},at=(j-outgoing boundary 1)];
			\node at ([xshift=50pt, yshift=20pt]m-incoming boundary 1) [font=\huge] {\(=\)};
			\node at ([xshift=20pt, yshift=20pt]m-incoming boundary 1) {$\figureud$};
		\end{tikzpicture} 
		\begin{tikzpicture}[scale=0.2]
			\hspace*{-0.5cm} 
			\node[inner sep=0mm] (P1) at (4.9,6.7) {};
			\node[inner sep=0mm] (P2) at (2.5,5.4) {};
			\node[inner sep=0mm] (P3) at (1.6,4) {};
			\Coordinate{e5b}{3.6,1.3}
			\Coordinate{e4l}{0.7,4.8}
			\Coordinate{e4r}{9.9,3.7}
			\Coordinate{si}{4.5,7.4}
			\Coordinate{bottom}{6.8,0.9}
			{[thick,black]
				\draw (e4l) to[out=270,in=160,looseness=1] (P3);
				\draw (P3) to[out=340,in=270,looseness=0.3]  (e4r);
				\draw[name path=P2e4r] (P2) to[out=120,in=80,looseness=3.7] (e4r);
				\draw[name path=P1P1] (P1) to[out=160,in=270,looseness=1] (2.6,9) to[out=90,in=90,looseness=1.3]  (6.6,9.6) to[out=270,in=40,looseness=1]  (P1) ;
				\draw (P2) to[out=315,in=315,looseness=0.5](P1);
				\draw[dashed] (P2) to[out=135,in=135,looseness=0.5] (P1);
				\draw (P2) to[out=220,in=90,looseness=1] (P3);
				\draw (P3) to[out=270,in=150,looseness=1] (e5b);
				\draw (e4l) to[out=270,in=160,looseness=1.3] (e5b);
				\draw (e5b) to[out=340,in=260,looseness=1.1] (e4r);
				\draw[dashed] (P2) to[out=300,in=330,looseness=1] (e5b);
				\draw[dashed] (e4l) to[out=90,in=90,looseness=0.6] (e4r);
				\draw[dashed] (P1) to[out=320,in=190,looseness=0.4] (bottom);
				\draw (P1) to[out=110,in=300,looseness=1] (si);
				\draw (si) to[out=120,in=270,looseness=1] (4,8.5) to[out=90,in=180,looseness=1] (5.2,9.7) to[out=0,in=90,looseness=1] (6,9) to[out=270,in=20,looseness=1] (si);
				\path[name path=e4lsi] (e4l) to[out=90,in=200,looseness=0.8] (si);
				\draw[name intersections={of=e4lsi and P2e4r}] (e4l) to[out=90,in=210,looseness=1] (intersection-1) coordinate (h1);
				\draw[dashed] (intersection-1) to[out=30,in=200,looseness=0.6] (si);
			}
			\node at (13,5) {,};
		\end{tikzpicture}
		\begin{tikzpicture}[tqft/cobordism/.style={draw,thick},
			tqft/view from=outgoing, tqft/boundary separation=30pt,
			tqft/cobordism height=40pt, tqft/circle x radius=8pt,
			tqft/circle y radius=4.5pt, tqft/every boundary component/.style={draw,rotate=90},tqft/every incoming
			boundary component/.style={draw,dotted,thick},tqft/every outgoing
			boundary component/.style={draw,thick},scale=0.8]
			\hspace*{-0.5cm}
			\pic[tqft,
			incoming boundary components=0,
			outgoing boundary components=2,
			rotate=90,name=j,anchor={(1,0)}];
			\node at ([xshift=10pt]j-outgoing boundary 1) {$\Arrow$};
			\node at ([xshift=10pt]j-outgoing boundary 2) {$\Arrow$};
			\pic[tqft,
			incoming boundary components=2,
			outgoing boundary components=0,
			rotate=90,name=m,anchor={(1,-0.5)},at=(j-outgoing boundary 1)];
			\node at ([xshift=50pt, yshift=20pt]m-incoming boundary 1) [font=\huge] {\(=\)};
		\end{tikzpicture}
		\begin{tikzpicture}
			\begin{scope}[scale=0.8,shift={(0,10)}]		
				\hspace*{-0.5cm}
				\draw[fill=white,style=thick] (-17, 1.5) ellipse (1.5cm and 0.8cm);
				\draw[style=thick] (-17.3,1.5) to[out=90, looseness = 0.7,in=90] (-16.6,1.5); 
				\draw[style=thick] (-17.5,1.6) to[out=270, looseness = 0.4,in=270] (-16.4,1.6);
				\node at (-15,1.5) {.};
			\end{scope}  
		\end{tikzpicture}
	\end{example}
	
	\begin{definition}
		An n-dimensional $\mathcal C$-valued \emph{unoriented topological quantum field theory (TQFT)}  is a symmetric monoidal functor $\text{UCob}_n \to \mathcal C$ for some symmetric monoidal category $\mathcal C$.
	\end{definition}
	
	There is a similar definition of $\mathcal C$-valued \emph{oriented} TQFTs $\operatorname{Cob}_n\to \mathcal C$, and it is well known they are in bijection with Frobenius algebras in $\mathcal C$, see for example \cite{Ko}. We discuss in the following section the relation between unoriented TQFTs and extended Frobenius algebras. 
	
	\subsection{Extended Frobenius algebras}
	In this section we recall some algebraic structures in monoidal categories, with the goal of talking about extended Frobenius algebras and their relation to unoriented TQFTs.
	
	\begin{definition}Let $\mathcal C$ be a monoidal category. 
		\begin{itemize}
			\item An \emph{algebra} in $\mathcal C$ is a triple $(A, m_A, u_A),$ consisting of an object $A \in \mathcal C$, and morphisms $m_A : A \otimes  A \to A$ and $u_A :\mathbb  1 \to A$ in $\mathcal C$, called \emph{multiplication} and \emph{unit}, satisfying the following commutative diagrams:
			\begin{equation}
				\begin{tikzcd}  [row sep=large,column sep=huge]
					A\otimes A \otimes A \arrow{r}{m_A\otimes \Id_A } \arrow{d}{ \Id_A\otimes m_A} &  A\otimes A\arrow{d}{m_A} \\
					A\otimes A \arrow{r}{m_A} &A 
				\end{tikzcd}
				,
				\begin{tikzcd}  [row sep=large,column sep=huge]
					\mathbb 1\otimes A \arrow{rd}{\sim}\arrow{r}{u_A\otimes \Id_A} \arrow{d}{ \Id_A\otimes u_A} &  A\otimes A\arrow{d}{m_A}\\
					A\otimes A \arrow{r}{m_A} &  A.
				\end{tikzcd}
			\end{equation}
			A \emph{ morphism of algebras} $(A, m_A, u_A)$ and $(B, m_B, u_B)$ is a morphism $f : A \to B$ in $\mathcal C$ so that	$fm_A = m_B(f \otimes f)$ and $fu_A = u_B$. 
			
			\item A \emph{coalgebra} in $\mathcal C$ is a triple $(A, \Delta_A, \varepsilon_A),$ consisting of an object $A \in \mathcal C$, and morphisms $\Delta_A : A \to  A\otimes A$ and $\varepsilon_A  : A \to \mathbb 1$ in $\mathcal C$, called \emph{comultiplication} and \emph{counit},   satisfying the following commutative diagrams:
			\begin{equation}\label{frob algebra diagram}
				\begin{tikzcd}  [row sep=large,column sep=huge]
					A\arrow{r}{\Delta_A } \arrow{d}{\Delta_A} &  A\otimes A \arrow{d}{ \Delta_A\otimes \Id_A}\\
					A\otimes A \arrow{r}{\Id_A\otimes \Delta_A} &A \otimes A \otimes A
				\end{tikzcd}
				,
				\begin{tikzcd}  [row sep=large,column sep=huge]
					A \arrow{rd}{\sim}\arrow{r}{\Delta_A} \arrow{d}{ \Delta_A} &  A\otimes A\arrow{d}{\varepsilon_A\otimes \Id_A}\\
					A\otimes A \arrow{r}{\Id_A \otimes \Delta_A} &  \mathbb 1\otimes A.
				\end{tikzcd}
			\end{equation}
			A \emph{morphism of coalgebras} $(A, \Delta_A, \varepsilon_A)$ and $(B, \Delta_B, \varepsilon_B)$ is a morphism $f : A \to B$ in $\mathcal C$ so that $\Delta_Bf = (f \otimes  f)\Delta_A$ and $\varepsilon_Bf = \varepsilon_A$.
			
			\item A \emph{Frobenius algebra} in $\mathcal C$ is a 5-tuple  $(A, m_A, u_A, \Delta_A, \varepsilon_A)$ where $(A, m_A, u_A)$ is an algebra in $\mathcal C$ and
			$(A, \Delta_A, \varepsilon_A)$ is a coalgebra in $\mathcal C$, so that the following diagram commutes
			\begin{equation}
				\begin{tikzcd}  [row sep=large,column sep=huge]
					A\otimes A\arrow{r}{\Id_A\otimes \Delta_A } \arrow{d}{\Delta_A\otimes \Id_A}\arrow{dr}{\Delta_A m_A} &  A\otimes A\otimes A \arrow{d}{ m_A\otimes \Id_A}\\
					A\otimes A\otimes A \arrow{r}{\Id_A\otimes m_A} &A \otimes A.
				\end{tikzcd}
			\end{equation}
			
			A \emph{morphism of Frobenius algebras} $f : A \to B$ is a both an algebra and coalgebra map. 
		\end{itemize}
	\end{definition}
	\begin{definition}
		Let $\mathcal C$ be a braided monoidal category with braiding $c_{X,Y}:X\otimes Y \xrightarrow{\sim}Y\otimes X$ for all $X,Y\in \mathcal C$.
		\begin{itemize}
			\item  A \emph{commutative algebra} in $\mathcal C$ is an algebra $(A, m_A , u_A)$ in $\mathcal C$ such that $m_A \ c_{A,A} = m_A$. 
			\item A \emph{commutative Frobenius algebra} in $\mathcal C$ is a Frobenius algebra that is commutative as an algebra. 
		\end{itemize}
		
	\end{definition}

	It is well known that commutative Frobenius algebras in a symmetric tensor category $\mathcal C$ are in  bijection with 2-dimensional oriented TQFTs $\text{Cob}_2 \to \mathcal C$ (see e.g. \cite[Theorem 0.1]{SP}). 
	An analogous correspondence takes place between extended Frobenius algebras and unoriented TQFTs. We give a precise statement in what follows.
	
	\begin{definition}\cite{TT}
		Let $\mathcal C$ be a braided monoidal category. An \emph{extended Frobenius algebra} in $\mathcal C$ is a 7-tuple $(A, m_A, u_A, \Delta_A, \varepsilon_A, \phi_A, \theta_A),$ where $(A, m_A, u_A, \Delta_A, \varepsilon_A)$ is a commutative Frobenius algebra in $\mathcal C,$  $\phi_A:A \to A$ and $\theta_A: \mathbb{1} \to A$ are morphisms of Frobenius algebras, $\phi_A^2=\text{id}_A,$ and the diagrams
		\begin{equation}\label{extended Frob alg} 
			\begin{tikzcd}  
				\mathbb{1}\otimes A \arrow{rr}{\theta_A\otimes \text{id}_A  } \arrow{d}{ \theta_A\otimes 1} &&  A\otimes A\arrow{d}{m_A} \\
				A\otimes A \arrow{r}{m_A} &A \arrow{r}{\phi_A} & A
			\end{tikzcd}
			,
			\begin{tikzcd}  
			\mathbb 	1 \arrow{r}{u_A} \arrow{d}{ \sim} &  A\arrow{r}{\Delta_A}&A\otimes A \arrow{r}{\phi_A \otimes \text{id}_A}& A\otimes A\arrow{d}{m_A}\\
			\mathbb 	1\otimes  \mathbb 1 \arrow{r}{\theta_A \otimes \theta_A} &  A\otimes A\arrow{rr}{m_A}&&A
			\end{tikzcd}
		\end{equation}
		commute. 
	\end{definition}

\begin{example} \label{algebra of functions}
	Let $n\geq 1$ and let $S_n\wr \mathbb Z_2$ denote the wreath product of $S_n$ and $\mathbb Z_2,$ where $S_n$ acts on $\mathbb Z_2^n$ by permutation of indices. 
	 Let $\mathcal C=\Rep(S_n\wr \mathbb Z_2)$. For $\lambda \in \mathbf k^{\times}$, we define $A=A_{\lambda}\in \mathcal C$ to be the algebra of functions on the set $\{x_1, x_{-1},\dots, x_n, x_{-n}\}$ over \textbf{k}, with unit and multiplication maps given by 
	\begin{align*}
		u_A : \mathbf k &\to A,& 	 m_A:A\otimes A &\to A,\\
		1&\mapsto \sum\limits_{i=1}^n (\delta_i+\delta_{-i}) &	f\otimes g &\mapsto fg
	\end{align*} 
	respectively, where $\delta_i(x_j)=\delta_{i,j}$ and $\delta_{-i}(x_{-j})=\delta_{i,j}$ for all $1\leq i,j\leq n$. 
		Note that $A$ is in fact in $\mathcal C$, since $S_n\wr \mathbb Z_2$ acts on $\{x_1, x_{-1}, \dots, x_{n}, x_{-n}\}$ by permuting indices, and thus also acts on its algebra of functions $A$.

	We show that $(A, m_A, u_A, \Delta_A, \varepsilon_A, \phi_A, \theta_A)$ is an extended Frobenius algebra in $\mathcal C$, where
	\begin{align*}
		\varepsilon_A: A &\to \mathbf k, &\Delta_A:A&\to A\otimes A, 	&\phi_A :A &\to A, &\theta_A : \mathbf k &\to A\\
		f&\mapsto \frac{1}{\lambda}\sum\limits_{i=1}^n(f(x_i)+f(x_{-i}))&\delta_i &\mapsto  \lambda\delta_i\otimes \delta_i, 	&\delta_{i}&\mapsto \delta_{-i}, & 1 &\mapsto 0\\
		&&\delta_{-i} &\mapsto  \lambda\delta_{-i}\otimes \delta_{-i}, &\delta_{-i} &\mapsto \delta_{i},
	\end{align*} 
	for all $1\leq i \leq n$ and for all $f,g : \{x_1, x_{-1}, \dots, x_{n},x_{-n}\} \to \textbf{k}$.

	It is easy to check that  $(A, u_A, m_A)$ is an algebra and $(A, \varepsilon_A, \Delta_A)$ is a coalgebra in $\mathcal C$. We compute
	\begin{align*}
		&(m_A\otimes \Id_A)(\Id_A\otimes \Delta_A)(\delta_i\otimes \delta_i)=\lambda (m_A\otimes \Id_A)(\delta_i\otimes \delta_i \otimes \delta_i)= \lambda \delta_i\otimes \delta_i = \Delta_A m_A(\delta_i \otimes \delta_i), \text{ and}\\
		&(m_A\otimes \Id_A)(\Id_A\otimes \Delta_A)(\delta_i\otimes \delta_j)=\lambda (m_A\otimes \Id_A)(\delta_i\otimes \delta_j \otimes \delta_j)= 0 = \Delta_A m_A(\delta_i \otimes \delta_i), \
	\end{align*}
	for  $ i \ne j \in\{1,-1, \dots, n,-n\}$.  Thus $(A, u_A, m_A,\varepsilon_A, \Delta_A)$ is a  Frobenius algebra in $\mathcal C$, which is clearly commutative. 
	On the other hand, $\phi(\theta_A) =\theta_A$ trivially, and 
	\begin{align*}
		m_A(\phi_A \otimes \Id)(\Delta_A(1)) &=  \lambda m_A(\phi_A \otimes \Id)\left(\sum\limits_{i=1}^n (\delta_i\otimes \delta_i)+(\delta_{-i}\otimes\delta_{-i})\right) \\
		&	= \lambda m_A\left(\sum\limits_{i=1}^n (\delta_{-i}\otimes \delta_i)+(\delta_{i}\otimes\delta_{-i})\right) \\&=0, \text{ and}\\
		m_A(\theta_A \otimes \theta_A)&=m_A(0)=0.
	\end{align*}
	Hence all the conditions for an extended Frobenius algebra are satisfied, see Definition \ref{extended Frob alg}. 
\end{example}
	
	\begin{proposition}\label{frob algebras and TQFTs}Let $\mathcal C$ be a symmetric monoidal category.
		Isomorphism classes of unoriented 2–dimensional $\mathcal C$-valued TQFTs are in bijective correspondence with isomorphism classes of extended Frobenius algebras in $\mathcal C$.
	\end{proposition}
	
	A proof for the Proposition above is given in \cite[Proposition 2.9]{TT}, for the case where $\mathcal C$ is the category of modules over a commutative ring $R$. Their proof works verbatim for the general case.

	Explicitly, an extended Frobenius algebra $A$ in a symmetric tensor category $\mathcal C$ gives rise to a symmetric monoidal functor $F_A: \UCob\to \mathcal C$ by mapping the circle object in $\UCob$ to $A$ in $\mathcal C$.

	\subsection{The category $\Rep(S_t \wr G)$} \label{section: StC}
	For this section, we assume $\textbf k$ has characteristic zero. Fix a finite group $G$ and let $\textbf k[G]$ be the regular representation of $G$. 
	
	The symmetric tensor category $\Rep(S_t \wr G),$ introduced by Knop in \cite{Kn}, interpolates the category $\Rep(S_n\wr G)$ of representations of the wreath product $ S_n\wr G$ from $n\in \mathbb N$ to $t\in \mathbf k$. In \cite{M}, Mori introduced a 2-functor  which sends a tensor category $\mathcal C$ to a new tensor category $\StC$. When $\mathcal C=\Rep(G)$, we have that $S_t(\mathcal C)=\Rep(S_t\wr G)$ as defined by Knop, see \cite[Remark 4.14]{M}.
	
	
	We will give a brief summary of the construction and graphical description of $S_t(\mathcal C)$ as presented in \cite{M}. We will use this description later on to prove Theorems I and II. 
	\begin{definition}
		Let $I_1, \dots, I_l$ be finite sets. A \emph{recollement} of $ I_1, \dots , I_l$ is an equivalence relation $r$ on  $I_1\sqcup \dots· \sqcup I_l$ such that for any $k = 1, \dots , l $ and $i, i'  \in I_k$, if $i \sim_r i'$ then $i = i'$.
	\end{definition}
	We denote by $R(I_1, \dots , I_l)$ the set of recollements of $ I_1, \dots , I_l$. For example, any element in $R(I, J)$ is  of the form
	$$r = \{\{i, j\}, \dots , \{i'\}, \dots , \{j'\}, \dots \}$$
	where $i, i'\in I$ and $j, j' \in J$.  
	
	For $\{a_1,\dots, a_p\} \subseteq \{1, \dots , l\}$, let $\pi_{a_1\dots,a_p}: R(I_1, \dots , I_l) \to  R(I_{a_1}, \dots , I_{a_p})$ be the map that restricts the equivalence relation $R(I_1, \dots, I_l)$ to  $I_{a_1}\sqcup \dots \sqcup I_{a_p} \subset I_1 \sqcup \dots \sqcup I_l$. Given finite sets $I,J,K$ and $r\in R(I,J), s\in R(J,K)$, let
	\begin{align*}
		R(s \circ r) = \{u \in  R(I, J, K) \  |\  \pi_{1,2}(u) = r, \pi_{2,3}(u) = s\}.
	\end{align*}

	\begin{definition} \cite[Definition 2.13]{M}
		Let $\mathcal C$ be a \textbf{k}-linear category, and let $t\in \textbf k$. Then $\StC$ is the pseudo-abelian envelope of the category defined as follows:
		
		$\diamondsuit$ \textbf{Objects:} Finite families $U_I=(U_i)_{i\in I}$ of objects in $\mathcal C$. We denote them by $\langle U_I\rangle_t.$
		
		$\diamondsuit$ \textbf{Morphisms:} For objects $\langle U_I \rangle_t$ and $\langle V_J\rangle_t$,
		\begin{align}\label{morphisms in StC}
			\Hom_{\StC}(\langle U_I \rangle_t, \langle V_J \rangle_t) \simeq \bigoplus_{r\in R(I,J)} \left(    \bigotimes_{(i,j)\in r} \Hom_{\mathcal C}(U_i, V_j)\right),
		\end{align}
		where for each $\Phi$ on the right hand side, we denote by $\langle \Phi \rangle_t$ the corresponding morphism in $\StC$. 
		
		$\diamondsuit$  \textbf{Composition:} For $\Phi \in \bigotimes\limits_{(i,j)\in r} \Hom_{\mathcal C}(U_i, V_j)$ and $\Psi \in \bigotimes\limits_{(j,k)\in s} \Hom_{\mathcal C}(V_j, W_k)$, composition is given by
		\begin{align*}
			\langle \Psi \rangle_t \circ \langle \Phi \rangle_t  := \sum_{u\in R(s\circ r)} P_u(t) \langle \Psi \circ_u \Phi \rangle_t,
		\end{align*}
		where  we denote by $ \Psi \circ_u \Phi$
		the element obtained by composing terms of $\Phi \otimes \Psi$ using 
		$$\Hom_{\mathcal C}(U_i, V_j) \otimes \Hom_{\mathcal C}(V_j, W_k) \to \Hom_{\mathcal C}(U_i, W_k),$$
		for all $(i, j, k) \in u$, and $P_u(t)$ is the polynomial
		$$P_u(t) =\prod_{\# \pi_{1,3}(u)\leq a < \# u} (t -a).$$
	\end{definition}
	The unit object $ \mathbb 1_{\StC}$ of $\StC$ is the object corresponding to the empty family. 
	
	\begin{remark}
		In this language, Deligne’s category $\Rep(S_t, \mathbb C)$ as defined in \cite{D} is equivalent to $S_t(\text{Vec})$, see \cite[Remark 4.14]{M}.
	\end{remark}
	
	Having a monoidal structure in $\mathcal C$ induces a monoidal structure in $S_t(\mathcal C)$, with tensor products defined as follows.
	
	\begin{definition}\cite[Definition 4.16]{M}
		Let $\mathcal C$ be a tensor category. Define tensor products in $S_t(\mathcal C)$ by:
		\begin{itemize}
			\item For families $U_I=(U_i)_{i\in I}$ and $V_J=(V_j)_{j\in J}$ in $\mathcal C$, 
			\begin{align*}
				\langle U_I\rangle_t\otimes  \langle V_J\rangle_t :=\bigoplus_{r\in R(I,J)} \langle  (U_i\otimes V_j)_{(i,j)\in r} \rangle_t.
			\end{align*}
			\item For families $U_I=(U_i)_{i\in I}, V_J=(V_j)_{j\in J}, W_K=(W_k)_{k\in K}$ and $X_L=(X_l)_{l\in L}$, and morphisms $\Phi \in \bigotimes\limits_{(i,j)\in r} \Hom_{\mathcal C}(U_i, V_j)$ and $\Psi \in \bigotimes\limits_{(k,l)\in s} \Hom_{\mathcal C}(W_k, X_l)$,
			\begin{align*}
				\langle \Phi\rangle_t \otimes \langle \Psi \rangle_t :=\sum\limits_{u\in R(r\otimes s)} \langle \Phi \otimes_u \Psi\rangle_t,
			\end{align*}
			where $\Phi\otimes_u\Psi$ is obtained by composing terms of $\Phi \otimes \Psi$ using tensor products 
			\begin{align*}
				\Hom_{\mathcal C}(U_i, V_j)\otimes \Hom_{\mathcal C}(W_k, X_l)\to \Hom_{\mathcal C}(U_i\otimes W_k, V_j\otimes X_l),
			\end{align*}
			for all $(i,k,j,l)\in u$.
		\end{itemize}
	\end{definition}

	\subsubsection{Graphical description of $S_t(\mathcal C)$}
	When $\mathcal C$ is a braided tensor category, it induces a natural braiding in $\mathcal S_t(\mathcal C)$, see \cite[Section 4.5]{M}. In particular, when $\mathcal C$ is symmetric then so is $S_t(\mathcal C)$. 
	
	For $\mathcal C$ a braided tensor category, \cite[Section 4.5]{M} shows  a useful graphical description of  morphisms in $S_t(\mathcal C)$. We will use it later on for $\mathcal C=\Rep(G)$, as it is easier to work with than the definitions above. We give now a quick summary of this graphical description. We note however that we make a minor change, as we represent morphisms from left to right, rather than from top to bottom.

	$\diamondsuit$  Represent objects $\langle U_1\rangle_t  \otimes  \dots \otimes \langle U_l\rangle_t $  by labeled points placed vertically:
	\begin{figure}[H]
		\begin{center}
			\begin{tikzpicture}[scale=0.6]
				\begin{scope}[scale=0.8,shift={(0,0)}]
					\node at (-1,0) {$U_1$};
					\node at (-1,1) {$U_2$};
					\node at (0,2) {$\vdots$};
					\node at (-1,3) {$U_l$};
					\draw[fill=black,style=thick] (0,0) circle (0.05cm and 0.05cm);
					\draw[fill=black,style=thick] (0,1) circle (0.05cm and 0.05cm);
					\draw[fill=black,style=thick] (0,3) circle (0.05cm and 0.05cm);
				\end{scope}   
			\end{tikzpicture}
		\end{center}
	\end{figure} 
	
	The unit object $\mathbb 1_{\StC}$ is represented by  \textquotedblleft no points". Morphisms between objects of this form are represented by strings which connect points from left to right. Since objects of the form $\langle U_1\rangle_t  \otimes  \dots \otimes \langle U_l\rangle_t $ generate $S_t(\mathcal C)$ as a pseudo-abelian category, it is enough to consider morphisms between them \cite[Corollary 4.18]{M}.

	$\diamondsuit$ For each morphism $\Phi: U \to V$ in $\mathcal C$, represent the corresponding morphism in $\StC$ by a string with a label:
	\begin{align*}
		\begin{aligned}
			\begin{tikzpicture}[block/.style={draw, rectangle, minimum height=0.5cm}]
				\node (box) [block] {$\Phi$}
				node (in1)     [left=-.1cm and 1.1cm of box]     {$U$}
				node (out1)    [right=-.1cm and 1.1cm of box]    {$V$}
				node at (-3,0) {$\langle \Phi \rangle_t:=$}
				node at (2,0) {.}
				{(in1) edge[-] (box.west |- in1)
					(box.east |- out1) edge[-] (out1)};
				\draw[fill=black,style=thick] (-1.3,0) circle (0.05cm and 0.05cm);
				\draw[fill=black,style=thick] (1.3,0) circle (0.05cm and 0.05cm);
			\end{tikzpicture}
		\end{aligned}
	\end{align*}
	When $\Phi=\text{id}$, we will sometimes omit the label.

	$\diamondsuit$ Take morphisms $u_{\mathcal C}\in \Hom_{S_t(\mathcal C)}( \mathbb 1_{\StC}, \langle \mathbb 1_{\mathcal C}\rangle_t)$ and $ \varepsilon_{\mathcal C}\in \Hom_{S_t(\mathcal C)}(\langle \mathbb 1_{\mathcal C}\rangle_t, \mathbb 1_{\StC} )$  which correspond to $\Id_{\mathbb 1_{\mathcal C}}$ via their respective isomorphisms to $\Hom_{\mathcal C}(\mathbb 1_{\mathcal C}, \mathbb 1_{\mathcal C})$, and represent them by broken strings:
	\begin{figure}[H]
		\begin{center}
			\begin{tikzpicture}[block/.style={draw, rectangle, minimum height=0.5cm}]
				\node at (-2,0)  {$u_{\mathcal C} :=$};
				\node   at  (-1.3,0)    {$\figureXv$};
				\node   at  (1.5,0)    {$ \mathbb{1}_{\mathcal C}$,};
				\draw[fill=black,style=thick] (1.1,0) circle (0.05cm and 0.05cm);
				\draw[style=thick]  (-1.3,0) to (1.1,0);
			\end{tikzpicture}
			\begin{tikzpicture}[block/.style={draw, rectangle, minimum height=0.5cm}]
				\node at (-2,0)  {$\varepsilon_{\mathcal C} :=$};
				\node   at  (1.5,0)    {$\figureXv$};
				\node   at  (1.7,0)    {.};
				\node   at  (-1.2,0)    {$\mathbb{1}_{\mathcal C}$};
				\draw[fill=black,style=thick] (-0.8,0) circle (0.05cm and 0.05cm);
				\draw[style=thick]  (-0.8,0) to (1.5,0);
			\end{tikzpicture}
		\end{center}
	\end{figure}
	
	$\diamondsuit$ Since $\langle U \otimes V\rangle_t$ is a direct summand of $\langle U\rangle_t \otimes \langle V\rangle_t$ (see \cite[Corollary 4.18]{M}), we have retraction $\mu_{\mathcal C}(U, V): \langle U\rangle_t \otimes \langle V \rangle_t \to \langle U \otimes V\rangle_t$ and section $\Delta_{\mathcal C}(U, V): \langle U \otimes V\rangle_t \to \langle U\rangle_t \otimes \langle V\rangle_t$ morphisms, which we represent by ramifications of strings:
	\begin{figure}[H]
		\begin{center}\resizebox{170pt}{!}{%
				\begin{tikzpicture}[block/.style={draw, rectangle, minimum height=0.5cm}]
					\node at (-3,0.5)  {$\mu_{\mathcal C}(U,V) :=$};
					\node   at  (-1.7,0)    {$U$};
					\node   at  (-1.7,1)    {$V$};
					\node   at  (1.9,0.5)    {$U\otimes V$,};
					\draw[fill=black,style=thick] (-1.4,0) circle (0.05cm and 0.05cm);
					\draw[fill=black,style=thick] (-1.4,1) circle (0.05cm and 0.05cm);
					\draw[fill=black,style=thick] (1.1,0.5) circle (0.05cm and 0.05cm);
					\draw[style=thick]  (-0.1,0.5) to   (1.1,0.5);
					\draw[style=thick]  (-1.4,1) to [out=0, looseness= 0.8,in=90] (-0.1,0.5);
					\draw[style=thick]  (-1.4,0) to [out=0, looseness =0.8,in=-90](-0.1,0.5);
			\end{tikzpicture}}\resizebox{170pt}{!}{%
				\begin{tikzpicture}[block/.style={draw, rectangle, minimum height=0.5cm}]
					\node at (-4,0.5)  {$\Delta_{\mathcal C}(U,V) :=$};
					\node   at  (1.6,0)    {$U$};
					\node   at  (1.6,1)    {$V$};
					\node   at  (-2,0.5)    {$U\otimes V$};
					\node   at  (1.8,0.5)    {,};
					\draw[fill=black,style=thick] (-1.4,0.5) circle (0.05cm and 0.05cm);
					\draw[fill=black,style=thick] (1.2,1) circle (0.05cm and 0.05cm);
					\draw[fill=black,style=thick] (1.2,0) circle (0.05cm and 0.05cm);
					\draw[style=thick]  (-1.4,0.5) to   (-0.2,0.5);
					\draw[style=thick]  (-0.2,0.5) to [out=90, looseness= 1,in=0] (1.1,1);
					\draw[style=thick]  (-0.2,0.5) to [out=-90, looseness =1,in=0](1.1,0);
			\end{tikzpicture}}
		\end{center}
	\end{figure}
	
	for all $U,V\in \mathcal C$.
	
The	tensor product of these maps is represented by stackings of diagrams, and composition by connecting them from left to right. 
	
	\begin{proposition}\cite[Section 4.6]{M} \label{generators of StC}
		The morphisms $\langle \Phi \rangle_t$, $\mu_{\mathcal C}(U, V)$, $u_{\mathcal C}$, $\Delta_{\mathcal C}(U, V)$ and $\varepsilon_{\mathcal C}$, for $U,V \in \mathcal C$ and arbitrary $\Phi:U\to V,$ generate the category $\StC$ as a pseudo-abelian braided tensor category. 
	\end{proposition}
	
	It is shown in \cite[Section 4.6]{M} that $S_t(\mathcal C)$ is defined by generators and relations, with the generators given in the proposition above. We include below the relations, which we will use throughout this work.

	\underline{Relations:}
	
	$\diamondsuit$ $\langle \cdot \rangle_t : \mathcal C \to S_t(\mathcal C)$ is \textbf{k}-linear and compatible with the composition in $\mathcal C$:
	\begin{align*}
		\begin{aligned}\resizebox{100pt}{!}{%
				\begin{tikzpicture}[block/.style={draw, rectangle, minimum height=0.5cm}]
					\draw (-2,0) to (2,0);
					\draw[fill=black,style=thick] (-2,0) circle (0.05cm and 0.05cm);
					\draw[fill=black,style=thick] (2,0) circle (0.05cm and 0.05cm);
					\node (box) [block, fill=white] {$a\Phi+b\Psi$};
			\end{tikzpicture}}
		\end{aligned}&=
		\begin{aligned}
			\resizebox{90pt}{!}{%
				\begin{tikzpicture}[block/.style={draw, rectangle, minimum height=0.5cm}]
					\node   at  (-1.5,0)    {$a$};
					\draw (-1,0) to (1,0);
					\draw[fill=black,style=thick] (-1,0) circle (0.05cm and 0.05cm);
					\draw[fill=black,style=thick] (1,0) circle (0.05cm and 0.05cm);
					\node (box) [block, fill=white] {$\Phi$};
					\node   at  (1.5,0)    {$+$};
			\end{tikzpicture}}
		\end{aligned}
		\begin{aligned}
			\resizebox{80pt}{!}{%
				\begin{tikzpicture}[block/.style={draw, rectangle, minimum height=0.5cm}]
					\node   at  (-1.5,0)    {$b$};
					\draw (-1,0) to (1,0);
					\draw[fill=black,style=thick] (-1,0) circle (0.05cm and 0.05cm);
					\draw[fill=black,style=thick] (1,0) circle (0.05cm and 0.05cm);
					\node (box) [block, fill=white] {$\Psi$};
					\node   at  (1.5,0)    {$,$};
			\end{tikzpicture}}
		\end{aligned}\\
		\begin{aligned}\resizebox{100pt}{!}{%
				\begin{tikzpicture}[block/.style={draw, rectangle, minimum height=0.5cm}]
					\draw (-2,0) to (2,0);
					\draw[fill=black,style=thick] (-2,0) circle (0.05cm and 0.05cm);
					\draw[fill=black,style=thick] (2,0) circle (0.05cm and 0.05cm);
					\node (box) at (-0.5,0)  [block, fill=white] {$\Phi$};
					\node (box) at (0.5,0)  [block, fill=white] {$\Psi$};
			\end{tikzpicture}}
		\end{aligned}&=
		\begin{aligned}
			\resizebox{60pt}{!}{%
				\begin{tikzpicture}[block/.style={draw, rectangle, minimum height=0.5cm}]
					\node   at  (1.2,0)    {$.$};
					\draw (-1,0) to (1,0);
					\draw[fill=black,style=thick] (-1,0) circle (0.05cm and 0.05cm);
					\draw[fill=black,style=thick] (1,0) circle (0.05cm and 0.05cm);
					\node (box) [block, fill=white] {$\Psi\Phi$};
			\end{tikzpicture}}
		\end{aligned}
	\end{align*}

	$\diamondsuit$ $\mu_{\mathcal C} : \langle \cdot \rangle_t \otimes \langle \cdot \rangle_t \to \langle \cdot \otimes \cdot  \rangle_t$ and $\Delta_{\mathcal C} : \langle \cdot \otimes \cdot  \rangle_t \to \langle \cdot \rangle_t \otimes \langle \cdot \rangle_t$  are \textbf{k}-linear and compatible with the  tensor product in $\mathcal C$:
	\begin{align}\label{relation for mu}&&\begin{aligned}
			\resizebox{80pt}{!}{%
				\begin{tikzpicture}[block/.style={draw, rectangle, minimum height=0.5cm}]
					\draw[fill=black,style=thick] (-1.4,-0.5) circle (0.05cm and 0.05cm);
					\draw[fill=black,style=thick] (-1.4,0.5) circle (0.05cm and 0.05cm);
					\draw[style=thick]  (-0.1,0) to   (1,0);
					\draw[style=thick]  (-1.4,0.5) to [out=0, looseness= 0.8,in=90] (-0.1,0);
					\draw[fill=black,style=thick] (1,0) circle (0.05cm and 0.05cm);
					\draw[style=thick]  (-1.4,-0.5) to [out=0, looseness =0.8,in=-90](-0.1,0);
					\node (box) at (-0.8,-0.5)  [block, fill=white] {$\Psi$};
					\node (box) at (-0.8,0.5)  [block, fill=white] {$\Phi$};
					\node   at  (1.5,0)    {$=$};
			\end{tikzpicture}}
		\end{aligned}
		\begin{aligned}
			\resizebox{90pt}{!}{%
				\begin{tikzpicture}[block/.style={draw, rectangle, minimum height=0.5cm}]
					\draw[fill=black,style=thick] (-1.4,0) circle (0.05cm and 0.05cm);
					\draw[fill=black,style=thick] (-1.4,1) circle (0.05cm and 0.05cm);
					\draw[style=thick]  (-0.1,0.5) to   (1.1,0.5);
					\draw[style=thick]  (1.1,0.5) to   (2.5,0.5);
					\draw[style=thick]  (-1.4,1) to [out=0, looseness= 0.8,in=90] (-0.1,0.5);
					\draw[fill=black,style=thick] (2.5,0.5) circle (0.05cm and 0.05cm);
					\draw[style=thick]  (-1.4,0) to [out=0, looseness =0.8,in=-90](-0.1,0.5);
					\node (box) at (1,0.5)  [block, fill=white] {$\Phi\otimes \Psi$};
			\end{tikzpicture}}
		\end{aligned}\ \ ,
		&&\begin{aligned}
			\resizebox{90pt}{!}{%
				\begin{tikzpicture}[block/.style={draw, rectangle, minimum height=0.5cm}]
					\draw[fill=black,style=thick] (1.2,1) circle (0.05cm and 0.05cm);
					\draw[fill=black,style=thick] (-2.5,0.5) circle (0.05cm and 0.05cm);
					\draw[fill=black,style=thick] (1.2,0) circle (0.05cm and 0.05cm);
					\draw[style=thick]  (-2.5,0.5) to   (-0.2,0.5);
					\draw[style=thick]  (-1.7,0.5) to  (-1.1,0.5);
					\draw[style=thick]  (-0.2,0.5) to [out=90, looseness= 1,in=0] (1.1,1);
					\draw[style=thick]  (-0.2,0.5) to [out=-90, looseness =1,in=0](1.1,0);
					\node (box) at (-1.2,0.5)  [block, fill=white] {$\Phi\otimes \Psi$};
			\end{tikzpicture}}
		\end{aligned}=
		\begin{aligned}
			\resizebox{70pt}{!}{%
				\begin{tikzpicture}[block/.style={draw, rectangle, minimum height=0.5cm}]
					\draw[fill=black,style=thick] (1.2,1) circle (0.05cm and 0.05cm);
					\draw[fill=black,style=thick] (-1.5,0.5) circle (0.05cm and 0.05cm);
					\draw[fill=black,style=thick] (1.2,0) circle (0.05cm and 0.05cm);
					\draw[style=thick]  (-1.5,0.5) to   (-0.2,0.5);
					\draw[style=thick]  (-0.2,0.5) to [out=90, looseness= 1,in=0] (1.1,1);
					\draw[style=thick]  (-0.2,0.5) to [out=-90, looseness =1,in=0](1.1,0);
					\node (box) at (0.7,1)  [block, fill=white] {$\Phi$};
					\node (box) at (0.7,0)  [block, fill=white] {$\Psi$};
			\end{tikzpicture}}
		\end{aligned}.
	\end{align}
	
	$\diamondsuit$ Associativity and coassociativity:
	\begin{align*}
		\begin{aligned}
			\resizebox{85pt}{!}{%
				\begin{tikzpicture}[block/.style={draw, rectangle, minimum height=0.5cm},scale=0.7]
					\draw[fill=black,style=thick] (-1.4,0) circle (0.05cm and 0.05cm);
					\draw[fill=black,style=thick] (-1.4,1) circle (0.05cm and 0.05cm);
					\draw[style=thick]  (-0.1,0.5) to   (1,0.5);
					\draw[style=thick]  (-1.4,1) to [out=0, looseness= 0.8,in=90] (-0.1,0.5);
					\draw[style=thick]  (-1.4,0) to [out=0, looseness =0.8,in=-90](-0.1,0.5);
					\draw[fill=black,style=thick] (-1.4,2) circle (0.05cm and 0.05cm);
					\draw[style=thick]  (-1.4,2) to   (1,2);
					\draw[style=thick]  (2.3,1.25) to   (4,1.25);
					\draw[style=thick]  (1,2) to [out=0, looseness= 0.8,in=90] (2.3,1.25);
					\draw[fill=black,style=thick] (4,1.25) circle (0.05cm and 0.05cm);
					\draw[style=thick]  (1,0.5) to [out=0, looseness =0.8,in=-90](2.3,1.25);				
			\end{tikzpicture}}
		\end{aligned}
		& \ \ \ 	\text{=}\ \ \ 
		\begin{aligned}	
			\resizebox{85pt}{!}{%
				\begin{tikzpicture}[block/.style={draw, rectangle, minimum height=0.5cm},scale=0.7]
					\draw[fill=black,style=thick] (-1.4,0) circle (0.05cm and 0.05cm);
					\draw[fill=black,style=thick] (-1.4,1) circle (0.05cm and 0.05cm);
					\draw[style=thick]  (-0.1,1.5) to   (1,1.5);
					\draw[style=thick]  (-1.4,2) to [out=0, looseness= 0.8,in=90] (-0.1,1.5);
					\draw[style=thick]  (-1.4,1) to [out=0, looseness =0.8,in=-90](-0.1,1.5);
					\draw[fill=black,style=thick] (-1.4,2) circle (0.05cm and 0.05cm);
					\draw[style=thick]  (-1.4,0) to   (0.8,0);			
					\draw[style=thick]  (2.1,0.75) to   (4,0.75);
					\draw[style=thick]  (0.8,1.5) to [out=0, looseness= 0.8,in=90] (2.1,0.75);
					\draw[fill=black,style=thick] (4,0.75) circle (0.05cm and 0.05cm);
					\draw[style=thick]  (0.8,0) to [out=0, looseness =0.8,in=-90](2.1,0.75);				
			\end{tikzpicture}}
		\end{aligned},
		\begin{aligned}
			\resizebox{100pt}{!}{%
				\begin{tikzpicture}[block/.style={draw, rectangle, minimum height=0.5cm}]
					\draw[fill=black,style=thick] (1.2,1) circle (0.05cm and 0.05cm);
					\draw[fill=black,style=thick] (1.2,0) circle (0.05cm and 0.05cm);
					\draw[style=thick]  (-0.4,0.5) to   (-0.2,0.5);
					\draw[style=thick]  (-0.2,0.5) to [out=90, looseness= 1,in=0] (1.1,1);
					\draw[style=thick]  (-0.2,0.5) to [out=-90, looseness =1,in=0](1.1,0);
					\draw[fill=black,style=thick] (1.2,1.5) circle (0.05cm and 0.05cm);
					\draw[fill=black,style=thick] (-3.5,1) circle (0.05cm and 0.05cm);
					\draw[style=thick]  (-0.5,1.5) to   (1.2,1.5);
					\draw[style=thick]  (-2,1) to   (-3.5,1);
					\draw[style=thick]  ( -2,1) to [out=90, looseness= 0.8,in=0] (-0.3,1.5);
					\draw[style=thick]  (-2,1) to [out=-90, looseness =0.8,in=0](-0.3,0.5);
			\end{tikzpicture}}
		\end{aligned}=
		\begin{aligned}
			\resizebox{100pt}{!}{%
				\begin{tikzpicture}[block/.style={draw, rectangle, minimum height=0.5cm}]
					\draw[fill=black,style=thick] (1.2,1.5) circle (0.05cm and 0.05cm);
					\draw[fill=black,style=thick] (1.2,0.5) circle (0.05cm and 0.05cm);
					\draw[style=thick]  (-0.2,1) to [out=90, looseness= 1,in=0] (1.1,1.5);
					\draw[style=thick]  (-0.2,1) to [out=-90, looseness =1,in=0](1.1,0.5);
					\draw[fill=black,style=thick] (1.2,0) circle (0.05cm and 0.05cm);
					\draw[fill=black,style=thick] (-3.5,0.5) circle (0.05cm and 0.05cm);
					\draw[style=thick]  (-0.2,0) to   (1.2,0);
					\draw[style=thick]  (-2,0.5) to   (-3.5,0.5);
					\draw[style=thick]  ( -2,0.5) to [out=90, looseness= 0.8,in=0] (-0.3,1);
					\draw[style=thick]  (-2,0.5) to [out=-90, looseness =0.8,in=0](-0.3,0);
			\end{tikzpicture}}
		\end{aligned}.
	\end{align*}
	
	$\diamondsuit$ Unitality and counitality:
	\begin{align*}&&\begin{aligned}
			\resizebox{60pt}{!}{%
				\begin{tikzpicture}[block/.style={draw, rectangle, minimum height=0.5cm}]
					\draw[fill=black,style=thick] (-1.4,-0.5) circle (0.05cm and 0.05cm);
					\draw[style=thick]  (-0.1,0) to   (1,0);
					\draw[style=thick]  (-1.4,0.5) to [out=0, looseness= 0.8,in=90] (-0.1,0);
					\draw[fill=black,style=thick] (1,0) circle (0.05cm and 0.05cm);
					\draw[style=thick]  (-1.4,-0.5) to [out=0, looseness =0.8,in=-90](-0.1,0);
					\node at (-1.4,0.5) {$\figureXv$};
			\end{tikzpicture}}
		\end{aligned}=
		\begin{aligned}
			\resizebox{30pt}{!}{%
				\begin{tikzpicture}[block/.style={draw, rectangle, minimum height=0.5cm,fill=white}]
					\draw[fill=black,style=thick] (-4,1.2) circle (0.05cm and 0.05cm);
					\draw[fill=black,style=thick] (-5,1.2) circle (0.05cm and 0.05cm);
					\draw[style=thick]  (-4,1.2) to   (-5,1.2);
			\end{tikzpicture}}
		\end{aligned}=
		\begin{aligned}
			\resizebox{60pt}{!}{%
				\begin{tikzpicture}[block/.style={draw, rectangle, minimum height=0.5cm}]
					\node at (-1,0.5) {};
					\draw[fill=black,style=thick] (-1.4,0) circle (0.05cm and 0.05cm);
					\draw[style=thick]  (-0.1,-0.5) to   (1,-0.5);
					\draw[style=thick]  (-1.4,0) to [out=0, looseness= 0.8,in=90] (-0.1,-0.5);
					\draw[fill=black,style=thick] (1,-0.5) circle (0.05cm and 0.05cm);
					\draw[style=thick]  (-1.4,-1) to [out=0, looseness =0.8,in=-90](-0.1,-0.5);
					\node at (-1.4,-1) {$\figureXv$};
			\end{tikzpicture}}
		\end{aligned}\ ,\ 
		\begin{aligned}
			\resizebox{70pt}{!}{%
				\begin{tikzpicture}[block/.style={draw, rectangle, minimum height=0.5cm}]
					\node at (1.2,1) {$\figureXv$};
					\draw[fill=black,style=thick] (-1.5,0.5) circle (0.05cm and 0.05cm);
					\draw[fill=black,style=thick] (1.2,0) circle (0.05cm and 0.05cm);
					\draw[style=thick]  (-1.5,0.5) to   (-0.2,0.5);
					\draw[style=thick]  (-0.2,0.5) to [out=90, looseness= 1,in=0] (1.1,1);
					\draw[style=thick]  (-0.2,0.5) to [out=-90, looseness =1,in=0](1.1,0);
			\end{tikzpicture}}
		\end{aligned}= 
		\begin{aligned}
			\resizebox{30pt}{!}{%
				\begin{tikzpicture}[block/.style={draw, rectangle, minimum height=0.5cm,fill=white}]
					\draw[fill=black,style=thick] (-4,1.2) circle (0.05cm and 0.05cm);
					\draw[fill=black,style=thick] (-5,1.2) circle (0.05cm and 0.05cm);
					\draw[style=thick]  (-4,1.2) to   (-5,1.2);
			\end{tikzpicture}}
		\end{aligned}=
		\begin{aligned}
			\resizebox{70pt}{!}{%
				\begin{tikzpicture}[block/.style={draw, rectangle, minimum height=0.5cm}]
					\node at (1.2,0) {$\figureXv$};
					\node at (1.2,1.5) {};
					\draw[fill=black,style=thick] (-1.5,0.5) circle (0.05cm and 0.05cm);
					\draw[fill=black,style=thick] (1.2,1) circle (0.05cm and 0.05cm);
					\draw[style=thick]  (-1.5,0.5) to   (-0.2,0.5);
					\draw[style=thick]  (-0.2,0.5) to [out=90, looseness= 1,in=0] (1.1,1);
					\draw[style=thick]  (-0.2,0.5) to [out=-90, looseness =1,in=0](1.1,0);
			\end{tikzpicture}}
		\end{aligned}.
	\end{align*}
	
	$\diamondsuit$ Compatibility between $\mu_{\mathcal C}$ and $\Delta_{\mathcal C}$:
	\begin{align*}
		\begin{aligned}
			\resizebox{100pt}{!}{%
				\begin{tikzpicture}[block/.style={draw, rectangle},scale=0.7]
					\draw[fill=black,style=thick] (-3.4,0) circle (0.05cm and 0.05cm);
					\draw[style=thick]  (-3.4,0) to   (-1.5,0);
					\draw[style=thick]  (-1.5,0) to [out=90, looseness= 1,in=0] (-0.2,0.5);
					\draw[style=thick]  (-1.5,0) to [out=-90, looseness =1,in=0](-0.2,-0.5);
					\draw[style=thick]  (-0.2,1) to [out=0, looseness= 0.4,in=45] (1.1,0.75);
					\draw[style=thick]  (-0.2,0.5) to [out=0, looseness =0.4,in=-45](1.1,0.75);
					\draw[style=thick]  (1.1,0.75) to   (3,0.75);
					\draw[fill=black,style=thick] (3,0.75) circle (0.05cm and 0.05cm);
					\draw[style=thick]  (-0.2,-0.5) to   (1,-0.5);
					\draw[style=thick]  (1,-0.5) to   (3,-0.5);
					\draw[fill=black,style=thick] (3,-0.5) circle (0.05cm and 0.05cm);
					\draw[fill=black,style=thick] (-3.4,1) circle (0.05cm and 0.05cm);
					\draw[style=thick]  (-3.4,1) to   (-0.2,1);
			\end{tikzpicture}}
		\end{aligned}=
		\begin{aligned}
			\resizebox{90pt}{!}{%
				\begin{tikzpicture}[block/.style={draw, rectangle},scale=0.7]
					\draw[fill=black,style=thick] (-1.4,0) circle (0.05cm and 0.05cm);
					\draw[fill=black,style=thick] (-1.4,1) circle (0.05cm and 0.05cm);
					\draw[style=thick]  (-0.1,0.5) to   (1.1,0.5);
					\draw[style=thick]  (1.1,0.5) to   (1.4,0.5);
					\draw[style=thick]  (1.4,0.5) to [out=90, looseness= 1,in=0] (2.7,1);
					\draw[style=thick]  (1.4,0.5) to [out=-90, looseness =1,in=0](2.7,0);
					\draw[style=thick]  (-1.4,1) to [out=0, looseness= 0.8,in=90] (-0.1,0.5);
					\draw[style=thick]  (-1.4,0) to [out=0, looseness =0.8,in=-90](-0.1,0.5);
					\draw[fill=black,style=thick] (2.8,1) circle (0.05cm and 0.05cm);
					\draw[fill=black,style=thick] (2.8,0) circle (0.05cm and 0.05cm);
			\end{tikzpicture}}
		\end{aligned}, \ 
		\begin{aligned}
			\resizebox{100pt}{!}{%
				\begin{tikzpicture}[block/.style={draw, rectangle},scale=0.7]
					\draw[fill=black,style=thick] (-3.4,2) circle (0.05cm and 0.05cm);
					\draw[style=thick]  (-3.4,2) to   (-1.5,2);
					\draw[style=thick]  (-1.5,2) to [out=90, looseness= 1,in=0] (-0.2,2.5);
					\draw[style=thick]  (-1.5,2) to [out=-90, looseness =1,in=0](-0.2,1.5);
					\draw[style=thick]  (-0.2,1.5) to [out=0, looseness= 0.4,in=45] (1.1,1.25);
					\draw[style=thick]  (-0.2,1) to [out=0, looseness =0.4,in=-45](1.1,1.25);
					\draw[style=thick]  (1.1,1.25) to   (3,1.25);
					\draw[fill=black,style=thick] (3,1.25) circle (0.05cm and 0.05cm);
					\draw[style=thick]  (-0.2,2.5) to   (3,2.5);
					\draw[fill=black,style=thick] (-3.4,1) circle (0.05cm and 0.05cm);
					\draw[fill=black,style=thick] (3,2.5) circle (0.05cm and 0.05cm);
					\draw[style=thick]  (-3.4,1) to   (-0.2,1);
			\end{tikzpicture}}
		\end{aligned}=
		\begin{aligned}
			\resizebox{90pt}{!}{%
				\begin{tikzpicture}[block/.style={draw, rectangle},scale=0.7]
					\draw[fill=black,style=thick] (-1.4,0) circle (0.05cm and 0.05cm);
					\draw[fill=black,style=thick] (-1.4,1) circle (0.05cm and 0.05cm);
					\draw[style=thick]  (-0.1,0.5) to   (1.1,0.5);
					\draw[style=thick]  (1.1,0.5) to   (1.4,0.5);
					\draw[style=thick]  (1.4,0.5) to [out=90, looseness= 1,in=0] (2.7,1);
					\draw[style=thick]  (1.4,0.5) to [out=-90, looseness =1,in=0](2.7,0);
					\draw[style=thick]  (-1.4,1) to [out=0, looseness= 0.8,in=90] (-0.1,0.5);
					\draw[style=thick]  (-1.4,0) to [out=0, looseness =0.8,in=-90](-0.1,0.5);
					\draw[fill=black,style=thick] (2.8,1) circle (0.05cm and 0.05cm);
					\draw[fill=black,style=thick] (2.8,0) circle (0.05cm and 0.05cm);
			\end{tikzpicture}}
		\end{aligned}.
	\end{align*}
	
	$\diamondsuit$ $\mu_{\mathcal C}$ is a retraction and $\Delta_{\mathcal C}$ is a section:
	\begin{align*}
		\begin{aligned}
			\resizebox{90pt}{!}{%
				\begin{tikzpicture}[block/.style={draw, rectangle, minimum height=0.5cm,fill=white}]
					\draw[fill=black,style=thick] (-3,0.5) circle (0.05cm and 0.05cm);
					\draw[fill=black,style=thick] (1.5,0.5) circle (0.05cm and 0.05cm);
					\draw[style=thick]  (-2,0.5) to [out=90, looseness= 1,in=0] (-0.7,1);
					\draw[style=thick]  (-2,0.5) to [out=-90, looseness =1,in=0](-0.7,0);
					\draw[style=thick]  (-0.7,1) to [out=0, looseness= 0.9,in=90] (0.6,0.5);
					\draw[style=thick]  (-0.7,0) to [out=0, looseness =0.9,in=-90](0.6,0.5);
					\draw[style=thick]  (-3,0.5) to   (-2,0.5);
					\draw[style=thick]  (0.6,0.5) to   (1.5,0.5);
			\end{tikzpicture}}
		\end{aligned} =
		\begin{aligned}
			\resizebox{45pt}{!}{%
				\begin{tikzpicture}[block/.style={draw, rectangle, minimum height=0.5cm,fill=white}]
					\draw[fill=black,style=thick] (-4,1.2) circle (0.05cm and 0.05cm);
					\draw[fill=black,style=thick] (-6,1.2) circle (0.05cm and 0.05cm);
					\draw[style=thick]  (-4,1.2) to   (-6,1.2);
			\end{tikzpicture}}
		\end{aligned}\ \  .
	\end{align*}
	
	$\diamondsuit$ The object $\langle \mathbb 1_{\mathcal C}\rangle_t$ is of dimension $t$:
	\begin{align*}
		&\begin{aligned}
			\resizebox{50pt}{!}{%
				\begin{tikzpicture}[block/.style={draw, rectangle, minimum height=0.5cm,fill=white}]
					\node at (-6,0.5) {$\figureXv$};
					\node at (-4,0.5) {$\figureXv$};
					\node at (-5,1) {$\mathbb{1}_{\mathcal C}$};
					\draw[fill=black,style=thick] (-5,0.5) circle (0.05cm and 0.05cm);
					\draw[style=thick]  (-4,0.5) to   (-6,0.5);
			\end{tikzpicture}}
		\end{aligned} 
		= t \ \text{id}_{\mathbb{1}}.
	\end{align*}
	
	$\diamondsuit$ Relations concerning the braiding (see relations (5) and (8) in \cite[Section 4.6]{M}).

	\medspace
In \cite[Section 4.6]{M}, it is discussed that $C \to \StC$ is a Frobenius functor \cite[Definition 4.28]{M}. The following proposition shows (graphically) how to get extended Frobenius algebras in $\StC$ from extended Frobenius algebras in $\mathcal C$.

	\begin{proposition}\label{extended in StC}
		\begin{enumerate}
			\item	\label{frob in C to frob in StC} 	
			A Frobenius algebra $A \in \mathcal C$ induces a Frobenius algebra $\langle A \rangle_t $  in $\StC$, with structure maps
			\begin{figure}[H]
				\begin{center}	\resizebox{170pt}{!}{%
}
					
				\end{center}
			\end{figure}
			
			\item 	\label{extended frob in C to frob in StC} 		An extended Frobenius algebra $A \in \mathcal C$  induces an extended Frobenius algebra $\langle A \rangle_t $  in $\StC$, with multiplication, comultiplication, unit and counit maps as above, and
			\begin{figure}[H]
				\begin{center}	\resizebox{170pt}{!}{%
						%
}
				\end{center}
			\end{figure}
		\end{enumerate}
	\end{proposition}
	
	\begin{proof}
		Let $A\in \mathcal C$ be a Frobenius algebra. We show first that $\langle A \rangle_t$ with the maps given in \eqref{frob in C to frob in StC} is a Frobenius algebra in $\StC$ by graphical calculus.
		
		\emph{Associativity:}
		\begin{align*}
			\begin{aligned}
				\resizebox{150pt}{!}{%
					%
}
			\end{aligned}
			&	\tiny{\text{unitality of }} \scriptstyle{\StC} \tiny{\text{ and }} \scriptstyle{u_{A}}.
		\end{align}	
		
		This shows that  $(\langle A \rangle_t, u_{\langle A \rangle_t}, m_{\langle A \rangle_t})$ is an algebra in $\StC$. Similarly, coassociativity and counitality of $\StC$ make  $(\langle A \rangle_t, \varepsilon_{\langle A \rangle_t}, \Delta_{\langle A \rangle_t})$ a coalgebra. Moreover, by the  equalities below
		\begin{align*}
			\begin{aligned}
				\resizebox{160pt}{!}{%
					%
}
			\end{aligned}
			&	\tiny{\text{compatibility of }} {\scriptstyle \mu_{\mathcal C}} \tiny\text{ and } {\scriptstyle \Delta_{\mathcal C}}.
		\end{align*}
		multiplication and comultiplication satisfy diagram \eqref{frob algebra diagram} and thus we have a Frobenius algebra structure in $\langle A \rangle_t$.
		
		We show now that if $A \in \mathcal C$ is an extended Frobenius algebra, then $\langle A \rangle_t$ with maps as in \eqref{extended frob in C to frob in StC} defines an extended Frobenius algebra in $\StC$.
		Commutativity follows from commutativity of $m_A$ and $\mu_{\mathcal C}$.  Since $\phi_A$ is an involution in $\mathcal C,$ then  $\phi_{\langle A\rangle_t}$ is an involution in $\StC$.  We show below that the first diagram in Equation \eqref{extended Frob alg} commutes. In fact, we arrive at the same equality when computing
		\begin{align*}
			\begin{aligned}
				\resizebox{130pt}{!}{%
					\begin{tikzpicture}[block/.style={draw, rectangle, minimum height=0.5cm},scale=0.7]
						\draw[fill=black,style=thick] (-4.2,0) circle (0.05cm and 0.05cm);
						\node at (-4.2, 0.4) {$A$};
						\draw[style=thick]  (-4.2,0) to   (-1.4,0);
						\node at (-4.2,1) {$\figureXv$};
						\draw[style=thick]  (-4.2,1) to   (-1.4,1);
						\draw[fill=black,style=thick] (-3.4,1) circle (0.05cm and 0.05cm);
						\node at (-3.4,1.5) {$ \mathbb 1_{\mathcal C}$};
						\node (box) at (-2.4,1)  [block, fill=white] {$\theta_A$};
						\node (box) at (-2.4,0)  [block, fill=white] {$\text{id}_A$};
						\draw[style=thick]  (-1.8,1) to   (-1.4,1);
						\node   at  (-1.4,0.4)    {$A$};
						\node   at  (-1.4,1.4)    {$A$};
						\node   at  (1.1,1)    {$A^{\otimes 2}$};
						\draw[fill=black,style=thick] (-1.3,0) circle (0.05cm and 0.05cm);
						\draw[fill=black,style=thick] (-1.3,1) circle (0.05cm and 0.05cm);
						\draw[fill=black,style=thick] (1.1,0.5) circle (0.05cm and 0.05cm);
						\draw[style=thick]  (-0.1,0.5) to   (1.1,0.5);
						\draw[style=thick]  (-1.4,1) to [out=0, looseness= 0.8,in=90] (-0.1,0.5);
						\draw[style=thick]  (-1.4,0) to [out=0, looseness =0.8,in=-90](-0.1,0.5);
						\draw[style=thick]  (1.1,0.5) to   (1.5,0.5);
						\node (box) at (2.1,0.5)  [block] {$m_A$};
						\draw[style=thick]  (2.7,0.5) to   (3,0.5);
						\draw[fill=black,style=thick] (3,0.5) circle (0.05cm and 0.05cm);
						\node at (3,1) {$A$};
				\end{tikzpicture}}
			\end{aligned}
			&	\text{=}
			\begin{aligned}
				\resizebox{155pt}{!}{%
					\begin{tikzpicture}[block/.style={draw, rectangle}, scale=0.7]
						\draw[style=thick]  (-2.5,0) to   (-1.4,0);
						\draw[style=thick]  (-2.5,1) to   (-1.4,1);
						\node   at  (-2.5,0.4)    {$A$};
						\node   at  (-1.4,0.4)    {$A$};
						\node at (-2.5,1) {$\figureXv$};
						\draw[style=thick]  (-2.5,1) to   (-1.4,1);
						\node   at  (-1.4,1.4)    {$\mathbb 1_{\mathcal C}$};
						\node   at  (1.1,1)    {$ \mathbb 1_{\mathcal C}\otimes A$};
						\draw[fill=black,style=thick] (-2.5,0) circle (0.05cm and 0.05cm);
						\draw[fill=black,style=thick] (-1.3,0) circle (0.05cm and 0.05cm);
						\draw[fill=black,style=thick] (-1.3,1) circle (0.05cm and 0.05cm);
						\draw[fill=black,style=thick] (1.1,0.5) circle (0.05cm and 0.05cm);
						\draw[style=thick]  (-0.1,0.5) to   (1.1,0.5);
						\draw[style=thick]  (-1.4,1) to [out=0, looseness= 0.8,in=90] (-0.1,0.5);
						\draw[style=thick]  (-1.4,0) to [out=0, looseness =0.8,in=-90](-0.1,0.5);
						\draw[style=thick]  (1.1,0.5) to   (1.5,0.5);
						\draw[style=thick]  (0,0.5) to   (6.5,0.5);
						\node (box) at (3.9,0.5)  [block, fill=white] {$m_A(\theta_A\otimes \text{id}_A)$};
						\draw[fill=black,style=thick] (6.5,0.5) circle (0.05cm and 0.05cm);
						\node at (6.5,1) {$A$};
				\end{tikzpicture}}
			\end{aligned}
			&	\tiny{\text{linearity of }} {\scriptstyle \mu_{\mathcal C}},	\\
			&=	\begin{aligned}
				\resizebox{120pt}{!}{%
					\begin{tikzpicture}[block/.style={draw, rectangle, minimum height=0.5cm,fill=white}]
						\node   at  (-1,1)    {$A$};
						\node at (-6,0.5) {$\figureXv$};
						\node at (-5,1) {$\mathbb{1}_{\mathcal C}$};
						\draw[fill=black,style=thick] (-5,0.5) circle (0.05cm and 0.05cm);
						\draw[fill=black,style=thick] (-1,0.5) circle (0.05cm and 0.05cm);
						\draw[style=thick]  (-6,0.5) to   (-1,0.5);
						\node (box) at (-3,0.5)  [block, fill=white] {$m_A(\theta_A \otimes \text{id}_A)$};
				\end{tikzpicture}}
			\end{aligned} 
			&	\tiny{\text{unitality of }} {\scriptstyle \mu_{\mathcal C}}.
		\end{align*}
		and
		\begin{align*}
			\begin{aligned}
				\resizebox{160pt}{!}{%
					\begin{tikzpicture}[block/.style={draw, rectangle, minimum height=0.5cm},scale=0.7]
						\draw[fill=black,style=thick] (-4.2,0) circle (0.05cm and 0.05cm);
						\node at (-4.2, 0.4) {$A$};
						\draw[style=thick]  (-4.2,0) to   (-1.4,0);
						\node at (-4.2,1) {$\figureXv$};
						\draw[style=thick]  (-4.2,1) to   (-1.4,1);
						\draw[fill=black,style=thick] (-3.4,1) circle (0.05cm and 0.05cm);
						\node at (-3.4,1.5) {$ \mathbb 1_{\mathcal C}$};
						\node (box) at (-2.4,1)  [block, fill=white] {$\theta_A$};
						\node (box) at (-2.4,0)  [block, fill=white] {$\text{id}_A$};
						\draw[style=thick]  (-1.8,1) to   (-1.4,1);
						\node   at  (-1.4,0.4)    {$A$};
						\node   at  (-1.4,1.4)    {$A$};
						\node   at  (1.1,1)    {$A^{\otimes 2}$};
						\draw[fill=black,style=thick] (-1.3,0) circle (0.05cm and 0.05cm);
						\draw[fill=black,style=thick] (-1.3,1) circle (0.05cm and 0.05cm);
						\draw[fill=black,style=thick] (1.1,0.5) circle (0.05cm and 0.05cm);
						\draw[style=thick]  (-0.1,0.5) to   (1.1,0.5);
						\draw[style=thick]  (-1.4,1) to [out=0, looseness= 0.8,in=90] (-0.1,0.5);
						\draw[style=thick]  (-1.4,0) to [out=0, looseness =0.8,in=-90](-0.1,0.5);
						\draw[style=thick]  (1.1,0.5) to   (1.5,0.5);
						\node (box) at (2.1,0.5)  [block] {$m_A$};
						\draw[style=thick]  (2.7,0.5) to   (5,0.5);
						\draw[fill=black,style=thick] (3,0.5) circle (0.05cm and 0.05cm);
						\node at (3,1) {$A$};
						\draw[fill=black,style=thick] (5,0.5) circle (0.05cm and 0.05cm);
						\node at (5,1) {$A$};
						\node (box) at (4,0.5)  [block, fill=white] {$\phi_A$};
				\end{tikzpicture}}
			\end{aligned}
			&	\text{=}
			\begin{aligned}
				\resizebox{160pt}{!}{%
					\begin{tikzpicture}[block/.style={draw, rectangle}, scale=0.7]
						\draw[style=thick]  (-2.5,0) to   (-1.4,0);
						\draw[style=thick]  (-2.5,1) to   (-1.4,1);
						\node   at  (-2.5,0.4)    {$A$};
						\node   at  (-1.4,0.4)    {$A$};
						\node at (-2.5,1) {$\figureXv$};
						\draw[style=thick]  (-2.5,1) to   (-1.4,1);
						\node   at  (-1.4,1.4)    {$\mathbb 1_{\mathcal C}$};
						\node   at  (1.1,1)    {$ \mathbb 1_{\mathcal C}\otimes A$};
						\draw[fill=black,style=thick] (-2.5,0) circle (0.05cm and 0.05cm);
						\draw[fill=black,style=thick] (-1.3,0) circle (0.05cm and 0.05cm);
						\draw[fill=black,style=thick] (-1.3,1) circle (0.05cm and 0.05cm);
						\draw[fill=black,style=thick] (1.1,0.5) circle (0.05cm and 0.05cm);
						\draw[style=thick]  (-0.1,0.5) to   (1.1,0.5);
						\draw[style=thick]  (-1.4,1) to [out=0, looseness= 0.8,in=90] (-0.1,0.5);
						\draw[style=thick]  (-1.4,0) to [out=0, looseness =0.8,in=-90](-0.1,0.5);
						\draw[style=thick]  (1.1,0.5) to   (1.5,0.5);
						\draw[style=thick]  (0,0.5) to   (6.5,0.5);
						\node (box) at (3.9,0.5)  [block, fill=white] {$\phi_A m_A(\theta_A\otimes \text{id}_A)$};
						\draw[fill=black,style=thick] (6.5,0.5) circle (0.05cm and 0.05cm);
						\node at (6.5,1) {$A$};
				\end{tikzpicture}}
			\end{aligned}
			&	\tiny{\text{linearity of }} {\scriptstyle \mu_{\mathcal C}},	\\
			&=	\begin{aligned}
				\resizebox{130pt}{!}{%
					\begin{tikzpicture}[block/.style={draw, rectangle, minimum height=0.5cm,fill=white}]
						\node   at  (0,1)    {$A$};
						\node at (-6,0.5) {$\figureXv$};
						\node at (-5,1) {$\mathbb{1}_{\mathcal C}$};
						\draw[fill=black,style=thick] (-5,0.5) circle (0.05cm and 0.05cm);
						\draw[fill=black,style=thick] (0,0.5) circle (0.05cm and 0.05cm);
						\draw[style=thick]  (-6,0.5) to   (0,0.5);
						\node (box) at (-2.5,0.5)  [block, fill=white] {$m_A(\theta_A \otimes \text{id}_A)$};
				\end{tikzpicture}}
			\end{aligned} 
			&	\tiny{\text{unitality of }} {\scriptstyle \mu_{\mathcal C}}.
		\end{align*}
		Lastly, the second diagram in Equation \eqref{extended Frob alg} also commutes since
		\begin{align*}
			\begin{aligned}
				\resizebox{210pt}{!}{%
					\begin{tikzpicture}[block/.style={draw, rectangle, minimum height=0.5cm,fill=white}]
						\node   at  (1.1,0.5)    {$A$};
						\node   at  (1.1,1.5)    {$A$};
						\node   at  (-1.3,1)    {$A^{\otimes 2}$};
						\node   at  (-3,1)    {$A$};
						\node   at  (2.3,1.5)    {$A$};
						\node   at  (2.3,0.5)    {$A$};
						\node   at  (6,1)    {$A$};
						\node at (-5.5,0.5) {$\figureXv$};
						\node at (-5,1) {$\mathbb{1}_{\mathcal C}$};
						\node   at  (-1.3,1)    {$A^{\otimes 2}$};
						\node   at  (4,1)    {$A^{\otimes 2}$};
						\node (box) at (-2.1,0.5)  [block] {$\Delta_A$};
						\node (box) at (-4,0.5)  [block] {$u_A$};
						\draw[fill=black,style=thick] (-5,0.5) circle (0.05cm and 0.05cm);
						\draw[fill=black,style=thick] (-1.3,0.5) circle (0.05cm and 0.05cm);
						\draw[fill=black,style=thick] (1.1,1) circle (0.05cm and 0.05cm);
						\draw[fill=black,style=thick] (-3,0.5) circle (0.05cm and 0.05cm);
						\draw[fill=black,style=thick] (1.1,0) circle (0.05cm and 0.05cm);
						\draw[fill=black,style=thick] (2.3,1) circle (0.05cm and 0.05cm);
						\draw[fill=black,style=thick] (2.3,0) circle (0.05cm and 0.05cm);
						\draw[fill=black,style=thick] (4,0.5) circle (0.05cm and 0.05cm);
						\draw[fill=black,style=thick] (6,0.5) circle (0.05cm and 0.05cm);
						\draw[style=thick]  (-4.4,0.5) to   (-5.5,0.5);
						\draw[style=thick]  (-3.6,0.5) to   (-3,0.5);
						\draw[style=thick]  (-3,0.5) to   (-2.5,0.5);
						\draw[style=thick]  (-1.3,0.5) to   (-0.2,0.5);
						\draw[style=thick]  (-1.7,0.5) to  (-1.1,0.5);
						\draw[style=thick]  (-0.2,0.5) to [out=90, looseness= 1,in=0] (1.1,1);
						\draw[style=thick]  (-0.2,0.5) to [out=-90, looseness =1,in=0](1.1,0);
						\draw[style=thick]  (1.1,1) to  (2,1);
						\draw[style=thick]  (1.1,0) to  (2,0);
						\draw[style=thick]  (2,1) to [out=0, looseness= 0.8,in=90] (3.3,0.5);
						\draw[style=thick]  (2,0) to [out=0, looseness =0.8,in=-90](3.3,0.5);	
						\draw[style=thick]  (3.3,0.5) to  (6,0.5);
						\node [fill=white](box) at (1.7,1)  [block,fill=white] {$\phi_A$};
						\node (box) at (5,0.5)  [block] {$m_A$};
				\end{tikzpicture}}
			\end{aligned}
			&=
			\begin{aligned}
				\resizebox{210pt}{!}{%
					\begin{tikzpicture}[block/.style={draw, rectangle, minimum height=0.5cm,fill=white}]
						\node   at  (1.1,0.5)    {$A$};
						\node   at  (1.1,1.5)    {$A$};
						\node   at  (5,1)    {$A$};
						\node at (-5.5,0.5) {$\figureXv$};
						\node at (-5,1) {$\mathbb{1}_{\mathcal C}$};
						\node   at  (-1,1)    {$A^{\otimes 2}$};
						\node   at  (3,1)    {$A^{\otimes 2}$};
						\draw[fill=black,style=thick] (-5,0.5) circle (0.05cm and 0.05cm);
						\draw[fill=black,style=thick] (-1,0.5) circle (0.05cm and 0.05cm);
						\draw[fill=black,style=thick] (1.1,1) circle (0.05cm and 0.05cm);
						\draw[fill=black,style=thick] (-3,0.5) circle (0.05cm and 0.05cm);
						\draw[fill=black,style=thick] (1.1,0) circle (0.05cm and 0.05cm);
						\draw[fill=black,style=thick] (3,0.5) circle (0.05cm and 0.05cm);
						\draw[fill=black,style=thick] (5,0.5) circle (0.05cm and 0.05cm);
						\draw[style=thick]  (-4.4,0.5) to   (-5.5,0.5);
						\draw[style=thick]  (-3.6,0.5) to   (-3,0.5);
						\draw[style=thick]  (-3,0.5) to   (-2.5,0.5);
						\draw[style=thick]  (-1.3,0.5) to   (-0.2,0.5);
						\draw[style=thick]  (-1.7,0.5) to  (-1.1,0.5);
						\draw[style=thick]  (-0.2,0.5) to [out=90, looseness= 1,in=0] (1.1,1);
						\draw[style=thick]  (-0.2,0.5) to [out=-90, looseness =1,in=0](1.1,0);
						\draw[style=thick]  (1.1,1) to [out=0, looseness= 0.8,in=90] (2.4,0.5);
						\draw[style=thick]  (1.1,0) to [out=0, looseness =0.8,in=-90](2.4,0.5);	
						\draw[style=thick]  (2.4,0.5) to  (5,0.5);
						\node (box) at (4,0.5)  [block] {$m_A$};
						\node (box) at (-3,0.5)  [block] {$(\phi_A \otimes \Id_A)\Delta_Au_A$};
				\end{tikzpicture}}
			\end{aligned}	\\
			&=	\begin{aligned}
				\resizebox{170pt}{!}{%
					\begin{tikzpicture}[block/.style={draw, rectangle, minimum height=0.5cm,fill=white}]
						\node   at  (2,1)    {$A$};
						\node at (-6,0.5) {$\figureXv$};
						\node at (-5,1) {$\mathbb{1}_{\mathcal C}$};
						\draw[fill=black,style=thick] (-5,0.5) circle (0.05cm and 0.05cm);
						\draw[fill=black,style=thick] (-1,0.5) circle (0.05cm and 0.05cm);
						\draw[fill=black,style=thick] (2,0.5) circle (0.05cm and 0.05cm);
						\draw[style=thick]  (-6,0.5) to   (2,0.5);
						\node (box) at (-2,0.5)  [block] {$m_A(\phi_A \otimes \Id_A)\Delta_Au_A$};
				\end{tikzpicture}}
			\end{aligned}\\
			&=	\begin{aligned}
				\resizebox{170pt}{!}{%
					\begin{tikzpicture}[block/.style={draw, rectangle, minimum height=0.5cm,fill=white}]
						\node   at  (2,1)    {$A$};
						\node at (-6,0.5) {$\figureXv$};
						\node at (-5,1) {$\mathbb{1}_{\mathcal C}$};
						\draw[fill=black,style=thick] (-5,0.5) circle (0.05cm and 0.05cm);
						\draw[fill=black,style=thick] (-1,0.5) circle (0.05cm and 0.05cm);
						\draw[fill=black,style=thick] (2,0.5) circle (0.05cm and 0.05cm);
						\draw[style=thick]  (-6,0.5) to   (2,0.5);
						\node (box) at (-2,0.5)  [block] {$m_A(\theta_A \otimes \theta_A)$};
				\end{tikzpicture}}
			\end{aligned}
		\end{align*}
		is equal to
		\begin{align*}
		\begin{aligned}
			\resizebox{130pt}{!}{%
				\begin{tikzpicture}[block/.style={draw, rectangle, minimum height=0.5cm},scale=0.7]
					\draw[fill=black,style=thick] (-4.2,0) circle (0.05cm and 0.05cm);
					\draw[style=thick]  (-4.2,0) to   (-1.4,0);
					\node at (-4.2,1) {$\figureXv$};
					\node at (-4.2,0) {$\figureXv$};
					\draw[style=thick]  (-4.2,1) to   (-1.4,1);
					\draw[fill=black,style=thick] (-3.4,1) circle (0.05cm and 0.05cm);
					\draw[fill=black,style=thick] (-3.4,0) circle (0.05cm and 0.05cm);
					\node at (-3.4,1.5) {$\mathbb{1}_{\mathcal C}$};
					\node at (-3.4,0.5) {$\mathbb{1}_{\mathcal C}$};
					\node (box) at (-2.4,1)  [block, fill=white] {$\theta_A$};
					\node (box) at (-2.4,0)  [block, fill=white] {$\theta_A$};
					\draw[style=thick]  (-1.8,1) to   (-1.4,1);
					\node   at  (-1.4,0.4)    {$A$};
					\node   at  (-1.4,1.4)    {$A$};
					\node   at  (1.1,1)    {$A^{\otimes 2}$};
					\draw[fill=black,style=thick] (-1.3,0) circle (0.05cm and 0.05cm);
					\draw[fill=black,style=thick] (-1.3,1) circle (0.05cm and 0.05cm);
					\draw[fill=black,style=thick] (1.1,0.5) circle (0.05cm and 0.05cm);
					\draw[style=thick]  (-0.1,0.5) to   (1.1,0.5);
					\draw[style=thick]  (-1.4,1) to [out=0, looseness= 0.8,in=90] (-0.1,0.5);
					\draw[style=thick]  (-1.4,0) to [out=0, looseness =0.8,in=-90](-0.1,0.5);
					\draw[style=thick]  (1.1,0.5) to   (1.5,0.5);
					\node (box) at (2.1,0.5)  [block] {$m_A$};
					\draw[style=thick]  (2.7,0.5) to   (3,0.5);
					\draw[fill=black,style=thick] (3,0.5) circle (0.05cm and 0.05cm);
					\node at (3,1) {$A$};
			\end{tikzpicture}}
		\end{aligned}
		&	\text{=}
		\begin{aligned}
			\resizebox{150pt}{!}{%
				\begin{tikzpicture}[block/.style={draw, rectangle}, scale=0.7]
					\draw[style=thick]  (-2.5,0) to   (-1.4,0);
					\draw[style=thick]  (-2.5,1) to   (-1.4,1);
					\node   at  (-1.4,0.5)    {	$\mathbb{1}_{\mathcal C}$};
					\node at (-2.5,1) {$\figureXv$};
					\node at (-2.5,0) {$\figureXv$};
					\draw[style=thick]  (-2.5,1) to   (-1.4,1);
					\node   at  (-1.4,1.4)    {$\mathbb{1}_{\mathcal C}$};
					\node   at  (1.1,1)    {$ \mathbb{1}_{\mathcal C}\otimes A$};
					\draw[fill=black,style=thick] (-2.5,0) circle (0.05cm and 0.05cm);
					\draw[fill=black,style=thick] (-1.3,0) circle (0.05cm and 0.05cm);
					\draw[fill=black,style=thick] (-1.3,1) circle (0.05cm and 0.05cm);
					\draw[fill=black,style=thick] (1.1,0.5) circle (0.05cm and 0.05cm);
					\draw[style=thick]  (-0.1,0.5) to   (1.1,0.5);
					\draw[style=thick]  (-1.4,1) to [out=0, looseness= 0.8,in=90] (-0.1,0.5);
					\draw[style=thick]  (-1.4,0) to [out=0, looseness =0.8,in=-90](-0.1,0.5);
					\draw[style=thick]  (1.1,0.5) to   (1.5,0.5);
					\draw[style=thick]  (0,0.5) to   (6.5,0.5);
					\node (box) at (3.9,0.5)  [block, fill=white] {$m_A(\theta_A\otimes \theta_A)$};
					\draw[fill=black,style=thick] (6.5,0.5) circle (0.05cm and 0.05cm);
					\node at (6.5,1) {$A$};
			\end{tikzpicture}}
		\end{aligned}
		&	\tiny{\text{linearity of }} {\scriptstyle \mu_{\mathcal C}},	\\
		&=	\begin{aligned}
			\resizebox{120pt}{!}{%
				\begin{tikzpicture}[block/.style={draw, rectangle, minimum height=0.5cm,fill=white}]
					\node   at  (-1,1)    {$A$};
					\node at (-6,0.5) {$\figureXv$};
					\node at (-5,1) {$\mathbb{1}_{\mathcal C}$};
					\draw[fill=black,style=thick] (-5,0.5) circle (0.05cm and 0.05cm);
					\draw[fill=black,style=thick] (-1,0.5) circle (0.05cm and 0.05cm);
					\draw[fill=black,style=thick] (-1,0.5) circle (0.05cm and 0.05cm);
					\draw[style=thick]  (-6,0.5) to   (-1,0.5);
					\node (box) at (-3,0.5)  [block] {$m_A(\theta_A \otimes \theta_A)$};
			\end{tikzpicture}}
		\end{aligned} 
		&	\tiny{\text{unitality of }} {\scriptstyle \mu_{\mathcal C}}.
	\end{align*}
		
		Hence $\langle A \rangle_t$ is an extended Frobenius algebra in $\StC$.
	\end{proof}
	
	\section{Unoriented 2-dimensional cobordisms} \label{section: unoriented 2dim}
	In this section, we will follow \cite{T, TT} to describe the category $\UCob$ of $2$-dimensional unoriented cobordisms, introduced in Definition \ref{def of UCobn}. More precisely, we describe the skeletal version of this category: objects are non-negative integers $n\in \mathbb Z_{\geq 0}$, where $0$ correponds to the empty space and $1$ to the isomorphism class of the circle. A morphism in $\Hom_{\UCob}(n,m)$ corresponds to an unoriented cobordism from $n$ disjoint circles to $m$ disjoint circles.    
	
	Similarly to the oriented case, $\UCob$ is a symmetric monoidal category, with tensor product given by disjoint union, and braiding given by the cobordism $\tau=$
	\resizebox{30pt}{!}{%
		\begin{tikzpicture}[tqft/cobordism/.style={draw,thick},
			tqft/view from=outgoing, tqft/boundary separation=30pt,
			tqft/cobordism height=40pt, tqft/circle x radius=8pt,
			tqft/circle y radius=4.5pt, tqft/every boundary component/.style={rotate=90}
			]
			\pic[tqft/cylinder to next,rotate=90, name=i, anchor=incoming boundary 1,boundary separation=55pt, anchor={(1.3,-0.5)},every incoming
			boundary component/.style={draw,dotted,thick},every outgoing
			boundary component/.style={draw,thick}];
			\pic[tqft/cylinder to prior,rotate=90, anchor={(0.8,-0.5)},boundary separation=55pt,every incoming
			boundary component/.style={draw,dotted,thick},every outgoing
			boundary component/.style={draw,thick},name=h];
	\end{tikzpicture}} and its generalization.

	\subsection{Generators and relations for $\text{UCob}_2$}\label{section:gens and rels}
	Consider the category $\UCob$ of unoriented 2-dimensional cobordisms, as defined in Definition \ref{def of UCobn}. In \cite[Section 2.2]{T}, Tubbenhauer defines a category $\operatorname{uCob^2}_R(\emptyset)^*$ by generators and relations, which has  an obvious functor to $\UCob$. Arguments in \cite[Section 2.2]{TT} essentially prove that this functor is an equivalence, giving a description by generators and relations of the category of 2-dimensional unoriented cobordisms. 
	
	Specifically, every morphism can be obtained by composition (from left to right) and disjoint union (vertical stacking) of the following 8 cobordisms:
	\begin{align}\label{unoriented generators}
		\begin{tikzpicture}[tqft/cobordism/.style={draw,thick},
			tqft/view from=outgoing, tqft/boundary separation=30pt,
			tqft/cobordism height=40pt, tqft/circle x radius=8pt,
			tqft/circle y radius=4.5pt, tqft/every boundary component/.style={rotate=90}
			]
			\pic[tqft/cylinder,rotate=90,name=a,anchor={(1,0)}, every incoming
			boundary component/.style={draw,dotted,thick},every outgoing
			boundary component/.style={draw,thick}
			];
			\node at ([yshift=-20pt,xshift=-20pt]a-outgoing boundary 1){\small{Id}};
			\pic[tqft/pair of pants,
			rotate=90,name=b,at=(a-outgoing boundary),anchor={(1,-0.5)}, every incoming
			boundary component/.style={draw,dotted,thick},every outgoing
			boundary component/.style={draw,thick}];
			\node at ([yshift=-15pt,xshift=-20pt]b-outgoing boundary 1){\small{$\Delta$}};
			\pic[tqft/reverse pair of pants, name=c,rotate=90,at=(b-outgoing boundary),anchor={(1,-0.5)},every incoming
			boundary component/.style={draw,thick,dotted},every outgoing
			boundary component/.style={draw,thick}];
			\node at ([yshift=-18pt,xshift=20pt]c-incoming boundary 1){\small{$m$}};
			\pic[tqft/cup, 
			rotate=90,name=d,at=(c-outgoing boundary),anchor={(1,-0.5)},every incoming
			boundary component/.style={draw,dotted,thick},every outgoing
			boundary component/.style={draw,thick}];
			\node at ([yshift=-18pt,xshift=2pt]d-incoming boundary 1){\small{$\varepsilon$}};
			\pic[tqft/cap, 
			rotate=90,name=e,anchor={(1,-4.3)}, every incoming
			boundary component/.style={draw,dotted,thick},every outgoing
			boundary component/.style={draw,thick}];
			\node at ([yshift=-18pt,xshift=-2pt]e-outgoing boundary 1){\small{$u$}};
			\pic[tqft/cylinder to next,rotate=90, anchor=incoming boundary 1,boundary separation=60pt, anchor={(1.07,-5.8)},every incoming
			boundary component/.style={draw,dotted,thick},every outgoing
			boundary component/.style={draw,thick}];
			\pic[tqft/cylinder to prior,rotate=90, anchor={(0.6,-5.8)},boundary separation=60pt,every incoming
			boundary component/.style={draw,dotted,thick},every outgoing
			boundary component/.style={draw,thick}];
			\pic[tqft/cylinder,rotate=90,name=k,anchor={(1,-7.3)},every incoming
			boundary component/.style={draw,dotted,thick},every outgoing
			boundary component/.style={draw,thick}];
			\node at ([yshift=-20pt,xshift=-80pt]k-outgoing boundary 1){\small{$\tau$}};
			\node at ([yshift=-18pt,xshift=-20pt]k-outgoing boundary 1){\small{$\phi$}};
			\node at ([xshift=20pt]k-incoming boundary 1){$\leftrightarrow$};
			\pic[tqft/cap, 
			rotate=90,name=d,at=(a-outgoing boundary),anchor={(1,-7)},every incoming
			boundary component/.style={draw,dotted,thick},every outgoing
			boundary component/.style={draw,thick}];
			\node at ([xshift=-7pt]d-outgoing boundary 1) {$\figureXv$};
			\node at ([xshift=10pt]d-outgoing boundary 1) {.};
			\node at ([yshift=-18pt,xshift=-2pt]d-outgoing boundary 1){\small{$\theta$}};
		\end{tikzpicture}
	\end{align}
	The first 6 cobordisms are the usual generators for $\text{Cob}_2$, the category of oriented 2-dimensional cobordisms. As is traditional, we will refer to these morphisms,  from left to right, in the following way: identity, pair of pants, reverse pair of pants, cap, cup and twist. The last 2 cobordisms are new:  $\phi:1 \to 1$ denotes  the orientation reversing diffeomorphism of the circle, see Example \ref{example}, and $\theta: 0 \to 1$ the once punctured projective plane, also called the Möbius or \emph{crosscap} cobordism, which is non-orientable.
	
	We include below the relations, as described in \cite{T, TT}.

	$\diamondsuit$ Associativity and coassociativity:
	\begin{align*}
		&\begin{aligned}
			\resizebox{140pt}{!}{%

			}
		\end{aligned}
	\end{align*}
	
	$\diamondsuit$ Unit and counit:
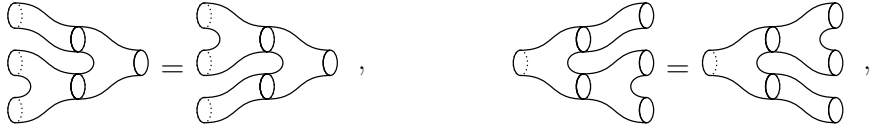
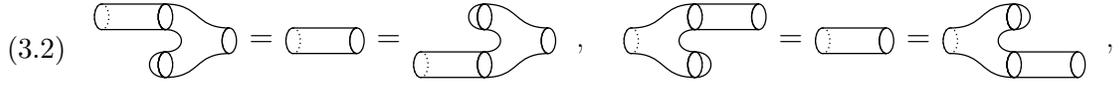
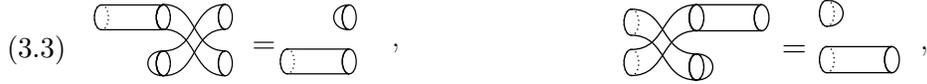
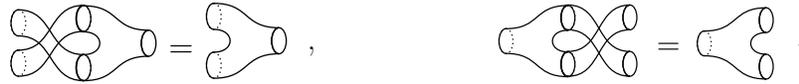
	\begin{align}\label{unit and counit}
		&\begin{aligned}
			\resizebox{190pt}{!}{%
				%
			}
		\end{aligned}
	\end{align}
	
	$\diamondsuit$ Commutativity and cocommutativity:
	\begin{align*}
		&\begin{aligned}
			\resizebox{120pt}{!}{%
				%
			}
		\end{aligned}
	\end{align*}
	
	$\diamondsuit$ First and second permutation relations:
	\begin{align*}
		&\begin{aligned}
			\resizebox{110pt}{!}{%
				%
			}
		\end{aligned},
	\end{align*}
	$\diamondsuit$ Third permutation relations:
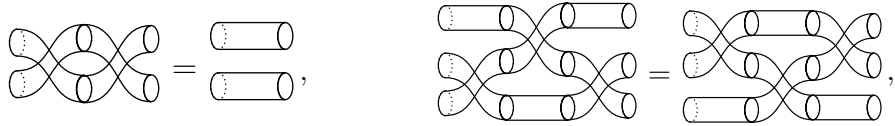
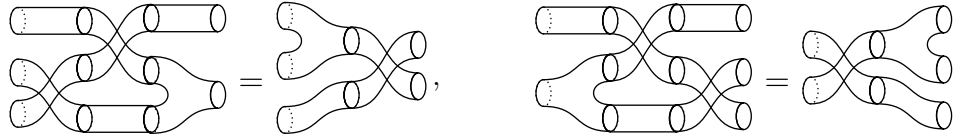
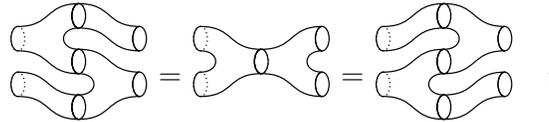
	\begin{align*}
		&\begin{aligned}
			\resizebox{160pt}{!}{%
				%
			}
		\end{aligned}
	\end{align*}
	
	$\diamondsuit$ $\phi$ is an involution:
	\begin{center}
		\resizebox{110pt}{!}{%
			%
		}
	\end{center}
	
	$\diamondsuit$ $\phi$ is multiplicative and comultiplicative:
	\begin{align}\label{phi is mult and comult}
		&\begin{aligned}
			\resizebox{150pt}{!}{%
				%
			}
		\end{aligned}
	\end{align}
	
	$\diamondsuit$ $\phi$ is unital and counital:
	\begin{align}\label{phi is unitual and counital}
		&\begin{aligned}
			\resizebox{80pt}{!}{%
				%
			}
		\end{aligned}
	\end{align}
	
	$\diamondsuit$ Extended Frobenius relations:
	\begin{align}
		&	\begin{aligned}
			\label{crosscap and involution}
			\resizebox{120pt}{!}{%
				%
 }
		\end{aligned}
	\end{align}

	Using the previous relations we can deduce the following one,
	\begin{align}\label{3 crosscaps}
		\begin{aligned}
			\resizebox{130pt}{!}{%
				%
 }
		\end{aligned}
	\end{align}
	which allows us to identify a triple crosscap with a handle with a unique crosscap. 
	
	Next we work towards a description of $\Hom_{\text{UCob}_2}(m,n)$ for $m,n\geq 0$. A morphism from $m$ to $n$ has a finite number of closed connected components,  and a finite number of connected components with boundary. Each of these can be orientable or unorientable. We consider orientable and unorientable connected components separately in what follows. 
	
	\begin{remark}
		Note that $\Hom_{\UCob}(m,n) \simeq \Hom_{\UCob}(0,n+m)$. This isomorphism is given by bending the left $n$ circles to the rigth, see the following figure:
	\end{remark}
	
	\begin{align*}
		\resizebox{270pt}{!}
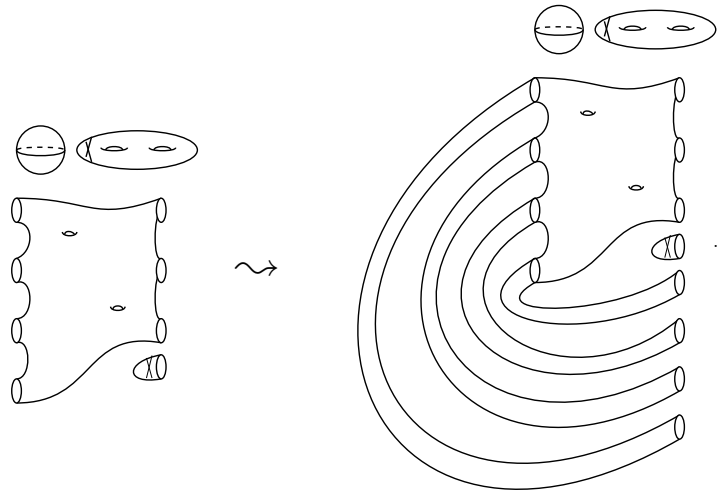
{%
			\begin{tikzpicture}[scale=0.6]
				\begin{scope}[scale=0.8]
					\node at (5.5,-4) {\figureA};
					\node at (10,0) [font=\Huge]{$\leadsto$};
					\draw[fill=white,style=thick] (5, 5) ellipse (2.5cm and 0.8cm);
					\draw[style=thick] (3.7,5) to[out=90, looseness = 0.7,in=90] (4.4,5); 
					\draw[style=thick] (3.5,5.1) to[out=270, looseness = 0.4,in=270] (4.6,5.1);
					\draw[style=thick] (5.7,5) to[out=90, looseness = 0.7,in=90] (6.4,5); 
					\draw[style=thick] (5.5,5.1) to[out=270, looseness = 0.4,in=270] (6.6,5.1);
					\node[above, font={\small}] at (3,4.2)  {\figureXv};
					\draw[fill=white,style=thick] (1, 5) ellipse (1cm and 1cm);
					\draw[style=thick] (0,5) to[out=270, looseness = 0.4,in=270] (2,5);
					\draw[style=thick,dashed] (0,5) to[out=90, looseness = 0.2,in=90] (2,5); 
					\draw[style=thick] (0,-0.5) to (0,0.5) to[out=20, looseness =1.3,in=0] (0,2) ;
					\draw[style=thick] (0,-0.5) to[out=20, looseness =1.3,in=0] (0,-2);
					\draw[style=thick] (6,-3)  to [out=170, looseness =1.3,in=0] (0,-5.5);
					\draw[style=thick] (0,-3) to[out=20, looseness =1.1,in=0] (0,-4.5);
					\draw[style=thick] (6,-2) to[out=-150, looseness =0.5,in=-150] (6,0);
					\draw[style=thick] (6,2.5) to[out=-150, looseness =0.5,in=-150] (6,0.5);
					\draw[style=thick] (0,3)to [out=0, looseness =1.3,in=-160] (6,3);
					\draw[style=thick] (2.4,1.5) to[out=120, looseness = 1.1,in=90] (2,1.5); 
					\draw[style=thick] (2.5,1.6) to[out=270, looseness = 0.9,in=270] (1.9,1.6); 
					\draw[style=thick] (4.4,-1.6) to[out=90, looseness = 1.1,in=90] (4,-1.6); 
					\draw[style=thick] (4.5,-1.5) to[out=270, looseness = 0.9,in=270] (3.9,-1.5); 
					\draw[style=thick] (6,-4.5)  to[out=190, looseness =4,in=190] (6,-3.5);
					\draw[fill=white,style=thick](6,-4) ellipse (0.2cm and 0.5cm);
					\draw[fill=white,style=thick](6,-2.5) ellipse (0.2cm and 0.5cm);
					\draw[fill=white,style=thick](0,-2.5) ellipse (0.2cm and 0.5cm);
					\draw[fill=white,style=thick](0,-5) ellipse (0.2cm and 0.5cm);
					\draw[fill=white,style=thick] (0,0) ellipse (0.2cm and 0.5cm);
					\draw[fill=white,style=thick](0,2.5) ellipse (0.2cm and 0.5cm);
					\draw[fill=white,style=thick] (6,2.5)ellipse (0.2cm and 0.5cm);
					\draw[fill=white,style=thick] (6,0)ellipse (0.2cm and 0.5cm);
				\end{scope}   
				\begin{scope}[scale=0.8,shift={(21.5,5)}]
					\node at (5.5,-4) {\figureA};
					\node at (7.5,-4) {.};
					\draw[fill=white,style=thick] (5, 5) ellipse (2.5cm and 0.8cm);
					\draw[style=thick] (3.7,5) to[out=90, looseness = 0.7,in=90] (4.4,5); 
					\draw[style=thick] (3.5,5.1) to[out=270, looseness = 0.4,in=270] (4.6,5.1);
					\draw[style=thick] (5.7,5) to[out=90, looseness = 0.7,in=90] (6.4,5); 
					\draw[style=thick] (5.5,5.1) to[out=270, looseness = 0.4,in=270] (6.6,5.1);
					\node[above, font={\small}] at (3,4.2)  {\figureXv};
					\draw[fill=white,style=thick] (1, 5) ellipse (1cm and 1cm);
					\draw[style=thick] (0,5) to[out=270, looseness = 0.4,in=270] (2,5);
					\draw[style=thick,dashed] (0,5) to[out=90, looseness = 0.2,in=90] (2,5); 
					\draw[style=thick] (0,-0.5) to (0,0.5) to[out=20, looseness =1.3,in=0] (0,2) ;
					\draw[style=thick] (0,-0.5) to[out=20, looseness =1.3,in=0] (0,-2);
					\draw[style=thick] (6,-3)  to [out=170, looseness =1.3,in=0] (0,-5.5);
					\draw[style=thick] (0,-3) to[out=20, looseness =1.3,in=0] (0,-4.5);
					\draw[style=thick] (6,-2) to[out=-150, looseness =0.5,in=-150] (6,0);
					\draw[style=thick] (6,2.5) to[out=-150, looseness =0.5,in=-150] (6,0.5);
					\draw[style=thick] (0,3)to [out=0, looseness =1.3,in=-160] (6,3);
					\draw[style=thick] (2.4,1.5) to[out=120, looseness = 1.1,in=90] (2,1.5); 
					\draw[style=thick] (2.5,1.6) to[out=270, looseness = 0.9,in=270] (1.9,1.6); 
					\draw[style=thick] (4.4,-1.6) to[out=90, looseness = 1.1,in=90] (4,-1.6); 
					\draw[style=thick] (4.5,-1.5) to[out=270, looseness = 0.9,in=270] (3.9,-1.5); 
					\draw[style=thick] (6,-4.5)  to[out=190, looseness =4,in=190] (6,-3.5);
					\draw[fill=white,style=thick](6,-4) ellipse (0.2cm and 0.5cm);
					\draw[fill=white,style=thick](6,-2.5) ellipse (0.2cm and 0.5cm);
					\draw[fill=white,style=thick](0,-2.5) ellipse (0.2cm and 0.5cm);
					\draw[fill=white,style=thick](0,-5) ellipse (0.2cm and 0.5cm);
					\draw[fill=white,style=thick] (0,0) ellipse (0.2cm and 0.5cm);
					\draw[fill=white,style=thick](0,2.5) ellipse (0.2cm and 0.5cm);
					\draw[fill=white,style=thick] (6,2.5)ellipse (0.2cm and 0.5cm);
					\draw[fill=white,style=thick] (6,0)ellipse (0.2cm and 0.5cm);
					\draw[style=thick] (0,-4.5)  to[out=-150, looseness =2,in=-150] (6,-6);
					\draw[style=thick] (0,-5.5)  to[out=-150, looseness =1.3,in=-140] (6,-5);	
					\draw[fill=white,style=thick] (6,-5.5)ellipse (0.2cm and 0.5cm);
					\draw[style=thick] (0,-2)  to[out=-150, looseness =2.4,in=-150] (6,-8);
					\draw[style=thick] (0,-3)  to[out=-150, looseness =2.3,in=-140] (6,-7);	
					\draw[fill=white,style=thick] (6,-7.5)ellipse (0.2cm and 0.5cm);
					\draw[style=thick] (0,0.5)  to[out=-150, looseness =2.3,in=-150] (6,-10);
					\draw[style=thick] (0,-0.5)  to[out=-150, looseness =2.4,in=-150] (6,-9);	
					\draw[fill=white,style=thick] (6,-9.5)ellipse (0.2cm and 0.5cm);
					\draw[style=thick] (0,3)  to[out=-150, looseness =2.4,in=-150] (6,-12);
					\draw[style=thick] (0,2)  to[out=-150, looseness =2.5,in=-150] (6,-11);	
					\draw[fill=white,style=thick] (6,-11.5)ellipse (0.2cm and 0.5cm);		
		\end{scope}   \end{tikzpicture}}
	\end{align*} 	
	
	\vspace{-2.3cm}
	\subsection{Orientable connected cobordisms with boundary}\label{Orientable connected cobordisms with boundary}
	The goal of this section is to give a graphical description of orientable connected morphisms, which will be useful later on. We note that orientable morphisms in $\UCob$ do not contain any crosscaps. 
	
	What we will be doing in the following Proposition is similar to what is done in \cite[Section 1.4.16]{Ko}.
	
	\begin{proposition}\label{prop:orientable with boundary}
		Any orientable connected cobordism  with boundary $m\to n$ in $\UCob$ can be decomposed into three parts, which we will call the \emph{in, mid} and \emph{out} parts:
		
		$\diamondsuit$ The \emph{in} part consists of:
		\begin{itemize}
			\item If $m=0$, a cup.
			\item If $m>0$, a composition of cylinders and reverse pairs of pants. The cobordism starts with $m$ stacked cylinders, each of which is either the identity or the involution $\phi$. Following the cylinders we have a composition of multiplication maps $m\to 1$. 
		\end{itemize}

		$\diamondsuit$ The \emph{mid} part consists of a composition of \emph{handles} \resizebox{40pt}{!}{%
			\begin{tikzpicture}[tqft/cobordism/.style={draw,thick},
				tqft/view from=outgoing, tqft/boundary separation=30pt,
				tqft/cobordism height=40pt, tqft/circle x radius=8pt,
				tqft/circle y radius=4.5pt, tqft/every boundary component/.style={draw,rotate=90},tqft/every incoming
				boundary component/.style={draw,dotted,thick},tqft/every outgoing
				boundary component/.style={draw,thick}]
				\pic[tqft/pair of pants, 
				rotate=90,name=a,];
				\pic[tqft/reverse pair of pants,
				rotate=90,name=j,at=(a-outgoing boundary)]; 
		\end{tikzpicture}}. The number of handles is unique and gives the genus of the cobordism. When the genus is zero, the  \emph{mid} part is empty.
		
		$\diamondsuit$  The \emph{out} part consists of:
		\begin{itemize}
			\item If $n=0$, a cap.
			\item If $n>0$, a composition of cylinders and pairs of pants. The cobordisms ends with a composition of pairs of pants $1 \to n,$ followed by $n$ stacked cylinders, each of which can be either the identity or $\phi$.
		\end{itemize}
		
	\end{proposition}
	
	\begin{example}
		The following is an  orientable cobordism $4\to 3$ of genus 3, 
		\begin{align*}
			&	\begin{aligned}
				\resizebox{320pt}{!}
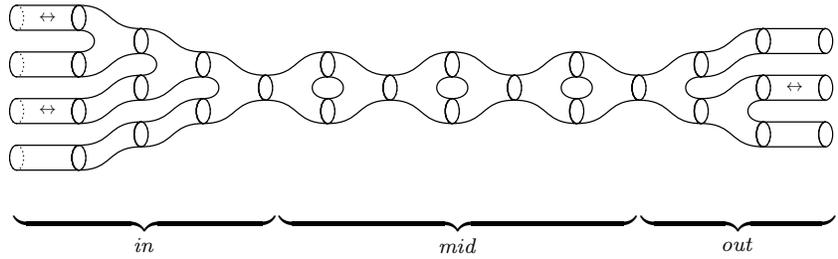
{%
					\begin{tikzpicture}[tqft/cobordism/.style={draw,thick},
						tqft/view from=outgoing, tqft/boundary separation=30pt,
						tqft/cobordism height=40pt, tqft/circle x radius=8pt,
						tqft/circle y radius=4.5pt, tqft/every boundary component/.style={draw,rotate=90},tqft/every incoming
						boundary component/.style={draw,dotted,thick},tqft/every outgoing
						boundary component/.style={draw,thick}]
						\pic[tqft/cylinder,rotate=90,name=a,anchor={(0,0)}];
						\node at ([xshift=20pt]a-incoming boundary 1){$\leftrightarrow$};
						\pic[tqft/cylinder,rotate=90,name=b,anchor={(1,0)}];
						\pic[tqft/cylinder,rotate=90,name=c,anchor={(2,0)}];
						\node at ([xshift=20pt]c-incoming boundary 1){$\leftrightarrow$};
						\pic[tqft/cylinder,rotate=90,name=d,anchor={(3,0)}];
						\pic[tqft/reverse pair of pants,
						rotate=90,name=e,at=(b-outgoing boundary)]; 
						\pic[tqft/cylinder to next,rotate=90,name=f, at=(c-outgoing boundary)];
						\pic[tqft/cylinder to next,rotate=90,name=g, at=(d-outgoing boundary)];
						\pic[tqft/reverse pair of pants,
						rotate=90,name=h,at=(f-outgoing boundary)]; 
						\pic[tqft/cylinder to next,rotate=90,name=i, at=(g-outgoing boundary)];
						\pic[tqft/reverse pair of pants,
						rotate=90,name=j,at=(i-outgoing boundary)]; 
						\pic[tqft/pair of pants,
						rotate=90,name=k,at=(j-outgoing boundary)]; 
						\pic[tqft/reverse pair of pants,
						rotate=90,name=l,at=(k-outgoing boundary)]; 
						\pic[tqft/pair of pants,
						rotate=90,name=m,at=(l-outgoing boundary)]; 
						\pic[tqft/reverse pair of pants,
						rotate=90,name=n,at=(m-outgoing boundary)]; 
						\pic[tqft/pair of pants,
						rotate=90,name=o,at=(n-outgoing boundary)]; 
						\pic[tqft/reverse pair of pants,
						rotate=90,name=p,at=(o-outgoing boundary)]; 
						\pic[tqft/pair of pants,
						rotate=90,name=r,at=(p-outgoing boundary)]; 
						\pic[tqft/cylinder to next,rotate=90,name=u,   at=(r-outgoing boundary 2)];
						\pic[tqft/pair of pants,
						rotate=90,name=s,at=(r-outgoing boundary)]; 
						\pic[tqft/cylinder,rotate=90,name=x, at=(u-outgoing boundary)];
						\pic[tqft/cylinder,rotate=90,name=y, at=(s-outgoing boundary)];
						\pic[tqft/cylinder,rotate=90,name=z, at=(s-outgoing boundary 2)];
						\node at ([xshift=20pt]z-incoming boundary 1){$\leftrightarrow$};
						\node at ([xshift=50pt]z-incoming boundary 1){.};
				\end{tikzpicture} }
			\end{aligned}\\
			&\underbrace{\ \ \ \ \ \ \ \ \ \ \ \ \ \ \ \ \ \ \ \ \ \ \ \ \ \ \ }_\text{\emph{in}} 
			\underbrace{\ \ \ \ \ \ \ \ \ \ \ \ \ \ \ \ \ \ \ \ \ \ \ \ \ \ \ \ \ \ \ \ \ \ \ \ \ }_\text{\emph{mid}} 
			\underbrace{\ \ \ \ \ \ \ \ \ \ \ \ \ \ \ \ \ \ \ \ }_\text{\emph{out}} 
		\end{align*}
	\end{example}

	\begin{proof}
We note that by the realtions involving the twist in $\UCob$, the twist $\tau$ can always be vanished from connected cobordisms, see also \cite[Section 1.4.16]{Ko}.	Start with any connected orientable cobordism $m\to n$. Our first step will be to move all shapes of the form
		\begin{align}\label{shape 1}
			&	\begin{aligned}
				\resizebox{30pt}{!}{%
					\begin{tikzpicture}[tqft/cobordism/.style={draw,thick},
						tqft/view from=outgoing, tqft/boundary separation=30pt,
						tqft/cobordism height=40pt, tqft/circle x radius=8pt,
						tqft/circle y radius=4.5pt, tqft/every boundary component/.style={draw,rotate=90},tqft/every incoming
						boundary component/.style={draw,dotted,thick},tqft/every outgoing
						boundary component/.style={draw,thick}]
						\pic[tqft/reverse pair of pants,
						rotate=90]; 
				\end{tikzpicture} }
			\end{aligned}
			&&	&	\begin{aligned}
				\resizebox{50pt}{!}{%
					\begin{tikzpicture}[tqft/cobordism/.style={draw,thick},
						tqft/view from=outgoing, tqft/boundary separation=30pt,
						tqft/cobordism height=40pt, tqft/circle x radius=8pt,
						tqft/circle y radius=4.5pt, tqft/every boundary component/.style={draw,rotate=90},tqft/every incoming
						boundary component/.style={draw,dotted,thick},tqft/every outgoing
						boundary component/.style={draw,thick}]
						\pic[tqft/cylinder,rotate=90,name=a,anchor={(0,0)}];
						\node at ([xshift=20pt]a-incoming boundary 1){$   \leftrightarrow$};
						\pic[tqft/cylinder,rotate=90,name=b,anchor={(1,0)}];
						\pic[tqft/reverse pair of pants,
						rotate=90,name=e,at=(b-outgoing boundary)]; 
				\end{tikzpicture} }
			\end{aligned}
			&&&&	\begin{aligned}
				\resizebox{50pt}{!}{%
					\begin{tikzpicture}[tqft/cobordism/.style={draw,thick},
						tqft/view from=outgoing, tqft/boundary separation=30pt,
						tqft/cobordism height=40pt, tqft/circle x radius=8pt,
						tqft/circle y radius=4.5pt, tqft/every boundary component/.style={draw,rotate=90},tqft/every incoming
						boundary component/.style={draw,dotted,thick},tqft/every outgoing
						boundary component/.style={draw,thick}]
						\pic[tqft/cylinder,rotate=90,name=a,anchor={(0,0)}];
						\node at ([xshift=20pt]b-incoming boundary 1){$   \leftrightarrow$};
						\pic[tqft/cylinder,rotate=90,name=b,anchor={(1,0)}];
						\pic[tqft/reverse pair of pants,
						rotate=90,name=e,at=(b-outgoing boundary)]; 
				\end{tikzpicture} }
			\end{aligned}
			&&&&\begin{aligned}
				\resizebox{50pt}{!}{%
					\begin{tikzpicture}[tqft/cobordism/.style={draw,thick},
						tqft/view from=outgoing, tqft/boundary separation=30pt,
						tqft/cobordism height=40pt, tqft/circle x radius=8pt,
						tqft/circle y radius=4.5pt, tqft/every boundary component/.style={draw,rotate=90},tqft/every incoming
						boundary component/.style={draw,dotted,thick},tqft/every outgoing
						boundary component/.style={draw,thick}]
						\pic[tqft/cylinder,rotate=90,name=a,anchor={(0,0)}];
						\node at ([xshift=20pt]a-incoming boundary 1){$   \leftrightarrow$};
						\node at ([xshift=20pt]b-incoming boundary 1){$   \leftrightarrow$};
						\pic[tqft/cylinder,rotate=90,name=b,anchor={(1,0)}];
						\pic[tqft/reverse pair of pants,
						rotate=90,name=e,at=(b-outgoing boundary)]; 
				\end{tikzpicture} }
			\end{aligned},
		\end{align}
		if any, to the left of the cobordism, so that they consitute the \emph{in} part. In order to do so, we must compute the composition of these shapes with the ones we may have to their left, that is, a cup, a pair of pants, or an orientation reversing cylinder (recall that this cobordism has no crosscaps, as it is orientable). We list below all possible cases.
		
		$\star$ Precomposition with a cup:
		\begin{align*}
			&	\begin{aligned}
				\resizebox{70pt}{!}{%
					\begin{tikzpicture}[tqft/cobordism/.style={draw,thick},
						tqft/view from=outgoing, tqft/boundary separation=30pt,
						tqft/cobordism height=40pt, tqft/circle x radius=8pt,
						tqft/circle y radius=4.5pt, tqft/every boundary component/.style={draw,rotate=90},tqft/every incoming
						boundary component/.style={draw,dotted,thick},tqft/every outgoing
						boundary component/.style={draw,thick}]
						\pic[tqft/reverse pair of pants,
						rotate=90,name=e]; 
						\pic[tqft/cap,
						rotate=90,name=d,at=(e-incoming boundary), anchor={(0,1)}]; 
						\node at ([xshift=20pt]e-outgoing boundary 1) [font=\huge] {\(=\)};
						\pic[tqft/cylinder,
						rotate=90,name=f, anchor={(0.5,-2)}]; 
				\end{tikzpicture} }
			\end{aligned}
			&&&	\begin{aligned}
				\resizebox{100pt}{!}{%
					\begin{tikzpicture}[tqft/cobordism/.style={draw,thick},
						tqft/view from=outgoing, tqft/boundary separation=30pt,
						tqft/cobordism height=40pt, tqft/circle x radius=8pt,
						tqft/circle y radius=4.5pt, tqft/every boundary component/.style={draw,rotate=90},tqft/every incoming
						boundary component/.style={draw,dotted,thick},tqft/every outgoing
						boundary component/.style={draw,thick}]
						\pic[tqft/cylinder,rotate=90,name=a,anchor={(0,0)}];
						\node at ([xshift=20pt]a-incoming boundary 1){$   \leftrightarrow$};
						\pic[tqft/cylinder,rotate=90,name=b,anchor={(1,0)}];
						\pic[tqft/reverse pair of pants,
						rotate=90,name=e,at=(b-outgoing boundary)]; 
						\pic[tqft/cap,
						rotate=90,name=d,at=(a-outgoing boundary),anchor={(1,2)}]; 
						\node at ([xshift=20pt]e-outgoing boundary 1) [font=\huge] {\(=\)};
						\pic[tqft/cylinder,
						rotate=90,name=f,at=(e-outgoing boundary), anchor={(1,-1)}]; 
				\end{tikzpicture} }
			\end{aligned}
			&&&&	\begin{aligned}
				\resizebox{100pt}{!}{%
					\begin{tikzpicture}[tqft/cobordism/.style={draw,thick},
						tqft/view from=outgoing, tqft/boundary separation=30pt,
						tqft/cobordism height=40pt, tqft/circle x radius=8pt,
						tqft/circle y radius=4.5pt, tqft/every boundary component/.style={draw,rotate=90},tqft/every incoming
						boundary component/.style={draw,dotted,thick},tqft/every outgoing
						boundary component/.style={draw,thick}]
						\node at ([xshift=20pt]b-incoming boundary 1){$   \leftrightarrow$};
						\pic[tqft/cylinder,rotate=90,name=b,anchor={(1,0)}];
						\pic[tqft/reverse pair of pants,
						rotate=90,name=e,at=(b-outgoing boundary)]; 
						\pic[tqft/cap,
						rotate=90,name=d,at=(e-outgoing boundary),anchor={(0.5,2)}]; 
						\node at ([xshift=20pt]e-outgoing boundary 1) [font=\huge] {\(=\)};
						\pic[tqft/cylinder,
						rotate=90,name=f,at=(e-outgoing boundary), anchor={(1,-1)}]; 
						\node at ([xshift=20pt]f-incoming boundary 1){$   \leftrightarrow$};
				\end{tikzpicture} }
			\end{aligned}
			&&&	\begin{aligned}
				\resizebox{100pt}{!}{%
					\begin{tikzpicture}[tqft/cobordism/.style={draw,thick},
						tqft/view from=outgoing, tqft/boundary separation=30pt,
						tqft/cobordism height=40pt, tqft/circle x radius=8pt,
						tqft/circle y radius=4.5pt, tqft/every boundary component/.style={draw,rotate=90},tqft/every incoming
						boundary component/.style={draw,dotted,thick},tqft/every outgoing
						boundary component/.style={draw,thick}]
						\pic[tqft/cylinder,rotate=90,name=a,anchor={(0,0)}];
						\node at ([xshift=20pt]a-incoming boundary 1){$   \leftrightarrow$};
						\node at ([xshift=20pt]b-incoming boundary 1){$   \leftrightarrow$};
						\pic[tqft/cylinder,rotate=90,name=b,anchor={(1,0)}];
						\pic[tqft/reverse pair of pants,
						rotate=90,name=e,at=(b-outgoing boundary)]; 
						\pic[tqft/cap,
						rotate=90,name=d,at=(a-outgoing boundary),anchor={(1,2)}]; 
						\node at ([xshift=20pt]e-outgoing boundary 1) [font=\huge] {\(=\)};
						\pic[tqft/cylinder,
						rotate=90,name=f,at=(e-outgoing boundary), anchor={(1,-1)}]; 
						\node at ([xshift=20pt]f-incoming boundary 1){$   \leftrightarrow$};
				\end{tikzpicture} }
			\end{aligned}.
		\end{align*}
		We get cylinders in every case, which we will move to the outmost left in the last step. 
		
		$\star$ Precomposition with a pair of pants, forming a handle:
		
		If  the shape \eqref{shape 1} has one orientation reversing cylinder, the composition would result in a non-orientable surface:
		\begin{align*}
			&\begin{aligned}
				\resizebox{100pt}{!}{%
					\begin{tikzpicture}[tqft/cobordism/.style={draw,thick},
						tqft/view from=outgoing, tqft/boundary separation=30pt,
						tqft/cobordism height=40pt, tqft/circle x radius=8pt,
						tqft/circle y radius=4.5pt, tqft/every boundary component/.style={draw,rotate=90},tqft/every incoming
						boundary component/.style={draw,dotted,thick},tqft/every outgoing
						boundary component/.style={draw,thick}]
						\pic[tqft/cylinder,rotate=90,name=a,anchor={(0,0)}];
						\node at ([xshift=20pt]a-incoming boundary 1)[color=black]{$   \leftrightarrow$};
						\pic[tqft/cylinder,rotate=90,name=b,anchor={(-1,0)}];
						\pic[tqft/reverse pair of pants,
						rotate=90,name=e,at=(a-outgoing boundary)]; 
						\pic[tqft/pair of pants,
						rotate=90,name=d,at=(a-outgoing boundary),anchor={(1,2)}]; 
						\node at ([xshift=20pt]e-outgoing boundary 1) [font=\huge] {\(=\)};
						\pic[tqft/cap, 
						rotate=90,name=z,anchor={(1.5,0)}, at=(e-outgoing boundary)];
						\node at ([xshift=-7pt]z-outgoing boundary 1) {$\figureXv$};
						\pic[tqft/reverse pair of pants,
						rotate=90,name=f,at=(z-outgoing boundary)]; 
						\pic[tqft/cap, 
						rotate=90,name=g,at=(f-incoming boundary 2),anchor={(1,1)}];
						\node at ([xshift=-7pt]g-outgoing boundary 1) {$\figureXv$};
				\end{tikzpicture} }
			\end{aligned}
			&&	\begin{aligned}
				\resizebox{100pt}{!}{%
					\begin{tikzpicture}[tqft/cobordism/.style={draw,thick},
						tqft/view from=outgoing, tqft/boundary separation=30pt,
						tqft/cobordism height=40pt, tqft/circle x radius=8pt,
						tqft/circle y radius=4.5pt, tqft/every boundary component/.style={draw,rotate=90},tqft/every incoming
						boundary component/.style={draw,dotted,thick},tqft/every outgoing
						boundary component/.style={draw,thick}]
						\pic[tqft/cylinder,rotate=90,name=a,anchor={(0,0)}];
						\node at ([xshift=20pt]b-incoming boundary 1)[color=black]{$   \leftrightarrow$};
						\pic[tqft/cylinder,rotate=90,name=b,anchor={(-1,0)}];
						\pic[tqft/reverse pair of pants,
						rotate=90,name=e,at=(a-outgoing boundary)]; 
						\pic[tqft/pair of pants,
						rotate=90,name=d,at=(a-outgoing boundary),anchor={(1,2)}]; 
						\node at ([xshift=20pt]e-outgoing boundary 1) [font=\huge] {\(=\)};
						\node at ([xshift=90pt]e-outgoing boundary 1)  {.};
						\pic[tqft/cap, 
						rotate=90,name=z,anchor={(1.5,0)}, at=(e-outgoing boundary)];
						\node at ([xshift=-7pt]z-outgoing boundary 1) {$\figureXv$};
						\pic[tqft/reverse pair of pants,
						rotate=90,name=f,at=(z-outgoing boundary)]; 
						\pic[tqft/cap, 
						rotate=90,name=g,at=(f-incoming boundary 2),anchor={(1,1)}];
						\node at ([xshift=-7pt]g-outgoing boundary 1) {$\figureXv$};
				\end{tikzpicture} }
			\end{aligned},
		\end{align*}
		So this case is not possible.

		If the shape \eqref{shape 1} has two orientation reversing cylinders, we can use comultiplicativity of $\phi$ to get an involution composed with a handle,
		\begin{align*}
			\begin{aligned}
				\resizebox{150pt}{!}{%
					\begin{tikzpicture}[tqft/cobordism/.style={draw,thick},
						tqft/view from=outgoing, tqft/boundary separation=30pt,
						tqft/cobordism height=40pt, tqft/circle x radius=8pt,
						tqft/circle y radius=4.5pt, tqft/every boundary component/.style={draw,rotate=90},tqft/every incoming
						boundary component/.style={draw,dotted,thick},tqft/every outgoing
						boundary component/.style={draw,thick}]
						\pic[tqft/cylinder,rotate=90,name=a,anchor={(0,0)}];
						\node at ([xshift=20pt]a-incoming boundary 1)[color=black]{$   \leftrightarrow$};
						\pic[tqft/cylinder,rotate=90,name=b,anchor={(-1,0)}];
						\node at ([xshift=20pt]b-incoming boundary 1)[color=black]{$   \leftrightarrow$};
						\pic[tqft/reverse pair of pants,
						rotate=90,name=e,at=(a-outgoing boundary)]; 
						\pic[tqft/pair of pants,
						rotate=90,name=d,at=(a-outgoing boundary),anchor={(1,2)}]; 
						\node at ([xshift=20pt]e-outgoing boundary 1) [font=\huge] {\(=\)};
						\pic[tqft/cylinder,rotate=90,name=g, at=(e-outgoing boundary),anchor={(1,-1)}];
						\node at ([xshift=20pt]g-incoming boundary 1)[color=black]{$   \leftrightarrow$};
						\pic[tqft/pair of pants,
						rotate=90,name=h,at=(g-outgoing boundary)]; 
						\pic[tqft/reverse pair of pants,
						rotate=90,name=i,at=(h-outgoing boundary)];  
				\end{tikzpicture} }
			\end{aligned},
		\end{align*}
	see relations \eqref{phi is mult and comult}. The handle will become part of the \emph{mid} part of the cobordism, and we will move the orientation reversing cylinder to the outmost left in the last step. 
		
		If the shape \eqref{shape 1} has no orientation reversing cylinders, similarly to the previous case, precomposition with a pair of pants forms a handle, which stays in the \emph{mid} part of the cobordism. 
		
		$\star$ Precomposition with a pair of pants, without forming a handle: it will look like one of the following,
		\begin{align*}
			&	\begin{aligned}
				\resizebox{120pt}{!}
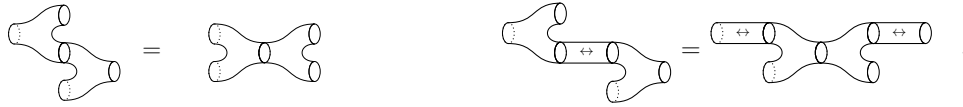
{%
					\begin{tikzpicture}[tqft/cobordism/.style={draw,thick},
						tqft/view from=outgoing, tqft/boundary separation=30pt,
						tqft/cobordism height=40pt, tqft/circle x radius=8pt,
						tqft/circle y radius=4.5pt, tqft/every boundary component/.style={draw,rotate=90},tqft/every incoming
						boundary component/.style={draw,dotted,thick},tqft/every outgoing
						boundary component/.style={draw,thick}]
						\pic[tqft/reverse pair of pants,
						rotate=90,name=e]; 
						\pic[tqft/pair of pants,
						rotate=90,name=d,at=(e-incoming boundary 2),anchor={(1,1)}]; 
						\node at ([xshift=30pt,yshift=15pt]e-outgoing boundary 1) [font=\huge] {\(=\)};
						\pic[tqft/reverse pair of pants,
						rotate=90,name=h,at=(d-outgoing boundary),anchor={(1.5,-3)}];				
						\pic[tqft/pair of pants,
						rotate=90,name=i,at=(h-outgoing boundary)];		
				\end{tikzpicture} }
			\end{aligned}
			&&&\begin{aligned}
				\resizebox{180pt}{!}{%
					\begin{tikzpicture}[tqft/cobordism/.style={draw,thick},
						tqft/view from=outgoing, tqft/boundary separation=30pt,
						tqft/cobordism height=40pt, tqft/circle x radius=8pt,
						tqft/circle y radius=4.5pt, tqft/every boundary component/.style={draw,rotate=90},tqft/every incoming
						boundary component/.style={draw,dotted,thick},tqft/every outgoing
						boundary component/.style={draw,thick}]
						\pic[tqft/cylinder,rotate=90,name=b,anchor={(-1,0)}];
						\node at ([xshift=20pt]b-incoming boundary 1)[color=black]{$   \leftrightarrow$};
						\pic[tqft/reverse pair of pants,
						rotate=90,name=e,at=(a-outgoing boundary)]; 
						\pic[tqft/pair of pants,
						rotate=90,name=d,at=(b-outgoing boundary),anchor={(1,2)}]; 
						\node at ([xshift=20pt,yshift=15pt]e-outgoing boundary 1) [font=\huge] {\(=\)};
						\node at ([xshift=230pt,yshift=15pt]e-outgoing boundary 1) [font=\huge] {\(.\)};
						\pic[tqft/cylinder,rotate=90,name=g, at=(e-outgoing boundary),anchor={(0,-1)}];
						\node at ([xshift=20pt]g-incoming boundary 1)[color=black]{$   \leftrightarrow$};
						\pic[tqft/reverse pair of pants,
						rotate=90,name=h,at=(g-outgoing boundary),anchor={(2,0)}];				
						\pic[tqft/pair of pants,
						rotate=90,name=i,at=(h-outgoing boundary)];		
						\pic[tqft/cylinder,rotate=90,name=l, at=(i-outgoing boundary 2)];
						\node at ([xshift=20pt]l-incoming boundary)[color=black]{$   \leftrightarrow$};
				\end{tikzpicture} }
			\end{aligned}
		\end{align*}	
		The picture above shows how we can move shape $\eqref{shape 1}$ to the left in each case. 
		
		$\star$ Precomposition with an orientation reversing cylinder: this just changes the shape from $\eqref{shape 1}$ into one of the other possible ones shown in $\eqref{shape 1}$.

		We have covered all the possible cases for precomposition, and so we are done with moving the shapes $\eqref{shape 1}$ to the left of the cobordism, which completes the first step.That is, all reverse pair of pants are now to the left of the cobordism.
		
		Our second step is to move all figures of the form 
		\begin{align*}
			&	\begin{aligned}
				\resizebox{30pt}{!}
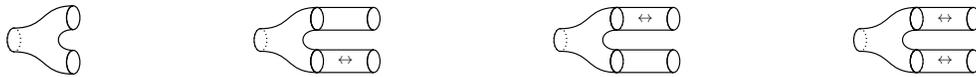
{%
					\begin{tikzpicture}[tqft/cobordism/.style={draw,thick},
						tqft/view from=outgoing, tqft/boundary separation=30pt,
						tqft/cobordism height=40pt, tqft/circle x radius=8pt,
						tqft/circle y radius=4.5pt, tqft/every boundary component/.style={draw,rotate=90},tqft/every incoming
						boundary component/.style={draw,dotted,thick},tqft/every outgoing
						boundary component/.style={draw,thick}]
						\pic[tqft/pair of pants,
						rotate=90]; 
				\end{tikzpicture} }
			\end{aligned}
			&&	&	\begin{aligned}
				\resizebox{50pt}{!}{%
					\begin{tikzpicture}[tqft/cobordism/.style={draw,thick},
						tqft/view from=outgoing, tqft/boundary separation=30pt,
						tqft/cobordism height=40pt, tqft/circle x radius=8pt,
						tqft/circle y radius=4.5pt, tqft/every boundary component/.style={draw,rotate=90},tqft/every incoming
						boundary component/.style={draw,dotted,thick},tqft/every outgoing
						boundary component/.style={draw,thick}]
						\pic[tqft/pair of pants,
						rotate=90,name=e]; 
						\pic[tqft/cylinder,rotate=90,name=a,at=(e-outgoing boundary 1)];
						\node at ([xshift=20pt]a-incoming boundary 1){$   \leftrightarrow$};
						\pic[tqft/cylinder,rotate=90,name=b,at=(e-outgoing boundary 2)];
				\end{tikzpicture} }
			\end{aligned}
			&&&&	\begin{aligned}
				\resizebox{50pt}{!}{%
					\begin{tikzpicture}[tqft/cobordism/.style={draw,thick},
						tqft/view from=outgoing, tqft/boundary separation=30pt,
						tqft/cobordism height=40pt, tqft/circle x radius=8pt,
						tqft/circle y radius=4.5pt, tqft/every boundary component/.style={draw,rotate=90},tqft/every incoming
						boundary component/.style={draw,dotted,thick},tqft/every outgoing
						boundary component/.style={draw,thick}]
						\pic[tqft/pair of pants,
						rotate=90,name=e]; 
						\pic[tqft/cylinder,rotate=90,name=a,at=(e-outgoing boundary 1)];
						\node at ([xshift=20pt]b-incoming boundary 1){$   \leftrightarrow$};
						\pic[tqft/cylinder,rotate=90,name=b,at=(e-outgoing boundary 2)];
				\end{tikzpicture} }
			\end{aligned}
			&&&&	\begin{aligned}
				\resizebox{50pt}{!}{%
					\begin{tikzpicture}[tqft/cobordism/.style={draw,thick},
						tqft/view from=outgoing, tqft/boundary separation=30pt,
						tqft/cobordism height=40pt, tqft/circle x radius=8pt,
						tqft/circle y radius=4.5pt, tqft/every boundary component/.style={draw,rotate=90},tqft/every incoming
						boundary component/.style={draw,dotted,thick},tqft/every outgoing
						boundary component/.style={draw,thick}]
						\pic[tqft/pair of pants,
						rotate=90,name=e]; 
						\pic[tqft/cylinder,rotate=90,name=a,at=(e-outgoing boundary 1)];
						\node at ([xshift=20pt]a-incoming boundary 1){$   \leftrightarrow$};
						\node at ([xshift=20pt]b-incoming boundary 1){$   \leftrightarrow$};
						\pic[tqft/cylinder,rotate=90,name=b,at=(e-outgoing boundary 2)];
				\end{tikzpicture} }
			\end{aligned}
		\end{align*}
		to the right. This can be achieved analogously to the previous step.
		
		Finally, we move any orientation reversing cylinder to either the outmost left or right of the cobordism, as follows. 
		
		$\star$ We move any involution precomposed with a handle to its left as in the figure below,
		\begin{align*}
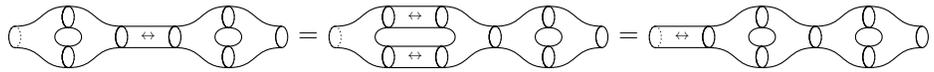
			\begin{aligned}
				\resizebox{350pt}{!}{%
					\begin{tikzpicture}[tqft/cobordism/.style={draw,thick},
						tqft/view from=outgoing, tqft/boundary separation=30pt,
						tqft/cobordism height=40pt, tqft/circle x radius=8pt,
						tqft/circle y radius=4.5pt, tqft/every boundary component/.style={draw,rotate=90},tqft/every incoming
						boundary component/.style={draw,dotted,thick},tqft/every outgoing
						boundary component/.style={draw,thick}]
						\pic[tqft/pair of pants,rotate=90,name=x,anchor={(0,7)}];
						\pic[tqft/reverse pair of pants,rotate=90,name=y,at=(x-outgoing boundary),];
						\pic[tqft/cylinder,rotate=90,name=z,at=(y-outgoing boundary)];
						\node at ([xshift=20pt]z-incoming boundary 1)[color=black]{$   \leftrightarrow$};
						\pic[tqft/pair of pants,rotate=90,name=xx,at=(z-outgoing boundary)];
						\pic[tqft/reverse pair of pants,rotate=90,name=yy,at=(xx-outgoing boundary),];
						\node at ([xshift=20pt]yy-outgoing boundary 1) [font=\huge] {\(=\)};
						\pic[tqft/cylinder,rotate=90,name=a,anchor={(0,0)}];
						\node at ([xshift=20pt]a-incoming boundary 1)[color=black]{$   \leftrightarrow$};
						\pic[tqft/cylinder,rotate=90,name=b,anchor={(-1,0)}];
						\node at ([xshift=20pt]b-incoming boundary 1)[color=black]{$   \leftrightarrow$};
						\pic[tqft/reverse pair of pants,
						rotate=90,name=e,at=(a-outgoing boundary)]; 
						\pic[tqft/pair of pants,
						rotate=90,name=d,at=(a-outgoing boundary),anchor={(1,2)}]; 
						\pic[tqft/pair of pants,rotate=90,name=aa,at=(e-outgoing boundary)];
						\pic[tqft/reverse pair of pants,rotate=90,name=bb,at=(aa-outgoing boundary),];
						\node at ([xshift=20pt]bb-outgoing boundary 1) [font=\huge] {\(=\)};
						\pic[tqft/cylinder,rotate=90,name=g, at=(bb-outgoing boundary),anchor={(1,-1)}];
						\node at ([xshift=20pt]g-incoming boundary 1)[color=black]{$   \leftrightarrow$};
						\pic[tqft/pair of pants,
						rotate=90,name=h,at=(g-outgoing boundary)]; 
						\pic[tqft/reverse pair of pants,
						rotate=90,name=i,at=(h-outgoing boundary)];  
						\pic[tqft/pair of pants,rotate=90,name=cc,at=(i-outgoing boundary)];
						\pic[tqft/reverse pair of pants,rotate=90,name=dd,at=(cc-outgoing boundary),];
				\end{tikzpicture} }
			\end{aligned}
			.
		\end{align*}
		
		$\star$ Then, we move any involution precomposed by a reverse pair of pants (but not a handle) to its left as in the figure below, 
		\begin{align}
			\begin{aligned}
				\resizebox{220pt}{!}{%
					\begin{tikzpicture}[tqft/cobordism/.style={draw,thick},
						tqft/view from=outgoing, tqft/boundary separation=30pt,
						tqft/cobordism height=40pt, tqft/circle x radius=8pt,
						tqft/circle y radius=4.5pt, tqft/every boundary component/.style={draw,rotate=90},tqft/every incoming
						boundary component/.style={draw,dotted,thick},tqft/every outgoing
						boundary component/.style={draw,thick}]
						\pic[tqft/reverse pair of pants,
						rotate=90,name=h,,anchor={(0,0)}]; 
						\pic[tqft/cylinder to next,rotate=90,name=i];
						\pic[tqft/reverse pair of pants,
						rotate=90,name=j,at=(i-outgoing boundary)]; 
						\pic[tqft/cylinder,
						rotate=90,name=l,at=(j-outgoing boundary)]; 
						\node at ([xshift=20pt]l-incoming boundary 1)[color=black]{$   \leftrightarrow$};
						\node at ([xshift=20pt]l-outgoing boundary 1) [font=\huge] {\(=\)};
						\pic[tqft/reverse pair of pants,rotate=90,name=g, at=(l-outgoing boundary),anchor={(1,-1)}];
						\pic[tqft/cylinder to next, rotate=90,name=m,at=(l-outgoing boundary),anchor={(2,-1)}];
						\pic[tqft/cylinder, rotate=90,name=n, at=(m-outgoing boundary)];
						\node at ([xshift=20pt]n-incoming boundary 1)[color=black]{$   \leftrightarrow$};
						\pic[tqft/cylinder, rotate=90,name=g, at=(g-outgoing boundary)];
						\node at ([xshift=20pt]g-incoming boundary 1)[color=black]{$   \leftrightarrow$};
						\pic[tqft/reverse pair of pants,rotate=90,name=k, at=(g-outgoing boundary),,anchor={(2,0)}];
						\node at ([xshift=20pt]k-outgoing boundary 1) [font=\huge] {\(=\)};
						\node at ([xshift=200pt]k-outgoing boundary 1) [font=\huge] {\(.\)};
						\pic[tqft/cylinder,rotate=90,name=a,at=(k-outgoing boundary),anchor={(1,-1)}];
						\node at ([xshift=20pt]a-incoming boundary 1){$\leftrightarrow$};
						\pic[tqft/cylinder,rotate=90,name=b,at=(k-outgoing boundary),anchor={(0,-1)}];
						\node at ([xshift=20pt]b-incoming boundary 1){$\leftrightarrow$};
						\pic[tqft/cylinder,rotate=90,name=c,at=(k-outgoing boundary),anchor={(2,-1)}];
						\node at ([xshift=20pt]c-incoming boundary 1){$\leftrightarrow$};
						\pic[tqft/reverse pair of pants,rotate=90,name=x,at=(a-outgoing boundary)];
						\pic[tqft/cylinder to next,rotate=90,name=z,at=(c-outgoing boundary)];
						\pic[tqft/reverse pair of pants,rotate=90,name=y,at=(z-outgoing boundary)];
				\end{tikzpicture} }
			\end{aligned}
		\end{align}

		$\star$ Lastly, we move any remaining involution composed with a pair of pants to its right as in the figure below, 
		
		\begin{align}
			\begin{aligned}
				\resizebox{220pt}{!}{%
					\begin{tikzpicture}[tqft/cobordism/.style={draw,thick},
						tqft/view from=outgoing, tqft/boundary separation=30pt,
						tqft/cobordism height=40pt, tqft/circle x radius=8pt,
						tqft/circle y radius=4.5pt, tqft/every boundary component/.style={draw,rotate=90},tqft/every incoming
						boundary component/.style={draw,dotted,thick},tqft/every outgoing
						boundary component/.style={draw,thick}]
						\pic[tqft/cylinder,
						rotate=90,name=l]; 
						\node at ([xshift=20pt]l-incoming boundary 1)[color=black]{$   \leftrightarrow$};
						\pic[tqft/pair of pants, rotate=90,name=j,at=(l-outgoing boundary)]; 
						\pic[tqft/pair of pants, rotate=90,name=a,at=(j-outgoing boundary)]; 
						\pic[tqft/cylinder to next,rotate=90,name=b,at=(j-outgoing boundary 2)];
						\node at ([xshift=20pt]a-outgoing boundary 2) [font=\huge] {\(=\)};
						\pic[tqft/pair of pants,rotate=90,name=c, at=(b-outgoing boundary),anchor={(2,-1)}];
						\pic[tqft/cylinder, rotate=90,name=d, at=(c-outgoing boundary)]; 
						\node at ([xshift=20pt]d-incoming boundary 1)[color=black]{$   \leftrightarrow$};
						\pic[tqft/cylinder, rotate=90,name=e, at=(c-outgoing boundary 2)]; 
						\node at ([xshift=20pt]e-incoming boundary 1)[color=black]{$   \leftrightarrow$};
						\pic[tqft/cylinder to next,rotate=90,name=f,at=(e-outgoing boundary)];
						\pic[tqft/pair of pants, rotate=90,name=g,at=(d-outgoing boundary)]; 
						\node at ([xshift=20pt]g-outgoing boundary 2) [font=\huge] {\(=\)};
						\pic[tqft/pair of pants,rotate=90,name=h, at=(g-outgoing boundary),anchor={(0,-1)}];
						\pic[tqft/cylinder to next,rotate=90,name=m,at=(h-outgoing boundary 2)];
						\pic[tqft/pair of pants,rotate=90,name=n, at=(h-outgoing boundary)];
						\pic[tqft/cylinder, rotate=90,name=p, at=(n-outgoing boundary)]; 
						\node at ([xshift=20pt]p-incoming boundary 1)[color=black]{$   \leftrightarrow$};
						\pic[tqft/cylinder, rotate=90,name=q, at=(n-outgoing boundary 2)]; 
						\node at ([xshift=20pt]q-incoming boundary 1)[color=black]{$   \leftrightarrow$};
						\pic[tqft/cylinder, rotate=90,name=r, at=(m-outgoing boundary)]; 
						\node at ([xshift=20pt]r-incoming boundary 1)[color=black]{$   \leftrightarrow$};
				\end{tikzpicture} }
			\end{aligned}.
		\end{align}	
		
		If $m,n\ne 0$, we are done. If $m=0$, then every orientation reversing cylinder disappears using relation 
		$\eqref{phi is unitual and counital}$. Moreover, by  the relation \eqref{unit and counit}, which concerns the cup and reverse pair of pants,  we are left with just one cup on the \emph{in} part. Analogously, if $n=0$, using relations \eqref{phi is unitual and counital} and \eqref{unit and counit} concerning the cap, we are left with just one cap on the \emph{out} part.
	\end{proof}

	\begin{remark}
		The decomposition described in Proposition \ref{prop:orientable with boundary} is not unique. For example, the following two morphisms are equal in $\UCob$,
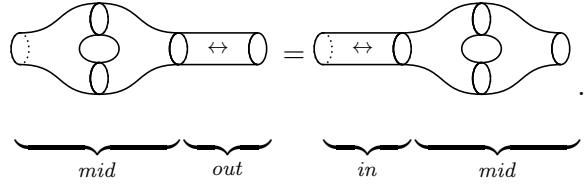
		\begin{align*}
			&	\begin{aligned}
				\resizebox{100pt}{!}{%
					\begin{tikzpicture}[tqft/cobordism/.style={draw,thick},
						tqft/view from=outgoing, tqft/boundary separation=30pt,
						tqft/cobordism height=40pt, tqft/circle x radius=8pt,
						tqft/circle y radius=4.5pt, tqft/every boundary component/.style={draw,rotate=90},tqft/every incoming
						boundary component/.style={draw,dotted,thick},tqft/every outgoing
						boundary component/.style={draw,thick}]
						\pic[tqft/pair of pants,
						rotate=90,name=k]; 
						\pic[tqft/reverse pair of pants,
						rotate=90,name=l,at=(k-outgoing boundary)]; 
						\pic[tqft/cylinder,rotate=90,name=x, at=(l-outgoing boundary)];
						\node at ([xshift=20pt]x-incoming boundary 1){$\leftrightarrow$};
				\end{tikzpicture} }
			\end{aligned}=
			\begin{aligned}
				\resizebox{100pt}{!}{%
					\begin{tikzpicture}[tqft/cobordism/.style={draw,thick},
						tqft/view from=outgoing, tqft/boundary separation=30pt,
						tqft/cobordism height=40pt, tqft/circle x radius=8pt,
						tqft/circle y radius=4.5pt, tqft/every boundary component/.style={draw,rotate=90},tqft/every incoming
						boundary component/.style={draw,dotted,thick},tqft/every outgoing
						boundary component/.style={draw,thick}]
						\pic[tqft/cylinder,rotate=90,name=x];
						\pic[tqft/pair of pants,
						rotate=90,name=k,at=(x-outgoing boundary)]; 
						\pic[tqft/reverse pair of pants,
						rotate=90,name=l,at=(k-outgoing boundary)]; 
						\node at ([xshift=20pt]x-incoming boundary 1){$\leftrightarrow$};
				\end{tikzpicture} }.
			\end{aligned}\\
			&\underbrace{\ \ \ \ \ \ \ \ \ \ \ \ \ \ \ \  \  }_\text{\emph{mid}} 
			\underbrace{\ \ \ \ \ \ \ \ \ }_\text{\emph{out}} \ \ \ \  \ 
			\underbrace{\ \ \ \ \ \ \ \ \ }_\text{\emph{in}} 
			\underbrace{\ \ \ \ \ \ \ \ \ \ \ \ \ \ \ \  \  }_\text{\emph{mid}} 
		\end{align*}
	\end{remark}

	\begin{remark}
		In the case $m = 0$ (respectively, $n=0$), the  \emph{in} part (respectively, the  \emph{out} part) consists of just a cap (respectively, just a cup).
	\end{remark}
	
	\begin{example}
		Pictured below is a connected orientable cobordism $0\to 4$,
		\begin{align*}
			& \ \	\begin{aligned}
				\resizebox{140pt}{!}
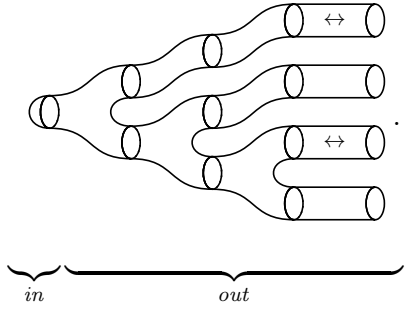
{%
					\begin{tikzpicture}[tqft/cobordism/.style={draw,thick},
						tqft/view from=outgoing, tqft/boundary separation=30pt,
						tqft/cobordism height=40pt, tqft/circle x radius=8pt,
						tqft/circle y radius=4.5pt, tqft/every boundary component/.style={draw,rotate=90},tqft/every incoming
						boundary component/.style={draw,dotted,thick},tqft/every outgoing
						boundary component/.style={draw,thick}]
						\pic[tqft/pair of pants,
						rotate=90,name=q]; 
						\pic[tqft/cylinder to next,rotate=90,name=t, at=(q-outgoing boundary 2)];
						\pic[tqft/cylinder to next,rotate=90,name=v, at=(t-outgoing boundary)];
						\pic[tqft/pair of pants,
						rotate=90,name=r,at=(q-outgoing boundary)]; 
						\pic[tqft/cylinder to next,rotate=90,name=u,   at=(r-outgoing boundary 2)];
						\pic[tqft/pair of pants,
						rotate=90,name=s,at=(r-outgoing boundary)]; 
						\pic[tqft/cap,
						rotate=90,at=(q-outgoing boundary),anchor={(0.5,2)}]; 
						\pic[tqft/cylinder,rotate=90,name=x, at=(u-outgoing boundary)];
						\pic[tqft/cylinder,rotate=90,name=y, at=(s-outgoing boundary)];
						\pic[tqft/cylinder,rotate=90,name=z, at=(s-outgoing boundary 2)];
						\node at ([xshift=20pt]z-incoming boundary 1){$\leftrightarrow$};	\pic[tqft/cylinder,rotate=90,name=u, at=(v-outgoing boundary)];
						\node at ([xshift=20pt]u-incoming boundary 1){$\leftrightarrow$};
				\end{tikzpicture} }
			\end{aligned}.\\
			&\underbrace{ }_\text{\emph{in}} 
			\underbrace{\ \ \ \ \ \ \ \ \ \ \ \ \ \ \ \ \ \ \ \ \ \ \ \ \ \ \ \ \ \ \ \ \ \ \ }_\text{\emph{out}} 
		\end{align*}
	\end{example}

\begin{definition}\label{xi 1st def}
		Let  $\xi^m_{\{i_1, \dots, i_l\}}$ denote the connected cobordism in $\Hom_{\UCob}(0,m)$ that has genus zero, and orientation reversing cylinders in its \emph{out} part exactly in the positions $1\leq i_1< \dots < i_l\leq m$, where $1\leq l \leq m$. That is, if we denote by $\Delta^m$ the composition 
	\begin{align*}
		\Delta^{m-1}:=(\text{id}^{\otimes (m-2)}\otimes \Delta)\dots (\text{id}\otimes \Delta)\Delta :1 \to m,
	\end{align*}
	and by $\phi_{i_1, \dots, i_l}$ the cobordism $m\to m$ given by
	\begin{align*}
		&\phi_{i_1, \dots, i_l} = c_1 \otimes \dots \otimes c_m, \ \ \text{where } c_j = \begin{cases}
			\text{id} &\text{ if } j\not \in \{i_1, \dots, i_l\}\\
			\phi &\text{ if } j\in \{i_1, \dots, i_l\},
		\end{cases}
	\end{align*}
	then 
	\begin{align*}
		\xi^m_{\{i_1, \dots, i_l\}}:= \phi_{i_1, \dots, i_l} \Delta^{m-1}u.
	\end{align*}
\end{definition}
	
For instance, the cobordism in the previous example is $\xi_{\{1,3\}}^4$.

	\begin{lemma}\label{xi=xj}
		If we have a partition $\{1,\dots, m\}=\{i_1, \dots, i_l \} \sqcup \{j_1,\dots, j_s \},$   then  $$\xi^m_{\{i_1, \dots, i_l\}}=\xi^m_{\{j_1, \dots, j_s\}}.$$
	\end{lemma}
	
	\begin{proof}
		Let $\{1,\dots, m\}=\{i_1, \dots, i_l \} \sqcup \{j_1,\dots, j_s \}$. Recall that by relation \eqref{phi is mult and comult}, $\Delta\phi = (\phi\otimes \phi) \Delta$. In general, this implies that $\Delta^{m-1}\phi= \phi^{\otimes m}\Delta^{m-1}$. We compute
		\begin{align*}
			\xi_{i_1, \dots, i_l}^m&= \phi_{i_1, \dots, i_l} \Delta^{m-1}u\\
			&= \phi_{i_1, \dots, i_l} \Delta^{m-1}\phi^2u\\
			&= \phi_{i_1, \dots, i_l}\phi^{\otimes m} \Delta^{m-1}\phi u\\
			&= (c_1\phi \otimes \dots c_m\phi) \Delta^{m-1}u,\\
			&=\xi_{j_1, \dots, j_s}^m, 
		\end{align*}
		where we are also using the relations $\phi^2=\text{id}$, $\phi u=u$, and that $c_k \phi=\phi$ if $k \not \in \{i_1, \dots, i_l\}$ and $c_k\phi=\text{id}$ if $k\in \{i_1, \dots, i_l\}$.
	\end{proof}
	
	\begin{example}
		To illustrate the previous Lemma, we show that $\xi_{\{1,3\}}^4=\xi_{\{2,4\}}^4$ using graphical calculus:
		\begin{align*}
			\begin{aligned}
				\resizebox{140pt}{!}{%
					\begin{tikzpicture}[tqft/cobordism/.style={draw,thick},
						tqft/view from=outgoing, tqft/boundary separation=30pt,
						tqft/cobordism height=40pt, tqft/circle x radius=8pt,
						tqft/circle y radius=4.5pt, tqft/every boundary component/.style={draw,rotate=90},tqft/every incoming
						boundary component/.style={draw,dotted,thick},tqft/every outgoing
						boundary component/.style={draw,thick}]
						\pic[tqft/pair of pants,
						rotate=90,name=q,at=(p-outgoing boundary), anchor={(1,-2)}]; 
						\pic[tqft/cylinder to next,rotate=90,name=t, at=(q-outgoing boundary 2)];
						\pic[tqft/cylinder to next,rotate=90,name=v, at=(t-outgoing boundary)];
						\pic[tqft/pair of pants,
						rotate=90,name=r,at=(q-outgoing boundary)]; 
						\pic[tqft/cylinder to next,rotate=90,name=u,   at=(r-outgoing boundary 2)];
						\pic[tqft/pair of pants,
						rotate=90,name=s,at=(r-outgoing boundary)]; 
						\pic[tqft/cap,
						rotate=90,name=z,at=(q-outgoing boundary),anchor={(0.5,2)},name=z]; 
						\node at ([xshift=-170pt,yshift=15pt]z-incoming boundary) [font=\huge] {	$\xi_{\{1,3\}}^4 $  \(=\) };
						\pic[tqft/cylinder,rotate=90,name=x, at=(u-outgoing boundary)];
						\pic[tqft/cylinder,rotate=90,name=y, at=(s-outgoing boundary)];
						\pic[tqft/cylinder,rotate=90,name=z, at=(s-outgoing boundary 2)];
						\node at ([xshift=20pt]z-incoming boundary 1){$\leftrightarrow$};	\pic[tqft/cylinder,rotate=90,name=u, at=(v-outgoing boundary)];
						\node at ([xshift=20pt]u-incoming boundary 1){$\leftrightarrow$};
						\node at ([xshift=20pt, yshift=-15pt]x-outgoing boundary 1) [font=\huge] {\(=\)};
				\end{tikzpicture} }
			\end{aligned}
			\begin{aligned}
				\resizebox{150pt}{!}{%
					\begin{tikzpicture}[tqft/cobordism/.style={draw,thick},
						tqft/view from=outgoing, tqft/boundary separation=30pt,
						tqft/cobordism height=40pt, tqft/circle x radius=8pt,
						tqft/circle y radius=4.5pt, tqft/every boundary component/.style={draw,rotate=90},tqft/every incoming
						boundary component/.style={draw,dotted,thick},tqft/every outgoing
						boundary component/.style={draw,thick}]
						\pic[tqft/cylinder,rotate=90,name=aa, anchor={(1,-2)}];
						\pic[tqft/cylinder,rotate=90,name=bb, at=(aa-outgoing boundary)];
						\node at ([xshift=20pt]aa-incoming boundary 1){$\leftrightarrow$};
						\pic[tqft/pair of pants,
						rotate=90,name=q,at=(bb-outgoing boundary)]; 
						\node at ([xshift=20pt]bb-incoming boundary 1){$\leftrightarrow$};
						\pic[tqft/cylinder to next,rotate=90,name=t, at=(q-outgoing boundary 2)];
						\pic[tqft/cylinder to next,rotate=90,name=v, at=(t-outgoing boundary)];
						\pic[tqft/pair of pants,
						rotate=90,name=r,at=(q-outgoing boundary)]; 
						\pic[tqft/cylinder to next,rotate=90,name=u,   at=(r-outgoing boundary 2)];
						\pic[tqft/pair of pants,
						rotate=90,name=s,at=(r-outgoing boundary)];
						\pic[tqft/cap,
						rotate=90,at=(q-outgoing boundary),anchor={(0.5,4)}]; 
						\pic[tqft/cylinder,rotate=90,name=x, at=(u-outgoing boundary)];
						\pic[tqft/cylinder,rotate=90,name=y, at=(s-outgoing boundary)];
						\pic[tqft/cylinder,rotate=90,name=z, at=(s-outgoing boundary 2)];
						\node at ([xshift=20pt]z-incoming boundary 1){$\leftrightarrow$};	\pic[tqft/cylinder,rotate=90,name=u, at=(v-outgoing boundary)];
						\node at ([xshift=20pt]u-incoming boundary 1){$\leftrightarrow$};
						\node at ([xshift=20pt, yshift=-15pt]x-outgoing boundary 1) [font=\huge] {\(=\)};
				\end{tikzpicture} }
			\end{aligned}
			\begin{aligned}
				\resizebox{120pt}{!}{%
					\begin{tikzpicture}[tqft/cobordism/.style={draw,thick},
						tqft/view from=outgoing, tqft/boundary separation=30pt,
						tqft/cobordism height=40pt, tqft/circle x radius=8pt,
						tqft/circle y radius=4.5pt, tqft/every boundary component/.style={draw,rotate=90},tqft/every incoming
						boundary component/.style={draw,dotted,thick},tqft/every outgoing
						boundary component/.style={draw,thick}]
						\pic[tqft/cylinder,rotate=90,name=aa, anchor={(1,-2)}];
						\node at ([xshift=20pt]aa-incoming boundary 1){$\leftrightarrow$};
						\pic[tqft/pair of pants,
						rotate=90,name=q,at=(aa-outgoing boundary)]; 
						\pic[tqft/cylinder,rotate=90,name=bb, at=(q-outgoing boundary)];
						\pic[tqft/cylinder,rotate=90,name=cc, at=(q-outgoing boundary 2)];
						\node at ([xshift=20pt]bb-incoming boundary 1){$\leftrightarrow$};
						\node at ([xshift=20pt]cc-incoming boundary 1){$\leftrightarrow$};
						\pic[tqft/cylinder to next,rotate=90,name=t, at=(cc-outgoing boundary)];
						\pic[tqft/pair of pants,
						rotate=90,name=r,at=(bb-outgoing boundary)]; 
						\pic[tqft/cylinder to next,rotate=90,name=v, at=(t-outgoing boundary)];
						\pic[tqft/cylinder to next,rotate=90,name=u,   at=(r-outgoing boundary 2)];
						\pic[tqft/pair of pants,
						rotate=90,name=s,at=(r-outgoing boundary)];
						\pic[tqft/cap,
						rotate=90,at=(q-outgoing boundary),anchor={(0.5,3)}]; 
						\pic[tqft/cylinder,rotate=90,name=x, at=(u-outgoing boundary)];
						\pic[tqft/cylinder,rotate=90,name=y, at=(s-outgoing boundary)];
						\pic[tqft/cylinder,rotate=90,name=z, at=(s-outgoing boundary 2)];
						\node at ([xshift=20pt]z-incoming boundary 1){$\leftrightarrow$};	\pic[tqft/cylinder,rotate=90,name=u, at=(v-outgoing boundary)];
						\node at ([xshift=20pt]u-incoming boundary 1){$\leftrightarrow$};
				\end{tikzpicture} }
			\end{aligned}\\
			\begin{aligned}
				\resizebox{130pt}{!}{%
					\begin{tikzpicture}[tqft/cobordism/.style={draw,thick},
						tqft/view from=outgoing, tqft/boundary separation=30pt,
						tqft/cobordism height=40pt, tqft/circle x radius=8pt,
						tqft/circle y radius=4.5pt, tqft/every boundary component/.style={draw,rotate=90},tqft/every incoming
						boundary component/.style={draw,dotted,thick},tqft/every outgoing
						boundary component/.style={draw,thick}]
						\pic[tqft/pair of pants,
						rotate=90,name=q,, anchor={(1,-2)}]; 
						\pic[tqft/cylinder to next,rotate=90,name=t, at=(q-outgoing boundary 2)];
						\pic[tqft/pair of pants,
						rotate=90,name=r,at=(q-outgoing boundary)]; 
						\pic[tqft/cylinder,rotate=90,name=bb, at=(r-outgoing boundary 2)];
						\pic[tqft/cylinder,rotate=90,name=cc, at=(t-outgoing boundary)];
						\node at ([xshift=20pt]bb-incoming boundary 1){$\leftrightarrow$};
						\node at ([xshift=20pt]cc-incoming boundary 1){$\leftrightarrow$};
						\pic[tqft/cylinder to next,rotate=90,name=v, at=(cc-outgoing boundary)];
						\pic[tqft/cylinder to next,rotate=90,name=u,   at=(bb-outgoing boundary)];
						\pic[tqft/cylinder,rotate=90,name=dd, at=(r-outgoing boundary)];
						\node at ([xshift=20pt]dd-incoming boundary 1){$\leftrightarrow$};
						\pic[tqft/pair of pants,
						rotate=90,name=s,at=(dd-outgoing boundary)];
						\pic[tqft/cap,
						rotate=90,at=(q-outgoing boundary),anchor={(0.5,2)}]; 
						\node at ([xshift=-70pt, yshift=15pt]q-outgoing boundary) [font=\huge] {\(=\)};
						\pic[tqft/cylinder,rotate=90,name=x, at=(u-outgoing boundary)];
						\pic[tqft/cylinder,rotate=90,name=y, at=(s-outgoing boundary)];
						\pic[tqft/cylinder,rotate=90,name=z, at=(s-outgoing boundary 2)];
						\node at ([xshift=20pt]z-incoming boundary 1){$\leftrightarrow$};	\pic[tqft/cylinder,rotate=90,name=u, at=(v-outgoing boundary)];
						\node at ([xshift=20pt]u-incoming boundary 1){$\leftrightarrow$};
						\node at ([xshift=30pt, yshift=-15pt]x-outgoing boundary 1) [font=\huge] {\(=\)};
				\end{tikzpicture} }
			\end{aligned}
			\begin{aligned}
				\resizebox{120pt}{!}{%
					\begin{tikzpicture}[tqft/cobordism/.style={draw,thick},
						tqft/view from=outgoing, tqft/boundary separation=30pt,
						tqft/cobordism height=40pt, tqft/circle x radius=8pt,
						tqft/circle y radius=4.5pt, tqft/every boundary component/.style={draw,rotate=90},tqft/every incoming
						boundary component/.style={draw,dotted,thick},tqft/every outgoing
						boundary component/.style={draw,thick}]
						\pic[tqft/pair of pants,
						rotate=90,name=q,at=(p-outgoing boundary), anchor={(1,-2)}]; 
						\pic[tqft/cylinder to next,rotate=90,name=t, at=(q-outgoing boundary 2)];
						\pic[tqft/cylinder to next,rotate=90,name=v, at=(t-outgoing boundary)];
						\pic[tqft/pair of pants,
						rotate=90,name=r,at=(q-outgoing boundary)]; 
						\pic[tqft/cylinder to next,rotate=90,name=u,   at=(r-outgoing boundary 2)];
						\pic[tqft/pair of pants,
						rotate=90,name=s,at=(r-outgoing boundary)]; 
						\pic[tqft/cap,
						rotate=90,at=(q-outgoing boundary),anchor={(0.5,2)}]; 
						\pic[tqft/cylinder,rotate=90,name=x, at=(u-outgoing boundary)];
						\pic[tqft/cylinder,rotate=90,name=xx, at=(x-outgoing boundary)];
						\pic[tqft/cylinder,rotate=90,name=y, at=(s-outgoing boundary)];
						\pic[tqft/cylinder,rotate=90,name=yy, at=(y-outgoing boundary)];
						\pic[tqft/cylinder,rotate=90,name=z, at=(s-outgoing boundary 2)];
						\pic[tqft/cylinder,rotate=90,name=zz, at=(z-outgoing boundary)];
						\node at ([xshift=20pt]z-incoming boundary 1){$\leftrightarrow$};	\pic[tqft/cylinder,rotate=90,name=u, at=(v-outgoing boundary)];
						\pic[tqft/cylinder,rotate=90,name=uu, at=(u-outgoing boundary)];
						\node at ([xshift=20pt]u-incoming boundary 1){$\leftrightarrow$};
						\node at ([xshift=20pt]xx-incoming boundary 1){$\leftrightarrow$};
						\node at ([xshift=20pt]yy-incoming boundary 1){$\leftrightarrow$};
						\node at ([xshift=20pt]zz-incoming boundary 1){$\leftrightarrow$};
						\node at ([xshift=20pt]uu-incoming boundary 1){$\leftrightarrow$};
						\node at ([xshift=70pt, yshift=-15pt]x-outgoing boundary 1) [font=\huge] {\(=\)};
				\end{tikzpicture} }
			\end{aligned}
			\begin{aligned}
				\resizebox{130pt}{!}{%
					\begin{tikzpicture}[tqft/cobordism/.style={draw,thick},
						tqft/view from=outgoing, tqft/boundary separation=30pt,
						tqft/cobordism height=40pt, tqft/circle x radius=8pt,
						tqft/circle y radius=4.5pt, tqft/every boundary component/.style={draw,rotate=90},tqft/every incoming
						boundary component/.style={draw,dotted,thick},tqft/every outgoing
						boundary component/.style={draw,thick}]
						\pic[tqft/pair of pants,
						rotate=90,name=q,at=(p-outgoing boundary), anchor={(1,-2)}]; 
						\pic[tqft/cylinder to next,rotate=90,name=t, at=(q-outgoing boundary 2)];
						\pic[tqft/cylinder to next,rotate=90,name=v, at=(t-outgoing boundary)];
						\pic[tqft/pair of pants,
						rotate=90,name=r,at=(q-outgoing boundary)]; 
						\pic[tqft/cylinder to next,rotate=90,name=u,   at=(r-outgoing boundary 2)];
						\pic[tqft/pair of pants,
						rotate=90,name=s,at=(r-outgoing boundary)]; 
						\pic[tqft/cap,
						rotate=90,at=(q-outgoing boundary),anchor={(0.5,2)}]; 
						\pic[tqft/cylinder,rotate=90,name=x, at=(u-outgoing boundary)];
						\pic[tqft/cylinder,rotate=90,name=y, at=(s-outgoing boundary)];
						\pic[tqft/cylinder,rotate=90,name=z, at=(s-outgoing boundary 2)];
						\node at ([xshift=20pt]x-incoming boundary 1){$\leftrightarrow$};	\pic[tqft/cylinder,rotate=90,name=u, at=(v-outgoing boundary)];
						\node at ([xshift=20pt]y-incoming boundary 1){$\leftrightarrow$};
						\node at ([xshift=40pt, yshift=-15pt]x-outgoing boundary 1) [font=\huge] {\(=\) $\xi_{\{2,4\}}^4$ .};
				\end{tikzpicture} }
			\end{aligned}
		\end{align*}
	\end{example}
	
	\begin{remark}\label{gen to all genus} We note that 
		Lemma \ref{xi=xj} generalizes to any genus.  That is, if	 $\xi^m_{\{i_1, \dots, i_l\},g}$ denotes the connected cobordism in $\Hom_{\UCob}(0,m)$ that has genus $g\geq 0,$ and orientation reversing cylinders in its \emph{out} part exactly in the positions $1\leq i_1<\dots <i_l\leq m$, then $$\xi^m_{\{i_1, \dots, i_l\},g}=\xi^m_{\{j_1, \dots, j_s\},g},$$
		where $\{1,\dots, m\}=\{i_1, \dots, i_l \} \sqcup \{j_1,\dots, j_s \}$. We leave the proof to the reader. 
	\end{remark}

	\subsection{Unorientable connected cobordisms with boundary}  \label{unorientable connected cobordisms with boundary}
	We describe now connected cobordisms with boundary in $\UCob$ that have at least one crosscap, i.e., unorientable ones.
	
	\begin{proposition}\label{prop unorientable with boundary}
		Any unorientable connected cobordism with bounday $m\to n$ in $\UCob$ can be decomposed in three parts, as follows:
		
		$\diamondsuit$ The \emph{in} part, which consists of one or two crosscaps and stacked (identity) cylinders, followed by a composition of reverse pairs of pants. 
		
		$\diamondsuit$ The \emph{mid} part, given by a composition of handles (this part is empty when the genus is zero). 
		
		$\diamondsuit$ The \emph{out} part, which is either:
		\begin{itemize}
			\item a cap, when $n=0$, or
			\item a composition of pairs of pants and (identity) cylinders $1\to n$.
		\end{itemize}
	\end{proposition}

	\begin{example}
		The following is an unorientable cobordism $2\to 3$ of genus $3$,
		\begin{align*}
			&\begin{aligned}
				\resizebox{280pt}{!}
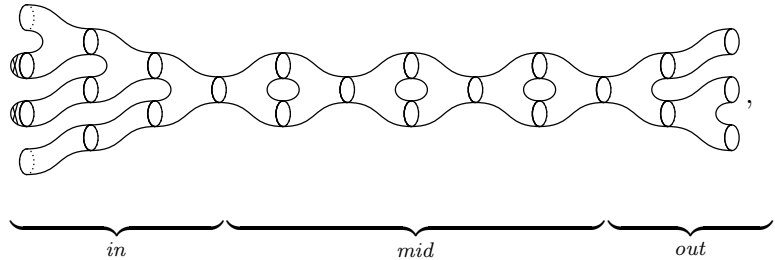
{%
					\begin{tikzpicture}[tqft/cobordism/.style={draw,thick},
						tqft/view from=outgoing, tqft/boundary separation=30pt,
						tqft/cobordism height=40pt, tqft/circle x radius=8pt,
						tqft/circle y radius=4.5pt, tqft/every boundary component/.style={draw,rotate=90},tqft/every incoming
						boundary component/.style={draw,dotted,thick},tqft/every outgoing
						boundary component/.style={draw,thick}]
						\pic[tqft/cap,rotate=90,name=c,anchor={(2,0)}];
						\pic[tqft/cap,rotate=90,name=x,anchor={(1,0)}];
						\node at ([xshift=-7pt]x-outgoing boundary 1) {$\figureXv$};
						\node at ([xshift=-7pt]c-outgoing boundary 1) {$\figureXv$};
						\pic[tqft/reverse pair of pants,
						rotate=90,name=e,anchor={(1,-1)}]; 
						\pic[tqft/cylinder to next,rotate=90,name=f, at=(c-outgoing boundary)];
						\pic[tqft/cylinder to next,rotate=90,name=g, anchor={(3,-1)}];
						\pic[tqft/reverse pair of pants,
						rotate=90,name=h,at=(f-outgoing boundary)]; 
						\pic[tqft/cylinder to next,rotate=90,name=i, at=(g-outgoing boundary)];
						\pic[tqft/reverse pair of pants,
						rotate=90,name=j,at=(i-outgoing boundary)]; 
						\pic[tqft/pair of pants,
						rotate=90,name=k,at=(j-outgoing boundary)]; 
						\pic[tqft/reverse pair of pants,
						rotate=90,name=l,at=(k-outgoing boundary)]; 
						\pic[tqft/pair of pants,
						rotate=90,name=m,at=(l-outgoing boundary)]; 
						\pic[tqft/reverse pair of pants,
						rotate=90,name=n,at=(m-outgoing boundary)]; 
						\pic[tqft/pair of pants,
						rotate=90,name=o,at=(n-outgoing boundary)]; 
						\pic[tqft/reverse pair of pants,
						rotate=90,name=p,at=(o-outgoing boundary)]; 
						\pic[tqft/pair of pants,
						rotate=90,name=r,at=(p-outgoing boundary)]; 
						\pic[tqft/cylinder to next,rotate=90,name=u,   at=(r-outgoing boundary 2)];
						\pic[tqft/pair of pants,
						rotate=90,name=s,at=(r-outgoing boundary)]; 
				\end{tikzpicture} }
			\end{aligned},\\
			&\underbrace{\ \ \ \ \ \ \ \ \ \ \ \ \ \ \ \ \ \ \ \ \ \ }_\text{\emph{in}} 
			\underbrace{\ \ \ \ \ \ \ \ \ \ \ \ \ \ \ \ \ \ \ \ \ \ \ \ \ \ \ \ \ \ \ \ \ \ \ \ \ \ \ }_\text{\emph{mid}} 
			\underbrace{\ \ \ \ \ \ \ \ \ \ \ \ \ \ \ \ \ }_\text{\emph{out}} 
		\end{align*}
		which has two crosscaps. 
	\end{example}

	\begin{proof}
		Unorientable cobordisms have at least one crosscap. By Relation $\eqref{3 crosscaps},$ if we have three crosscaps we can replace them by a handle with a crosscap. Hence, we can always reduce to having either one or two crosscaps. On the other hand, relations \eqref{crosscap and involution} can be used to get rid of orientation reversing cylinders, as they either disappear or get replaced by two crosscaps. Lastly, analogously to the orientable case in Proposition \ref{prop:orientable with boundary}, we can move pairs of pants (respectively, reverse pairs of pants) to the \emph{out} part (respectively, to the \emph{in} part) forming handles in the \emph{mid} part.
	\end{proof}
	\subsection{Closed connected components.}  \label{Closed connected components.} 
	It will be convenient to divide closed connected surfaces in three types, according to their number of crosscaps. Recall that three crosscaps can be replaced by a handle with a crosscap, see Equation $\eqref{3 crosscaps}.$ Hence we have three type of surfaces, described below. 
	
	$\diamondsuit$ Type 1: orientable surfaces of genus $g$, denoted $M_g$, for all $g\geq 0$. See below illustrations for $g=1, 2 $ and $3$, respectively:
	
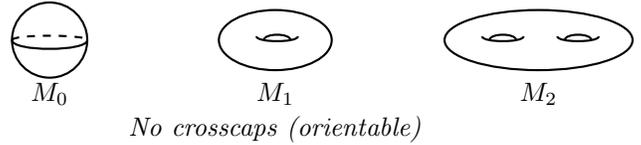
\begin{figure}[H]
		\begin{center}
			\begin{tikzpicture}
				\begin{scope}[scale=0.5,shift={(0,0)}]

					\draw[fill=white,style=thick] (-23, 1) ellipse (1cm and 1cm);
					
					\draw[style=thick] (-24,1) to[out=270, looseness = 0.4,in=270] (-22,1);
					\draw[style=thick,dashed] (-24,1) to[out=90, looseness = 0.2,in=90] (-22,1); 
					
					\draw[fill=white,style=thick] (-17, 1) ellipse (1.5cm and 0.8cm);
					
					\draw[style=thick] (-17.3,1) to[out=90, looseness = 0.7,in=90] (-16.6,1); 
					\draw[style=thick] (-17.5,1.1) to[out=270, looseness = 0.4,in=270] (-16.4,1.1);
					
					\node[above, font={\small}] at (-17,-1)  {$M_1$};
					
					\node[above, font={\small}] at (-23,-1)  {$M_0$};
					
					\draw[fill=white,style=thick] (-10, 1) ellipse (2.5cm and 0.8cm);
					\draw[style=thick] (-11.3,1) to[out=90, looseness = 0.7,in=90] (-10.6,1); 
					\draw[style=thick] (-11.5,1.1) to[out=270, looseness = 0.4,in=270] (-10.4,1.1);
					
					\draw[style=thick] (-9.3,1) to[out=90, looseness = 0.7,in=90] (-8.6,1); 
					\draw[style=thick] (-9.5,1.1) to[out=270, looseness = 0.4,in=270] (-8.4,1.1);
					\node[above, font={\small}] at (-10,-1)  {$M_2$};		
					\node[above, font={\small}] at (-17,-2)  {\emph{No crosscaps (orientable)}};			
					\node[above, font={\small}] at (-7,0)  {.};				  
				\end{scope}  
			\end{tikzpicture}
		\end{center}
	\end{figure}
	$\diamondsuit$ Type 2: unorientable surfaces with one crosscap and genus $g$, denoted $M_g^1$, for all $g\geq 0$. See below illustrations for $g=1, 2 $ and $3$, respectively:
	
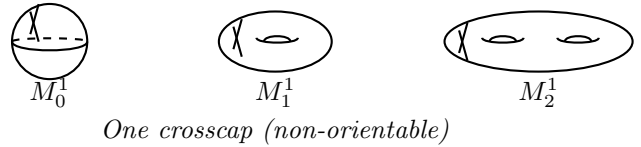
\begin{align*}
		\begin{aligned}
			\begin{tikzpicture}
				\begin{scope}[scale=0.5,shift={(0,0)}]
					\draw[fill=white,style=thick] (-23, 1) ellipse (1cm and 1cm);
					\node[above, font={\tiny}] at (-23.5,0.7)  {\tiny\figureXv};
					\node[above, font={\small}] at (-23,-1)  {\small{$M_0^1$}};
					\draw[style=thick] (-24,1) to[out=270, looseness = 0.4,in=270] (-22,1);
					\draw[style=thick,dashed] (-24,1) to[out=90, looseness = 0.15,in=90] (-22,1); 
					\draw[fill=white,style=thick] (-17, 1) ellipse (1.5cm and 0.8cm);
					\draw[style=thick] (-17.3,1) to[out=90, looseness = 0.7,in=90] (-16.6,1); 
					\draw[style=thick] (-17.5,1.1) to[out=270, looseness = 0.4,in=270] (-16.4,1.1);
					\node[above, font={\small}] at (-17,-1)  {$M_1^1$};
					\node[above, font={\small}] at (-18,0.3)  {\figureXv};
					\draw[fill=white,style=thick] (-10, 1) ellipse (2.5cm and 0.8cm);
					\draw[style=thick] (-11.3,1) to[out=90, looseness = 0.7,in=90] (-10.6,1); 
					\draw[style=thick] (-11.5,1.1) to[out=270, looseness = 0.4,in=270] (-10.4,1.1);
					\draw[style=thick] (-9.3,1) to[out=90, looseness = 0.7,in=90] (-8.6,1); 
					\draw[style=thick] (-9.5,1.1) to[out=270, looseness = 0.4,in=270] (-8.4,1.1);
					\node[above, font={\small}] at (-10,-1)  {$M_2^1$};						
					\node[above, font={\small}] at (-12,0.2)  {\figureXv};
					\node[above, font={\small}] at (-17,-2)  {\emph{One crosscap (non-orientable)}};					 
					\node[above, font={\small}] at (-7,0)  {.};				  
				\end{scope}  
			\end{tikzpicture}
		\end{aligned}
	\end{align*}
	$\diamondsuit$ Type 3: unorientable surfaces with two crosscaps and genus $g$, denoted $M_g^2$, for all $g\geq 0$. See below illustrations for $g=1, 2 $ and $3$, respectively:
	
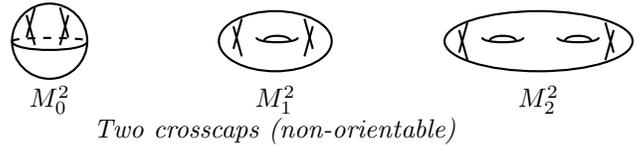
\begin{align*}
		\begin{aligned}
			\begin{tikzpicture}
				\begin{scope}[scale=0.5,shift={(0,0)}]
					\draw[fill=white,style=thick] (-23, 1) ellipse (1cm and 1cm);
					\node[above, font={\small}] at (-23.5,0.6)  {\figureXv};
					\node[above, font={\small}] at (-23,-1.2)  {$M_0^2$};
					\node[above, font={\small}] at (-22.5,0.6)  {\figureS};
					\draw[style=thick] (-24,1) to[out=270, looseness = 0.4,in=270] (-22,1);
					\draw[style=thick,dashed] (-24,1) to[out=90, looseness = 0.15,in=90] (-22,1); 
					\draw[fill=white,style=thick] (-17, 1) ellipse (1.5cm and 0.8cm);
					\draw[style=thick] (-17.3,1) to[out=90, looseness = 0.7,in=90] (-16.6,1); 
					\draw[style=thick] (-17.5,1.1) to[out=270, looseness = 0.4,in=270] (-16.4,1.1);
					\node[above, font={\small}] at (-17,-1.2)  {$M_1^2$};
					\node[above, font={\small}] at (-18,0.3)  {\figureXv};
					\node[above, font={\small}] at (-16,0.3)  {\figureS};
					\draw[fill=white,style=thick] (-10, 1) ellipse (2.5cm and 0.8cm);
					\draw[style=thick] (-11.3,1) to[out=90, looseness = 0.7,in=90] (-10.6,1); 
					\draw[style=thick] (-11.5,1.1) to[out=270, looseness = 0.4,in=270] (-10.4,1.1);
					\draw[style=thick] (-9.3,1) to[out=90, looseness = 0.7,in=90] (-8.6,1); 
					\draw[style=thick] (-9.5,1.1) to[out=270, looseness = 0.4,in=270] (-8.4,1.1);
					\node[above, font={\small}] at (-10,-1.2)  {$M_2^2$};						
					\node[above, font={\small}] at (-12,0.2)  {\figureXv};
					\node[above, font={\small}] at (-8,0.2)  {\figureS};
					\node[above, font={\small}] at (-17,-2)  {\emph{Two crosscaps (non-orientable)}};					
					\node[above, font={\small}] at (-7,0)  {.};				  
				\end{scope}  
			\end{tikzpicture}
		\end{aligned}
	\end{align*}
	
	\section{Constructions with rational sequences $\alpha, \beta, \gamma$}
	
	\subsection{The category $\VUCob$}\label{section:VUCob}
	Given three sequences $\alpha=(\alpha_0, \alpha_2, \dots)$, $\beta=(\beta_0, \beta_1, \dots )$ and $\gamma=(\gamma_0, \gamma_1, \dots)$ with $\alpha_i, \beta_i, \gamma_i\in \textbf{k}$, we define a linearization of the category $\text{UCob}_2$, denoted by $\VUCob$. This is the analogue of the linearization of $\Cob$ by a sequence $\alpha$ denoted $\text{VCob}_{\alpha}$ in \cite{KKO} and $\text{Cob}'_{\alpha}$ in \cite{KS}.
	
	\begin{definition}\label{def of VUCob}
		We define $\VUCob$ as the category with:
		\begin{itemize}
			\item Objects: Same as in $\UCob$, objects are non-negative integers.
			\item Morphisms: Morphisms $m\to n$ are \textbf k-linear combinations of unoriented $2$-cobordisms $m\to n$, modulo the following relations. For the connected closed cobordisms $M_g, M_g^1$ and $M_g^2$ as in Subsection  \ref{Closed connected components.}, we set 
			\begin{align*}
				&&M_g=\alpha_g, &&M_g^1=\beta_g  &&\text{ and } &&M_g^2=\gamma_g. 
			\end{align*}
			\item Composition: Given by glueing as induced from $\UCob,$ followed by evaluating closed components.
		\end{itemize}
	\end{definition}
	
	By our definition, a closed surface $M$ in $\UCob$ is evaluated to 
	\begin{align}\label{evaluation}
		M\mapsto {\displaystyle \prod_{g\geq 0}}\alpha_g^{a_g}  \beta_g^{b_g} \gamma_g^{c_g},
	\end{align}
	in $\VUCob$, where $a_g, b_g, c_g\in \mathbb Z_{\geq 0}$ denote the number of connected components of type $M_g, M_g^1$ and $M_g^2$ of $M$, respectively. 
	
	The category $\VUCob$  is a rigid symmetric tensor category, with tensor product and braiding induced from those on $\UCob$. We call a 2-dimensional cobordism \emph{viewable} if it has no closed components. Thus morphisms $m\to n$ in $\VUCob$ are linear combinations of unoriented viewable cobordisms  $m\to n$. Since a cobordism can have any number of handles, Hom spaces in this category are infinite-dimensional.
	
	\begin{definition}\label{trace}
		Let $n\geq 0$. We define a trace $$\tr_{\alpha,\beta, \gamma}:\Hom_{\VUCob}(n,n)\to \textbf k,$$ as follows. Closing an unoriented cobordism $M$ from $n\to n$ by connecting its $n$ source circles with its $n$ target circles via $n$ annuli results in an unoriented closed connected surface $M'$. Then $\tr=\operatorname{tr}_{\alpha,\beta, \gamma}(M)$  is  the evaluation of $M'$ as in equation \eqref{evaluation}.
	\end{definition}
	
	\begin{example}
		Consider the map $M'\in \Hom(2,2)$  given in the following picture,
		
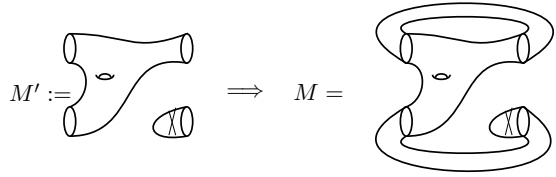
\begin{figure}[H]
			\begin{center}
				\resizebox{250pt}{!}{%
					\begin{tikzpicture}[scale=0.6]
						\begin{scope}[scale=0.8,shift={(0,0)}]
							\node at (3.5,0) {\figureA};
							\node at (-1,1) {$M':=$};
							\node at (6,1) {$\implies$};
							
							\draw[style=thick] (0,-0.5) to (0,0.5) to[out=20, looseness =1.3,in=0] (0,2)  to (0,3)to [out=0, looseness =1.3,in=-160] (4,3)to (4,2)  to [out=170, looseness =1.3,in=0] (0,-0.5);
							
							\draw[style=thick] (1.4,1.5) to[out=90, looseness = 1.1,in=90] (1,1.5); 
							\draw[style=thick] (1.5,1.6) to[out=270, looseness = 0.9,in=270] (0.9,1.6); 
							
							\draw[style=thick] (4,-0.5)  to[out=190, looseness =4,in=190] (4,0.5);

							\draw[fill=white,style=thick] (0,0) ellipse (0.2cm and 0.5cm);
							\draw[fill=white,style=thick](0,2.5) ellipse (0.2cm and 0.5cm);
							\draw[fill=white,style=thick] (4,2.5)ellipse (0.2cm and 0.5cm);
							\draw[fill=white,style=thick] (4,0)ellipse (0.2cm and 0.5cm);
						\end{scope}   
						\begin{scope}[scale=0.8,shift={(11.5,0)}]
							\node at (3.5,0) {\figureA};
							\node at (-3,1) {$M=$};
							
							\draw[style=thick] (0,-0.5) to (0,0.5) to[out=20, looseness =1.3,in=0] (0,2)  to (0,3)to [out=0, looseness =1.3,in=-160] (4,3)to (4,2)  to [out=170, looseness =1.3,in=0] (0,-0.5);
							
							\draw[style=thick] (1.4,1.5) to[out=90, looseness = 1.1,in=90] (1,1.5); 
							\draw[style=thick] (1.5,1.6) to[out=270, looseness = 0.9,in=270] (0.9,1.6); 
							
							\draw[style=thick] (4,-0.5)  to[out=190, looseness =4,in=190] (4,0.5);

							\draw[fill=white,style=thick] (0,0) ellipse (0.2cm and 0.5cm);
							\draw[fill=white,style=thick](0,2.5) ellipse (0.2cm and 0.5cm);
							\draw[fill=white,style=thick] (4,2.5)ellipse (0.2cm and 0.5cm);
							\draw[fill=white,style=thick] (4,0)ellipse (0.2cm and 0.5cm);
							
							\draw[style=thick] (0,3)  to[out=150, looseness =1,in=30] (4,3);
							\draw[style=thick] (0,2)  to[out=150, looseness =3.5,in=30] (4,2);
							
							\draw[style=thick] (0,-0.5)  to[out=-150, looseness =1,in=-30] (4,-0.5);
							\draw[style=thick] (0,0.5)  to[out=-150, looseness =3.5,in=-35] (4,0.5);				
						\end{scope}   
				\end{tikzpicture}}
			\end{center}
		\end{figure} 
		
		In this case, $\text{tr}(M')=\beta_2$.
	\end{example}

	\begin{definition}\label{inner product}
		The sequences $(\alpha, \beta, \gamma)$ determine a \textbf{k}-bilinear symmetric form $$(\cdot, \cdot)_{\alpha, \beta, \gamma}:\Hom_{\VUCob}(0,m)\times \Hom_{\VUCob}(0,m)\to \mathbf k,$$ 
		as follows. Given $S_1, S_2$ unoriented cobordisms $0\to m$, define
		\begin{align*}
			(S_1, S_2)_{\alpha, \beta, \gamma}:=	(S_1,S_2)=\operatorname{tr}_{\alpha,\beta,\gamma}((-S_1)\sqcup S_2),
		\end{align*}
		where $(-S_1)\sqcup S_2$ is the closed surface obtained by gluing the $m$ target circles of $S_1$ with the respective $m$ target circles of $S_2$, and $\operatorname{tr}_{\alpha,\beta,\gamma}$ is as  in Definition \ref{trace}. Then
		extend linearly to all of  $\Hom_{\VUCob}(0,m)$.
	\end{definition}
	
	The notion of trace above is the analogue of the one defined for oriented cobordisms in \cite[Section 2.1]{KKO}.
	
	\begin{example}
		Consider the cobordisms
		\begin{align*}&&M:=\begin{aligned}
				\resizebox{60pt}{!}{%
					\begin{tikzpicture}[tqft/cobordism/.style={draw,thick},
						tqft/view from=outgoing, tqft/boundary separation=30pt,
						tqft/cobordism height=40pt, tqft/circle x radius=8pt,
						tqft/circle y radius=4.5pt, tqft/every boundary component/.style={draw,rotate=90},tqft/every incoming
						boundary component/.style={draw,dotted,thick},tqft/every outgoing
						boundary component/.style={draw,thick}]
						\pic[tqft/cap,
						rotate=90,name=a,anchor={(0,0)}]; 
						\pic[tqft/pair of pants,
						rotate=90,name=b,at=(a-outgoing boundary)]; 
						\pic[tqft/reverse pair of pants,
						rotate=90,name=c,at=(b-outgoing boundary)]; 
				\end{tikzpicture} }
			\end{aligned}&&\text{and} &&N:=
			\begin{aligned}
				\resizebox{15pt}{!}{%
					\begin{tikzpicture}[tqft/cobordism/.style={draw,thick},
						tqft/view from=outgoing, tqft/boundary separation=30pt,
						tqft/cobordism height=40pt, tqft/circle x radius=8pt,
						tqft/circle y radius=4.5pt, tqft/every boundary component/.style={draw,rotate=90},tqft/every incoming
						boundary component/.style={draw,dotted,thick},tqft/every outgoing
						boundary component/.style={draw,thick}]
						\pic[tqft/cap,
						rotate=90,name=z,anchor={(1,0)}]; 
						\node at ([xshift=-7pt]z-outgoing boundary 1)[color=gray]{$   \figureXv$};
				\end{tikzpicture} }
			\end{aligned}.
		\end{align*}
		
		Then 
		\begin{align*}&&M\sqcup N=
			&&	\begin{aligned}
				\resizebox{65pt}{!}{%
					\begin{tikzpicture}[tqft/cobordism/.style={draw,thick},
						tqft/view from=outgoing, tqft/boundary separation=30pt,
						tqft/cobordism height=40pt, tqft/circle x radius=8pt,
						tqft/circle y radius=4.5pt, tqft/every boundary component/.style={draw,rotate=90},tqft/every incoming
						boundary component/.style={draw,dotted,thick},tqft/every outgoing
						boundary component/.style={draw,thick}]
						\pic[tqft/cap,
						rotate=90,name=a,anchor={(0,0)}]; 
						\pic[tqft/pair of pants,
						rotate=90,name=b,at=(a-outgoing boundary)]; 
						\pic[tqft/reverse pair of pants,
						rotate=90,name=c,at=(b-outgoing boundary)]; 
						\pic[tqft/cap,
						rotate=90,name=z,anchor={(1,-2)}]; 
						\node at ([xshift=-7pt]z-outgoing boundary 1)[color=gray]{$   \figureXv$};
				\end{tikzpicture} }
			\end{aligned}&&\leadsto 
			&&	\begin{aligned}
				\resizebox{80pt}{!}{%
					\begin{tikzpicture}[tqft/cobordism/.style={draw,thick},
						tqft/view from=outgoing, tqft/boundary separation=30pt,
						tqft/cobordism height=40pt, tqft/circle x radius=8pt,
						tqft/circle y radius=4.5pt, tqft/every boundary component/.style={draw,rotate=90},tqft/every incoming
						boundary component/.style={draw,dotted,thick},tqft/every outgoing
						boundary component/.style={draw,thick}]
						\pic[tqft/cap,
						rotate=90,name=a,anchor={(0,0)}]; 
						\pic[tqft/pair of pants,
						rotate=90,name=b,at=(a-outgoing boundary)]; 
						\pic[tqft/reverse pair of pants,
						rotate=90,name=c,at=(b-outgoing boundary)]; 
						\pic[tqft/cap,
						rotate=90,name=z,anchor={(1,-2)}]; 
						\node at ([xshift=-7pt]z-outgoing boundary 1)[color=gray]{$   \figureXv$};
						\pic[tqft,
						incoming boundary components=2,
						outgoing boundary components=0,
						rotate=90,name=j,anchor={(1,-3)}];
				\end{tikzpicture} }
			\end{aligned},
		\end{align*}
	and the last cobordism has trace $\beta_1$. Hence
		\begin{align*}
			(M,N)=\beta_1.
		\end{align*}
	\end{example}

	We distinguish the following morphisms in $\Hom_{\VUCob}(1,1)$:
\begin{align}\label{handle}
	&	\begin{aligned}
		\begin{tikzpicture}[scale=0.4]
			\begin{scope}[scale=0.8]
				\node at (0,4) {\underline{Handle}};
				\draw[style=thick]  (-3,2) to [out=0, looseness= 0.6,in=180](0,3)to [out=0, looseness= 0.6,in=180] (3,2);
				\draw[style=thick]  	(3,1)	to  [out=170, looseness= 0.8,in=0] (0,0)to  [out=170, looseness= 0.8,in=0] (-3,1);
				\draw[style=thick] (-0.3,1.5) to[out=90, looseness = 0.7,in=90] (0.4,1.5); 
				\draw[style=thick] (-0.5,1.6) to[out=270, looseness = 0.4,in=270] (0.6,1.6);
				\node at (0,-1.5) {$x: 1\xrightarrow{\Delta} 1\otimes 1\xrightarrow{m} 1$};
				\draw[fill=white,style=thick] (-3,1.5) ellipse (0.2cm and 0.5cm);
				\draw[fill=white,style=thick] (3,1.5) ellipse (0.2cm and 0.5cm);
			\end{scope}   
		\end{tikzpicture}
	\end{aligned} ,
	&\begin{aligned}
		\begin{tikzpicture}[scale=0.4]
			\begin{scope}[scale=0.8,shift={(0,0)}]
				\draw[style=thick]  (-5,4) to (-1,4);
				\draw[style=thick]  	(-5,5)	to (-1,5);
				\node at (-1,7) {\underline{Cross} };
				\node[above, font={\tiny}] at (-4,3.7)  {\smallfigureXv};
				\node at (0,1.5) {$	y: 1 \xrightarrow{\sim} 0\otimes 1 \xrightarrow{\theta\otimes 1}1\otimes 1 \xrightarrow{m}1$};
				\draw[fill=white,style=thick] (-5,4.5) ellipse (0.2cm and 0.5cm);
				\draw[fill=white,style=thick] (-1,4.5) ellipse (0.2cm and 0.5cm);
			\end{scope}   
		\end{tikzpicture}
	\end{aligned},
	&&\begin{aligned}
		\begin{tikzpicture}[scale=0.4]
			\begin{scope}[scale=0.8,shift={(0,0)}]
				\node at (-1,7) {\underline{Cup-Cap} };
				\node at (-0.5,1.5) {$	c: 1 \xrightarrow{\varepsilon} 0 \xrightarrow{u}1$};
				\draw[fill=white,style=thick] (-3,4.5) ellipse (0.2cm and 0.5cm);
				\draw[fill=white,style=thick] (1,4.5) ellipse (0.2cm and 0.5cm);
				\draw[style=thick]  (-3,5) to [out=10, looseness= 4,in=0](-3,4);
				\draw[style=thick]  (1,5) to [out=180, looseness= 4,in=180](1,4);
			\end{scope}   
		\end{tikzpicture}
	\end{aligned}.
\end{align}

\vspace{-2.5cm}
\begin{align*}
	\hspace{3.2cm}			\resizebox{55pt}{!}{%
		\begin{tikzpicture}[tqft/cobordism/.style={draw,thick},
			tqft/view from=outgoing, tqft/boundary separation=30pt,
			tqft/cobordism height=40pt, tqft/circle x radius=8pt,
			tqft/circle y radius=4.5pt, tqft/every boundary component/.style={draw,rotate=90},tqft/every incoming
			boundary component/.style={draw,dotted,thick},tqft/every outgoing
			boundary component/.style={draw,thick}]
			\pic[tqft/reverse pair of pants,
			rotate=90,name=c,anchor={(0,0)}]; 
			\node at (-1,1.5) { \textbf{:=} };
			\pic[tqft/cap,
			rotate=90,name=z,anchor={(-1,1)}]; 
			\node at ([xshift=-7pt]z-outgoing boundary 1)[color=gray]{$   \figureXv$};
	\end{tikzpicture} }
\end{align*}
\vspace{0.5cm}

	Note that $x^n, x^ny$ and $x^ny^2$ are connected cobordisms $1\to 1$ of genus n and $0, 1$ and $2$ crosscaps, respectively. To take trace, we close these cobordisms by an annulus connecting their in and out boundaries, obtaining a connected closed surface of genus $n+1$ and $0, 1$ and $2$ crosscaps, respectively, see Definition \ref{trace}. Hence
	\begin{align*}
		\tr(x^n)=\alpha_{n+1}, &&\tr(x^ny)=\beta_{n+1} &&\text{and} &&\tr(x^ny^2)=\gamma_{n+1}, &&\text{ for all } n\in \mathbb Z_{\geq 0}.
	\end{align*}

	\begin{example}\label{Hom 0 to 1}
		The cobordisms
		\begin{align*}
			&\resizebox{360pt}{!}{%
				\begin{tikzpicture}[tqft/cobordism/.style={draw,thick},
					tqft/view from=outgoing, tqft/boundary separation=30pt,
					tqft/cobordism height=40pt, tqft/circle x radius=8pt,
					tqft/circle y radius=4.5pt, tqft/every boundary component/.style={draw,rotate=90},tqft/every incoming
					boundary component/.style={draw,dotted,thick},tqft/every outgoing
					boundary component/.style={draw,thick}]
					\pic[tqft/cap,
					rotate=90,name=a,anchor={(0,-1)}]; 
					\node at ([xshift=-40pt]a-outgoing boundary 1)[color=black][font=\huge]{$	x^n u =$};
					\node at ([xshift=100pt, yshift=-40pt]a-outgoing boundary 1)[color=black][font=\huge]{$\underbrace{\ \ \ \ \ \ \ \ \ \  \ \ \ \ \ \ \ \ \ \ \ \ \ \ \ \ \ \ \ \ \ \ }_n$};
					\pic[tqft/pair of pants,
					rotate=90,name=b,at=(a-outgoing boundary)]; 
					\pic[tqft/reverse pair of pants,
					rotate=90,name=c,at=(b-outgoing boundary)]; 
					\node at ([xshift=20pt]c-outgoing boundary 1)[color=black]{$\dots$};
					\pic[tqft/pair of pants,
					rotate=90,name=d,anchor={(0,-5)}]; 
					\pic[tqft/reverse pair of pants,
					rotate=90,name=e,at=(d-outgoing boundary)]; 
					\node at ([xshift=20pt]e-outgoing boundary 1)[color=black][font=\huge]{$=$};
					\node at ([xshift=145pt, yshift=-40pt]e-outgoing boundary 1)[color=black][font=\huge]{$\underbrace{\ \ \ \ \ \ \ \ \ \ \ \ \ \ \ \ \ \ \ \ \ \ \ \ \ \ \ \ \ \ \ \ }_n$};
					\pic[tqft/cap,
					rotate=90,name=f,anchor={(0,-7)}]; 
					\pic[tqft/pair of pants,
					rotate=90,name=g,at=(f-outgoing boundary)]; 
					\pic[tqft/reverse pair of pants,
					rotate=90,name=h,at=(g-outgoing boundary)]; 
					\node at ([xshift=20pt]h-outgoing boundary)[color=black]{$\dots$};
					\pic[tqft/pair of pants,
					rotate=90,name=i,anchor={(0,-11)}]; 
					\pic[tqft/reverse pair of pants,
					rotate=90,name=j,at=(i-outgoing boundary)]; 
					\pic[tqft/cylinder,rotate=90,name=m,at=(j-outgoing boundary)];
					\node at ([xshift=20pt]m-incoming boundary 1)[color=black]{$  \leftrightarrow$};
					\node at ([xshift=20pt]m-outgoing boundary 1)[color=black][font=\huge]{$,$};
			\end{tikzpicture} }
		\end{align*}
		\begin{align*}
			\resizebox{230pt}{!}{%
				\begin{tikzpicture}[tqft/cobordism/.style={draw,thick},
					tqft/view from=outgoing, tqft/boundary separation=30pt,
					tqft/cobordism height=40pt, tqft/circle x radius=8pt,
					tqft/circle y radius=4.5pt, tqft/every boundary component/.style={draw,rotate=90},tqft/every incoming
					boundary component/.style={draw,dotted,thick},tqft/every outgoing
					boundary component/.style={draw,thick}]
					\pic[tqft/cylinder,rotate=90,name=g,anchor={(1,-1)}];
					\node at ([xshift=30pt]g-incoming boundary 1)[color=gray]{$   \figureXv$};
					\node at ([xshift=-90pt]g-outgoing boundary 1)[color=black][font=\huge]{$	x^n y u =$};
					\node at ([xshift=105pt, yshift=-40pt]g-outgoing boundary 1)[color=black][font=\huge]{$\underbrace{\ \ \ \ \ \ \ \ \ \ \ \  \ \ \ \ \ \ \ \ \ \ \ \ \ \ \ \ \ \ \ \ }_n$};
					\pic[tqft/cap,
					rotate=90,name=z,anchor={(1,0)}]; 
					\pic[tqft/pair of pants,
					rotate=90,name=h,at=(g-outgoing boundary)]; 
					\pic[tqft/reverse pair of pants,
					rotate=90,name=i,at=(h-outgoing boundary)];  
					\node at ([xshift=20pt]i-outgoing boundary 1)[color=black]{$\dots$};
					\pic[tqft/pair of pants,
					rotate=90,name=l,anchor={(1,-5)}]; 
					\pic[tqft/reverse pair of pants,
					rotate=90,name=j,at=(l-outgoing boundary)];  
					\node at ([xshift=20pt]j-outgoing boundary 1)[color=black][font=\huge]{$,$};
			\end{tikzpicture} }
		\end{align*}
		\begin{align*}
			\resizebox{230pt}{!}{%
				\begin{tikzpicture}[tqft/cobordism/.style={draw,thick},
					tqft/view from=outgoing, tqft/boundary separation=30pt,
					tqft/cobordism height=40pt, tqft/circle x radius=8pt,
					tqft/circle y radius=4.5pt, tqft/every boundary component/.style={draw,rotate=90},tqft/every incoming
					boundary component/.style={draw,dotted,thick},tqft/every outgoing
					boundary component/.style={draw,thick}]
					\pic[tqft/cylinder,rotate=90,name=g,anchor={(1,-1)}];
					\node at ([xshift=20pt]g-incoming boundary 1)[color=gray]{$   \figureXv$};
					\node at ([xshift=30pt]g-incoming boundary 1)[color=gray]{$   \figureXv$};
					\node at ([xshift=-90pt]g-outgoing boundary 1)[color=black][font=\huge]{$	x^n y^2 u =$};
					\node at ([xshift=100pt, yshift=-40pt]g-outgoing boundary 1)[color=black][font=\huge]{$\underbrace{\ \ \ \ \ \ \ \ \ \ \ \ \ \ \ \ \ \ \ \ \ \ \ \ \ \ \ \ \ \ \ \ }_n$};
					\pic[tqft/cap,
					rotate=90,name=z,anchor={(1,0)}]; 
					\pic[tqft/pair of pants,
					rotate=90,name=h,at=(g-outgoing boundary)]; 
					\pic[tqft/reverse pair of pants,
					rotate=90,name=i,at=(h-outgoing boundary)];  
					\node at ([xshift=20pt]i-outgoing boundary 1)[color=black]{$\dots$};
					\pic[tqft/pair of pants,
					rotate=90,name=l,anchor={(1,-5)}]; 
					\pic[tqft/reverse pair of pants,
					rotate=90,name=j,at=(l-outgoing boundary)];  
					\node at ([xshift=20pt]j-outgoing boundary 1)[color=black][font=\huge]{$,$};
			\end{tikzpicture} }
		\end{align*}
		for	$n\geq 0$, generate $\Hom_{\VUCob}(0,1)$ as a vector space, see Propositions  \ref{prop:orientable with boundary} and \ref{prop unorientable with boundary}.
	\end{example}

	\begin{example}\label{Hom 1 to 1}
		The cobordisms
		\begin{align*}
			\resizebox{170pt}{!}{%
				\begin{tikzpicture}[tqft/cobordism/.style={draw,thick},
					tqft/view from=outgoing, tqft/boundary separation=30pt,
					tqft/cobordism height=40pt, tqft/circle x radius=8pt,
					tqft/circle y radius=4.5pt, tqft/every boundary component/.style={draw,rotate=90},tqft/every incoming
					boundary component/.style={draw,dotted,thick},tqft/every outgoing
					boundary component/.style={draw,thick}]
					\pic[tqft/pair of pants,
					rotate=90,name=h,anchor={(1,0)}]; 
					\node at ([xshift=-30pt]h-incoming boundary 1)[color=black][font=\huge]{$	x^n=$};
					\node at ([xshift=60pt, yshift=-25pt]h-outgoing boundary 1)[color=black][font=\huge]{$\underbrace{\ \ \ \ \ \ \ \  \  \ \ \ \ \ \ \ \ \ \ \ \ \ \ \ \ \ \ \ \ \ \ \ \ }_n$};
					\pic[tqft/reverse pair of pants,
					rotate=90,name=i,at=(h-outgoing boundary)];  
					\node at ([xshift=20pt]i-outgoing boundary 1)[color=black]{$\dots$};
					\pic[tqft/pair of pants,
					rotate=90,name=l,anchor={(1,-3)}]; 
					\pic[tqft/reverse pair of pants,
					rotate=90,name=j,at=(l-outgoing boundary)];  
					\node at ([xshift=15pt, yshift=-5pt]j-outgoing boundary 1)[color=black][font=\huge]{$,$};
			\end{tikzpicture} }
			\resizebox{200pt}{!}{%
				\begin{tikzpicture}[tqft/cobordism/.style={draw,thick},
					tqft/view from=outgoing, tqft/boundary separation=30pt,
					tqft/cobordism height=40pt, tqft/circle x radius=8pt,
					tqft/circle y radius=4.5pt, tqft/every boundary component/.style={draw,rotate=90},tqft/every incoming
					boundary component/.style={draw,dotted,thick},tqft/every outgoing
					boundary component/.style={draw,thick}]
					\pic[tqft/cylinder,rotate=90,name=g,anchor={(1,-1)}];
					\node at ([xshift=20pt]g-incoming boundary 1)[color=black]{$   \leftrightarrow$};
					\node at ([xshift=-75pt]g-outgoing boundary 1)[color=black][font=\huge]{$	x^n\phi =$};
					\node at ([xshift=100pt, yshift=-40pt]g-outgoing boundary 1)[color=black][font=\huge]{$\underbrace{\ \ \ \ \ \ \ \ \ \ \ \ \ \ \ \ \ \ \ \ \ \ \ \ \ \ \ \ \ \ \ \ }_n$};
					\pic[tqft/pair of pants,
					rotate=90,name=h,at=(g-outgoing boundary)]; 
					\pic[tqft/reverse pair of pants,
					rotate=90,name=i,at=(h-outgoing boundary)];  
					\node at ([xshift=20pt]i-outgoing boundary 1)[color=black]{$\dots$};
					\pic[tqft/pair of pants,
					rotate=90,name=l,anchor={(1,-5)}]; 
					\pic[tqft/reverse pair of pants,
					rotate=90,name=j,at=(l-outgoing boundary)];  
					\node at ([xshift=15pt, yshift=-5pt]j-outgoing boundary 1)[color=black][font=\huge]{$,$};
			\end{tikzpicture} }
		\end{align*}
		\begin{align*}
			\resizebox{200pt}{!}{%
				\begin{tikzpicture}[tqft/cobordism/.style={draw,thick},
					tqft/view from=outgoing, tqft/boundary separation=30pt,
					tqft/cobordism height=40pt, tqft/circle x radius=8pt,
					tqft/circle y radius=4.5pt, tqft/every boundary component/.style={draw,rotate=90},tqft/every incoming
					boundary component/.style={draw,dotted,thick},tqft/every outgoing
					boundary component/.style={draw,thick}]
					\pic[tqft/cylinder,rotate=90,name=g,anchor={(1,-1)}];
					\node at ([xshift=30pt]g-incoming boundary 1)[color=gray]{$   \figureXv$};
					\node at ([xshift=-75pt]g-outgoing boundary 1)[color=black][font=\huge]{$	x^n y =$};
					\node at ([xshift=105pt, yshift=-40pt]g-outgoing boundary 1)[color=black][font=\huge]{$\underbrace{\ \ \ \ \ \ \ \ \ \ \ \ \ \ \ \ \ \ \ \ \ \ \ \ \ \ \ \ \ \ \ \ }_n$};
					\pic[tqft/pair of pants,
					rotate=90,name=h,at=(g-outgoing boundary)]; 
					\pic[tqft/reverse pair of pants,
					rotate=90,name=i,at=(h-outgoing boundary)];  
					\node at ([xshift=20pt]i-outgoing boundary 1)[color=black]{$\dots$};
					\pic[tqft/pair of pants,
					rotate=90,name=l,anchor={(1,-5)}]; 
					\pic[tqft/reverse pair of pants,
					rotate=90,name=j,at=(l-outgoing boundary)];  
					\node at ([xshift=15pt, yshift=-5pt]j-outgoing boundary 1)[color=black][font=\huge]{$,$};
			\end{tikzpicture} }
			\resizebox{210pt}{!}{%
				\begin{tikzpicture}[tqft/cobordism/.style={draw,thick},
					tqft/view from=outgoing, tqft/boundary separation=30pt,
					tqft/cobordism height=40pt, tqft/circle x radius=8pt,
					tqft/circle y radius=4.5pt, tqft/every boundary component/.style={draw,rotate=90},tqft/every incoming
					boundary component/.style={draw,dotted,thick},tqft/every outgoing
					boundary component/.style={draw,thick}]
					\pic[tqft/cylinder,rotate=90,name=g,anchor={(1,-1)}];
					\node at ([xshift=20pt]g-incoming boundary 1)[color=gray]{$   \figureXv$};
					\node at ([xshift=30pt]g-incoming boundary 1)[color=gray]{$   \figureXv$};
					\node at ([xshift=-75pt]g-outgoing boundary 1)[color=black][font=\huge]{$	x^n y^2 =$};
					\node at ([xshift=105pt, yshift=-40pt]g-outgoing boundary 1)[color=black][font=\huge]{$\underbrace{\ \ \ \ \ \ \ \ \ \ \ \ \ \ \ \ \ \ \ \ \ \ \ \ \ \ \ \ \ \ \ \ }_n$};
					\pic[tqft/pair of pants,
					rotate=90,name=h,at=(g-outgoing boundary)]; 
					\pic[tqft/reverse pair of pants,
					rotate=90,name=i,at=(h-outgoing boundary)];  
					\node at ([xshift=20pt]i-outgoing boundary 1)[color=black]{$\dots$};
					\pic[tqft/pair of pants,
					rotate=90,name=l,anchor={(1,-5)}]; 
					\pic[tqft/reverse pair of pants,
					rotate=90,name=j,at=(l-outgoing boundary)];  
					\node at ([xshift=15pt, yshift=-5pt]j-outgoing boundary 1)[color=black][font=\huge]{$,$};
			\end{tikzpicture} }
		\end{align*}
		\begin{align*}
			\resizebox{430pt}{!}{%
				\begin{tikzpicture}[tqft/cobordism/.style={draw,thick},
					tqft/view from=outgoing, tqft/boundary separation=30pt,
					tqft/cobordism height=40pt, tqft/circle x radius=8pt,
					tqft/circle y radius=4.5pt, tqft/every boundary component/.style={draw,rotate=90},tqft/every incoming
					boundary component/.style={draw,dotted,thick},tqft/every outgoing
					boundary component/.style={draw,thick}]
					\pic[tqft/pair of pants,
					rotate=90,name=a,anchor={(1,0)}]; 
					\node at ([xshift=-130pt]a-incoming boundary 1)[color=black][font=\huge]{$	y^i x^n c x^m y^j,$ where $x^n c x^m=$};
					\node at ([xshift=60pt, yshift=-25pt]a-outgoing boundary 1)[color=black][font=\huge]{$\underbrace{\ \ \ \ \ \ \ \ \ \ \ \ \ \ \ \ \ \ \ \ \ \ \ \ \ \ \ \ \ \ \ \ }_m$};
					\pic[tqft/reverse pair of pants,
					rotate=90,name=b,at=(a-outgoing boundary)];  
					\node at ([xshift=20pt]b-outgoing boundary 1)[color=black]{$\dots$};
					\pic[tqft/pair of pants,
					rotate=90,name=l,anchor={(1,-3)}]; 
					\pic[tqft/reverse pair of pants,
					rotate=90,name=j,at=(l-outgoing boundary)];  
					\pic[tqft/cup,
					rotate=90,name=k,at=(j-outgoing boundary)];  
					\pic[tqft/cap,
					rotate=90,name=l,at=(k-incoming boundary)];  
					\pic[tqft/pair of pants,
					rotate=90,name=m,at=(l-outgoing boundary)];  
					\pic[tqft/reverse pair of pants,
					rotate=90,name=n,at=(m-outgoing boundary)];  
					\node at ([xshift=20pt]n-outgoing boundary 1)[color=black]{$\dots$};
					\pic[tqft/pair of pants,
					rotate=90,name=l,anchor={(1,-9)}]; 
					\pic[tqft/reverse pair of pants,
					rotate=90,name=l1,at=(l-outgoing boundary)]; 
					\node at ([xshift=15pt, yshift=-5pt]l1-outgoing boundary 1)[color=black][font=\huge]{$,$};
					\node at ([xshift=300pt, yshift=-25pt]a-outgoing boundary 1)[color=black][font=\huge]{$\underbrace{\ \ \ \ \ \ \ \ \ \ \ \ \ \ \ \ \ \ \ \ \ \ \ \ \ \ \ \ \ \ \ \ }_n$};
			\end{tikzpicture} }
		\end{align*}
		for	$n\geq 0$ and $i,j=0,1,2$, generate $\End_{\VUCob}(1)$ as a vector space, see Propositions  \ref{prop:orientable with boundary} and \ref{prop unorientable with boundary}.
	\end{example}

	\subsection{A distinguished subset in $\Hom_{\VUCob}(0,m)$}\label{section: xi}
	
	We prove a technical result in Theorem \ref{xi are li}, describing a linearly independent subset of orientable maps in $\Hom_{\VUCob}(0,m)$ for all $m\in \mathbb Z_{\geq 0}$, which will be useful later. 
	
	Let $m\in \mathbb Z_{\geq 0}$ and let $\mathcal P_m$ denote the power set of $\{1, \dots, m\}$. We define an equivalence relation $\sim$ in $\mathcal P_m$ by 
	\begin{align*}
		J \sim I\text{ iff } J=I \text{ or } J= I^c, \ \ \ \ \ \text{ for all }  J\in \mathcal P_m.
	\end{align*}
	Let $\mathcal R_m$ denote the set of equivalence classes given by this relation. Note that the size of this set is $|\mathcal R_m|=2^{m-1}$ if $m>0$.

\begin{definition}\label{xi definition}
		Let  $\xi^m_{\{i_1, \dots, i_l\}}\in \Hom_{\VUCob}(0,m)$ denote the connected cobordism that has genus zero and orientation reversing cylinders exactly in the positions $1\leq i_1, \dots, i_l\leq m$, for $1\leq l \leq m$, see Definition \ref{xi 1st def}. By Lemma \ref{xi=xj}, for each $\overline{J}\in \mathcal R_m$ we have a well-defined connected cobordism $$\xi^m_{\overline J}:0 \to m,$$ where $\overline J$ denotes the class of $J\in \mathcal P_m$ in $\mathcal R_m$. 
\end{definition}

	We will need the following auxiliary Lemma. 
	\begin{lemma}\label{matrix determinant}
		Let $a,b\in \mathbf k^{\times}$, and let $A$ be the $n\times n$ matrix given by
		\begin{align*}
			A:=	\begin{bmatrix}
				a &  b &\dots &b\\
				b & a &\ddots &b \\
				\vdots &\ddots &\ddots & \vdots \\
				b &\dots &b&a
			\end{bmatrix}.
		\end{align*}
		That is, $A$ has $a$'s on the diagonal and $b$'s everywhere else. Then
		\begin{align*}
			\det(A)=(a-b)^{n-1} \cdot (a+(n-1)b).
		\end{align*} 
	\end{lemma}
	\begin{proof}
		We start by adding columns 2 to $n$ to the first column in the matrix, obtaining
		\begin{align*}
			\begin{bmatrix}
				a+(n-1)b &  b &\dots &b\\
				a+(n-1)b  & a &\ddots &b \\
				\vdots & &\ddots & \vdots \\
				a+(n-1)b  &b &\dots &a
			\end{bmatrix}.
		\end{align*}
		Now, substracting from rows 2 to n the first row, we get
		\begin{align*}
			\begin{bmatrix}
				a+(n-1)b &  b &\dots &&b\\
				0 & a-b &0 &\dots  &0 \\
				\vdots & 0&& \ddots &\vdots \\
				&&\ddots &\ddots &0	\\
				0 &0 &\dots &0 &a-b
			\end{bmatrix},
		\end{align*}
		which has determinant $(a-b)^{n-1} \cdot (a+(n-1)b)$, as desired. 
	\end{proof}
	
	\begin{theorem}\label{xi are li} Let $\alpha, \beta$ and $\gamma$ be sequences in $\mathbf{k},$ and let $m\geq 2$. Suppose that 
		\begin{align}
			\alpha_{m-1} \ne \gamma_{m-2},  (1- 2^{m-1})\gamma_{m-2}.
		\end{align}
		Then the set  $\{\xi^m_{\overline J}  : \overline J \in \mathcal R_m \},$ as given in Definition \ref{xi definition},  is a linearly independent subset of $\Hom_{\operatorname{VUCob}_{\alpha, \beta, \gamma}}(0,m)$.
	\end{theorem}

	\begin{proof}
		Fix $m\geq 2$. We want to compute the matrix of inner products (see Definition \ref{inner product}) for cobordisms in the set $\{\xi^m_{\overline J}  : \overline J \in \mathcal R_m \}$.
		
		Let $\overline J, \overline L \in \mathcal R_m$, with $\overline J\ne \overline L$. Since $\phi $ is an involution then the inner product of $\xi^m_{\overline{J}}$ with itself results in the surface of type $M_{m-1}$, see Subsection \ref{Closed connected components.}, which evaluates to $\alpha_{m-1}$. On the other hand, relations \eqref{crosscap and involution} imply that the inner product of  $\xi^m_{\overline{J}}$ and $\xi^m_{\overline{L}}$ results in the surface of type $M_{m-2}^2$, which evaluates to $\gamma_{m-2}$. Hence, the resulting matrix of inner products is the $2^{m-1}\times 2^{m-1}$ matrix given by
		\begin{align*}
			M:=	\begin{bmatrix}
				\alpha_{m-1} &  \gamma_{m-2} &\dots &\gamma_{m-2}\\
				\gamma_{m-2} & \alpha_{m-1}  &\ddots &\gamma_{m-2} \\
				\vdots &\ddots &\ddots & \vdots \\
				\gamma_{m-2}&\dots &\gamma_{m-2}&\alpha_{m-1} 
			\end{bmatrix}.
		\end{align*}
		By Lemma \ref{matrix determinant}, this matrix has determinant
		\begin{align*}
			\det(M)= (\alpha_{m-1} -\gamma_{m-2})^{2^{m-1}-1} \cdot (\alpha_{m-1} + (2^{m-1}-1)\gamma_{m-2}),
		\end{align*}
		and thus cobordisms in the set $\{\xi^m_{\overline J}  : \overline J \in \mathcal R_m \}$  are linearly independent, since by assumption
		\begin{align*}
			\alpha_{m-1} \ne \gamma_{m-2} && \text{and} & &\alpha_{m-1} \ne (1- 2^{m-1})\gamma_{m-2}.
		\end{align*}
	\end{proof}
	
	We generalize the previous result to any genus. For $g\in \mathbb N$. Consider the maps  $\xi^m_{\overline{\{i_1, \dots, i_l\}},g}\in \Hom_{\VUCob}(0,m)$, where $\xi^m_{\overline{\{i_1, \dots, i_l\}},g}$  represents the connected cobordism of genus $g$ that has orientation reversing cylinders exactly in the positions $1\leq i_1< \dots< i_l\leq m$, where $1\leq l \leq m$. Note that this is well defined by Remark \ref{gen to all genus}.
	
	\begin{theorem}\label{xi are lin gen} Let $\alpha, \beta$ and $\gamma$ be sequences in $\mathbf{k},$ and let $m\geq 2$, $g\geq 1$. Suppose that
		\begin{align}
			\alpha_{2g+m-1} \ne \gamma_{m-2},  (1- 2^{m-1})\gamma_{2g+m-2}.
		\end{align}
		Then the set  $\{\xi^m_{\overline J,g}  : \overline J \in \mathcal R_m \}$ is a linearly independent subset of $\Hom_{\operatorname{VUCob}_{\alpha, \beta, \gamma}}(0,m)$.
	\end{theorem}
	\begin{proof}
		Fix $m\geq 2, g\geq 1$. We want to compute the matrix of inner products of the set  $\{\xi^m_{\overline J,g}  : \overline J \in \mathcal R_m \}$. Same as in the proof of Proposition \ref{xi are li}, this matrix has size $2^{m-1}\times 2^{m-1}$ and is given by 
		\begin{align*}
			M:=	\begin{bmatrix}
				\alpha_{2g+m-1} &  \gamma_{2g+m-2} &\dots &\gamma_{2g+m-2}\\
				\gamma_{2g+m-2} & \alpha_{2g+m-1}  &\ddots &\gamma_{2g+m-2} \\
				\vdots &\ddots &\ddots & \vdots \\
				\gamma_{2g+m-2}&\dots &\gamma_{2g+m-2}&\alpha_{2g+m-1} 
			\end{bmatrix}.
		\end{align*}
		Thus by Lemma \ref{matrix determinant} it has determinant
		\begin{align*}
			\det(M)= (\alpha_{2g+m-1} -\gamma_{2g+m-2})^{2^{m-1}-1} \cdot (\alpha_{2g+m-1} + (2^{m-1}-1)\gamma_{2g+m-2}).
		\end{align*}
		The result follows. 
	\end{proof}

	\subsection{The skein category $\SUCob$}\label{SUCob} 
	Given three sequences $\alpha=(\alpha_0, \alpha_2, \dots)$, $\beta=(\beta_0, \beta_1, \dots ),$ and $\gamma=(\gamma_0, \gamma_1, \dots),$ with $\alpha_i, \beta_i, \gamma_i\in \textbf{k}$, we define a rigid symmetric tensor category with finite dimensional Hom spaces, denoted by $\SUCob$. This is the analogue of the category $\text{SCob}_{\alpha}$ in \cite{KKO}, and $\text{Cob}_{\alpha}$ in \cite{KS}.
	\begin{definition}
		We say that a sequence $\eta=(\eta_0, \eta_1, \dots)$ in $\textbf k$ satifies a  \emph{linear recurrence} (or is \emph{linearly recurrent})  if there exist  fixed $K\geq r$ and $a_1, \dots, a_{r}\in \textbf{k}$, such that 
		\begin{align}\label{equation for K}
			\eta_{l}= a_1\eta_{l-1} +\dots+ a_r\eta_{l-r}, \  \  \text{ for all } l\geq K.
		\end{align}
	\end{definition}

	A sequence $\eta$ is linearly recurrent if and only if it has a rational generating function 
	\begin{align*}
		Z_{\eta}(T)=\frac{p_{\eta}(T)}{q(T)} = \sum\limits_{k\geq 0} \eta_k T^k,  
	\end{align*}
	where  $p_{\eta}(T),  q(T)\in \textbf k[T]$ are relatively prime, see for example \cite{KP}.  Normalizing so that $q(0)=1$, the polynomials $p_{\eta}(T)$ and $q(T)$ are uniquely determined by the sequence $\eta$. Assume from now on that 
	\begin{align}\label{alpha gen function}
		Z_{\eta}(T)=\frac{p_{\eta}(T)}{q(T)} = \sum\limits_{k\geq 0} \eta_k T^k,  \ \ (p_{\eta}(T), q(T))=1, \ \ q(0)=1,
	\end{align}
	and let
	\begin{align}\label{M y N}
		&&N=\deg(p_{\eta}(T)), &&M=\deg(q(T)), &&\text{and} &&K=\max(N+1, M).
	\end{align}
	Here our convention is that a constant polynomial has degree zero. 
	Then, if
	\begin{align}\label{equation for q(t)}
		q(T)=1-a_1T+a_2T^2+ \dots + (-1)^M a_MT^M,
	\end{align}
	we have that
	\begin{align}\label{linear recurrence for alpha}
		\eta_{l}= a_1 \eta_{l-1} - a_{2} \eta_{l-2} + \dots + (-1)^{M-1} a_{M} \eta_{l-M}, \ \ \text{for all } l\geq K. 
	\end{align}

	Now we apply this to our sequences $(\alpha, \beta, \gamma)$. We want to take the category $\VUCob$ and quotient by relations defined by the sequence $\alpha$ in order to get finite dimensional Hom spaces. 
	
	Assume from now on that $\alpha$ is linearly recurrent, with generating function satisfying \eqref{alpha gen function} and \eqref{equation for q(t)}. Recall that in the category $\VUCob$ we have a trace, see Definition \ref{trace}. Let $x$ denote the handle cobordism, as defined on \eqref{handle}. 
	It follows from the linear recurrence on $\alpha$, see equation \eqref{linear recurrence for alpha}, that the trace of the map $\sigma$ given by
	\begin{align}\label{skein relation}
		\sigma:= x^K + \sum\limits_{i=1}^M (-1)^ia_ix^{K-i},
	\end{align}
where $K$ is as in \eqref{M y N},	is zero in $\VUCob$,  see also \cite[Section 2.4]{Kh}.  We call the equation $\sigma=0$ the \emph{handle relation}. 

We want to show that by imposing conditions on $\beta$ and $\gamma$, the quotient of $\VUCob$ by the tensor ideal generated by $\sigma$ will be non-trivial. We recall below the definition of a tensor ideal. 

	\begin{definition}
		A \emph{tensor ideal} $I$ in a tensor category $\mathcal C$ is a collection of subspaces $I(X,Y)\subset \Hom(X,Y),$ for every pair of objects $X,Y\in \mathcal C$, closed under composition and tensor product in the following way:
		\begin{enumerate}
			\item for $f \in I(X, Y ),$  $g_1\in \Hom(Y, Z)$ and $g_2 \in \Hom(Z, X)$,  the compositions $g_1\circ f$ and $f\circ g_2$ are in $ I(X,Z)$ and $I(Z, Y)$, respectively;
			\item  for $f \in I(X, Y )$ and $g \in \Hom(Z, W)$, the products $f\otimes g$ and $g\otimes f$ are in  $I(X \otimes Z, Y \otimes W)$ and $I(Z \otimes X, W \otimes Y )$, respectively;
		\end{enumerate}
		for all $X,Y,Z,W \in \mathcal C$.
	\end{definition}
	For a tensor ideal $I$ in $\mathcal C$, we can define the \emph{quotient} $\mathcal C'$ of $\mathcal C$ by $I$, see \cite[Section 2.1]{EO}. The new category $\mathcal C'$ is again a tensor category, with objects the objects of $\mathcal C$, and $\Hom$ spaces defined by $\Hom_{\mathcal C'}(X, Y )  :=\Hom_{\mathcal C}(X, Y )/I(X, Y )$. Composition of morphisms and tensor product are induced from $\mathcal C$.
	\begin{definition} \cite[Exercise 8.18.9]{EGNO} Let $\tr=\tr_{\alpha,\beta, \gamma}$ be as in Definition \ref{trace}. 
		We call a morphism $x\in \Hom_{\VUCob}(n,m)$  \emph{negligible} if, for any $y\in \Hom_{\VUCob}(m,n),$ we have that $\tr(yx)=0$. 
	\end{definition}
	
	The set of all negligible morphisms in  a spherical tensor category $\mathcal C$ is a proper tensor ideal, see \cite[Proposition 2.4]{EO}. Hence, if $\sigma$ is negligible, then the quotient  of $\VUCob$ by the tensor ideal generated by $\sigma$ is non-trivial. We prove in what follows that, under certain conditions on $\alpha, \beta$ and $\gamma$, the handle morphism $\sigma$ is negligible. We will need the following auxiliary Lemma.
	
	\begin{lemma}\label{lemma auxiliar gen functions}
		Fix an integer $B\geq 0$. Let $\eta$ be a linearly recurrent sequence in $\mathbf{k}$, with generating function 
		\begin{align*}
			Z_{\eta}(T)=\frac{p_{\eta}(T)}{q(T)}, \ \ \text{where} \ \ p_{\eta}(T), q(T)\in \mathbf{k}[T].
		\end{align*}
		 Let $q(T)=q_0+q_1T+\dots +q_LT^L,$ for $q_0, \dots, q_{L-1}\in \mathbf k$ and $q_L\in \mathbf k^{\times}$. A sequence $\mu$ in $\mathbf k$ satisfies
		\begin{align*}
			\sum\limits_{j=0}^l q_{l-j}\mu_j =0, \ \ \text{for all } l\geq B, 
		\end{align*}
		 if and only if it has a generating function
		\begin{align*}
			Z_{\mu}(T)=\frac{p_{\mu}(T)}{q(T)},
		\end{align*}
		for some $p_{\mu}(T)\in \mathbf k[T]$ with $\deg(p_{\mu}(T))<B$.
	\end{lemma}
	
	\begin{proof}
		Define
		\begin{align*}
			p_{\mu}(T):=	\left(	\sum\limits_{j\geq 0} \mu_j T^j\right) q_{\eta}(T).
		\end{align*}
		Then 	\begin{align*}
			\sum\limits_{j=0}^l q_{l-j}\mu_j =0, \ \ \text{for all } l\geq B ,
		\end{align*}
		is equivalent to $p_{\mu}(T)$ being a polynomial of degree at most $B-1$, as desired.  
	\end{proof}
	
	As a consequence, we get the following. 
	
	\begin{lemma}\label{handle is negligible}
		Consider a linearly recurrent sequence $\alpha$  in $\mathbf k$, with generating function satisfying \eqref{alpha gen function}, \eqref{M y N} and \eqref{equation for q(t)}. Let $\beta$ and $\gamma$ be sequences in $\mathbf{k}$ with generating functions  
		\begin{align*}
			&Z_{\beta}(T)=\frac{p_{\beta}(T)}{q(T)} = \sum\limits_{k\geq 0} \beta_k T^k  &&\text{and}
			&&Z_{\gamma}(T)=\frac{p_{\gamma}(T)}{q(T)} = \sum\limits_{k\geq 0} \gamma_k T^k,
		\end{align*}
		respectively, where   $p_{\beta}(T), p_{\gamma}(T)\in \mathbf  k[T]$ satisfy
		\begin{align*}
			\deg(p_{\beta}(T)), \deg(p_{\gamma}(T))< K=\max(N+1,M).
		\end{align*}
		Then the handle morphism $\sigma$ from \eqref{handle} is negligible in $\operatorname{VUCob}_{\alpha,\beta,\gamma}$. 
	\end{lemma} 
	\begin{proof}
		We want to show that the trace of $\sigma \circ z$ is zero for every $z\in \End_{\VUCob}(1,1)$. By Example \ref{Hom 1 to 1}, it is enough to show that the traces of  $\sigma x^n y^i$, $\sigma x^n \phi$ and $\sigma y^ix^mux^ny^j$ are zero for all $m,n\geq 0$ and $ i,j=0,1,2$. 
		
		As in Equation \eqref{linear recurrence for alpha}, we have that 
		\begin{align*}
			\alpha_{l}+ \sum\limits_{s=1}^M (-1)^s a_s \alpha_{l-s}=0, \ \ \text{for all } l\geq K. 
		\end{align*}
		On the other hand, since $\deg(p_{\beta}(T)), \deg(p_{\gamma}(T))\leq K,$ taking $B=K$  in Lemma \ref{lemma auxiliar gen functions} we get
		\begin{align*}
			&&	\beta_{l}+ \sum\limits_{s=1}^M (-1)^s a_s \beta_{l-s}=0 &&\text{ and}
			&&\gamma_{l}+ \sum\limits_{s=1}^M (-1)^s a_s \gamma_{l-s}=0, \ \ \text{for all } l\geq K.
		\end{align*}
		Using this, we compute
		\begin{align*}
			tr(\sigma y^ix^n)=\begin{cases}
				\alpha_{K+n+1}-\sum\limits_{s=1}^M  a_{s}\alpha_{K+n+1-s} =0 &\text{if } i=0, \\
				\beta_{K+n+1}-\sum\limits_{s=1}^M a_{s}\beta_{K+n+1-s} =0 &\text{if } i=1, \\
				\gamma_{K+n+1}-\sum\limits_{s=1}^M a_{s}\gamma_{K+n+1-s} =0 &\text{if } i=2.
			\end{cases}
		\end{align*}
		On the other hand
		\begin{align*}
			tr(\sigma y^i x^m u x^n y^j )=\begin{cases}
				\alpha_{K+n+m}-\sum\limits_{s=1}^M a_{s}\alpha_{K+n+m-s} =0 &\text{if } i=j=0, \\
				\beta_{K+n+m}-\sum\limits_{s=1}^M a_{s}\beta_{K+n+m-s} =0 &\text{if } i=0, j=1, \text{ or } i=1, j=0,\\
				\gamma_{K+n+m}-\sum\limits_{s=1}^M a_{s}\gamma_{K+n+m-s} =0 &\text{if } i=2, j=0,  \text{ or } i=0, j=2, \\&\text{or } i=j=1,\\
				\beta_{K+n+m+1}-\sum\limits_{s=1}^M a_{s}\beta_{K+n+m+1-s} =0 &\text{if } i=2, j=1, \text{ or } i=1, j=2 ,\\
				\gamma_{K+n+m+1}-\sum\limits_{s=1}^M a_{s}\gamma_{K+n+m+1-s} =0 &\text{if } i=2, j=2.
			\end{cases}
		\end{align*}
		Lastly, since closing the cobordism $x^n \phi$ via an annulus between its boundary circles results in the surface $u x^{n+1} \varepsilon$, then $\tr(\sigma x^n\phi)=\tr(\sigma x^n)=0$ by the equations above, for all $n\geq 0$.
		
		Thus $\sigma$ is negligible, as desired. 
	\end{proof}
	
	\begin{definition}\label{def of skein category}
		Let $(\alpha, \beta, \gamma)$ be as in Lemma \ref{handle is negligible}. We define the unoriented \emph{skein} category $\SUCob$ as the quotient of $\VUCob$ by the tensor ideal generated by $\sigma$. 
	\end{definition}
	The category $\SUCob$  is a non-trivial rigid symmetric tensor category, with tensor product and braiding induced from $\VUCob$.
	Hom spaces in $\SUCob$ consist of linear combinations of unoriented cobordisms whose connected components have genus strictly less than $K$. Thus they are finite dimensional.

	\begin{example}\label{example 0 to 1}
		Consider the sequences $\alpha, \beta, \gamma$ with generating functions 
		\begin{align*}
			&&	Z_{\alpha}(t)=\frac{\alpha_0}{1-\lambda t}, &&Z_{\beta}(t)=\frac{\beta_0}{1-\lambda t}, &&\text{and} &&Z_{\gamma}(t)=\frac{\gamma_0}{1-\lambda t},
		\end{align*}
		respectively, for $\alpha_0, \beta_0, \gamma_0, \lambda \in \textbf k^{\times}$. That is, 
		\begin{align*}
			&&\alpha=(\alpha_0, \lambda \alpha_0, \lambda^2\alpha_0, \dots), &&\beta=(\beta_0, \lambda\beta_0, \lambda^2\beta_0, \dots),&&\text{and} &&\gamma=(\gamma_0, \lambda\gamma_0, \lambda^2\gamma_0, \dots).
		\end{align*}
		Suppose that $\lambda \alpha_0\ne \gamma_0$ and $\gamma_0\ne \pm \sqrt{\lambda} \beta_0$. Then $\Hom_{\SUCob}(0,1)$ has dimension $3$, with basis
		\begin{align*}
			&u=	\begin{aligned}
				\resizebox{20pt}{!}{%
					\begin{tikzpicture}[tqft/cobordism/.style={draw,thick},
						tqft/view from=outgoing, tqft/boundary separation=30pt,
						tqft/cobordism height=40pt, tqft/circle x radius=8pt,
						tqft/circle y radius=4.5pt, tqft/every boundary component/.style={draw,rotate=90},tqft/every incoming
						boundary component/.style={draw,dotted,thick},tqft/every outgoing
						boundary component/.style={draw,thick}]
						\pic[tqft/cap,rotate=90,name=a, anchor={(1,-2)}]; 
				\end{tikzpicture} }
			\end{aligned}
			&	&\theta=	\begin{aligned}
				\resizebox{20pt}{!}{%
					\begin{tikzpicture}[tqft/cobordism/.style={draw,thick},
						tqft/view from=outgoing, tqft/boundary separation=30pt,
						tqft/cobordism height=40pt, tqft/circle x radius=8pt,
						tqft/circle y radius=4.5pt, tqft/every boundary component/.style={draw,rotate=90},tqft/every incoming
						boundary component/.style={draw,dotted,thick},tqft/every outgoing
						boundary component/.style={draw,thick}]
						\pic[tqft/cap,rotate=90,name=c, anchor={(1,-2)}]; 
						\node at ([xshift=-7pt]c-outgoing boundary)[color=gray]{$   \figureXv$};
				\end{tikzpicture} }
			\end{aligned},
			&	&\theta[2]=	\begin{aligned}
				\resizebox{20pt}{!}{%
					\begin{tikzpicture}[tqft/cobordism/.style={draw,thick},
						tqft/view from=outgoing, tqft/boundary separation=30pt,
						tqft/cobordism height=40pt, tqft/circle x radius=8pt,
						tqft/circle y radius=4.5pt, tqft/every boundary component/.style={draw,rotate=90},tqft/every incoming
						boundary component/.style={draw,dotted,thick},tqft/every outgoing
						boundary component/.style={draw,thick}]
						\pic[tqft/cap,rotate=90,name=c, anchor={(1,-2)}]; 
						\node at ([xshift=-7pt]c-outgoing boundary)[color=gray]{$   \figureXv$};
						\node at ([xshift=-3pt]c-outgoing boundary)[color=gray]{$   \figureXv$};
				\end{tikzpicture} }
			\end{aligned}:=
			\begin{aligned}
				\resizebox{50pt}{!}{%
					\begin{tikzpicture}[tqft/cobordism/.style={draw,thick},
						tqft/view from=outgoing, tqft/boundary separation=30pt,
						tqft/cobordism height=40pt, tqft/circle x radius=8pt,
						tqft/circle y radius=4.5pt, tqft/every boundary component/.style={draw,rotate=90},tqft/every incoming
						boundary component/.style={draw,dotted,thick},tqft/every outgoing
						boundary component/.style={draw,thick}]
						\pic[tqft/cap,rotate=90,name=c, anchor={(1,-2)}]; 
						\pic[tqft/cap,rotate=90,name=a, anchor={(0,-2)}]; 
						\node at ([xshift=-7pt]c-outgoing boundary)[color=gray]{$   \figureXv$};
						\node at ([xshift=-7pt]a-outgoing boundary)[color=gray]{$   \figureXv$};
						\pic[tqft/reverse pair of pants,
						rotate=90,name=b,at=(c-outgoing boundary)]; 
				\end{tikzpicture} }
			\end{aligned}.
		\end{align*}
		In fact, in this case $K=1$ and the handle relation is
		\begin{align*}
			x-\lambda \text{id}=0,
		\end{align*}
		see equation \eqref{skein relation}. Hence, by definition, morphisms in $\SUCob$ are linear combinations of cobordisms whose connected components have no handles. 
		Hence, by Propositions \ref{prop:orientable with boundary}  and \ref{prop unorientable with boundary}, the set $\{u, \theta, \theta[2]\}$ generates $\Hom_{\SUCob}(0,1)$. Moreover, the matrix of inner product (see Definition \ref{inner product}) of this set is given by 
		\begin{align*}
			A:=	\begin{bmatrix}
				\alpha_0 &\beta_0 &\gamma_0\\
				\beta_0 &\gamma_0 &\lambda\beta_0\\
				\gamma_0 &\lambda\beta_0 &\lambda\gamma_0
			\end{bmatrix},
		\end{align*}
		with determinant
		\begin{align*}
			\det(A)=\alpha_0(\lambda\gamma_0^2-\lambda^2\beta_0^2)+\gamma_0(\lambda\beta_0^2-\gamma_0^2)=(\alpha_0\lambda -\gamma_0)(\gamma_0^2-\lambda\beta_0^2),
		\end{align*}
		which is nonzero by our assumptions on $\alpha_0, \beta_0$ and $\gamma_0$. Hence $\{u, \theta, \theta[2]\}$  is  a basis for $\Hom_{\SUCob}(0,1)$.
	\end{example}

	\begin{definition}
		Let $(\alpha, \beta, \gamma)$ be as  in Lemma \ref{handle is negligible}. We denote by $\PsUCob$ the pseudo-abelian envelope of  $\SUCob$, i.e., $$\PsUCob=\mathcal K(\mathcal A(\SUCob)),$$ see Definition \ref{karoubian}.
	\end{definition}
	
	Thus, $\PsUCob$ is a pseudo-abelian (that is, additive and idempotent-complete) rigid symmetric tensor category, with finite dimensional Hom spaces. Our main results, Theorems \ref{maintheorem1} and \ref{maintheorem2}, describe the relationship between these categories and $\Rep(S_t\wr \mathbb Z_2)$.

	\subsection{Universal properties}\label{universal property} In this section we state the universal properties of the categories $\UCob$, $\VUCob$ and $\SUCob$. Let $\mathcal C$ be a symmetric monoidal category and let $A$ be an  Frobenius algebra in $\mathcal C$. Consider the following morphisms in $\End_{\mathcal C}(A)$,
	\begin{align}\label{x,y}
		\begin{aligned}
			&x: A\xrightarrow{\Delta_A} A\otimes A \xrightarrow{m_A} A, \ \ \text{and}\\
			&y: A \xrightarrow{\sim} 1\otimes A \xrightarrow{\theta_A\otimes 1}A\otimes A \xrightarrow{m_A}A.
		\end{aligned}
	\end{align}
	We define
	\begin{align*}
		&&a_n:=\varepsilon_A x^n u_A, \ \ \ b_{n}:=\varepsilon_A x^ny u_A, \ \ \text{and} \ 
		c_{n}:=	\varepsilon_A x^ny^2 u_A\in \End_{\mathcal C}(\mathbb 1_{\mathcal C}) \ \ \ \text{ for all } n\in \mathbb Z_{\geq 0}.
	\end{align*}
	Since $\mathbf k\simeq \End_{\mathcal C}(\mathbb{1}_{\mathcal C})$, we have that $a_n=\alpha_n \text{id}_{\mathbb 1_{\mathcal C}}$, $b_n=\beta_n \text{id}_{\mathbb 1_{\mathcal C}}$ and $c_n= \gamma_n \Id_{\mathbb 1_{\mathcal C}},$ for some $\alpha_n,\beta_n, \gamma_n\in \textbf{k}$. The sequences $\alpha:=(\alpha_n)_{n\geq 0}$, $\beta:=(\beta_n)_{n\geq 0}$ and $\gamma:=(\gamma_n)_{n\geq 0}$ will be called the \emph{evaluation} of $A$, and we say that $A$ is a $\emph{realization}$ of $(\alpha,\beta, \gamma)$ if the evaluation of $A$ is  $(\alpha,\beta,\gamma)$.

		Proposition \ref{frob algebras and TQFTs} induces an equivalence of categories
	\begin{align*}
		\begin{Bmatrix}\text{symmetric monoidal functors}  \UCob \to \mathcal C  
		\end{Bmatrix}
		\leftrightarrow 
		\begin{Bmatrix}
			\text{extended Frobenius algebras in } \mathcal C  
		\end{Bmatrix}.
	\end{align*}
\begin{proposition}
	When $\mathcal C$ is a tensor category, this results in an equivalence of categories
	\begin{align}\label{equivalence for VUCob}
		\begin{Bmatrix}
			\text{ symmetric tensor functors } \\
			\operatorname{VUCob}_{\alpha, \beta, \gamma} \to \mathcal C  
		\end{Bmatrix}
		\leftrightarrow 
		\begin{Bmatrix}
			\text{extended Frobenius algebras  in } \mathcal C  \\
			\text{ with evaluation } (\alpha, \beta, \gamma)
		\end{Bmatrix}.
	\end{align}
\end{proposition}	
\begin{proof}
	Let $F:\VUCob\to \mathcal C$ be a symmetric tensor functor and let $A=F(1)$ be the corresponding extended Frobenius algebra in $\mathcal C$. Note that 
	\begin{align*}
		a_n=\varepsilon_Ax^n u_A=F(\varepsilon)(F(m)F(\Delta))^nF(u)=F(\varepsilon (m\Delta)^nu)=F(\alpha_n)=\alpha_n.
	\end{align*} 
	Similarly, $b_n=\beta_n$ and $c_n=\gamma_n$. Hence $A$ has evaluation $(\alpha, \beta, \gamma)$.
	
	Conversely, let $A$ be an extended Frobenius algebra in $\mathcal C$ with evaluation $(\alpha, \beta, \gamma)$, and let $F_A:\UCob\to \mathcal C$ be the symmetric tensor functor mapping $1\to A$ and the generators of $\UCob$ to the corresponding structure maps of $A$. Then $F_A(\varepsilon (m\Delta)^nu)=\alpha_n$, $F_A(\varepsilon (m\Delta)^n y u)=\beta_n$ and $F_A(\varepsilon (m\Delta)^n y^2 u)=\gamma_n$. Hence $F_A$ factors through a symmetric tensor functor $F_A:\VUCob\to \mathcal C,$ as desired.
	\end{proof}

	Given an extended Frobenius algebra $A$ with extended evaluation $(\alpha, \beta, \gamma)$ in $\mathcal C$, denote by $F_A:\VUCob\to \mathcal C$ the corresponding functor  mapping $1$ to $A$. For $(\alpha,\beta, \gamma)$ as in Lemma \ref{handle is negligible}, $F_A$ factors through $\SUCob$ if and only if $F_A$ annhiliates the handle polynomial, see equation \eqref{skein relation}. When $\mathcal C$ is pseudo-abelian, there is a unique extension $F_A:\PsUCob \to \mathcal C$. 
	
	\subsection{Finite realizations}\label{finite realizations}
	Let $A\in \mathcal C$ be an extended Frobenius algebra with evaluation $(\alpha, \beta, \gamma)$. We say that the realization is \emph{finite} if the  Hom spaces $\Hom(\mathbb{1}, A^{\otimes n})$ in $\mathcal C$ are finite dimensional. 
	
	The following theorem is the unoriented analogue of  \cite[Theorem 3.1]{KKO}.

	\begin{theorem}\label{thm: finite realizations}
		The tuple $(\alpha,\beta,\gamma)$ admits a finite  realization if and only if the following conditions hold:
		\begin{itemize}
			\item $\alpha$ is linearly recurrent, with generating function $Z_{\alpha}(T)$ satisfying
			\begin{align*}
				Z_{\alpha}(T)=\frac{p_{\alpha}(T)}{q(T)}, &&(p_{\alpha}(T), q(T))=1 &&\text{and} &&q(0)=1;
			\end{align*}   
			\item $\beta$ and $\gamma$ have generating functions
			\begin{align*}
				Z_{\beta}(T)=\frac{p_{\beta}(T)}{q(T)} &&\text{and} &&	Z_{\gamma}(T)=\frac{p_{\gamma}(T)}{q(T)}, 
			\end{align*}
			where $\deg(p_{\beta}(T)), \deg(p_{\gamma}(T))\leq \max(\deg(p_{\alpha}(T))+1, \deg(q(T)))$.
		\end{itemize}
	\end{theorem}
	
	\begin{proof}
		Assume first that $(\alpha,\beta,\gamma)$ admits a finite realization $A$ in $\mathcal C$. Consider the map $x\in \End_{\mathcal C}(A)$ as defined in \eqref{x,y}. Since by assumption $\tr(x^n)=\alpha_n$, we aim to find a linear relation satisfied by powers of $x$, and thus also by $\alpha$. 
		
		Since $\End_{\mathcal C}(A)$ is finite dimensional, there exists a minimal polynomial $m(T)  \in \textbf k[T]$ for $x$. As $m(x)=0$, then $x^lm(x)=0$ for all $l\geq 0$. 
		Taking trace on both sides of these equations yields a linearly recurrent relation with coefficients in $\mathbf k$ satisfied by $\alpha_l$, for all $l\geq \deg(m(T))$. That is, $\alpha$ is linearly recurrent, and thus has a generating function	
		\begin{align}\label{gen equation}
			Z_{\alpha}(T)=\frac{p_{\alpha}(T)}{q(T)}=\sum\limits_{n\geq 0} \alpha_nT^n, &&\text{where} &&(p_{\alpha}(T), q(T))=1 &&\text{and} &&q(0)=1.
		\end{align}   
		Let $N=\deg(p_{\alpha}(T)), M=\deg(q(T)),$ and $K=\max(N+1, M)$. Since the first term at which the recurrence happens is $\alpha_K$, we must have $K=\deg(m(T))$.
		
		We now want to compute generating functions for $\beta$ and $\gamma$. Let $y$ as in \eqref{x,y}. Recall that $\tr(yx^n)=\beta_n$ and $\tr(y^2x^n)=\gamma_n$, for all $n\geq 0$. Since  $yx^im(x)=0$ and $y^2x^im(x)=0$, taking traces on these equations we obtain 
		linear relations  satisfied by $\beta_l$ and $\gamma_l$, for $l\geq K$, equal to the ones satisfied by $\alpha$. Hence by Lemma \ref{lemma auxiliar gen functions}, $\beta$ and $\gamma$ have generating functions 
		\begin{align*}
			&&Z_{\beta}(T)=\frac{p_{\beta}(T)}{q(T)} &&\text{and} &&Z_{\gamma}(T)=\frac{p_{\gamma}(T)}{q(T)},
		\end{align*}
		with $\deg(p_{\beta}(T)), \deg(p_{\gamma}(T))< K$, as desired.
		\medbreak
		
		Conversely, assume $\alpha, \beta$ and $\gamma$ are as in the statement. Under these assumptions, Lemma \ref{handle is negligible} implies that the cobordism $\sigma$  is negligible, and thus we can quotient $\VUCob$ by the handle relation. The resulting category $\SUCob$  (see Section \ref{SUCob}), with extended Frobenius algebra given by the circle $S^1,$ equipped with the orientation reversing diffeomorphism $\phi: S^1 \to S^1$ and the Möebius band $\theta:\emptyset \to S^1$, gives a finite extended realization of $(\alpha,\beta,\gamma)$.
	\end{proof}

	\section{Exterior product decompositions} \label{section: direct sums}
	
	This section is the analogue of \cite[Section 2.4]{KKO} for the unoriented case. We want  to show a decomposition of $\PsUCob$ as an exterior product of categories. 
	
	Let $\mathcal C$ be a symmetric tensor category and consider an extended Frobenius algebra $A\in \mathcal C$  with  evaluation $(\alpha, \beta, \gamma)$. Suppose $A$ has finite dimensional Hom spaces and let 
	\begin{align}\label{caso 1}
		&Z_{\alpha}(T)=\frac{p_{\alpha}(T)}{q(T)}, &&Z_{\beta}(T)=\frac{p_{\beta}(T)}{q(T)} & \text{ and } &&Z_{\gamma}(T)=\frac{p_{\gamma}(T)}{q(T)}, 
	\end{align}
	be the generating functions of $\alpha, \beta$ and  $\gamma$, respectively, satisfying the conditions of  Theorem \ref{thm: finite realizations}. That is, $\deg(p_{\beta}(T)), \deg(p_{\gamma}(T))\leq \max(\deg(p_{\alpha}(T)), \deg(q(T))+1).$
	
	Consider the algebra homomorphism
	\begin{align*}
		\Psi: \Hom_{\mathcal C}(\mathbb 1 , A) \to \End_{\mathcal C}(A),
	\end{align*}
	induced from the action of $\Hom_{\mathcal C}(\mathbb 1, A)$ on $A$. Let $m_A, \Delta_A, u_A, \theta_A$ denote the multiplication, comultiplication, unit and crosscap morphisms of $A$, respectively. Define the handle $x_0:= m_A\Delta_A u_A$  and cross $y_0:=m_A\Delta_A \theta_A$ morphisms of $A$, which are in $\Hom_{\mathcal C}(\mathbb 1, A)$. We denote by $A_0$  the subalgebra of $\Hom_{\mathcal C}(\mathbb 1, A)$ generated by $x_0$. Since $A_0$  is finite dimensional, $x_0$ has a minimal (monic) polynomial $m(T)$, where
	\begin{align}\label{Q and T}
		q(T)=T^dm(T^{-1}), \ \ d=\deg(m(T)).
	\end{align}
	
	Let $e\in A_0$ be an idempotent. Then $\Psi(e)A$ is an extended Frobenius subalgebra of $A$ with cross morphism $ \Psi(e)(\theta_A)$ and unit $\Psi(e)(u_A)$, and we have a decomposition
	\begin{align*}
		A=\Psi(e) A \oplus \Psi(1-e)A,
	\end{align*} 
	as a direct sum of extended Frobenius subalgebras. We denote their handle endomorphisms by $x_0'$ and $x_0''$, and their evaluations by $(\alpha', \beta', \gamma')$ and $(\alpha'', \beta'', \gamma'')$, respectively.
	
	\begin{lemma}\label{gen functions caso 2}
		If $Z_{\alpha}(T), Z_{\beta}(T), Z_{\gamma}(T)$ are the generating functions of $(\alpha,\beta,\gamma)$, then the generating functions $Z_{\alpha}', Z_{\beta}', Z_{\gamma}'$ and $Z_{\alpha}'', Z_{\beta}'', Z_{\gamma}''$ of $\Psi(e)A$ and $\Psi(1-e)(A)$, respectively,  satisfy
		\begin{align*}
			\begin{aligned}
				&&Z_{\alpha}(T)=  Z_{\alpha}'(T)+Z_{\alpha}''(T), &&	Z_{\beta}(T)= Z_{\beta}'(T) + Z_{\beta}''(T),  &&\text{ and}
				&&Z_{\gamma}(T)=  Z_{\gamma}'(T)+Z_{\gamma}''(T).
			\end{aligned}
		\end{align*}
	\end{lemma}
	\begin{proof}
		We compute 
		\begin{align*}
			&	\beta_n''=\varepsilon (y_0(x_0'')^n) = \varepsilon(y_0(1-e)x_0^n)=\varepsilon(y_0x_0^{n})-\varepsilon(y_0ex_0^n)=\beta_n-\varepsilon(y_0(x_0')^n)=\beta_n-\beta_n'.
		\end{align*} 
		That is, $\beta_n=\beta_n'+\beta_n''$, and so the statement follows for $\beta$. Analogously, we have that
		\begin{align*}
			&&\alpha_n=\alpha_n'+\alpha_n'' &&\text{and} &&\gamma_n=\gamma_n'+\gamma_n'',
		\end{align*}
		and so it is also true for $\alpha$ and $\gamma$.
	\end{proof}
	
	\begin{remark}\label{remark about idemp} The parametrization of idempotents given in \cite[Section 2.4]{KKO} is still true in this context. That is, idempotents $e\in A_0$ are labelled by factorizations 
		\begin{align}\label{fact of U}
			m(T)=m'(T)m''(T),
		\end{align}
		where $m'(T)$ and $m''(T)$ are relatively prime. This labelling goes as follows. Given said factorization, we have that $e=a(x_0)m(x_0)$ is an idempotent, where $a(T),b(T)\in \textbf k[T]$ satisfy $a(T)m'(T)+b(T)m''(T)=1$. Conversely, given an idempotent $e=s(x_0)$, setting $m'(T)=\gcd(m(T),s(T))$ and $m''(T)=\gcd(m(T),1-s(T)),$ we get the desired factorization.
	\end{remark}
	
	Note that a factorization \eqref{fact of U}  induces a factorization 
	\begin{align*}
		q(T)=q'(T)q''(T),
	\end{align*}
	where $q'(T), q''(T)$ are determined by $m(T)$ and $m'(T)$  as in \eqref{Q and T}. By assumption, $m'(T)$ and $m''(T)$ are relatively prime, and thus we may assume that $T$ does not divide $m'(T)$. It follows that we have partial fraction decompostitions
	\begin{align}\label{desc alpha}
		\begin{aligned}
			&Z_{\alpha}(T)=  \frac{v_{\alpha}'(T)}{q'(T)} + \frac{v_{\alpha}''(T)}{q''(T)}  =: Z_{\alpha}'(T)+Z_{\alpha}''(T),\\
			&	Z_{\beta}(T)=  \frac{v_{\beta}'(T)}{q'(T)} + \frac{v_{\beta}''(T)}{q''(T)}=: Z_{\beta}'(T) + Z_{\beta}''(T),  &\text{ and, }\\
			&Z_{\gamma}(T)=  \frac{v_{\gamma}'(T)}{q'(T)} + \frac{v_{\gamma}''(T)}{q''(T)}=:Z_{\gamma}'(T)+Z_{\gamma}''(T),
		\end{aligned}
	\end{align}
	so that $Z_{\alpha}'(T), Z_{\beta}'(T)$ and $Z_{\gamma}'(T)$ are proper fractions, but $Z_{\alpha}''(T), Z_{\beta}''(T)$ and $Z_{\gamma}''(T)$  may not be proper. 
	
	\begin{remark}\label{remark on degrees}
		Since $m''(x_0'')=0$, we get equations $(x_0'')^nm''(x_0'')=0$, $y(x_0'')^nm''(x_0'')=0$ and $y^2(x_0'')^nm''(x_0'')=0$, for all $n\geq 0$. From the first set of equations we get the linear recurrence relation satisfied by the sequence $\alpha''$, which begins at $K=\max(\deg(v_{\alpha}''(T))+1, q''(T))$. The other two sets of equations imply that $\beta$ and $\gamma$ also satisfy this recurrence relation, and thus by Lemma \ref{lemma auxiliar gen functions} we must have 
		\begin{align*}
			\deg(v_{\beta}''(T)), \deg(v_{\gamma}''(T))\leq K.
		\end{align*}
		The same can be done for $\alpha', \beta', \gamma'$. That is, we get that
		\begin{align*}
			\deg(v_{\beta}'(T)), \deg(v_{\gamma}'(T))\leq K.
		\end{align*}
		This will be important to us due to Lemma \ref{handle is negligible}.
	\end{remark}

	The following proposition is the analogue of \cite[Proposition 2.6]{KKO} for the case of extended Frobenius algebras. The proof is analogous but we include it for the sake of completeness.

	\begin{proposition}\label{gen functions caso 1}
		If the generating functions of $(\alpha,\beta,\gamma)$ are as in \eqref{caso 1}, then the generating functions for the evaluations $(\alpha',\beta', \gamma')$ and $(\alpha'', \beta'', \gamma'')$ of $\Psi(e)A$ and $\Psi(1-e)(A)$  are $Z_{\alpha}'(T), Z_{\beta}'(T), Z_{\gamma}'(T)$ and $Z_{\alpha}''(T), Z_{\beta}''(T), Z_{\gamma}''(T)$ as given in \eqref{desc alpha}, respectively.
	\end{proposition}
	
	\begin{proof}
		Let 
		\begin{align*}
			&&	x_0'=ex_0,  \ y_0'=ey_0&&\text{ and} &&x_0''=(1-e)x_0,  \ y_0''=(1-e)y_0,
		\end{align*}
		be the handle and cross endomorphisms of $\Psi(e)A$ and $\Psi(1-e)A$, respectively. Let $s(T)\in \textbf{k}(T)$ be such that $e=s(x_0)\in A_0$. The extended evaluation of $\Psi(e)(A)$ is given by
		\begin{align}\label{eq eval}
			\begin{aligned}
				&	\alpha_n'=\varepsilon (x_0')^n u = \varepsilon(ex_0^n)u=\varepsilon(s(x_0)x_0^n)u,\\
				&	\beta_n'=\varepsilon (y_0x_0')^nu = \varepsilon(ey_0x_0^n)u=\varepsilon(s(x_0)y_0x_0^n)u,\\
				&\gamma_n'=\varepsilon (y_0^2x_0')^n u= \varepsilon(ey_0^2x_0^n)u=\varepsilon(s(x_0)y_0^2x_0^n)u,
			\end{aligned}
		\end{align}
		and thus it is  uniquely determined by $e$. Analogously, the extended evaluation of $\psi(1-e)A$ is uniquely determined by $e$. By Remark \ref{remark about idemp}, this implies that the extended evaluations of $\psi(e)A$ and $\psi(1-e)A$ are uniquely determined by the factorization of $m(T)$.

		Since the evaluations depend only on the factorization of $m(T)$, we may compute them in a specific setting. Consider symmetric tensor categories $\mathcal C'$ and $\mathcal C''$. Let $A'\in \mathcal C'$ and $A''\in \mathcal C''$ be extended Frobenius algebras such that the generating functions of their evaluations are $Z_{\alpha}'(T), Z_{\beta}'(T), Z_{\gamma}'(T)$ and $ Z_{\alpha}''(T),Z_{\beta}''(T), Z_{\gamma}''(T)$ as in \eqref{desc alpha}, respectively, and  their handle endomorphisms $x_0'$ and $x_0''$ have minimal polynomials $m'(T)$ and  $m''(T)$ as in \eqref{fact of U}. By Remark \ref{remark on degrees} and Section \ref{SUCob}, we know such categories exist, see Definition \ref{def of skein category}. Denote by $y_0'$ and $y_0''$ the cross endomorphisms of $A'$ and $A''$.  
		
		Let $\mathcal C' \boxtimes \mathcal C''$ be the external tensor product of $\mathcal C'$ and $\mathcal C''$ \cite[Section 2.2]{O}, and consider the object
		\begin{align*}
			A:= (A'\boxtimes \mathbb 1) \oplus (\mathbb 1 \boxtimes A'') \in \mathcal C' \boxtimes \mathcal C'',
		\end{align*}
		with extended Frobenius algebra structure induced from that of $A'$ and $A''$.  Then the handle and cross endomorphisms of $A$ are $x_0=x_0' \oplus x_0''$ and $y_0=y_0'\oplus y_0''$, respectively, and the generating functions of $A$ are
		\begin{align*}
			&&Z_{\alpha}=Z_{\alpha}' + Z_{\alpha}'', && Z_{\beta}=Z_{\beta}'+Z_{\beta}'', &&\text{and} &&Z_{\gamma}= Z_{\gamma}'+Z_{\gamma}''.
		\end{align*}
		Since $m'(T)$ and $m''(T)$ are relatively prime, the minimal polynomial of $x_0$ is $m(T)=m'(T)m''(T)$. The idempotents determined by them (see Remark \ref{remark about idemp}) are the units of $A'\boxtimes \mathbb 1$ and $\mathbb 1 \boxtimes A''$, respectively. Thus, the respective evaluations can be computed as in \eqref{eq eval}, and the result follows. 
	\end{proof}

	We list in what follows a series of results from \cite[Section 2.6]{KKO} that still hold in our context. From now on, let $(\alpha', \beta', \gamma')$ and $(\alpha'', \beta'', \gamma'')$ have generating functions as in \eqref{desc alpha} (satisfying the conditions in  Remark \ref{remark on degrees}). Let 
	\begin{align*}
		&	\text{SUCob}':=\text{SUCob}_{\alpha', \beta', \gamma' }&&\text{and} &&\text{SUCob}'':=\text{SUCob}_{\alpha'', \beta'', \gamma''}.
	\end{align*}
	Denote by $A'$ and $A''$ the extended Frobenius algebras in $\text{SUCob}'$ and $\text{SUCob}''$ given by the circle object with structure maps the generating cobordisms of $\text{SUCob}'$ and $\text{SUCob}''$, respectively. Consider also the category $\text{SUCob}' \boxtimes \text{SUCob}''$
	with extended Frobenius algebra
	\begin{align}\label{A}
		\mathcal A:=A'\boxtimes \mathbb 1 \oplus \mathbb 1 \boxtimes A'',
	\end{align}
	and denote by $\mathcal U(T)$ the minimal polynomial for its handle endomorphism. Define 
	\begin{align}\label{sum of seq}
		&\alpha:=\alpha'+\alpha'', &\beta:=\beta' + \beta'' &&\text{ and } &&\gamma=\gamma'+\gamma'',
	\end{align}
	so that $\mathcal A$ has extended evaluation $(\alpha,\beta,\gamma).$
	
	\begin{lemma}\cite[Lemma 2.8]{KKO} Let $m'(T)$ and $m''(T)$ denote the handle polynomials of $A'$ and $A''$, respectively. Then
		\begin{align*}
			\mathcal U(T) = \text{lcm}(m'(T), m''(T)).
		\end{align*}
	\end{lemma}
	
	\begin{corollary}\cite[Corollary 1]{KKO}\label{U divisor}
		Consider the category $\operatorname{SUCob}_{\alpha, \beta, \gamma}$ with extended Frobenius algebra $A$, for $\alpha, \beta$ and $\gamma$ as in \eqref{sum of seq}. The minimal polynomial $m_{\alpha}(T)$  for its handle endomorphism is a divisor of $\mathcal U(T)$.
	\end{corollary}
	
	The following Proposition is the analogue of \cite[Proposition 2.10]{KKO} for the unoriented case. 
	\begin{proposition} With the notation above, 
		$\mathcal U(T)=m_{\alpha}(T)$ iff there exists a functor
		\begin{align}\label{decomposition functor}
			F:\operatorname{SUCob}_{\alpha,\beta,\gamma} \to \operatorname{SUCob}'\boxtimes \operatorname{SUCob}'',
		\end{align} 
		mapping the extended Frobenius algebra $A$ given by the circle object of $\SUCob$ to  $\mathcal A$, see \eqref{A},  and the extended Frobenius structure maps of $A$ to the ones of $\mathcal A$.
	\end{proposition}
	\begin{proof}
		Since both $A$ and $\mathcal A$ have extended evaluation $(\alpha,\beta,\gamma)$, by the universal property (see Section \ref{universal property}), there exists a symmetric tensor functor $F:\VUCob \to \text{SUCob}'\boxtimes \text{SUCob}''$ as described. Let $x_0$ and $z_0$ denote the handle endomorphisms of $\operatorname{SUCob}_{\alpha,\beta,\gamma}$ and $\operatorname{SUCob}'\boxtimes \operatorname{SUCob}$, respectively. Then $F$ factors through $\SUCob $ if and only if $F$ maps $m_{\alpha}(x_0)\mapsto 0$. If 	$\mathcal U(T)=m_{\alpha}(T)$, then $F(m_{\alpha}(x_0))=m_{\alpha}(z_0)=\mathcal U(z_0)=0$. For the converse, if the functor $F$ factors, then $m_{\alpha}(z_0)=0$, and so by Corollary \ref{U divisor} we must have $m_{\alpha}(T)=\mathcal U(T)$.
	\end{proof}
	
	Assume from now on that $m_{\alpha}(T)=\mathcal U(T)$. The following is the analogue of  \cite[Proposition 2.14]{KKO}.
	\begin{proposition}
		With the notation above, the functor \eqref{decomposition functor} induces an equivalence of pseudo-abelian symmetric tensor categories
		\begin{align*}
			F:\operatorname{UCob}_{\alpha, \beta, \gamma} \xrightarrow{\sim}\operatorname{UCob}'\boxtimes \operatorname{UCob}''.
		\end{align*}
	\end{proposition}
	
	\begin{proof}
		Consider the functor $F:\text{SUCob}_{(\alpha,\beta,\gamma) } \to \text{SUCob}'\boxtimes \text{SUCob}''$ as in \eqref{decomposition functor}, which maps $A\mapsto \mathcal A$. Taking the pseudo-abelian envelope, we get an induced functor $$F:\text{UCob}_{\alpha,\beta,\gamma} \to \text{UCob}'\boxtimes \text{UCob}''.$$
		
		On the other hand, consider the factorization $m_{\alpha}(T)=m'(T)m''(T),$ which induces a decomposition $Z_{\alpha}(T)=Z_{\alpha'}(T)+Z_{\alpha''}(T)$ as in equation \eqref{desc alpha}. By Remark \ref{remark about idemp} we have a corresponding idempotent $e\in \Hom(\mathbb 1, A),$ and a decomposition
		\begin{align*}
			A=\Psi(e)A \oplus \Psi(1-e)A,
		\end{align*}
		where the generating functions of $\Psi(e)A$ and $\Psi(1-e)A$ are as in Proposition \ref{gen functions caso 1}. Note that $A'$ in $\text{SUCob}'$ and $\Psi(e)A$ in $\PsUCob$  have the same extended evaluation and thus by the universal property, see Section \ref{universal property}, there exists a symmetric tensor functor $\text{SUCob}' \to \PsUCob$. Similarly, we have a functor  $\text{SUCob}'' \to \PsUCob$ mapping $A''\mapsto \Psi(1-e)A$. On the other hand, by the universal property of the external tensor product (see \cite[Section 2.2]{O}), we have a symmetric tensor functor
		\begin{align*}
			G:	\text{SUCob}' \boxtimes \text{SUCob}''  \to \PsUCob,
		\end{align*}
		mapping $A'\boxtimes  \mathbb 1$ to $\Psi(e)A$ and $\mathbb 1 \boxtimes A'$ to $\Psi(1-e)A$. This induces a unique tensor functor from the pseudo-abelian closure of the source category,
		\begin{align*}
			G:	\text{UCob}' \boxtimes \text{UCob}''  \to \PsUCob.
		\end{align*}
		Note that the compositions 
		\begin{align*}
			\PsUCob \xrightarrow{F} \text{UCob}'\boxtimes \text{UCob}'' \xrightarrow{G} \PsUCob,
		\end{align*}
		and 
		\begin{align*}
			\text{UCob}'\boxtimes \text{UCob}''\xrightarrow{G} \PsUCob \xrightarrow{F} \text{UCob}'\boxtimes \text{UCob}'',
		\end{align*}
		map $A$ to itself and $A'\boxtimes \mathbb 1$ and $\mathbb 1 \boxtimes A''$ to themselves, respectively. Hence $G\circ F$ and $F\circ G$ are isomorphic (as tensor functors) to the corresponding identity functor, and the statement follows. 
	\end{proof}

	\section{(Unoriented) orientable cobordisms}\label{section:SOCob}
	Let $\alpha$ be a linearly recurrent sequence in \textbf k. In this section, we define the category $\SOCob,$ obtained by modding out the crosscap cobordism $\theta$. We also show that  when specializing to the sequence $\alpha=(\alpha_0, \lambda \alpha_0, \lambda^2 \alpha_0, \dots)$, with $\alpha_0, \lambda \in \mathbf k^{\times}$ such that $\lambda\alpha_0$ is not a non-negative even integer, its pseudo-abelian envelope $\PsOCob$ is equivalent to $\Rep(S_{t}\wr \mathbb Z_2),$ for $t=\frac{\lambda \alpha_0}{2}$.
	
	\begin{definition}
		Let $\alpha=(\alpha_0, \alpha_1, \dots)$ be a linearly recurrent sequence in $\mathbf{k}$, and let
		$\beta= \gamma=(0, 0, \dots)$.  We define the \emph{orientable skein} category $\SOCob,$ with:
		\begin{itemize}
			\item Objects: Non-negative integers $n\in \mathbb Z_{\geq 0}.$
			\item Morphisms: $\Hom_{\SOCob}(m,n)$ is equal to $\Hom_{\SUCob}(m,n)$ modulo the relation $\theta=0$. 
			\item Composition: Induced from $\SUCob$.
		\end{itemize}
	\end{definition}
	
	The category $\SOCob$ is a rigid symmetric tensor category. Morphisms in $\SOCob$ are \textbf k-linear combinations of orientable cobordisms without closed components and genus at most $K$, for $K$ as in $\eqref{M y N}$. Hence, connected components  of cobordisms in $\SOCob$ are as in Proposition \ref{prop:orientable with boundary}. Same as in $\SUCob$,  Hom spaces in $\SOCob$ are finite dimensional. We denote its pseudo-abelian envelope by $\PsOCob$.  We note that the category $\SOCob$ of orientable cobordisms is different from the category of oriented cobordisms, see Example \ref{orientable example}.
	
	Let $\mathcal C$ be a symmetric tensor category with  an extended Frobenius algebra $A,$ such that $A$ has extended evaluation $\alpha$ and $\beta= \gamma=0,$ with $\alpha$  linearly recurrent. 
	Let $F_A:\VUCob\to \mathcal C$ be the symmetric tensor functor given by the universal property of $\VUCob$, mapping $1$ to $A$, see Section \ref{universal property}. Then $F_A$ factors through $\SUCob$ if and only if $F_A$ annhiliates the handle polynomial, and through $\SOCob$ if and only if it also anhiliates $\theta$, i.e, if and only if $0=F_A(\theta)=\theta_A$. In such case, when $\mathcal C$ is pseudo-abelian there is a unique extension $F_A:\PsOCob \to \mathcal C$. 
	\begin{example}\label{orientable example}
		The cobordisms below
		\begin{align*}
			x^n&=
			\begin{aligned}
				\resizebox{200pt}{!}{%
					\begin{tikzpicture}[tqft/cobordism/.style={draw,thick},
						tqft/view from=outgoing, tqft/boundary separation=30pt,
						tqft/cobordism height=40pt, tqft/circle x radius=8pt,
						tqft/circle y radius=4.5pt, tqft/every boundary component/.style={draw,rotate=90},tqft/every incoming
						boundary component/.style={draw,dotted,thick},tqft/every outgoing
						boundary component/.style={draw,thick}]
						\pic[tqft/pair of pants,
						rotate=90,name=h,anchor={(1,-1)}]; 
						\node at ([xshift=105pt, yshift=-30pt]h-outgoing boundary 1)[color=black][font=\huge]{$\underbrace{\ \ \ \ \ \ \ \ \ \ \ \ \ \ \ \ \ \ \ \ \ \ \ \ \ \  \ \ \ \ \ \ \ \ \ \ \ \ \ \ \ \ \ \ \ \ }_n$};
						\pic[tqft/reverse pair of pants,
						rotate=90,name=i,at=(h-outgoing boundary)];  
						\pic[tqft/pair of pants,
						rotate=90,name=a,at=(i-outgoing boundary)]; 
						\pic[tqft/reverse pair of pants,
						rotate=90,name=b,at=(a-outgoing boundary)];  
						\node at ([xshift=20pt]b-outgoing boundary 1)[color=black]{$\dots$};
						\pic[tqft/pair of pants,
						rotate=90,name=l,anchor={(1,-6)}]; 
						\pic[tqft/reverse pair of pants,
						rotate=90,name=j,at=(l-outgoing boundary)];  
				\end{tikzpicture} }
			\end{aligned}, \\
			\phi x^n&=
			\begin{aligned}
				\resizebox{220pt}{!}{%
					\begin{tikzpicture}[tqft/cobordism/.style={draw,thick},
						tqft/view from=outgoing, tqft/boundary separation=30pt,
						tqft/cobordism height=40pt, tqft/circle x radius=8pt,
						tqft/circle y radius=4.5pt, tqft/every boundary component/.style={draw,rotate=90},tqft/every incoming
						boundary component/.style={draw,dotted,thick},tqft/every outgoing
						boundary component/.style={draw,thick}]
						\pic[tqft/cylinder,rotate=90,name=g,anchor={(1,-1)}];
						\node at ([xshift=20pt]g-incoming boundary 1)[color=black]{$   \leftrightarrow$};
						\node at ([xshift=140pt, yshift=-40pt]g-outgoing boundary 1)[color=black][font=\huge]{$\underbrace{\ \ \ \ \ \ \ \ \ \ \ \ \ \ \ \ \ \ \ \ \ \ \ \ \ \ \ \ \ \ \ \ \ \ \ \ \ \ \ \ \ \ \ \ \ }_n$};
						\pic[tqft/pair of pants,
						rotate=90,name=h,at=(g-outgoing boundary)]; 
						\pic[tqft/reverse pair of pants,
						rotate=90,name=i,at=(h-outgoing boundary)];  
						\pic[tqft/pair of pants,
						rotate=90,name=a,at=(i-outgoing boundary)]; 
						\pic[tqft/reverse pair of pants,
						rotate=90,name=b,at=(a-outgoing boundary)];  
						\node at ([xshift=20pt]b-outgoing boundary 1)[color=black]{$\dots$};
						\pic[tqft/pair of pants,
						rotate=90,name=l,anchor={(1,-7)}]; 
						\pic[tqft/reverse pair of pants,
						rotate=90,name=j,at=(l-outgoing boundary)];  
				\end{tikzpicture} }
			\end{aligned}, \\
			x^mu \varepsilon x^n&=
			\begin{aligned}
				\resizebox{290pt}{!}{%
					\begin{tikzpicture}[tqft/cobordism/.style={draw,thick},
						tqft/view from=outgoing, tqft/boundary separation=30pt,
						tqft/cobordism height=40pt, tqft/circle x radius=8pt,
						tqft/circle y radius=4.5pt, tqft/every boundary component/.style={draw,rotate=90},tqft/every incoming
						boundary component/.style={draw,dotted,thick},tqft/every outgoing
						boundary component/.style={draw,thick}]
						\pic[tqft/pair of pants,
						rotate=90,name=a,anchor={(1,0)}]; 
						\node at ([xshift=60pt, yshift=-30pt]a-outgoing boundary 1)[color=black][font=\huge]{$\underbrace{\ \ \   \ \ \ \ \ \ \ \ \ \ \ \ \ \ \ \ \ \ \ \ \ \ \ \ \ \ \ \ }_n$};
						\pic[tqft/reverse pair of pants,
						rotate=90,name=b,at=(a-outgoing boundary)];  
						\node at ([xshift=20pt]b-outgoing boundary 1)[color=black]{$\dots$};
						\pic[tqft/pair of pants,
						rotate=90,name=l,anchor={(1,-3)}]; 
						\pic[tqft/reverse pair of pants,
						rotate=90,name=j,at=(l-outgoing boundary)];  
						\pic[tqft/cup,
						rotate=90,name=k,at=(j-outgoing boundary)];  
						\pic[tqft/cap,
						rotate=90,name=l,at=(k-incoming boundary)];  
						\pic[tqft/pair of pants,
						rotate=90,name=m,at=(l-outgoing boundary)];  
						\pic[tqft/reverse pair of pants,
						rotate=90,name=n,at=(m-outgoing boundary)];  
						\node at ([xshift=20pt, yshift=-40pt]n-outgoing boundary 1)[color=black][font=\huge]{$\underbrace{\ \ \   \ \ \ \ \ \ \ \ \ \ \ \ \ \ \ \ \ \ \ \ \ \ \ \ \ \ \ \ }_m$};
						\node at ([xshift=20pt]n-outgoing boundary 1)[color=black]{$\dots$};
						\pic[tqft/pair of pants,
						rotate=90,name=l,anchor={(1,-9)}]; 
						\pic[tqft/reverse pair of pants,
						rotate=90,name=l1,at=(l-outgoing boundary)]; 
				\end{tikzpicture} }
			\end{aligned}, 
		\end{align*} for $0\leq n,m \leq K$, where $K$ is as in $\eqref{linear recurrence for alpha}$, span $\End_{\SOCob}(1)$, see Proposition \ref{prop:orientable with boundary}.
	\end{example}

	\subsection{Equivalence with the category $\Rep(S_t \wr \mathbb Z_2)$} \label{section:theorem I} Throughout this section, let \textbf{k} be an algebraically closed field of characteristic zero.

	Let
	\begin{align*}
		&	\alpha=(\alpha_0,\lambda \alpha_0, \lambda^2 \alpha_0, \dots)&&\text{and} &&\beta=\gamma=(0, 0 , \dots),
	\end{align*}
	for some $\alpha_0, \lambda\in \mathbf k^\times.$ Here, the generating function for $\alpha$ is given by
	\begin{align*}
		Z_{\alpha}(T)=\frac{\alpha_0}{1-\lambda T}.
	\end{align*}
	Hence, $\SOCob$ is the quotient of $\VUCob$ by the relations $x-\lambda\Id=0$ and $\theta=0$. 
	
	We are interested in finding a spanning set for $\Hom_{\SOCob}(0,m)$ for all $m\in \mathbb N$. Let $\mathcal W_m$ denote the subspace spanned by connected cobordisms in $\Hom_{\SOCob}(0,m)$.
	For $m=1$, there is only one connected cobordism, namely $u$, see Section \ref{section:gens and rels}. Hence $\mathcal W_1$ has dimension 1. 
	
	\begin{proposition}\label{basis for connected unoriented orientable}
		Let $m\geq 2$.	The set $\{\xi^m_{\overline J}  : \overline J \in \mathcal R_m \}$ from Definition \ref{xi definition} is a basis for $\mathcal W_m$. In particular, $\dim(\mathcal W_m)=2^{m-1}$.
	\end{proposition}
	
	\begin{proof}
		We show first that  $\{\xi^m_{\overline J}  : \overline J \in \mathcal R_m \}$ spans $\mathcal W_m$. In fact, connected cobordisms in $\SOCob$ are orientable and thus they are of the form given in Proposition \ref{prop:orientable with boundary}. Moreover, by the relation $x=\lambda\Id,$  every handle gets replaced by a multiple of the identity and so every cobordism has genus zero. Hence every connected cobordism $0\to m$ in $\SOCob$ is in the set $\{\xi^m_{\overline J}  : \overline J \in \mathcal R_m \}$.
		
		On the other hand, note that $\alpha_{m-1}=\lambda^{m-1}\alpha_0\ne 0,$ as $\alpha_0, \lambda\in \mathbf k^{\times}$. Since the sequence $\gamma$ is constantly zero,  by Theorem \ref{xi are li} cobordisms in the set $\{\xi^m_{\overline J}  : \overline J \in \mathcal R_m \}$ are linearly independent.
	\end{proof}
	
	\begin{remark}\label{basis from basis for connected}
		Now that we have a basis for connected cobordisms $0 \to m$, we can describe all cobordisms $0\to m$ as follows. For every partition $P$ of $\{1, \dots, m\}$ and every $p\in P$, assign a connected cobordism $\xi_{\overline J}^{|p|}$ to $p$, where $\overline J\in \mathcal R_{|p|}$. 
		Then the set of cobordisms obtained by running through all possible partitions $P$ and all posible classes $\overline J\in \mathcal R_{|p|}$ for every $p\in P$ gives a  spanning set for $\Hom_{\SOCob}(0,m)$.
	\end{remark}

	\begin{example}\label{example S_m for m=2}
		If $\lambda \alpha_0 \ne 2$, the following is a basis for $\Hom_{\SOCob}(0,2)$:
		\begin{align*}
			&&	\begin{aligned}
				\resizebox{17pt}{!}
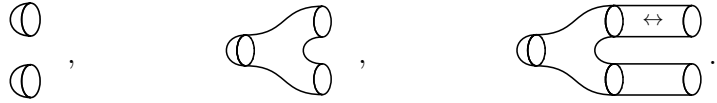
{%
					\begin{tikzpicture}[tqft/cobordism/.style={draw,thick},
						tqft/view from=outgoing, tqft/boundary separation=30pt,
						tqft/cobordism height=40pt, tqft/circle x radius=8pt,
						tqft/circle y radius=4.5pt, tqft/every boundary component/.style={draw,rotate=90},tqft/every incoming
						boundary component/.style={draw,dotted,thick},tqft/every outgoing
						boundary component/.style={draw,thick}]
						\pic[tqft/cap,
						rotate=90,name=a,anchor={(1,-2)}]; 
						\pic[tqft/cap,
						rotate=90,name=b,anchor={(2,-2)}]; 
				\end{tikzpicture} }
			\end{aligned}\ \ ,
			&&\begin{aligned}
				\resizebox{45pt}{!}{%
					\begin{tikzpicture}[tqft/cobordism/.style={draw,thick},
						tqft/view from=outgoing, tqft/boundary separation=30pt,
						tqft/cobordism height=40pt, tqft/circle x radius=8pt,
						tqft/circle y radius=4.5pt, tqft/every boundary component/.style={draw,rotate=90},tqft/every incoming
						boundary component/.style={draw,dotted,thick},tqft/every outgoing
						boundary component/.style={draw,thick}]
						\pic[tqft/cap,
						rotate=90,name=a,anchor={(1,-2)}]; 
						\pic[tqft/pair of pants,
						rotate=90,name=b,at=(a-outgoing boundary)]; 
				\end{tikzpicture} }
			\end{aligned}\ \ ,
			&&\begin{aligned}
				\resizebox{75pt}{!}{%
					\begin{tikzpicture}[tqft/cobordism/.style={draw,thick},
						tqft/view from=outgoing, tqft/boundary separation=30pt,
						tqft/cobordism height=40pt, tqft/circle x radius=8pt,
						tqft/circle y radius=4.5pt, tqft/every boundary component/.style={draw,rotate=90},tqft/every incoming
						boundary component/.style={draw,dotted,thick},tqft/every outgoing
						boundary component/.style={draw,thick}]
						\pic[tqft/cap,
						rotate=90,name=a,anchor={(1,-2)}]; 
						\pic[tqft/pair of pants,
						rotate=90,name=b,at=(a-outgoing boundary)]; 
						\pic[tqft/cylinder,rotate=90,name=g at=(b-outgoing boundary 1),anchor={(1.5,-4)}];
						\pic[tqft/cylinder,rotate=90,name=f, at=(b-outgoing boundary 2)];
						\node at ([xshift=20pt]f-incoming boundary 1){$\leftrightarrow$};
				\end{tikzpicture} }
			\end{aligned}.
		\end{align*}
		In fact, we know this is a spanning set by Remark \ref{basis from basis for connected}. On the other hand, the matrix of inner products is given by
		\begin{align*}
			A=\begin{bmatrix}
				\alpha_0^2 & \alpha_0 &\alpha_0\\
				\alpha_0 &\alpha_1 &0\\
				\alpha_0 & 0 &\alpha_1
			\end{bmatrix},
		\end{align*}
		which has determinant
		\begin{align*}
			\det(A)=\alpha_0^2\alpha_1^2(\alpha_1-2)= \lambda \alpha_0^3 (\lambda\alpha_0-2)\ne 0,
		\end{align*}
		so the set is also linearly independent. Hence 
		\begin{align*}
			\dim(\Hom_{\SOCob}(0,2))=3.
		\end{align*}
	\end{example}

	\begin{example}\label{example S_m for m=3}
		Suppose that $\lambda\alpha_0\ne2, 4$. Then	the following is a basis for $\Hom_{\SOCob}(0,3)$:
		\begin{align*}
			&&	\begin{aligned}
				\resizebox{17pt}{!}
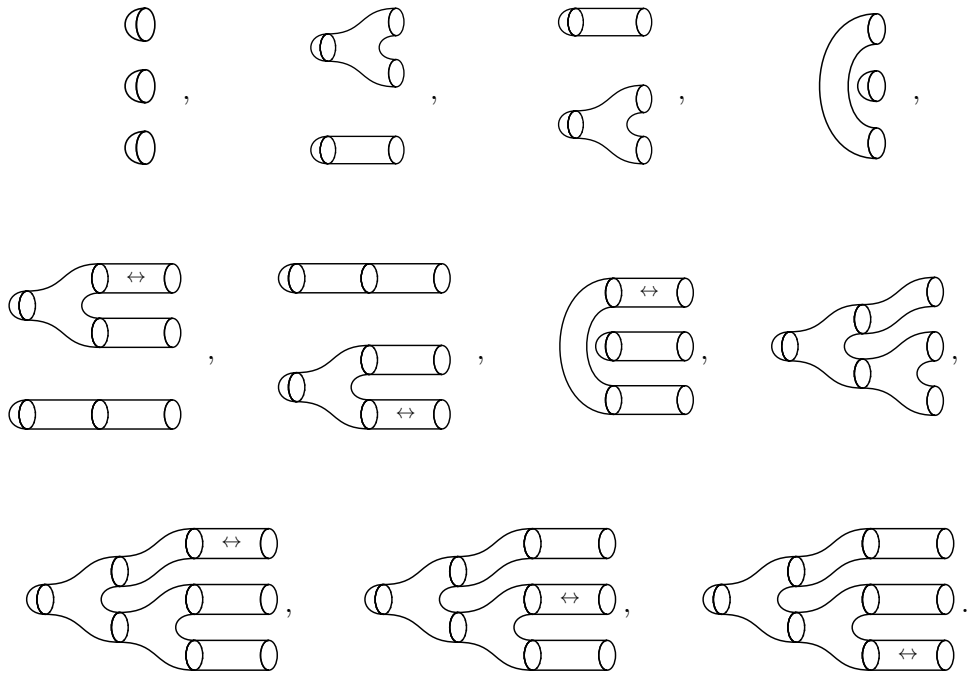
{%
					\begin{tikzpicture}[tqft/cobordism/.style={draw,thick},
						tqft/view from=outgoing, tqft/boundary separation=30pt,
						tqft/cobordism height=40pt, tqft/circle x radius=8pt,
						tqft/circle y radius=4.5pt, tqft/every boundary component/.style={draw,rotate=90},tqft/every incoming
						boundary component/.style={draw,dotted,thick},tqft/every outgoing
						boundary component/.style={draw,thick}]
						\pic[tqft/cap,
						rotate=90,name=a,anchor={(1,-2)}]; 
						\pic[tqft/cap,
						rotate=90,name=b,anchor={(2,-2)}]; 
						\pic[tqft/cap,
						rotate=90,name=c,anchor={(3,-2)}]; 
				\end{tikzpicture} }
			\end{aligned}\ \ ,
			&&\begin{aligned}
				\resizebox{40pt}{!}{%
					\begin{tikzpicture}[tqft/cobordism/.style={draw,thick},
						tqft/view from=outgoing, tqft/boundary separation=30pt,
						tqft/cobordism height=40pt, tqft/circle x radius=8pt,
						tqft/circle y radius=4.5pt, tqft/every boundary component/.style={draw,rotate=90},tqft/every incoming
						boundary component/.style={draw,dotted,thick},tqft/every outgoing
						boundary component/.style={draw,thick}]
						\pic[tqft/cap,
						rotate=90,name=a,anchor={(1,-2)}]; 
						\pic[tqft/pair of pants,
						rotate=90,name=b,at=(a-outgoing boundary)]; 
						\pic[tqft/cap,
						rotate=90,name=c,anchor={(3,-2)}]; 
						\pic[tqft/cylinder,
						rotate=90,at=(c-outgoing boundary),name=d]; 
				\end{tikzpicture} }
			\end{aligned}\ \ ,
			&&\begin{aligned}
				\resizebox{40pt}{!}{%
					\begin{tikzpicture}[tqft/cobordism/.style={draw,thick},
						tqft/view from=outgoing, tqft/boundary separation=30pt,
						tqft/cobordism height=40pt, tqft/circle x radius=8pt,
						tqft/circle y radius=4.5pt, tqft/every boundary component/.style={draw,rotate=90},tqft/every incoming
						boundary component/.style={draw,dotted,thick},tqft/every outgoing
						boundary component/.style={draw,thick}]
						\pic[tqft/cap,
						rotate=90,name=a,anchor={(1,-2)}]; 
						\pic[tqft/cap,
						rotate=90,name=c,anchor={(3,-2)}]; 
						\pic[tqft/pair of pants,
						rotate=90,name=b,at=(c-outgoing boundary)]; 
						\pic[tqft/cylinder,
						rotate=90,name=d,at=(a-outgoing boundary)]; 
				\end{tikzpicture} }
			\end{aligned}\ \ ,
			&&\begin{aligned}
				\resizebox{35pt}{!}{%
					\begin{tikzpicture}[tqft/cobordism/.style={draw,thick},
						tqft/view from=outgoing, tqft/boundary separation=30pt,
						tqft/cobordism height=40pt, tqft/circle x radius=8pt,
						tqft/circle y radius=4.5pt, tqft/every boundary component/.style={draw,rotate=90},tqft/every incoming
						boundary component/.style={draw,dotted,thick},tqft/every outgoing
						boundary component/.style={draw,thick}]
						\pic[tqft,
						incoming boundary components=0,
						outgoing boundary components=2,
						rotate=90,name=a,boundary separation=60pt,cobordism height=60pt,anchor={(1,0)}];
						\pic[tqft/cap,
						rotate=90,name=b,anchor={(0,-0.5)}]; 
				\end{tikzpicture} }
			\end{aligned}
			\ \ ,	\end{align*}
		
		\begin{align*}
			&&\begin{aligned}
				\resizebox{70pt}{!}{%
					\begin{tikzpicture}[tqft/cobordism/.style={draw,thick},
						tqft/view from=outgoing, tqft/boundary separation=30pt,
						tqft/cobordism height=40pt, tqft/circle x radius=8pt,
						tqft/circle y radius=4.5pt, tqft/every boundary component/.style={draw,rotate=90},tqft/every incoming
						boundary component/.style={draw,dotted,thick},tqft/every outgoing
						boundary component/.style={draw,thick}]
						\pic[tqft/cap,
						rotate=90,name=a,anchor={(1,-2)}]; 
						\pic[tqft/pair of pants,
						rotate=90,name=b,at=(a-outgoing boundary)]; 
						\pic[tqft/cylinder,rotate=90,name=g at=(b-outgoing boundary 1),anchor={(1.5,-4)}];
						\pic[tqft/cap,
						rotate=90,name=c,anchor={(3,-2)}]; 
						\pic[tqft/cylinder,
						rotate=90,name=h,at=(c-outgoing boundary)]; 
						\pic[tqft/cylinder,rotate=90, at=(h-outgoing boundary)];
						\pic[tqft/cylinder,rotate=90,name=f, at=(b-outgoing boundary 2)];
						\node at ([xshift=20pt]f-incoming boundary 1){$\leftrightarrow$};
				\end{tikzpicture} }
			\end{aligned} \ \ ,
			&&\begin{aligned}
				\resizebox{70pt}{!}{%
					\begin{tikzpicture}[tqft/cobordism/.style={draw,thick},
						tqft/view from=outgoing, tqft/boundary separation=30pt,
						tqft/cobordism height=40pt, tqft/circle x radius=8pt,
						tqft/circle y radius=4.5pt, tqft/every boundary component/.style={draw,rotate=90},tqft/every incoming
						boundary component/.style={draw,dotted,thick},tqft/every outgoing
						boundary component/.style={draw,thick}]
						\pic[tqft/cap,
						rotate=90,name=a,anchor={(1,-2)}]; 
						\pic[tqft/cap,
						rotate=90,name=c,anchor={(3,-2)}]; 
						\pic[tqft/pair of pants,
						rotate=90,name=b,at=(c-outgoing boundary)]; 
						\pic[tqft/cylinder,
						rotate=90,name=d,at=(a-outgoing boundary)]; 
						\pic[tqft/cylinder,
						rotate=90,name=e,at=(b-outgoing boundary)]; 
						\pic[tqft/cylinder,
						rotate=90,name=f,at=(b-outgoing boundary 2)]; 
						\pic[tqft/cylinder,
						rotate=90,name=g,at=(d-outgoing boundary)]; 
						\node at ([xshift=20pt]e-incoming boundary 1){$\leftrightarrow$};
				\end{tikzpicture} }
			\end{aligned}\ \  ,
			&&\begin{aligned}
				\resizebox{60pt}{!}{%
					\begin{tikzpicture}[tqft/cobordism/.style={draw,thick},
						tqft/view from=outgoing, tqft/boundary separation=30pt,
						tqft/cobordism height=40pt, tqft/circle x radius=8pt,
						tqft/circle y radius=4.5pt, tqft/every boundary component/.style={draw,rotate=90},tqft/every incoming
						boundary component/.style={draw,dotted,thick},tqft/every outgoing
						boundary component/.style={draw,thick}]
						\pic[tqft,
						incoming boundary components=0,
						outgoing boundary components=2,
						rotate=90,name=a,boundary separation=60pt,cobordism height=60pt,anchor={(1,0)}];
						\pic[tqft/cap,
						rotate=90,name=b,anchor={(0,-0.5)}]; 
						\pic[tqft/cylinder,
						rotate=90,name=d,at=(a-outgoing boundary 2)]; 
						\pic[tqft/cylinder,
						rotate=90,name=e,at=(a-outgoing boundary 1)]; 
						\pic[tqft/cylinder,
						rotate=90,name=f,at=(b-outgoing boundary)]; 
						\node at ([xshift=20pt]d-incoming boundary 1){$\leftrightarrow$};
				\end{tikzpicture} }
			\end{aligned},	&&	\begin{aligned}
				\resizebox{70pt}{!}{%
					\begin{tikzpicture}[tqft/cobordism/.style={draw,thick},
						tqft/view from=outgoing, tqft/boundary separation=30pt,
						tqft/cobordism height=40pt, tqft/circle x radius=8pt,
						tqft/circle y radius=4.5pt, tqft/every boundary component/.style={draw,rotate=90},tqft/every incoming
						boundary component/.style={draw,dotted,thick},tqft/every outgoing
						boundary component/.style={draw,thick}]
						\pic[tqft/pair of pants,
						rotate=90,name=a, anchor={(1,-2)}]; 
						\pic[tqft/cylinder to next,rotate=90,name=b, at=(a-outgoing boundary 2)];
						\pic[tqft/pair of pants,
						rotate=90,name=d,at=(a-outgoing boundary)]; 
						\pic[tqft/cap,
						rotate=90,at=(a-outgoing boundary),anchor={(0.5,2)}]; 
				\end{tikzpicture} }
			\end{aligned},\ \
		\end{align*}
		
		\begin{align*}
			&&	\begin{aligned}
				\resizebox{100pt}{!}{%
					\begin{tikzpicture}[tqft/cobordism/.style={draw,thick},
						tqft/view from=outgoing, tqft/boundary separation=30pt,
						tqft/cobordism height=40pt, tqft/circle x radius=8pt,
						tqft/circle y radius=4.5pt, tqft/every boundary component/.style={draw,rotate=90},tqft/every incoming
						boundary component/.style={draw,dotted,thick},tqft/every outgoing
						boundary component/.style={draw,thick}]
						\pic[tqft/pair of pants,
						rotate=90,name=a, anchor={(1,-2)}]; 
						\pic[tqft/cylinder to next,rotate=90,name=b, at=(a-outgoing boundary 2)];
						\pic[tqft/pair of pants,
						rotate=90,name=d,at=(a-outgoing boundary)]; 
						\pic[tqft/cylinder,rotate=90,name=e, at=(d-outgoing boundary)];
						\pic[tqft/cylinder,rotate=90,name=f, at=(d-outgoing boundary 2)];
						\pic[tqft/cylinder,rotate=90,name=g, at=(b-outgoing boundary)];
						\node at ([xshift=20pt]g-incoming boundary 1){$\leftrightarrow$};
						\pic[tqft/cap,
						rotate=90,at=(a-outgoing boundary),anchor={(0.5,2)}]; 
				\end{tikzpicture} }
			\end{aligned},
			&&	\begin{aligned}
				\resizebox{100pt}{!}{%
					\begin{tikzpicture}[tqft/cobordism/.style={draw,thick},
						tqft/view from=outgoing, tqft/boundary separation=30pt,
						tqft/cobordism height=40pt, tqft/circle x radius=8pt,
						tqft/circle y radius=4.5pt, tqft/every boundary component/.style={draw,rotate=90},tqft/every incoming
						boundary component/.style={draw,dotted,thick},tqft/every outgoing
						boundary component/.style={draw,thick}]
						\pic[tqft/pair of pants,
						rotate=90,name=a, anchor={(1,-2)}]; 
						\pic[tqft/cylinder to next,rotate=90,name=b, at=(a-outgoing boundary 2)];
						\pic[tqft/pair of pants,
						rotate=90,name=d,at=(a-outgoing boundary)]; 
						\pic[tqft/cylinder,rotate=90,name=e, at=(d-outgoing boundary)];
						\pic[tqft/cylinder,rotate=90,name=f, at=(d-outgoing boundary 2)];
						\node at ([xshift=20pt]f-incoming boundary 1){$\leftrightarrow$};
						\pic[tqft/cylinder,rotate=90,name=g, at=(b-outgoing boundary)];
						\pic[tqft/cap,
						rotate=90,at=(a-outgoing boundary),anchor={(0.5,2)}]; 
				\end{tikzpicture} }
			\end{aligned},
			&&	\begin{aligned}
				\resizebox{100pt}{!}{%
					\begin{tikzpicture}[tqft/cobordism/.style={draw,thick},
						tqft/view from=outgoing, tqft/boundary separation=30pt,
						tqft/cobordism height=40pt, tqft/circle x radius=8pt,
						tqft/circle y radius=4.5pt, tqft/every boundary component/.style={draw,rotate=90},tqft/every incoming
						boundary component/.style={draw,dotted,thick},tqft/every outgoing
						boundary component/.style={draw,thick}]
						\pic[tqft/pair of pants,
						rotate=90,name=a, anchor={(1,-2)}]; 
						\pic[tqft/cylinder to next,rotate=90,name=b, at=(a-outgoing boundary 2)];
						\pic[tqft/pair of pants,
						rotate=90,name=d,at=(a-outgoing boundary)]; 
						\pic[tqft/cylinder,rotate=90,name=e, at=(d-outgoing boundary)];
						\pic[tqft/cylinder,rotate=90,name=f, at=(d-outgoing boundary 2)];
						\node at ([xshift=20pt]e-incoming boundary 1){$\leftrightarrow$};
						\pic[tqft/cylinder,rotate=90,name=g, at=(b-outgoing boundary)];
						\pic[tqft/cap,
						rotate=90,at=(a-outgoing boundary),anchor={(0.5,2)}]; 
				\end{tikzpicture} }
			\end{aligned}.
		\end{align*}
		Again, this is a spanning set by Proposition \ref{prop:orientable with boundary} and Remark \ref{basis from basis for connected}. In this case, the matrix of inner products is 
		\begin{align*}
			\left[ \begin{array}{ccccccccccc}
				\alpha_0^3 & \alpha_0^2&\alpha_0^2&\alpha_0^2&\alpha_0^2& \alpha_0^2& \alpha_0^2& \alpha_0& \alpha_0&\alpha_0& \alpha_0 \\
				\alpha_0^2& \lambda\alpha_0^2 &  \alpha_0 & \alpha_0 & 0 & \alpha_0 & \alpha_0 & \lambda\alpha_0& 0 & 0 & \lambda\alpha_0 \\
				\alpha_0^2 &\alpha_0& \lambda\alpha_0^2& \alpha_0& \alpha_0& 0& \alpha_0& \lambda\alpha_0& \lambda\alpha_0& 0& 0 \\
				\alpha_0^2& \alpha_0& \alpha_0& \lambda\alpha_0^2& \alpha_0& \alpha_0& 0& \lambda\alpha_0& 0& \lambda\alpha_0& 0 \\
				\alpha_0^2& 0& \alpha_0& \alpha_0& \lambda\alpha_0^2& \alpha_0& \alpha_0& 0&\lambda \alpha_0& \lambda\alpha_0& 0 \\
				\alpha_0^2& \alpha_0& 0& \alpha_0&\alpha_0& \lambda\alpha_0^2& \alpha_0& 0& 0& \lambda\alpha_0&\lambda\alpha_0 \\
				\alpha_0^2& \alpha_0& \alpha_0& 0& \alpha_0& \alpha_0& \lambda\alpha_0^2& 0&\lambda\alpha_0& 0& \lambda\alpha_0 \\
				\alpha_0& \lambda\alpha_0& \lambda\alpha_0& \lambda\alpha_0& 0& 0& 0& \lambda^2\alpha_0& 0& 0& 0 \\
				\alpha_0&0&\lambda\alpha_0& 0& \lambda\alpha_0& 0& \lambda\alpha_0&0& \lambda^2\alpha_0& 0& 0 \\
				\alpha_0& 0& 0& \lambda\alpha_0& \lambda\alpha_0& \lambda\alpha_0& 0& 0& 0& \lambda^2\alpha_0& 0 \\
				\alpha_0& \lambda\alpha_0& 0& 0& 0& \lambda\alpha_0& \lambda\alpha_0& 0& 0& 0&\lambda^2 \alpha_0
			\end{array}\right],
		\end{align*}
		which has determinant 
		\begin{align*}
		\lambda^6\alpha_0^{11}(\lambda\alpha_0-2)^7(\lambda\alpha_0-4).
		\end{align*}
		Hence 
		\begin{align*}
			\dim(\Hom_{\SOCob}(0,3))=11.
		\end{align*}
	\end{example}
	
Let $P=\{p_1, \dots, p_k\}$ be a partition of $\{1, \dots, m\},$ and let $\overline{J_i}\in \mathcal R_{|p_i|}$ for all $1\leq i \leq k$. Denote by $c_{P, \overline{J_1}, \dots, \overline{J_k}}$	the cobordism $0\to m$ with $k$ connected components, where the $i$-th component is of the form $\xi^{m_i}_{\overline J_i}$, with out-boundary given by the circles in positions $j_{i,1}, \dots, j_{i,m_i},$ where $p_i=\{j_{i,1}, \dots, j_{i,m_i}\}$. For instance, in Example \ref{example S_m for m=3} the first and seventh cobordisms are $c_{\{\{1\},\{2\},\{3\}\}, \overline\emptyset, \overline\emptyset, \overline\emptyset}$ and $c_{\{\{1,3\},\{2\}\},\overline{\{1\}}, \overline\emptyset}$, respectively. 
Then a spanning set for $\Hom_{\SOCob}(0,m)$ is given by 
	\begin{align}\label{spanning set}
		\mathcal S_m:=	\left\{  c_{P, \overline{J_1}, \dots, \overline{J_k}}:  P=\{p_1, \dots, p_k\} \text{ is a partition of } \{1,\dots, m\} \text{ and } \overline{J_i}\in \mathcal R_{|p_i|} \text{ for } 1\leq i \leq k \right\},
	\end{align}
	see Remark \ref{basis from basis for connected}.
	
	\begin{lemma}\label{lemma: inner products matrix}
	The determinant of the matrix of inner products of morphisms in $\mathcal S_m$ is a non-zero polynomial in the variables $\lambda$ and $\alpha_0$, for all $m\geq 1$. 
\end{lemma}

\begin{proof}
	Fix $m\geq 1$. Let $A$ be the matrix of inner products of morphisms in $\mathcal S_m$, and let $p(\lambda, \alpha_0)$ be its determinant as a polynomial in $\lambda$ and $\alpha_0$.
	Note that to show $p(\lambda, \alpha_0)$ is a non-zero polynomial, it is enough to show that $p(\alpha_0, \alpha_0)\ne 0$. Hence we assume from now on that $\lambda=\alpha_0$. 
	
	We will show that every row of $A$ has the highest power of $\lambda$ only in its diagonal entry. Let $c:=c_{P,\overline{J_1},\dots, \overline{J_k}}$ in $\mathcal S_m$, and consider its corresponding row on $A$. Recall that, for any $d\in \mathcal S_m$, the inner product of $c$ and $d$ is given by evaluating the cobordism $c\sqcup (-d)$, see Definition \ref{inner product}. If $c\sqcup (-d)$ is unorientable, then $(c,d)=0$. Otherwise, 
	\begin{align*}
		(c,d)=\alpha_{l_1}\dots \alpha_{l_s} =\lambda^{l_1+\dots+l_s} \alpha_0^{s}=\alpha_0^{l_1+\dots +l_s+s},
	\end{align*}
	where $s$ is the number of connected components of $c\sqcup (-d),$ and $ l_i$ is the genus of its $i$-th component, for all $1\leq i \leq s$. That is, the power of $\alpha_0$ at the entry $(c,d)$ will be the sum of the genuses of the connected components of $c\sqcup (-d)$ and the number of connected components of $c\sqcup (-d)$. 	Hence the highest it can be is $$(|p_1|-1) + \dots+ (|p_k|-1)+ k =m-k+k=m.$$ 
	In fact, the number $s$ of connected components of $c\sqcup (-d)$ is at most the number of connected components $k$ of $c$, and the sum of the genuses is at most  $(|p_1|-1) + \dots+ (|p_k|-1)$ since that is the highest genus we can generate with $c$ (when every space between two connected circles is closed).
	
	We check now that this power is reached only in the diagonal entry $(c,c)$. Note first that the sum of genuses in $c\sqcup (-d)$ can be $(|p_1|-1) + \dots+ (|p_k|-1)$ only when every pair of connected circles in $c$ gets closed, and so $d$ must be of the form $d:c_{Q,\overline{L_1},\dots, \overline{L_k}}$ for some $ \overline{L_i} \in \mathcal R_{|p_i|}$. But by  relation \eqref{crosscap and involution} we have $(\xi^{|p_i|}_{\overline{J_i}} , \xi^{|p_i|}_{\overline{L_i}} )=0$ if $\overline{J_i}\ne \overline{L_i},$ in which case $(c,d)=0$. On the other hand,  involutions cancel each other out in $c\sqcup (-c)$ and so $(\xi^{|p_i|}_{\overline{J_i}} , \xi^{|p_i|}_{\overline{J_i}})=\lambda^{|p_i|-1}\alpha_0=\alpha_0^{|p_i|}$, which implies 
	$$(c,c)=\Pi_{i=1}^k \alpha_0^{|p_i|} = \alpha_0^m,$$
	as desired.
	
	We showed that in every row of the matrix $A$, the highest power of $\alpha_0$ in that row happens only at the diagonal entry. Hence the term of  $p(\alpha_0, \alpha_0)$ computed using the diagonal entries of $A$ will have degree (as a polynomial in $\alpha_0$) strictly greater than any other term, and thus does not cancel out. This shows the determinant of $A$ is not the zero polynomial in $\lambda$ and $\alpha_0$.  
\end{proof}


	\begin{corollary}\label{dimension of unoriented orientable} Fix $\alpha_0\in \mathbf k^{\times}$. For all but countably many $\lambda$ and $\alpha_0$ in $\mathbf k^{\times}$, the dimension $\dim_m$ of $\Hom_{\operatorname{SOCob}_{\alpha}}(0,m)$ is given by
		\begin{align*}
			\dim_m= \sum\limits_{l=1}^m 2^{m-l}  \begin{Bmatrix} 
				m \\l
			\end{Bmatrix}, \text{ for all } m\geq 1, 
		\end{align*}
		where $\begin{Bmatrix} 
			m \\l
		\end{Bmatrix}$ denotes the Stirling number of the second kind, which counts the number of partitions of a set with $m$ elements into $l$ non-empty subsets.
	\end{corollary}

	\begin{proof}
		As per Remark \ref{basis from basis for connected}, the set of cobordisms $\mathcal S_m$,  obtained by assigning to every partition $P$ of $\{1, \dots, m\}$ and every $p\in P$ a connected cobordism $\xi_{\overline J}^{|p|}$ (see \eqref{spanning set}),  gives a  spanning set for $\Hom_{\SOCob}(0,m)$. Let $P$ be a partition of $\{1,\dots, m\}$ with parts $\{p_1, \dots, p_l\},$ for $1\leq l \leq m$. The number of ways of assigning  connected cobordisms to $P$ is then
		\begin{align*}
			|\mathcal R_{|p_1|}|\dots  |\mathcal R_{|p_l|}|,
		\end{align*}
		see Section \ref{section: xi}.	Thus the number of cobordisms in the spanning set is
		\begin{align*}
			\sum\limits_{l=1}^m \left( \sum\limits_{\{p_1, \dots, p_l \}\in \text{Part(m)}} |\mathcal R_{|p_1|}|\dots  |\mathcal R_{|p_l|}| \right) 
			&	= \sum\limits_{l=1}^m \left( \sum\limits_{\{p_1, \dots, p_l \}\in \text{Part(m)}}  2^{|p_1|-1}\dots  2^{|p_l|-1} \right)\\
			&=\sum\limits_{l=1}^m \left(  \sum\limits_{\{p_1, \dots, p_l \}\in \text{Part(m)}}  2^{m-l}  \right)\\
			& =\sum\limits_{l=1}^m 2^{m-l}  \begin{Bmatrix} 
				m \\l
			\end{Bmatrix}.
		\end{align*}
		Lastly, note that by Lemma \ref{lemma: inner products matrix} the determinant of the matrix of inner products of cobordisms in $\mathcal S_m$ is a non-zero polynomial in $\lambda$ and $\alpha_0$, for all $m\geq 1$. Thus, for values of $\lambda$ and $\alpha_0$ such that these polynomials are not evaluated to zero, cobordisms in $\mathcal S_m$ are linearly independent, and the result follows. 
	\end{proof}

	\begin{remark}
		The generating function of the sequence given by $\sum\limits_{l=1}^m 2^{m-l}\begin{Bmatrix} 
			m \\l
		\end{Bmatrix}$ in the Lemma above is $\exp((\exp(2x)-1)/2)$. The first 6 terms are given by $1,1,3,11,49,257$. See sequence A004211 at https://oeis.org/A004211 for more information.
	\end{remark}

	\begin{remark}\label{remark: power of lambda}
		By  Corollary \ref{dimension of unoriented orientable} and the  proof of Lemma \ref{lemma: inner products matrix}, the highest power of $\lambda$ in  the determinant of the matrix of inner products in $\mathcal S_m$ is given by
		\begin{align*}
			\sum\limits_{l=1}^{m-1}  (m-l)\cdot 2^{m-l}\begin{Bmatrix} 
				m \\l
			\end{Bmatrix}.
			\end{align*}
		The first five terms in the sequence for $m\geq 1$ are 0, 2, 14, 92, 644.
		\end{remark}
	
	\begin{remark}
		It follows from Remark \ref{remark: power of lambda} and Corollary 	\ref{dimension of unoriented orientable}, that the highest power of $\alpha_0$ in the determinant of the matrix of inner products of $\mathcal S_m$ is given by 
		\begin{align*}
			\sum\limits_{l=1}^m l\cdot 2^{m-l}\begin{Bmatrix} 
				m \\l
			\end{Bmatrix}.
		\end{align*}
		The first five terms in the sequence for $m\geq 1$  are given by $1,4,19,106,641$.
	\end{remark}
	
	\begin{remark}
		We will see later on that the exceptional values of $\lambda$ and $\alpha_0$ in Corollary \ref{dimension of unoriented orientable} are those such that $\lambda\alpha_0$ is a non-negative even integer. We thus predict that the determinant of the matrix of inner products of $\mathcal S_m$, see \eqref{spanning set}, will factor into powers of the form $(\lambda \alpha_0-2k)$, for $k=0, \dots, m-1$. Note that we know this to be the case for $m=1,2,3$, see Examples \ref{example S_m for m=2} and \ref{example S_m for m=3}.
	\end{remark}
	
	\subsection{Extended Frobenius algebra in $\Rep(\mathbb Z_2)$}
	
	For this subsection, let $\mathcal C=\Rep(\mathbb Z_2)$. For $\lambda\in \mathbf k^{\times}$, let $A=A_{\lambda}$ be the extended Frobenius algebra given in Example \ref{algebra of functions} for $n=1$, and consider the induced extended Frobenius algebra $\langle A\rangle_t=\langle A_{\lambda}\rangle_t$ in $\Rep(S_t\wr\mathbb Z_2)$, as described in Proposition \ref{extended frob in C to frob in StC}. We remark that both of these algebras depend on $\lambda$, but we drop the subscript $\lambda$ to simplify the notation.
	
	\begin{lemma}\label{ext eval of A}
		The evaluation $(\alpha, \beta, \gamma)$ of $\langle A \rangle_t$ in $\Rep(S_t\wr \mathbb Z_2)$ is given by 
		\begin{align*}
			&&\alpha=(2\lambda^{-1}t, 2t, 2\lambda t, \dots) &&\text{and} &&\beta=(0, 0, \dots )= \gamma.
		\end{align*}
	\end{lemma}	
	\begin{proof}
		We will use the graphical description of maps in $\Rep(S_t\wr \mathbb Z_2)$ as shown in Section \ref{section: StC}. The structure maps of $\langle A \rangle_t$ are
		
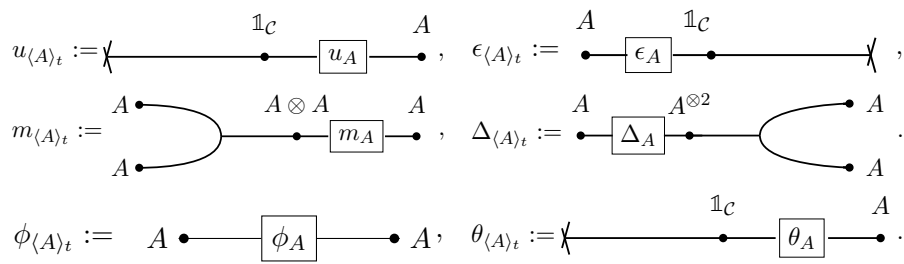
\begin{figure}[H]
			\begin{center}	\resizebox{170pt}{!}{%
					\begin{tikzpicture}[block/.style={draw, rectangle, minimum height=0.5cm}]
						\node (box) at (2.3,0)  [block] {$u_A$};
						\node at (-2.1,0)  {$u_{\langle A \rangle_t} :=$};
						\node   at  (-1.3,0)    {$\figureXv$};
						\node   at  (3.5,0.5)    {$A$};
						\node    at  (1.1,0.5)    {$ \mathbb 1_{\mathcal C}$};
						\draw[fill=black,style=thick] (1.1,0) circle (0.05cm and 0.05cm);
						\draw[style=thick]  (-1.3,0) to (1.1,0);
						\draw[style=thick]  (1.1,0) to (1.9,0);
						\draw[style=thick]  (2.7,0) to (3.5,0);
						\draw[fill=black,style=thick] (3.5,0) circle (0.05cm and 0.05cm);
						\node at (3.8,0)  {,};
				\end{tikzpicture}}
				\resizebox{170pt}{!}{%
					\begin{tikzpicture}[block/.style={draw, rectangle, minimum height=0.5cm, fill=white}]
						\node at (-3.6,0)  {$\varepsilon_{\langle A \rangle_t} :=$};
						\node (box) at (-1.7,0)  [block, fill=white] {$\varepsilon_A$};
						\node   at  (1.5,0)    {$\figureXv$};
						\node   at  (-2.6,0.5)    {$A$};
						\node   at  (-1,0.5)    {$ \mathbb 1_{\mathcal C}$};
						\draw[fill=black,style=thick] (-2.6,0) circle (0.05cm and 0.05cm);
						\draw[fill=black,style=thick] (-0.8,0) circle (0.05cm and 0.05cm);
						\draw[style=thick]  (-0.8,0) to (1.5,0);	
						\draw[style=thick]  (-0.8,0) to (-1.35,0);
						\draw[style=thick]  (-2.6,0) to (-2,0);
						\node at (1.9,0)  {,};
				\end{tikzpicture}}
				\resizebox{170pt}{!}{%
					\begin{tikzpicture}[block/.style={draw, rectangle, minimum height=0.5cm}]
						\node (box) at (2.05,0.5)  [block] {$m_A$};
						\node at (-2.7,0.5)  {$m_{\langle A \rangle_t}:=$};
						\node   at  (-1.7,0)    {$A$};
						\node   at  (3,1)    {$A$};
						\node   at  (-1.7,1)    {$A$};
						\node   at  (1.1,1)    {$A\otimes A$};
						\draw[fill=black,style=thick] (-1.4,0) circle (0.05cm and 0.05cm);
						\draw[fill=black,style=thick] (-1.4,1) circle (0.05cm and 0.05cm);
						\draw[fill=black,style=thick] (1.1,0.5) circle (0.05cm and 0.05cm);
						\draw[style=thick]  (-0.1,0.5) to   (1.1,0.5);
						\draw[style=thick]  (1.1,0.5) to   (1.6,0.5);
						\draw[style=thick]  (2.5,0.5) to   (3,0.5);
						\draw[style=thick]  (-1.4,1) to [out=0, looseness= 0.8,in=90] (-0.1,0.5);
						\draw[fill=black,style=thick] (3,0.5) circle (0.05cm and 0.05cm);
						\draw[style=thick]  (-1.4,0) to [out=0, looseness =0.8,in=-90](-0.1,0.5);
						\node at (3.4,0.5)  {,};
				\end{tikzpicture}}
				\resizebox{170pt}{!}{%
					\begin{tikzpicture}[block/.style={draw, rectangle, minimum height=0.5cm}]
						\node at (-4,0.5)  {$\Delta_{\langle A \rangle_t} :=$};
						\node   at  (1.6,0)    {$A$};
						\node   at  (1.6,1)    {$A$};
						\node   at  (-1.3,1)    {$A^{\otimes 2}$};
						\node   at  (-3,1)    {$A$};
						\node (box) at (-2.1,0.5)  [block] {$\Delta_A$};
						\draw[fill=black,style=thick] (-1.3,0.5) circle (0.05cm and 0.05cm);
						\draw[fill=black,style=thick] (1.2,1) circle (0.05cm and 0.05cm);
						\draw[fill=black,style=thick] (-3,0.5) circle (0.05cm and 0.05cm);
						\draw[fill=black,style=thick] (1.2,0) circle (0.05cm and 0.05cm);
						\draw[style=thick]  (-3,0.5) to   (-2.5,0.5);
						\draw[style=thick]  (-1.3,0.5) to   (-0.2,0.5);
						\draw[style=thick]  (-1.7,0.5) to  (-1.1,0.5);
						\draw[style=thick]  (-0.2,0.5) to [out=90, looseness= 1,in=0] (1.1,1);
						\draw[style=thick]  (-0.2,0.5) to [out=-90, looseness =1,in=0](1.1,0);
						\node at (2,0.5)  {.};
				\end{tikzpicture}}
				\resizebox{170pt}{!}{%
					\begin{tikzpicture}[block/.style={draw, rectangle, minimum height=0.5cm}]
						\node (box) [block] {$\phi_A$}
						node (in1)     [left=-.1cm and 1.1cm of box]     {$A$}
						node (out1)    [right=-.1cm and 1.1cm of box]    {$A$}
						node at (-3,0) {$\phi_{\langle A \rangle_t}:=$}
						{(in1) edge[-] (box.west |- in1)
							(box.east |- out1) edge[-] (out1)};
						\draw[fill=black,style=thick] (-1.4,0) circle (0.05cm and 0.05cm);
						\draw[fill=black,style=thick] (1.4,0) circle (0.05cm and 0.05cm);
						\node at (2,0)  {,};
				\end{tikzpicture}}
				\resizebox{170pt}{!}{%
					\begin{tikzpicture}[block/.style={draw, rectangle, minimum height=0.5cm}]
						\node (box) at (2.3,0)  [block] {$\theta_A$};
						\node at (-2.1,0)  {$\theta_{\langle A \rangle_t} :=$};
						\node   at  (-1.3,0)    {$\figureXv$};
						\node   at  (3.5,0.5)    {$A$};
						\node    at  (1.1,0.5)    {$ \mathbb 1_{\mathcal C}$};
						\draw[fill=black,style=thick] (1.1,0) circle (0.05cm and 0.05cm);
						\draw[style=thick]  (-1.3,0) to (1.1,0);
						\draw[style=thick]  (1.1,0) to (1.9,0);
						\draw[style=thick]  (2.7,0) to (3.5,0);
						\draw[fill=black,style=thick] (3.5,0) circle (0.05cm and 0.05cm);
						\node at (3.8,0)  {.};
				\end{tikzpicture}}
			\end{center}
		\end{figure}
		To obtain the sequences $\alpha, \beta$ and $\gamma$ we need to compute 
		\begin{align*}
			&&\alpha_n=\varepsilon_{\langle A\rangle_t}x^n u_{\langle A\rangle_t}, &&\beta_n=\varepsilon_{\langle A \rangle_t} x^n y u_{\langle A\rangle_t} &&\text{and} &&&&\gamma_n=\varepsilon_{\langle A \rangle_t} x^n y^2 u_{\langle A\rangle_t},
		\end{align*}
	for all $n\geq 0$,	where 
		\begin{align*}
			x:= m_{\langle A \rangle_t} \Delta_{\langle A \rangle_t} =
			\begin{aligned}\resizebox{210pt}{!}{%
					\begin{tikzpicture}[block/.style={draw, rectangle, minimum height=0.5cm}]
						\node   at  (1.6,-0.5)    {$A$};
						\node   at  (1.6,1.5)    {$A$};
						\node   at  (-1.3,1)    {$A^{\otimes 2}$};
						\node   at  (-3,1)    {$A$};
						\node (box) at (-2.1,0.5)  [block] {$\Delta_A$};
						\draw[fill=black,style=thick] (-1.3,0.5) circle (0.05cm and 0.05cm);
						\draw[fill=black,style=thick] (1.2,1) circle (0.05cm and 0.05cm);
						\draw[fill=black,style=thick] (-3,0.5) circle (0.05cm and 0.05cm);
						\draw[fill=black,style=thick] (1.2,0) circle (0.05cm and 0.05cm);
						\draw[style=thick]  (-3,0.5) to   (-2.5,0.5);
						\draw[style=thick]  (-1.3,0.5) to   (-0.2,0.5);
						\draw[style=thick]  (-1.7,0.5) to  (-1.1,0.5);
						\draw[style=thick]  (-0.2,0.5) to [out=90, looseness= 1,in=0] (1.1,1);
						\draw[style=thick]  (-0.2,0.5) to [out=-90, looseness =1,in=0](1.1,0);
						\node at (2,0.5)  {.};
						\draw[style=thick]  (1.1,1) to [out=0, looseness= 0.8,in=90] (2.4,0.5);
						\draw[fill=black,style=thick] (3.5,0.5) circle (0.05cm and 0.05cm);
						\node   at  (3.5,1)    {$A^{\otimes 2}$};
						\draw[style=thick]  (1.1,0) to [out=0, looseness =0.8,in=-90](2.4,0.5);
						\draw[style=thick]  (2.4,0.5) to   (3.6,0.5);
						\draw[style=thick]  (3.6,0.5) to   (4.1,0.5);
						\draw[style=thick]  (5,0.5) to   (5.6,0.5);
						\node (box) at (4.55,0.5)  [block] {$m_A$};
						\draw[fill=black,style=thick] (5.6,0.5) circle (0.05cm and 0.05cm);
						\node   at  (5.6,1)    {$A$};
				\end{tikzpicture}}
			\end{aligned},
		\end{align*}
		and 
		\begin{align*}
			y=m_{\langle A \rangle_t}(\theta_{\langle A \rangle_t}\otimes \text{id})=0,
		\end{align*}
		since $\theta_A=0$. It follows trivially that $\beta=(0, \dots)=\gamma$. To compute $\alpha_n$, we show first that $x=\lambda \text{Id}_A$ by graphical calculus:
		\begin{align}\label{computation of x in StC}
			\begin{aligned}	&\resizebox{220pt}{!}{%
					\begin{tikzpicture}[block/.style={draw, rectangle, minimum height=0.5cm}]
						\node   at  (1.6,-0.3)    {$A$};
						\node   at  (-3.8,0.5)    {\Large{$x=$}};
						\node   at  (1.6,1.3)    {$A$};
						\node   at  (-1.3,1)    {$A^{\otimes 2}$};
						\node   at  (-3,1)    {$A$};
						\node (box) at (-2.1,0.5)  [block] {$\Delta_A$};
						\draw[fill=black,style=thick] (-1.3,0.5) circle (0.05cm and 0.05cm);
						\draw[fill=black,style=thick] (1.2,1) circle (0.05cm and 0.05cm);
						\draw[fill=black,style=thick] (-3,0.5) circle (0.05cm and 0.05cm);
						\draw[fill=black,style=thick] (1.2,0) circle (0.05cm and 0.05cm);
						\draw[style=thick]  (-3,0.5) to   (-2.5,0.5);
						\draw[style=thick]  (-1.3,0.5) to   (-0.2,0.5);
						\draw[style=thick]  (-1.7,0.5) to  (-1.1,0.5);
						\draw[style=thick]  (-0.2,0.5) to [out=90, looseness= 1,in=0] (1.1,1);
						\draw[style=thick]  (-0.2,0.5) to [out=-90, looseness =1,in=0](1.1,0);
						\node at (2,0.5)  {.};
						\draw[style=thick]  (1.1,1) to [out=0, looseness= 0.8,in=90] (2.4,0.5);
						\draw[fill=black,style=thick] (3.5,0.5) circle (0.05cm and 0.05cm);
						\node   at  (3.5,1)    {$A^{\otimes 2}$};
						\draw[style=thick]  (1.1,0) to [out=0, looseness =0.8,in=-90](2.4,0.5);
						\draw[style=thick]  (2.4,0.5) to   (3.6,0.5);
						\draw[style=thick]  (3.6,0.5) to   (4.1,0.5);
						\draw[style=thick]  (5,0.5) to   (5.6,0.5);
						\node (box) at (4.55,0.5)  [block] {$m_A$};
						\draw[fill=black,style=thick] (5.6,0.5) circle (0.05cm and 0.05cm);
						\node   at  (5.6,1)    {$A$};
				\end{tikzpicture}}\\
				&\ \ =
				\resizebox{120pt}{!}{%
					\begin{tikzpicture}[block/.style={draw, rectangle, minimum height=0.5cm}]
						\node   at  (-1.3,1)    {$A^{\otimes 2}$};
						\node   at  (-3,1)    {$A$};
						\node (box) at (-2.1,0.5)  [block] {$\Delta_A$};
						\draw[fill=black,style=thick] (-1.3,0.5) circle (0.05cm and 0.05cm);
						\draw[fill=black,style=thick] (-3,0.5) circle (0.05cm and 0.05cm);
						\draw[style=thick]  (-3,0.5) to   (-2.5,0.5);
						\draw[style=thick]  (-1.3,0.5) to   (-0.2,0.5);
						\draw[style=thick]  (-1.7,0.5) to  (-1.1,0.5);
						\node   at  (-0.2,1)    {$A^{\otimes 2}$};
						\draw[style=thick]  (-0.2,0.5) to   (0.3,0.5);
						\draw[style=thick]  (1.2,0.5) to   (1.8,0.5);
						\node (box) at (0.75,0.5)  [block] {$m_A$};
						\draw[fill=black,style=thick] (-0.2,0.5) circle (0.05cm and 0.05cm);
						\draw[fill=black,style=thick] (1.8,0.5) circle (0.05cm and 0.05cm);
						\node   at  (1.8,1)    {$A$};
				\end{tikzpicture}}
				\\
				&\ \ =
				\resizebox{80pt}{!}{%
					\begin{tikzpicture}[block/.style={draw, rectangle, minimum height=0.5cm,fill=white}]
						\node   at  (-3,1)    {$A$};
						\node (box) at (-1.5,0.5)  [block] {$m_A\Delta_A$};
						\draw[fill=black,style=thick] (-3,0.5) circle (0.05cm and 0.05cm);
						\draw[style=thick]  (-3,0.5) to   (-2.2,0.5);
						\draw[style=thick]  (-0.8,0.5) to   (-0.2,0.5);
						\node   at  (-0.2,1)    {$A$};
						\draw[fill=black,style=thick] (-0.2,0.5) circle (0.05cm and 0.05cm);
				\end{tikzpicture}}\\
				&\ \ =\lambda \text{Id}_A,
			\end{aligned}
		\end{align}
		since $m_A\Delta_A=\lambda \text{Id}_A$ in $\Rep(\mathbb Z_2)$. Then $x^n=\lambda^n  \text{Id}_A$, and so
		\begin{align*}
			\varepsilon_{\langle A \rangle_t} x^n u_{\langle A \rangle_t} &=	\begin{aligned}
				\resizebox{230pt}{!}{%
					\begin{tikzpicture}[block/.style={draw, rectangle, minimum height=0.5cm}]
						\node (box) at (1,0)  [block] {$u_A$};
						\node   at  (-1.3,0)    {$\figureXv$};
						\node   at  (2,0.5)    {$A$};
						\node   at  (4,0.5)    {$A$};
						\node   at  (6,0.5)    {$A$};
						\node    at  (0,0.5)    {$ \mathbb 1_{\mathcal C}$};
						\draw[fill=black,style=thick] (0,0) circle (0.05cm and 0.05cm);
						\draw[style=thick]  (-1.3,0) to (0.65,0);
						\draw[style=thick]  (1.35,0) to (7.5,0);
						\draw[fill=black,style=thick] (2,0) circle (0.05cm and 0.05cm);
						\draw[fill=black,style=thick] (4,0) circle (0.05cm and 0.05cm);
						\draw[fill=black,style=thick] (6,0) circle (0.05cm and 0.05cm);
						\node (box) at (5,0)  [block, fill=white] {$\varepsilon_A$};
						\node (box) at (3,0)  [block, fill=white] {$\lambda^n\Id_A$};
						\node   at  (7.5,0)    {$\figureXv$};
				\end{tikzpicture}}
			\end{aligned}\\
			&=\lambda^n \begin{aligned}
				\resizebox{120pt}{!}{%
					\begin{tikzpicture}[block/.style={draw, rectangle, minimum height=0.5cm}]
						\node   at  (-1.3,0)    {$\figureXv$};
						\node   at  (2,0.5)    {$ \mathbb 1_{\mathcal C}$};
						\node    at  (0,0.5)    {$ \mathbb 1_{\mathcal C}$};
						\draw[fill=black,style=thick] (0,0) circle (0.05cm and 0.05cm);
						\draw[style=thick]  (-1.3,0) to (0.65,0);
						\draw[style=thick]  (1.35,0) to (3,0);
						\draw[fill=black,style=thick] (2,0) circle (0.05cm and 0.05cm);
						\node (box) at (1,0)  [block, fill=white] {$\varepsilon_A u_A$};
						\node   at  (3,0)    {$\figureXv$};
				\end{tikzpicture}}
			\end{aligned}\\
			&=2\lambda^{n-1}\begin{aligned}
				\resizebox{80pt}{!}{%
					\begin{tikzpicture}[block/.style={draw, rectangle, minimum height=0.5cm}]
						\node   at  (-1.3,0)    {$\figureXv$};
						\node    at  (0,0.5)    {$ \mathbb 1_{\mathcal C}$};
						\draw[fill=black,style=thick] (0,0) circle (0.05cm and 0.05cm);
						\draw[style=thick]  (-1.3,0) to (1.5,0);
						\node   at  (1.5,0)    {$\figureXv$};
				\end{tikzpicture}}
			\end{aligned}\\
			&=2\lambda^{n-1} t \text{Id}_{\mathbb 1}.
		\end{align*}
		Hence
		\begin{align*}
			\alpha_n=2\lambda^{n-1}t \ \ \text{ for all }n\geq 0,
		\end{align*}
		as desired.
	\end{proof}

	\begin{remark}\label{connected in St} Consider the graphical description of $\Rep(S_t \wr \mathbb Z_2)$ as given in Section \ref{section: StC}. 
		We will say that a map in $\Rep(S_{t}\wr \mathbb Z_2)$ is \emph{connected} if its graphical representation is a connected diagram. Thus maps given by stackings of connected diagrams generate $\Rep(S_t\wr \mathbb Z_2)$ as a pseudo-abelian category. 
	\end{remark}

	Let $u_{\mathcal C}, \varepsilon_{\mathcal C}, \mu_{\mathcal C}$ and $\Delta_{\mathcal C}$ in $\StC=\Rep(S_t\wr \mathbb Z_2)$ as defined in Section \ref{section: StC}.
	
	\begin{lemma}\label{standard form}
		Connected maps $\mathbb 1_{\Rep(S_{t}\wr \mathbb Z_2)}\to \langle A^{\otimes n} \rangle_t$ can be written as a composition 
		$$\Delta_{\mathcal C}^{n-1} \circ  \langle f \rangle_{t}\circ u_{\mathcal C},$$  where $f\in \Hom_{\mathcal C}(\mathbb 1_{\mathcal C}, A^{\otimes n})$ and $\Delta_{\mathcal C}^{n-1}$ is the appropriate composite of $\Delta_{\mathcal C}$'s.
	\end{lemma}
	\begin{proof}
		This follows from \cite[Proposition 4.24]{M}.
	\end{proof}
	
	\begin{example}
		Connected maps $\mathbb 1_{\Rep(S_{t}\wr \mathbb Z_2)}\to \langle A^{\otimes 3}\rangle_t$ in $\Rep(S_t\wr \mathbb Z_2)$ are given by
		\begin{align*}
			\Delta_{\mathcal C}^3 \circ  \langle f \rangle_{t}\circ u_{\mathcal C} =
			\begin{aligned}
				\resizebox{180pt}{!}{%
					\begin{tikzpicture}[block/.style={draw, rectangle, minimum height=0.5cm,fill=white}]
						\node   at  (-0.7,1.5)    {$A$};
						\node   at  (-0.7,0.5)    {$A^{\otimes 2}$};
						\node   at  (1.8,1.5)    {$A$};
						\node   at  (1.8,-0.1)    {$A$};
						\node   at  (-3,1)    {$A^{\otimes 3}$};
						\node at (-6,0.5) {$\figureXv$};
						\node at (-5,1) {$\mathbb 1_{\mathcal C}$};
						\draw[fill=black,style=thick] (-5,0.5) circle (0.05cm and 0.05cm);
						\draw[fill=black,style=thick] (1.9,0.5) circle (0.05cm and 0.05cm);
						\draw[fill=black,style=thick] (1.9,-0.5) circle (0.05cm and 0.05cm);
						\draw[fill=black,style=thick] (1.9,1) circle (0.05cm and 0.05cm);
						\draw[fill=black,style=thick] (-3,0.5) circle (0.05cm and 0.05cm);
						\draw[fill=black,style=thick] (-0.7,1) circle (0.05cm and 0.05cm);
						\draw[fill=black,style=thick] (-0.7,0) circle (0.05cm and 0.05cm);
						\draw[style=thick]  (-2,0.5) to [out=90, looseness= 1,in=0] (-0.7,1);
						\draw[style=thick]  (0.5,0) to [out=-90, looseness =1,in=0](1.9,-0.5);
						\draw[style=thick]  (0.5,0) to [out=90, looseness= 1,in=0] (1.9,0.5);
						\draw[style=thick]  (-2,0.5) to [out=-90, looseness =1,in=0](-0.7,0);
						\draw[style=thick]  (-2,0.5) to   (-6,0.5);
						\draw[style=thick]  (-0.7,0) to   (0.5,0);
						\draw[style=thick]  (-0.7,1) to   (1.9,1);
						\node (box) at (-4,0.5)  [block] {$f$};
				\end{tikzpicture}}
			\end{aligned}
			,	\end{align*}
		where $f : \mathbb 1 \to A^{\otimes 3}$ is a map in $\Rep(\mathbb Z_2)$.
	\end{example}
	
	\begin{corollary}\label{dim of connected in Rep(St wr Z2)}
		The subspace $\mathcal U_n$ spanned by connected maps  in $\Hom_{\Rep(S_t\wr \mathbb Z_2)}\left(\mathbb 1, \langle A^{\otimes n} \rangle_t \right)$ has dimension 
		\begin{align*}
			\dim(\mathcal U_n)=\dim(\Hom_{\mathcal C}(\mathbb{1}_{\mathcal C}, A^{\otimes n}))= 2^{n-1} \ \ \text{for all } n\geq 1.
		\end{align*}
	\end{corollary}
	\begin{proof}
		This follows from linearity of the functor $\Rep(\mathbb Z_2)\to \mathcal \Rep(S_t \wr \mathbb Z_2)$, see \cite[Proposition 4.26]{M},  and the previous lemma.
	\end{proof}
	
	\subsubsection{Proof of Theorem I}
	We give here a proof of our first main result, stated below. For a symmetric tensor category $\mathcal C$, we denote by $\underline{\mathcal C}$ its quotient by the tensor ideal of negligible morphisms, see Section \ref{SUCob}.

	\begin{maintheorem}\label{maintheorem1}
		Let $\alpha=(\alpha_0, \lambda \alpha_0, \lambda^2 \alpha_0, \dots)$ and $\beta=(0,0,  \dots)=\gamma$ be sequences in $\mathbf k$, for $\alpha_0, \lambda \in \mathbf k^{\times}$
		. We have an equivalence of symmetric tensor categories
		\begin{align*}
			\underline{\operatorname{OCob}_{\alpha}}\simeq  \underline{\Rep(\text{S}_t\wr \mathbb Z_2 )},
		\end{align*}
		where $t=\frac{\lambda \alpha_0}{2}$.
	\end{maintheorem}
	
	To prove this, we will use the \cite[Lemma 2.6]{BEEO}, specialized to $\OCob_{\alpha}$ see also \cite[Proposition 2.4]{KKO}. We write the statement below for clarity. 
	
	\begin{proposition}\label{proposition 1}
		Let $\mathcal C$ be a  semisimple Karoubian symmetric tensor category with finite dimensional $\Hom$ spaces. Suppose there is 
		a symmetric tensor functor $F:\operatorname{OCob}_{\alpha} \to \mathcal C$ that is surjective on $\Hom$'s. 
		 Then there is a unique fully faithful symmetric tensor functor 
		\begin{align*}
			F: \underline{\operatorname{OCob}_{\alpha}}\xrightarrow{\sim}  \mathcal C.
		\end{align*}
	\end{proposition}
	
	For the rest of this section, fix sequences 
	\begin{align*}
		&&\alpha=(\alpha_0, \lambda \alpha_0, \lambda^2 \alpha_0, \dots) &&\text{and} &&\beta=(0,0,  \dots)=\gamma,
	\end{align*}
	where $\alpha_0, \lambda \in \mathbf k^{\times},$ and set $t=\frac{\lambda\alpha_0}{2}$. We will show that $\underline{\Rep(S_t\wr \mathbb Z_2)}$ satisfies the conditions in the proposition above via a series of lemmas. 
	
	Recall we denote by $\langle A \rangle_t=\langle A_{\lambda} \rangle_t$  the extended Frobenius algebra in  $\Rep(S_t\wr \mathbb Z_2)$, where $A=A_{\lambda}$ in $\Rep(\mathbb Z_2)$ is defined in Example \ref{algebra of functions} for $n=1$, and structure maps are constructed as in Proposition \ref{extended in StC}. We recall that the structure maps of these algebras depend on $\lambda$, but we drop the subscript to simplify the notation. 
	
	\begin{lemma}
		There is a symmetric tensor functor $F_A:\operatorname{VUCob}_{\alpha, \beta, \gamma}\to \Rep(S_t\wr \mathbb Z_2)$, mapping the circle object of $\operatorname{VUCob}_{\alpha, \beta, \gamma}$ to the extended Frobenius algebra $\langle A \rangle_t$ in $\Rep(S_t\wr \mathbb Z_2)$.
	\end{lemma}
	
	\begin{proof}
		This follows from the universal property of $\VUCob$, see Section \ref{universal property}, and Lemma \ref{ext eval of A}.
	\end{proof}
	
	\begin{lemma}
		The functor $F_A:\operatorname{VUCob}_{\alpha, \beta, \gamma} \to \Rep(S_t\wr \mathbb Z_2)$ factors through $\operatorname{SOCob}_{\alpha}$.
	\end{lemma}
	\begin{proof}
		We need to check that $F_A$ annhilates the handle relation $x-\lambda \operatorname{Id}=0$ and the crosscap relation $\theta=0$. The latter is trivial since $F_A(\theta)=\theta_{\langle A \rangle_t}=0$. On the other hand, $F_A(x-\lambda I)=m_{\langle A \rangle_t} \Delta_{\langle A\rangle_t}- \lambda \text{Id}_{\langle A \rangle_t}=\lambda \text{Id}_{\langle A \rangle_t}=0$, see equation \eqref{computation of x in StC}.
	\end{proof}

	\begin{lemma}\label{lemma 1}
		The functor $F_A:\operatorname{SOCob}_{\alpha} \to \Rep(S_t\wr \mathbb Z_2)$ satisfies that any indecomposable object of $\Rep(S_t\wr \mathbb Z_2)$ is a direct summand of $F(n)$ for some $n$ in $\operatorname{SOCob}_{\alpha}$.
	\end{lemma}
	
	\begin{proof}
		By definition, $F_A(1)=	\langle A \rangle_t$, and so
		\begin{align*}
			F_A(n)=\langle A \rangle_t^{\otimes n}.
		\end{align*} The result follows since objects of the form $\langle A\rangle_t^{\otimes n}$ generate $\Rep(S_t\wr \mathbb Z_2)$ as a pseudo-abelian category, see \cite[Remark 4.25]{M}.
	\end{proof}
	
	Fix $n\geq 1$. Let $\mathcal W_n$ denote the subspace of $\Hom_{\SOCob}(0,n)$ spanned by connected cobordisms, 
	and let $\mathcal U_n$ be the subspace of $\Hom_{\Rep(S_t\wr \mathbb Z_2)}(\mathbb 1,\langle A\rangle_t^{\otimes n})$  spanned by connected diagrams, see Lemma \ref{standard form}.
	
	\begin{lemma} \label{surjective on connected}
		The functor $F_A:\operatorname{SOCob}_{\alpha}\to \Rep(S_t\wr \mathbb Z_2)$ induces an isomorphism
		\begin{align*}
			\mathcal W_n\xrightarrow{F_A} \mathcal U_n.
		\end{align*}
	\end{lemma}
	
	\begin{proof}
		Consider the basis $\{\xi_{\overline J}^n\}_{\overline J \in \mathcal R_n}$ of $\mathcal W_n$  as in Definition \ref{xi definition}, see also Proposition \ref{basis for connected unoriented orientable}. That is, 
		\begin{align*}
			\xi_{\overline J}^n= \phi_{J}\Delta^{n-1}u, 
		\end{align*}
		where $J$ is the representative of $\overline J$ with $|J|\leq n/2$,
		\begin{align*}
			\phi_{J}:=c_{1,J}\otimes \dots \otimes c_{n,J}, \ \ \text{with } c_{j,J}= \begin{cases} \text{id} &\text{if } j\in J,\\
				\phi &\text{if }  j\not\in J,
			\end{cases}
		\end{align*}
	 $\Delta^{n-1}$ denotes 
		\begin{align*}
			\Delta^{n-1}=(\text{id}^{\otimes (n-2)}\otimes \Delta)\dots (\text{id}\otimes \Delta)\Delta,
		\end{align*}
	and $u$ is the cap cobordism of $\SOCob$.
	
	We compute the image of $\xi_{\overline J}^n$ under $F_A$. By definition,  $F_A(\Delta)= \Delta_{\mathcal C}\Delta_A$, and so
		\begin{align*}
			F_A(\Delta^{n-1}) = \left(\text{id}^{\otimes(n-2)}\otimes \Delta_{\mathcal C}\langle \Delta_A \rangle_t\right)\dots \left(\text{id} \otimes \Delta_{\mathcal C}\langle \Delta_A\rangle_t\right)\Delta_{\mathcal C}\langle \Delta_A\rangle_t.
		\end{align*}
		Let 
		\begin{align*}
			\Delta_{\mathcal C}^k=\left(\text{id}^{\otimes(k-2)}\otimes \Delta_{\mathcal C}\right)\dots \left(\text{id} \otimes \Delta_{\mathcal C}\right)\Delta_{\mathcal C},
		\end{align*}
		for all $k\geq 1$. Using the relation $(\langle f\rangle_t\otimes \langle g\rangle_t)\Delta_{\mathcal C}= \Delta_{\mathcal C}\langle f \otimes g \rangle_t$ in $\Rep(S_t\wr \mathbb Z_2)$, see equation \eqref{relation for mu}, we have that
		\begin{align*}
			\left(\text{id}\otimes \Delta_{\mathcal C}\langle \Delta_A\rangle_t\right)\Delta_{\mathcal C}\langle \Delta_A\rangle_t= (\text{id} \otimes \Delta_{\mathcal C})\Delta_{\mathcal C} \langle (\text{id} \otimes \Delta_A)\Delta_A\rangle_t= \Delta_{\mathcal C}^2\langle \Delta_A^2\rangle_t.
		\end{align*}
		In general,
		\begin{align*}
			\left(\text{id}^{\otimes(n-2)}\otimes \Delta_{\mathcal C}\langle \Delta_A\rangle_t\right)\dots \left(\mathbb 1 \otimes \Delta_{\mathcal C}\langle \Delta_A\rangle_t\right)\Delta_{\mathcal C}\langle \Delta_A\rangle_t=\Delta_{\mathcal C}^{n-1}\langle \Delta_A^{n-1}\rangle_t,
		\end{align*}
		and thus $F_A(\Delta^{n-1})=	\Delta_{\mathcal C}^{n-1}\langle \Delta_A^{n-1}\rangle_t$. Thus if we define
		\begin{align*}
			\psi_{i,J}=\
			\begin{cases}
				\phi_A &\text{if } i \in J,\\
				\text{id}  &\text{if } i\not\in J,
			\end{cases}
		\end{align*} then
		\begin{align}  \label{image of xi under F}
			\begin{aligned}
			F(\xi_{\overline J}^n)&= \left( F(c_{1,J}) \otimes \dots \otimes F(c_{n,J}) \right) F(\Delta_{\mathcal C}^{n-1}) F(u_A)\\
			&= \left( \langle\psi_{1,J}\rangle_t\otimes \dots \otimes \langle \psi_{n,J} \rangle_t\right) \Delta_{\mathcal C}^{n-1}\langle \Delta_A^{n-1}\rangle_t \langle u_A\rangle_t u_{\mathcal C}\\
			&=\Delta_{\mathcal C}^{n-1} \langle (\psi_{1,J}\otimes \dots \otimes \psi_{n,J} )\Delta_A^{n-1} u_A\rangle_t u_{\mathcal C}.
			\end{aligned}
		\end{align}
	Hence under $F_A$, $\mathcal W_n$ is mapped to $\mathcal U_n$. We want  to show that the set $\{F(\xi_{\overline J}^n)\}_{{\overline J}\in \mathcal R_n}$ is a basis of $\mathcal U_n$. We prove first that it is linearly independent. By the equation above, it is enough to show that $\{f_{\overline J}^n\}_{\overline J \in \mathcal R_k}$ is linearly independent in $\Rep(\mathbb Z_2)$, where
		$$f_{\overline J}^n:=(\psi_{1,J}\otimes \dots \otimes \psi_{n,J} )\Delta_A^{n-1} u_A \in \Hom_{\Rep(\mathbb Z_2)}(\mathbb 1, A^{\otimes n}).$$
		For this, we check that the value of each $f_{\overline J}^n$ at $1\in \mathbb 1\cong \mathbf k$ is  
		\begin{align*}
			\delta_{\overline J} := \delta_{1,J}\otimes \dots \otimes \delta_{n,J} + \delta_{1,J^c}\otimes \dots \otimes \delta_{n,J^c}, \ \ \text{where} \ \delta_{i,J}:=\begin{cases}
				\delta_1 &\text{if } i\in J,\\
				\delta_{-1} &\text{if } i\not\in J,
			\end{cases}
		\end{align*}
	for all $1\leq i\leq n$,	and $\delta_1, \delta_{-1} \in A$ are such that $\delta_1(x_l)=\delta_{l,1}$ and $\delta_{-1}(x_l)=\delta_{l,-1}$ for $l=1,-1$, see Example  \ref{algebra of functions}. We compute
		\begin{align*}
			f_{\overline J}^n(1)&= \left(\psi_{1,J} \otimes \dots \otimes \psi_{n,J}\right) \Delta_A^{n-1}(\delta_{-1}+ \delta_1)\\
			&=\left(\psi_{1,J} \otimes \dots \otimes \psi_{n,J}\right)(\delta_{-1}^{\otimes n}+\delta_1^{\otimes n})\\
			&=\psi_{1,J}(\delta_{-1}) \otimes \dots \otimes \psi_{n,J}(\delta_{-1})+\psi_{1,J}(\delta_1) \otimes \dots \otimes \psi_{n,J}(\delta_1).
		\end{align*}
		Recall that $\phi_A(\delta_{-1})=\delta_1$ and $\phi_A(\delta_1)=\delta_{-1}$. Hence
		\begin{align*}
			\psi_{i,J}(\delta_{-1})=\begin{cases}
				\delta_1 &\text{if } i\in J\\
				\delta_{-1} &\text{if } i\not\in J
			\end{cases} = \delta_{i,J}  \ \ \text{and} \ \ \ 	\psi_{i,J}(\delta_1)=\begin{cases}
				\delta_{-1}&\text{if } i\in J\\
				\delta_1 &\text{if } i\not\in J
			\end{cases} = \delta_{i,J^c},
		\end{align*}
		so we get that 
		\begin{align*}
			f_{\overline J}^n(1) &=\delta_{1,J} \otimes \dots \otimes \delta_{n,J}+\delta_{1,J^c} \otimes \dots \otimes \delta_{n,J^c} = \delta_{\overline J}.
		\end{align*}
		So it is enough to note that $\{\delta_{\overline J}\}_{\overline J \in \mathcal R_n}$ is linearly independent in $A^{\otimes n}$. In fact, since $\{\delta_1, \delta_{-1}\}$ is a basis for $A$ as a vector space, then $\{\delta_{1,J}\otimes \dots \otimes \delta_{n,J}\}_{J\in \mathcal P(n)}$ is a basis for $A^{\otimes n}$, and so the set $$\{\delta_{1,J}\otimes \dots \otimes \delta_{n,J}+\delta_{1,J^c}\otimes \dots \otimes \delta_{n,J^c}\}_{J\in \mathcal P(n)}=\{\delta_{\overline J}\}_{\overline J \in \mathcal R_n},$$
		is linearly independent, as desired. 
		
		Lastly, by Proposition \ref{basis for connected unoriented orientable} and Lemma \ref{dim of connected in Rep(St wr Z2)}, we know that $\dim(\mathcal W_n)=2^{n-1}=\dim(\mathcal U_n)$, and so the statement follows. 
	\end{proof}
	
	Taking Karoubian envelope on the source category, $F_A:\SOCob \to \Rep(S_t\wr \mathbb Z_2)$  extends uniquely to a symmetric tensor functor
	\begin{align*}
		F_A:\PsOCob \to \Rep(S_t\wr \mathbb Z_2).
	\end{align*}
	Moreover, by Lemma \ref{lemma 1} $F_A$  is  essentially surjective.

	\begin{lemma}\label{lemma 2}
		The functor $F_A:\operatorname{OCob}_{\alpha} \to \Rep(S_t\wr \mathbb Z_2)$  is full.
	\end{lemma}
	
	\begin{proof}
	To show that $F_A$ is surjective on morphisms, it is enough to check that the maps 
		\begin{align*}
			\Hom_{\PsOCob}(n,m)  \xrightarrow{F_A} \Hom_{\Rep(\text{S}_t\wr \mathbb Z_2 )}( \langle A \rangle_t^{\otimes n},\langle A \rangle_t^{\otimes m})
		\end{align*}
		induced by $F_A$ are surjective for all $n,m\geq 1$. Note that since $A\in \Rep(\mathbb Z_2)$ is self-dual, then by 
		\cite[Appendix A]{M} so is $\langle A \rangle_t$ in $\Rep(S_t\wr \mathbb Z_2)$. Hence, by duality, it is enough to check surjectivity of the maps 
		\begin{align*}
			\Hom_{\PsOCob}(0,n)  \xrightarrow{F_A} \Hom_{\Rep(\text{S}_t\wr \mathbb Z_2 )}( \mathbb 1,\langle A \rangle_t^{\otimes n}),
		\end{align*}
		for all $n\geq 1$.  This follows from Lemma \ref{surjective on connected}, since connected diagrams generate $\Rep(S_t\wr \mathbb Z_2)$ as a pseudo-abelian tensor category.
	\end{proof}	

	We now show a proof for Theorem \ref{maintheorem1}.	
	
	\begin{proof}[Proof of Theorem \ref{maintheorem1}]
		Consider the symmetric tensor functor $$\underline{F_A}:=\PsOCob \xrightarrow{F_A}\Rep(\text{S}_t\wr \mathbb Z_2 )\to  \underline{\Rep(\text{S}_t\wr \mathbb Z_2 )},$$ where $F_A$ is as defined previously,  followed by the semisimplification functor. By Lemma \ref{lemma 2}, $\underline{F_A}$ satisfies the conditions of Proposition \ref{proposition 1}. Hence $\underline{F_A}$ induces a fully faithful symmetric tensor functor
		\begin{align*}
			\underline{F_A}: \underline{\PsOCob} \xrightarrow{} \underline{\Rep(\text{S}_t\wr \mathbb Z_2 )}.
		\end{align*}
	Moreover, since $F_A$ is essentially surjective then so is $\underline{F_A}$ and we have the desired equivalence.	\end{proof}	
	
	\begin{corollary}\label{Corollary I}
		If $\lambda\alpha_0$ is not a non-negative even integer, then 
		\begin{align*}
			\operatorname{OCob}_{\alpha}\cong \Rep(S_t\wr \mathbb Z_2).
		\end{align*}
	In particular, $\operatorname{OCob}_{\alpha}$ is semisimple, and $\dim_{\operatorname{OCob}_{\alpha}}(0,m)=\sum\limits_{l=1}^m 2^{m-l} \begin{Bmatrix} 
		m \\l
	\end{Bmatrix},$ for all $m\geq 1$.
	\end{corollary}
	
	\begin{proof}
		By Corollary \ref{dimension of unoriented orientable} we know that for all but countably many values of $\lambda$ and $\alpha_0$ in $\mathbf k^{\times}$, there are no negligible morphisms in $\PsOCob$. From this and Theorem \ref{maintheorem1}, it follows that  from all but countably many $\lambda$ and $\alpha_0$ we get an equivalence
		\begin{align*}
		\PsOCob=	\underline{\PsOCob}\cong \underline{\Rep(S_t\wr \mathbb Z_2)},
		\end{align*}  where $t=\frac{\lambda\alpha_0}{2}$.
	Hence, when $\lambda$ and $\alpha_0$ are not one of the exceptional values, we have that 
	\begin{align*}
		\dim \Hom_{\underline{\Rep(S_t\wr \mathbb Z_2)}}(\mathbb 1, \langle A \rangle_t^{\otimes m})= \dim \Hom_{\PsOCob}(0,m)=\sum\limits_{l=1}^m 2^{m-l}\begin{Bmatrix} 
			m \\l
		\end{Bmatrix},
	\end{align*}
	see Corollary \ref{dimension of unoriented orientable}. But by \cite[Proposition 5.5]{M}, for $t\not\in \mathbb Z_{\geq 0},$ the category $\Rep(S_t\wr \mathbb Z_2)$ is semisimple, and so $\dim_{\Rep(S_t\wr \mathbb Z_2)}(\mathbb 1, \langle A \rangle_t^{\otimes m})$ does not depend on $t$. Therefore, $\dim_{\Rep(S_t\wr \mathbb Z_2)}(\mathbb 1, \langle A \rangle_t^{\otimes m})= \sum\limits_{l=1}^m 2^{m-l}\begin{Bmatrix} 
		m \\l
	\end{Bmatrix}$ whenever $t=\frac{\lambda\alpha_0}{2}\not\in \mathbb Z_{\geq 0}$.
	
	Let $t=\frac{\lambda\alpha_0}{2}\not\in \mathbb Z_{\geq 0}$, and consider again the equivalence 
	\begin{align*}
		\underline{\PsOCob}=\underline{\Rep(S_t\wr \mathbb Z_2)}=\Rep(S_t\wr \mathbb Z_2).
	\end{align*}
Since 
\begin{align*}
\sum\limits_{l=1}^m 2^{m-l}\begin{Bmatrix} 
	m \\l
\end{Bmatrix}=	\dim \Hom_{\Rep(S_t\wr \mathbb Z_2)}(\mathbb 1, \langle A \rangle_t^{\otimes m})= \dim \Hom_{\PsOCob}(0,m)\leq \sum\limits_{l=1}^m 2^{m-l}\begin{Bmatrix} 
m \\l
\end{Bmatrix},
\end{align*}
then $\dim \Hom_{\PsOCob}(0,m)= \sum\limits_{l=1}^m 2^{m-l}\begin{Bmatrix} 
	m \\l
\end{Bmatrix}$. Hence, when $\lambda\alpha_0$ is not a non-negative even  integer, there are no negligible morphisms in $\PsOCob$  and thus by Theorem \ref{maintheorem1} we conclude
\begin{align*}
	\PsOCob=\underline{\PsOCob}\cong \underline{\Rep(S_t\wr \mathbb Z_2)}=\Rep(S_t\wr \mathbb Z_2),
\end{align*}
as desired. 
	\end{proof}

		\section{Equivalence with the category $\Rep(S_t \wr \mathbb Z_2) \boxtimes \Rep(S_{t_+}) \boxtimes \Rep(S_{t_-})$ }\label{section:theorem II}

	In this section we study the category $\SUCob$ for the sequences
	\begin{align}\label{sequences}
		&&\alpha=(\alpha_0, \lambda \alpha_0, \lambda^2 \alpha_0, \dots), &&\beta=(\beta_0, \lambda \beta_0, \lambda^2 \beta_0, \dots) &&\text{and} &&\gamma=(\gamma_0, \lambda \gamma_0, \lambda^2 \gamma_0, \dots),
	\end{align}
	where $\alpha_0, \beta_0, \gamma_0, \lambda\in \mathbf k^{\times}$. Here, the  generating functions for $\alpha, 
	\beta$ and $\gamma$ are $$Z_{\alpha}(T)=\frac{\alpha_0}{1-\lambda T}, \ \ Z_{\beta}(T)=\frac{\beta_0}{1-\lambda T}, \ \  Z_{\gamma}(T)=\frac{\gamma_0}{1-\lambda T},$$ respectively. 
	
	 Recall that for a symmetric tensor category $\mathcal C$, we denote by $\underline{\mathcal C}$ its quotient by the tensor ideal of negligible morpshims, see Section \ref{SUCob}.

	\begin{maintheorem}\label{maintheorem2}
		Let $\alpha, \beta$ and $\gamma$ be sequences as in \eqref{sequences}. 
		Then we have an equivalence of symmetric tensor categories
		\begin{align*}
			\underline{\operatorname{UCob}_{\alpha,\beta, \gamma}}\simeq  \underline{\Rep(\text{S}_t\wr \mathbb Z_2 )}\boxtimes \underline{\Rep(S_{t_+})}\boxtimes \underline{\Rep(S_{t_-})},
		\end{align*}
		where $t=\frac{1}{2}(\lambda \alpha_0-\gamma_0), t_{+}=\frac{1}{2}(\gamma_0+\sqrt{\lambda}\beta_0)$ and $t_{-}= \frac{1}{2}(\gamma_0-\sqrt{\lambda}\beta_0)$.
	\end{maintheorem}

The rest of this section is dedicated to giving a proof for Theorem \ref{maintheorem2}.
	
Recall that $\SUCob$ is a rigid symmetric tensor category with finite dimensional Hom spaces. In this case, the handle relation is $x-\lambda \text{Id}=0$, and so $\Hom$ spaces are spanned by cobordisms where all connected components have genus 0.
	
	\begin{proposition}
	For $m\geq 2$, let $\{\xi^m_{\overline J}  : \overline J \in \mathcal R_m\}$ be as in Definition \ref{xi definition}, and let $\theta_m:=\Delta^{m-1}\theta$ and  $\theta_m[2]:=\Delta^{m-1}m(\theta\otimes \theta)$. Then the set 
	\begin{align}\label{set t_m}
		\operatorname{\texttt{t}}_m:=\{\xi^m_{\overline J}  : \overline J \in \mathcal R_m \}\cup \{\theta_m, \theta_m[2]\},
	\end{align}
	is a spanning set for the subspace of $\Hom_{\operatorname{SUCob}_{\alpha,\beta,\gamma}}(0,m)$ spanned by connected cobordisms.
\end{proposition}

\begin{proof}
	This follows from Propositions \ref{prop:orientable with boundary} and \ref{prop unorientable with boundary}, since connected cobordisms have genus $0$. 
\end{proof}

Let $P=\{p_1, \dots, p_k\}$ be a partition of $\{1, \dots, m\}$. Let $P_1=\{p_{l_i}\}_{i=1}^{k_1},P_2$ and $P_3$ be disjoint subsets of $P$ such that $P=P_1\cup P_2\cup P_3.$ Choose $\overline{J_{l_i}}\in \mathcal R_{|p_{l_i}|}$ for all $1\leq i \leq k_1.$ 
Denote by $c_{P, \overline{J_{l_1}}, \dots, \overline{J_{l_{k_1}}}, \theta_{P_2}, \theta^2_{P_3}}$ the cobordism $0\to m$ with $k$ connected components where:
\begin{itemize}
	\item the $l_i$-th component is of the form $\xi^{|p_{l_i}|}_{\overline J_{l_i}}$ with out-boundary the circles in positions corresponding to the elements of $p_{l_i},$ for all $1\leq i \leq k_1$, 
	\item for $p\in P_2$, the corresponding connected component is of the form $\theta_{|p|}$ with out-boundary the circles in positions corresponding to the elements of $p,$ 
	\item for $p\in P_3$, the corresponding connected component is of the form $\theta_{|p|}[2]$ with out-boundary the circles in positions corresponding to the elements of $p$. 
\end{itemize}
Then 
\begin{align}\label{spanning set}
	\mathcal T_m:=	\left\{  c_{P, \overline{J_{l_1}}, \dots, \overline{J_{l_{k_1}}}, \theta_{P_2}, \theta^2_{P_3}} \right\}
\end{align}
moving over all possible such choices is a spanning set for $\Hom_{\SOCob}(0,m).$

	\begin{example}\label{example T_m for m=1} 
	Suppose that $\lambda \alpha_0\ne \gamma_0$ and $\gamma_0\ne \pm \sqrt{\lambda} \beta_0$. Then $\Hom_{\SUCob}(0,1)$ has dimension $3$, with basis
	\begin{align*}
		&u=	\begin{aligned}
			\resizebox{20pt}{!}{%
				\begin{tikzpicture}[tqft/cobordism/.style={draw,thick},
					tqft/view from=outgoing, tqft/boundary separation=30pt,
					tqft/cobordism height=40pt, tqft/circle x radius=8pt,
					tqft/circle y radius=4.5pt, tqft/every boundary component/.style={draw,rotate=90},tqft/every incoming
					boundary component/.style={draw,dotted,thick},tqft/every outgoing
					boundary component/.style={draw,thick}]
					\pic[tqft/cap,rotate=90,name=a, anchor={(1,-2)}]; 
			\end{tikzpicture} }
		\end{aligned}
		&	&\theta=	\begin{aligned}
			\resizebox{20pt}{!}{%
				\begin{tikzpicture}[tqft/cobordism/.style={draw,thick},
					tqft/view from=outgoing, tqft/boundary separation=30pt,
					tqft/cobordism height=40pt, tqft/circle x radius=8pt,
					tqft/circle y radius=4.5pt, tqft/every boundary component/.style={draw,rotate=90},tqft/every incoming
					boundary component/.style={draw,dotted,thick},tqft/every outgoing
					boundary component/.style={draw,thick}]
					\pic[tqft/cap,rotate=90,name=c, anchor={(1,-2)}]; 
					\node at ([xshift=-7pt]c-outgoing boundary)[color=gray]{$   \figureXv$};
			\end{tikzpicture} }
		\end{aligned},
		&	&\theta[2]=	\begin{aligned}
			\resizebox{20pt}{!}{%
				\begin{tikzpicture}[tqft/cobordism/.style={draw,thick},
					tqft/view from=outgoing, tqft/boundary separation=30pt,
					tqft/cobordism height=40pt, tqft/circle x radius=8pt,
					tqft/circle y radius=4.5pt, tqft/every boundary component/.style={draw,rotate=90},tqft/every incoming
					boundary component/.style={draw,dotted,thick},tqft/every outgoing
					boundary component/.style={draw,thick}]
					\pic[tqft/cap,rotate=90,name=c, anchor={(1,-2)}]; 
					\node at ([xshift=-7pt]c-outgoing boundary)[color=gray]{$   \figureXv$};
					\node at ([xshift=-3pt]c-outgoing boundary)[color=gray]{$   \figureXv$};
			\end{tikzpicture} }
		\end{aligned}:=
		\begin{aligned}
			\resizebox{50pt}{!}{%
				\begin{tikzpicture}[tqft/cobordism/.style={draw,thick},
					tqft/view from=outgoing, tqft/boundary separation=30pt,
					tqft/cobordism height=40pt, tqft/circle x radius=8pt,
					tqft/circle y radius=4.5pt, tqft/every boundary component/.style={draw,rotate=90},tqft/every incoming
					boundary component/.style={draw,dotted,thick},tqft/every outgoing
					boundary component/.style={draw,thick}]
					\pic[tqft/cap,rotate=90,name=c, anchor={(1,-2)}]; 
					\pic[tqft/cap,rotate=90,name=a, anchor={(0,-2)}]; 
					\node at ([xshift=-7pt]c-outgoing boundary)[color=gray]{$   \figureXv$};
					\node at ([xshift=-7pt]a-outgoing boundary)[color=gray]{$   \figureXv$};
					\pic[tqft/reverse pair of pants,
					rotate=90,name=b,at=(c-outgoing boundary)]; 
			\end{tikzpicture} }
		\end{aligned}.
	\end{align*}
In fact, from Example \ref{example 0 to 1} we know that the matrix of inner products has determinant
	\begin{align*}
\alpha_0(\lambda\gamma_0^2-\lambda^2\beta_0^2)+\gamma_0(\lambda\beta_0^2-\gamma_0^2)=(\alpha_0\lambda -\gamma_0)(\gamma_0^2-\lambda\beta_0^2).
\end{align*}

\end{example}

	\begin{example}\label{example T_m for m=2} Suppose that $\lambda\alpha_0- \gamma_0\ne 0, 2$ and $\gamma_0\pm \sqrt{\lambda}\beta_0\ne 0, 2$. Then the following is a  basis for $\Hom_{\SUCob}(0,2)$:
	\begin{align*}
		&	\begin{aligned}
			\resizebox{20pt}{!}
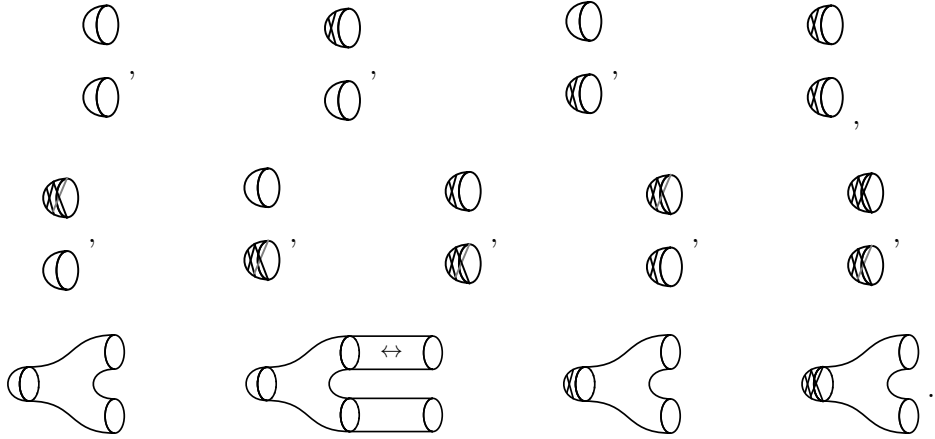
{%
				\begin{tikzpicture}[tqft/cobordism/.style={draw,thick},
					tqft/view from=outgoing, tqft/boundary separation=30pt,
					tqft/cobordism height=40pt, tqft/circle x radius=8pt,
					tqft/circle y radius=4.5pt, tqft/every boundary component/.style={draw,rotate=90},tqft/every incoming
					boundary component/.style={draw,dotted,thick},tqft/every outgoing
					boundary component/.style={draw,thick}]
					\pic[tqft/cap,
					rotate=90,name=a,anchor={(1,-2)}]; 
					\pic[tqft/cap,
					rotate=90,name=b,anchor={(2,-2)}]; 
			\end{tikzpicture} }
		\end{aligned},
		&&	\begin{aligned}
			\resizebox{20pt}{!}{%
				\begin{tikzpicture}[tqft/cobordism/.style={draw,thick},
					tqft/view from=outgoing, tqft/boundary separation=30pt,
					tqft/cobordism height=40pt, tqft/circle x radius=8pt,
					tqft/circle y radius=4.5pt, tqft/every boundary component/.style={draw,rotate=90},tqft/every incoming
					boundary component/.style={draw,dotted,thick},tqft/every outgoing
					boundary component/.style={draw,thick}]
					\pic[tqft/cap,
					rotate=90,name=a,anchor={(1,-2)}]; 
					\node at ([xshift=-7pt]a-outgoing boundary 1) {$\figureXv$};
					\pic[tqft/cap,
					rotate=90,name=b,anchor={(2,-2)}]; 
			\end{tikzpicture} }
		\end{aligned},
		&&	\begin{aligned}
			\resizebox{20pt}{!}{%
				\begin{tikzpicture}[tqft/cobordism/.style={draw,thick},
					tqft/view from=outgoing, tqft/boundary separation=30pt,
					tqft/cobordism height=40pt, tqft/circle x radius=8pt,
					tqft/circle y radius=4.5pt, tqft/every boundary component/.style={draw,rotate=90},tqft/every incoming
					boundary component/.style={draw,dotted,thick},tqft/every outgoing
					boundary component/.style={draw,thick}]
					\pic[tqft/cap,
					rotate=90,name=a,anchor={(1,-2)}]; 
					\pic[tqft/cap,
					rotate=90,name=b,anchor={(2,-2)}]; 
					\node at ([xshift=-7pt]b-outgoing boundary 1) {$\figureXv$};
			\end{tikzpicture} }
		\end{aligned},
		&&	\begin{aligned}
			\resizebox{20pt}{!}{%
				\begin{tikzpicture}[tqft/cobordism/.style={draw,thick},
					tqft/view from=outgoing, tqft/boundary separation=30pt,
					tqft/cobordism height=40pt, tqft/circle x radius=8pt,
					tqft/circle y radius=4.5pt, tqft/every boundary component/.style={draw,rotate=90},tqft/every incoming
					boundary component/.style={draw,dotted,thick},tqft/every outgoing
					boundary component/.style={draw,thick}]
					\pic[tqft/cap,
					rotate=90,name=a,anchor={(1,-2)}]; 
					\node at ([xshift=-7pt]a-outgoing boundary 1) {$\figureXv$};
					\pic[tqft/cap,
					rotate=90,name=b,anchor={(2,-2)}]; 
					\node at ([xshift=-7pt]b-outgoing boundary 1) {$\figureXv$};
			\end{tikzpicture} },
		\end{aligned}
	\end{align*}
	\begin{align*}
		&	\begin{aligned}
			\resizebox{20pt}{!}{%
				\begin{tikzpicture}[tqft/cobordism/.style={draw,thick},
					tqft/view from=outgoing, tqft/boundary separation=30pt,
					tqft/cobordism height=40pt, tqft/circle x radius=8pt,
					tqft/circle y radius=4.5pt, tqft/every boundary component/.style={draw,rotate=90},tqft/every incoming
					boundary component/.style={draw,dotted,thick},tqft/every outgoing
					boundary component/.style={draw,thick}]
					\pic[tqft/cap,
					rotate=90,name=a,anchor={(1,-2)}]; 
					\node at ([xshift=-7pt]a-outgoing boundary 1) {$\figureXv$};
					\node at ([xshift=-3pt]a-outgoing boundary)[color=gray]{$   \figureXvl$};
					\pic[tqft/cap,
					rotate=90,name=b,anchor={(2,-2)}]; 
			\end{tikzpicture} }
		\end{aligned},
		&&	\begin{aligned}
			\resizebox{20pt}{!}{%
				\begin{tikzpicture}[tqft/cobordism/.style={draw,thick},
					tqft/view from=outgoing, tqft/boundary separation=30pt,
					tqft/cobordism height=40pt, tqft/circle x radius=8pt,
					tqft/circle y radius=4.5pt, tqft/every boundary component/.style={draw,rotate=90},tqft/every incoming
					boundary component/.style={draw,dotted,thick},tqft/every outgoing
					boundary component/.style={draw,thick}]
					\pic[tqft/cap,
					rotate=90,name=a,anchor={(1,-2)}]; 
					\pic[tqft/cap,
					rotate=90,name=b,anchor={(2,-2)}]; 
					\node at ([xshift=-7pt]b-outgoing boundary 1) {$\figureXv$};
					\node at ([xshift=-3pt]b-outgoing boundary)[color=gray]{$   \figureXvl$};
			\end{tikzpicture} }
		\end{aligned},
		&&	\begin{aligned}
			\resizebox{20pt}{!}{%
				\begin{tikzpicture}[tqft/cobordism/.style={draw,thick},
					tqft/view from=outgoing, tqft/boundary separation=30pt,
					tqft/cobordism height=40pt, tqft/circle x radius=8pt,
					tqft/circle y radius=4.5pt, tqft/every boundary component/.style={draw,rotate=90},tqft/every incoming
					boundary component/.style={draw,dotted,thick},tqft/every outgoing
					boundary component/.style={draw,thick}]
					\pic[tqft/cap,
					rotate=90,name=a,anchor={(1,-2)}]; 
					\pic[tqft/cap,
					rotate=90,name=b,anchor={(2,-2)}]; 
					\node at ([xshift=-7pt]a-outgoing boundary 1) {$\figureXv$};
					\node at ([xshift=-7pt]b-outgoing boundary 1) {$\figureXv$};
					\node at ([xshift=-3pt]b-outgoing boundary)[color=gray]{$   \figureXvl$};
			\end{tikzpicture} }
		\end{aligned},
		&&	\begin{aligned}
			\resizebox{20pt}{!}{%
				\begin{tikzpicture}[tqft/cobordism/.style={draw,thick},
					tqft/view from=outgoing, tqft/boundary separation=30pt,
					tqft/cobordism height=40pt, tqft/circle x radius=8pt,
					tqft/circle y radius=4.5pt, tqft/every boundary component/.style={draw,rotate=90},tqft/every incoming
					boundary component/.style={draw,dotted,thick},tqft/every outgoing
					boundary component/.style={draw,thick}]
					\pic[tqft/cap,
					rotate=90,name=a,anchor={(1,-2)}]; 
					\pic[tqft/cap,
					rotate=90,name=b,anchor={(2,-2)}]; 
					\node at ([xshift=-7pt]a-outgoing boundary 1) {$\figureXv$};
					\node at ([xshift=-7pt]b-outgoing boundary 1) {$\figureXv$};
					\node at ([xshift=-3pt]a-outgoing boundary)[color=gray]{$   \figureXvl$};
			\end{tikzpicture} }
		\end{aligned},
		&&	\begin{aligned}
			\resizebox{20pt}{!}{%
				\begin{tikzpicture}[tqft/cobordism/.style={draw,thick},
					tqft/view from=outgoing, tqft/boundary separation=30pt,
					tqft/cobordism height=40pt, tqft/circle x radius=8pt,
					tqft/circle y radius=4.5pt, tqft/every boundary component/.style={draw,rotate=90},tqft/every incoming
					boundary component/.style={draw,dotted,thick},tqft/every outgoing
					boundary component/.style={draw,thick}]
					\pic[tqft/cap,
					rotate=90,name=a,anchor={(1,-2)}]; 
					\pic[tqft/cap,
					rotate=90,name=b,anchor={(2,-2)}]; 
					\node at ([xshift=-7pt]a-outgoing boundary 1) {$\figureXv$};
					\node at ([xshift=-7pt]b-outgoing boundary 1) {$\figureXv$};
					\node at ([xshift=-3pt]a-outgoing boundary 1) {$\figureXvl$};
					\node at ([xshift=-3pt]b-outgoing boundary)[color=gray]{$   \figureXvl$};
			\end{tikzpicture} }
		\end{aligned},
	\end{align*}
	\begin{align*}
		&\begin{aligned}
			\resizebox{50pt}{!}{%
				\begin{tikzpicture}[tqft/cobordism/.style={draw,thick},
					tqft/view from=outgoing, tqft/boundary separation=30pt,
					tqft/cobordism height=40pt, tqft/circle x radius=8pt,
					tqft/circle y radius=4.5pt, tqft/every boundary component/.style={draw,rotate=90},tqft/every incoming
					boundary component/.style={draw,dotted,thick},tqft/every outgoing
					boundary component/.style={draw,thick}]
					\pic[tqft/cap,
					rotate=90,name=a,anchor={(1,-2)}]; 
					\pic[tqft/pair of pants,
					rotate=90,name=b,at=(a-outgoing boundary)]; 
			\end{tikzpicture} }
		\end{aligned}
		&&\begin{aligned}
			\resizebox{80pt}{!}{%
				\begin{tikzpicture}[tqft/cobordism/.style={draw,thick},
					tqft/view from=outgoing, tqft/boundary separation=30pt,
					tqft/cobordism height=40pt, tqft/circle x radius=8pt,
					tqft/circle y radius=4.5pt, tqft/every boundary component/.style={draw,rotate=90},tqft/every incoming
					boundary component/.style={draw,dotted,thick},tqft/every outgoing
					boundary component/.style={draw,thick}]
					\pic[tqft/cap,
					rotate=90,name=a,anchor={(1,-2)}]; 
					\pic[tqft/pair of pants,
					rotate=90,name=b,at=(a-outgoing boundary)]; 
					\pic[tqft/cylinder,rotate=90,name=g at=(b-outgoing boundary 1),anchor={(1.5,-4)}];
					\pic[tqft/cylinder,rotate=90,name=f, at=(b-outgoing boundary 2)];
					\node at ([xshift=20pt]f-incoming boundary 1){$\leftrightarrow$};
			\end{tikzpicture} }
		\end{aligned}
		&&\begin{aligned}
			\resizebox{50pt}{!}{%
				\begin{tikzpicture}[tqft/cobordism/.style={draw,thick},
					tqft/view from=outgoing, tqft/boundary separation=30pt,
					tqft/cobordism height=40pt, tqft/circle x radius=8pt,
					tqft/circle y radius=4.5pt, tqft/every boundary component/.style={draw,rotate=90},tqft/every incoming
					boundary component/.style={draw,dotted,thick},tqft/every outgoing
					boundary component/.style={draw,thick}]
					\pic[tqft/cap,
					rotate=90,name=a,anchor={(1,-2)}]; 
					\pic[tqft/pair of pants,
					rotate=90,name=b,at=(a-outgoing boundary)]; 
					\node at ([xshift=-7pt]a-outgoing boundary 1) {$\figureXv$};
			\end{tikzpicture} }
		\end{aligned}
		&&	\begin{aligned}
			\resizebox{50pt}{!}{%
				\begin{tikzpicture}[tqft/cobordism/.style={draw,thick},
					tqft/view from=outgoing, tqft/boundary separation=30pt,
					tqft/cobordism height=40pt, tqft/circle x radius=8pt,
					tqft/circle y radius=4.5pt, tqft/every boundary component/.style={draw,rotate=90},tqft/every incoming
					boundary component/.style={draw,dotted,thick},tqft/every outgoing
					boundary component/.style={draw,thick}]
					\pic[tqft/cap,
					rotate=90,name=a,anchor={(1,-2)}]; 
					\node at ([xshift=-7pt]a-outgoing boundary 1) {$\figureXv$};
					\node at ([xshift=-3pt]a-outgoing boundary 1) {$\figureXvl$};
					\pic[tqft/pair of pants,at=(a-outgoing boundary 1),
					rotate=90]; 
			\end{tikzpicture} }
		\end{aligned}.
	\end{align*}
	
	In fact, it is a generating set by the proposition above, and the matrix of inner products is
	\footnotesize
	\begin{align*}
		\left[	 \begin{array}{c c ccccccccccc}
			\alpha_0^2&    \alpha_0  \beta_0&    \alpha_0  \beta_0&      \beta_0^2&    \alpha_0  \gamma_0&    \alpha_0  \gamma_0&      \beta_0  \gamma_0&      \beta_0  \gamma_0&      \gamma_0^2&    \alpha_0&    \alpha_0&      \beta_0&      \gamma_0\\
			\alpha_0  \beta_0&    \alpha_0  \gamma_0&    \beta_0^2&      \beta_0  \gamma_0&  \alpha_0  \beta_0  \lambda&    \beta_0  \gamma_0&      \gamma_0^2&    \beta_0^2  \lambda&    \beta_0  \gamma_0  \lambda&    \beta_0&    \beta_0&      \gamma_0&    \beta_0  \lambda\\
			\alpha_0  \beta_0&    \beta_0^2&    \alpha_0  \gamma_0&     \beta_0  \gamma_0&    \beta_0  \gamma_0&  \alpha_0  \beta_0  \lambda&    \beta_0^2  \lambda&      \gamma_0^2&    \beta_0  \gamma_0  \lambda&    \beta_0&    \beta_0&      \gamma_0&    \beta_0  \lambda\\
			\beta_0^2&    \beta_0  \gamma_0&    \beta_0  \gamma_0&      \gamma_0^2&  \beta_0^2  \lambda&  \beta_0^2  \lambda&    \beta_0  \gamma_0  \lambda&    \beta_0  \gamma_0  \lambda&  \beta_0^2  \lambda^2&    \gamma_0&    \gamma_0&    \beta_0  \lambda&    \gamma_0  \lambda\\
			\alpha_0  \gamma_0&  \alpha_0  \beta_0  \lambda&    \beta_0  \gamma_0&     \beta_0^2  \lambda&  \alpha_0  \gamma_0  \lambda&    \gamma_0^2&    \beta_0  \gamma_0  \lambda&    \beta_0  \gamma_0  \lambda&    \gamma_0^2  \lambda&    \gamma_0&    \gamma_0&    \beta_0  \lambda&    \gamma_0  \lambda\\
			\alpha_0  \gamma_0&    \beta_0  \gamma_0&  \alpha_0  \beta_0  \lambda&    \beta_0^2 l&    \gamma_0^2&  \alpha_0  \gamma_0  \lambda&    \beta_0  \gamma_0  \lambda&    \beta_0  \gamma_0  \lambda&    \gamma_0^2  \lambda&    \gamma_0&    \gamma_0&    \beta_0  \lambda&    \gamma_0  \lambda\\
			\beta_0  \gamma_0&    \gamma_0^2&  \beta_0^2  \lambda&     \beta_0  \gamma_0  \lambda&  \beta_0  \gamma_0  \lambda&  \beta_0  \gamma_0  \lambda&    \gamma_0^2  \lambda&  \beta_0^2  \lambda^2&  \beta_0  \gamma_0  \lambda^2&  \beta_0  \lambda&  \beta_0  \lambda&    \gamma_0  \lambda&  \beta_0  \lambda^2\\
			\beta_0  \gamma_0&  \beta_0^2  \lambda&    \gamma_0^2&     \beta_0  \gamma_0  \lambda&  \beta_0  \gamma_0  \lambda&  \beta_0  \gamma_0  \lambda&  \beta_0^2  \lambda^2&    \gamma_0^2  \lambda&  \beta_0  \gamma_0  \lambda^2&  \beta_0  \lambda&  \beta_0  \lambda&    \gamma_0  \lambda&  \beta_0  \lambda^2\\
			\gamma_0^2&  \beta_0  \gamma_0  \lambda&  \beta_0  \gamma_0  \lambda&  \beta_0^2  \lambda^2&  \gamma_0^2  \lambda&  \gamma_0^2  \lambda&  \beta_0  \gamma_0  \lambda^2&  \beta_0  \gamma_0  \lambda^2&  \gamma_0^2  \lambda^2&  \gamma_0  \lambda&  \gamma_0  \lambda&  \beta_0  \lambda^2&  \gamma_0  \lambda^2\\
			\alpha_0&      \beta_0&      \beta_0&        \gamma_0&      \gamma_0&      \gamma_0&      \beta_0  \lambda&      \beta_0  \lambda&      \gamma_0  \lambda&  \alpha_0  \lambda&    \gamma_0&    \beta_0  \lambda&    \gamma_0  \lambda\\
			\alpha_0&      \beta_0&      \beta_0&        \gamma_0&      \gamma_0&      \gamma_0&      \beta_0  \lambda&      \beta_0  \lambda&      \gamma_0  \lambda&    \gamma_0&  \alpha_0  \lambda&    \beta_0  \lambda&    \gamma_0  \lambda\\
			\beta_0&      \gamma_0&      \gamma_0&      \beta_0  \lambda&    \beta_0  \lambda&    \beta_0  \lambda&      \gamma_0  \lambda&      \gamma_0  \lambda&    \beta_0  \lambda^2&  \beta_0  \lambda&  \beta_0  \lambda&    \gamma_0  \lambda&  \beta_0  \lambda^2\\
			\gamma_0&    \beta_0  \lambda&    \beta_0  \lambda&      \gamma_0  \lambda&    \gamma_0  \lambda&    \gamma_0  \lambda&    \beta_0  \lambda^2&    \beta_0  \lambda^2&    \gamma_0  \lambda^2&  \gamma_0  \lambda&  \gamma_0  \lambda&  \beta_0  \lambda^2&  \gamma_0 \lambda^2
		\end{array}\right],\end{align*}
	\normalsize
	which has determinant 
	\begin{align*}
		\lambda^3(\gamma_0-\sqrt{\lambda}\beta_0)^6(\gamma_0+\sqrt{\lambda}\beta_0)^6(\gamma_0-\sqrt{\lambda}\beta_0-2)(\gamma_0+\sqrt{\lambda}\beta_0-2)(\lambda\alpha_0-\gamma_0)^7(\lambda\alpha_0-\gamma_0-2).
	\end{align*}
	Hence
	\begin{align*}
		\dim(\Hom_{\SUCob}(0,2))=13.
	\end{align*}
\end{example}

\begin{conjecture}\label{conjecture for product}
	We conjecture that for $\lambda, \alpha_0, \beta_0, \gamma_0$ in $\mathbf k^{\times}$ such that  $\lambda\alpha_0-\gamma_0$ and $\gamma_0\pm \sqrt{\lambda}\beta_0$ are  non-negative even integers, the determinants of the matrices of inner products of morphisms in $\mathcal T_m$ are non-zero polynomials on $\lambda, \alpha_0, \beta_0, \gamma_0$, for all $m\geq 1$.
	
	We know this to be true for $m=0,1,2$, see the examples above. The general computation would follow the same lines as the proof of Lemma \ref{lemma: inner products matrix}. However, this case requires a more careful combinatorial analysis, since for instance, for certain $a_i, b_i\geq 0$ such that $\sum\limits_{i=1}^l(a_i+b_i)=k$, we will get
	\begin{align*}
		(\theta_k, \theta_k)=\gamma_{k-1}=(\theta_k,\theta^{\otimes a_1}\otimes \theta[2]^{\otimes b_1}\otimes \theta^{\otimes a_2}\otimes \dots \otimes \theta[2]^{\otimes b_l}).
	\end{align*}
	That is, the diagonal entry in $\mathcal T_k$ corresponding to the row of $\theta_k$ will be repeated in another entry of the same row. Moreover, it will actually be repeated more than once in the same row, since
	\begin{align*}
		(\theta_k,\theta^{\otimes a_1}\otimes \theta[2]^{\otimes b_1}\otimes \theta^{\otimes a_2}\otimes \dots \otimes \theta[2]^{\otimes b_l})=(\theta_k,\theta^{\otimes a_1'}\otimes \theta[2]^{\otimes b_1'}\otimes \theta^{\otimes a_2'}\otimes \dots \otimes \theta[2]^{\otimes b_j'}),
	\end{align*} 
	whenever $a_i, a_i', b_i, b_i'\geq 0$ are such that $\sum\limits_{i=1}^l(a_i+b_i)=\sum\limits_{i=1}^j(a_i'+b_i')=k$,
	and $a_1+\dots+a_l+2(b_1+\dots+b_l)	=a_1'+\dots+a_l'+2(b_1'+\dots+b_l')	$.
\end{conjecture}

\begin{remark}
	We conjectured that the exceptional values of $\lambda, \alpha_0, \beta_0, \gamma_0$ are those such that $\lambda\alpha_0-\gamma_0$ and $\gamma_0\pm \sqrt{\lambda}\beta_0$ are  non-negative even integers. We thus predict that the determinant of the matrix of inner products of $\mathcal T_m$, see \eqref{spanning set for product}, will factor into powers of the form $\lambda \alpha_0-\gamma_0-2k$ and $\gamma_0\pm \sqrt{\lambda}\beta_0-2k$, for $k=0, \dots, m-1$. Note that we know this to be the case for $m=1$ and $2$, see Examples \ref{example 0 to 1} and \ref{example T_m for m=2}.
\end{remark}

	To prove this, we will use the \cite[Lemma 2.6]{BEEO}, specialized to $\UCob$. We write the statement below for clarity. 

	\begin{proposition}\label{proposition 2}
		Let $\mathcal C$ be a semisimple Karoubian symmetric tensor category with finite dimensional Hom spaces. Suppose there is a symmetric tensor functor $F:\operatorname{UCob}_{\alpha,\beta,\gamma}\to \mathcal C$ that is surjective on $\Hom$'s. 
		Then $F$ induces a fully faithful symmetric tensor functor  
		\begin{align*}
			F:\underline{\operatorname{UCob}_{\alpha, \beta, \gamma} }\xrightarrow{\sim} \mathcal C.
		\end{align*}
	\end{proposition}

	We define now Deligne’s  category $\Rep(S_t)$ for $t\in \mathbf k$, following \cite{CO,D}.
	\begin{definition}\cite[Definition 2.11]{CO}
		Consider the category $\Rep_0(S_t)$ given by:
		\begin{itemize}
			\item Objects are non-negative integers. We represent $n\in \mathbb Z_{\geq 0}$ by $n$ horizontal dots (zero is represented by ``no dots"). 
			\item Morphisms $m\to m'$ are given by $\mathbf k$-linear combinations of partitions of the set $\{1, \dots, m, 1', \dots, m'\}$. Such a partition is represented by a diagram with $m$ points on top labelled $1$ to $m$ and $m'$ points on the bottom labeled $1'$ to $m'$, such that points in the same part of the partition are connected by a path. 
			\item Composition is as described in \cite[Definition 2.11]{CO}.
		\end{itemize}
	\end{definition}

	\begin{definition}
		\cite{D} Let $\Rep(S_t)$ be the pseudo-abelian envelope of $\Rep_0(S_t)$.
	\end{definition}
	
	\begin{remark}
		Let $t\in \mathbf k^{\times}$. Recall that we represent morphisms in $\Rep_0(S_t)$ as going from top to bottom. There is a canonical Frobenius algebra $A=A_{\lambda}$ in $\Rep(S_t)$, given by the object 1 (represented as 1 point) and unit, mutiplication, counit and comultiplication maps given by
		\begin{align*}
			&	u=
			\begin{aligned}
				\resizebox{4pt}{!}{%
					\begin{tikzpicture}[block/.style={draw, rectangle, minimum height=0.5cm,fill=white}]
						\draw[fill=black] (0,0) circle (0.05cm and 0.05cm);
						\draw[fill=white,color=white] (0,1) circle (0.05cm and 0.05cm);
				\end{tikzpicture}}
			\end{aligned}\ \ \ ,
			&&	m=
			\begin{aligned}
				\resizebox{40pt}{!}{%
					\begin{tikzpicture}[block/.style={draw, rectangle, minimum height=0.5cm,fill=white}]
						\draw[fill=black] (-0.5,1) circle (0.05cm and 0.05cm);
						\draw[fill=black] (0.5,1) circle (0.05cm and 0.05cm);
						\draw[fill=black] (0,0) circle (0.05cm and 0.05cm);
						\draw  (0,0) to   (0.5,1);
						\draw  (0,0) to   (-0.5,1);
						\draw  (0.5,1) to   (-0.5,1);
				\end{tikzpicture}}
			\end{aligned},
			&&	\varepsilon= \frac{1}{\lambda} \
			\begin{aligned}
				\resizebox{4pt}{!}{%
					\begin{tikzpicture}[block/.style={draw, rectangle, minimum height=0.5cm,fill=white}]
						\draw[fill=black] (0,1) circle (0.05cm and 0.05cm);
						\draw[fill=white,color=white] (0,0) circle (0.05cm and 0.05cm);
				\end{tikzpicture}}
			\end{aligned}\ \ \ ,
			&&\Delta = \lambda
			\begin{aligned}
				\resizebox{40pt}{!}{%
					\begin{tikzpicture}[block/.style={draw, rectangle, minimum height=0.5cm,fill=white}]
						\draw[fill=black] (-0.5,0) circle (0.05cm and 0.05cm);
						\draw[fill=black] (0.5,0) circle (0.05cm and 0.05cm);
						\draw[fill=black] (0,1) circle (0.05cm and 0.05cm);
						\draw  (0,1) to   (0.5,0);
						\draw  (0,1) to   (-0.5,0);
						\draw  (0.5,0) to   (-0.5,0);
				\end{tikzpicture}}
			\end{aligned},
		\end{align*}
	respectively. Then if
		\begin{align*}
			&&\phi=
			\begin{aligned}
				\resizebox{4pt}{!}{%
					\begin{tikzpicture}[block/.style={draw, rectangle, minimum height=0.5cm,fill=white}]
						\draw[fill=black] (0,0) circle (0.05cm and 0.05cm);
						\draw[fill=black] (0,1) circle (0.05cm and 0.05cm);
						\draw  (0,1) to   (0,0);
				\end{tikzpicture}}
			\end{aligned}
			&&\text{and}
			&&	\theta=  \pm \sqrt{\lambda} u,
		\end{align*}
		we get that $A$ with these structure maps is an extended Frobenius algebra in $\Rep(S_t)$, see Definition 	\ref{extended Frob alg} 
	\end{remark}
	
	From now on, let 
	\begin{align*}
		&&	t=\frac{1}{2}(\lambda\alpha_0-\gamma_0), &&t_+=\frac{1}{2}(\gamma_0+\sqrt{\lambda}\beta_0) &&\text{and} &&t_-=\frac{1}{2}(\gamma_0-\sqrt{\lambda}\beta_0).
	\end{align*}
	Let $A_{\pm}=A_{\pm,\lambda}$ denote the extended Frobenius algebras in $\Rep(S_{t_{\pm}})$ as defined above, respectively, and let $A_t=A_{t,\lambda}:=\langle A \rangle_t$ in $\Rep(S_t \wr \mathbb Z_2)$ be the extended Frobenius algebra induced from the algebra of functions introduced in Example \ref{algebra of functions} for $n=1$. We remark that the structure maps of these algebras depend on $\lambda$, but we drop the subscript to simplify notation.
	
	\begin{lemma}
The extended Frobenius algebra
		\begin{align*}
			\mathcal A := (A_t \boxtimes \mathbb 1 \boxtimes \mathbb 1) \oplus (\mathbb 1 \boxtimes A_+ \boxtimes \mathbb 1) \oplus (\mathbb 1 \boxtimes \mathbb 1 \boxtimes A_-)
		\end{align*}
		in $\Rep(S_t \wr \mathbb Z_2) \boxtimes \Rep(S_{t_+}) \boxtimes \Rep(S_{t_-})$ has evaluation $\alpha, \beta, \gamma$ as defined on Equation \eqref{sequences}.
	\end{lemma}
	\begin{proof}
		The evaluation of $A_t\boxtimes \mathbb 1\boxtimes \mathbb 1$ was computed in Lemma \ref{ext eval of A}, and is given by
		\begin{align*}
			&&\alpha_t=(2\lambda^{-1}t, 2t, 2\lambda t, \dots) &&\text{and} &&\beta_t=\gamma_t=(0,\dots).
		\end{align*}
		On the other hand, it is easy to check that the evaluations of $\mathbb 1 \boxtimes A_+ \boxtimes \mathbb 1$ and $\mathbb 1 \boxtimes \mathbb 1 \boxtimes A_-$ are 
		\begin{align*}
			\alpha_{\pm}=(\lambda^{-1}t_{\pm}, t_{\pm}, \lambda t_{\pm}, \dots), &&\beta_{\pm}=(\pm\lambda^{-1/2}t_{\pm},\pm\lambda^{1/2} t_{\pm}, \pm\lambda^{3/2}  t_{\pm}, \dots),&&\gamma_{\pm}=(t_{\pm},\lambda t_{\pm}, \lambda^2 t_{\pm}, \dots),
		\end{align*}
		respectively.
		
	Call $\tilde \alpha, \tilde \beta$ and $\tilde \gamma$ the evaluation sequences of $\mathcal A$. Note that $2t+t_{+}+t_{-}=\lambda\alpha_0$. Hence the $\tilde\alpha$ sequence of the evaluation of $\mathcal A$ is
		\begin{align*}
		\tilde	\alpha&=\alpha_t+\alpha_{+}+\alpha_{-}\\
			&=(\lambda^{-1}(2t+t_++t_{-}), 2t+t_++t_{-}, \lambda(2t+t_++t_{-}), \dots)\\
			&=(\alpha_0, \lambda \alpha_0, \lambda^2\alpha_0, \dots),
		\end{align*}
		as desired. 
		On the other hand, $t_+ -t_- =\sqrt{\lambda}\beta_0$ and so
		\begin{align*}
		\tilde	\beta&=\beta_t+\beta_+ + \beta_-\\
			&=(\lambda^{-1/2}(t_+ -t_-),\lambda^{1/2} (t_+ -t_-), \lambda^{3/2}  (t_+ -t_-), \dots)\\
			&=(\beta_0, \lambda\beta_0, \lambda^2\beta_0, \dots).
		\end{align*}
		Lastly, $t_+ +t_- =\gamma_0$, thus
		\begin{align*}
		\tilde	\gamma&=\gamma_t+\gamma_+ + \gamma_-\\
			&=(t_+ +t_-,\lambda (t_+ +t_-), \lambda^{2}  (t_+ +t_-), \dots)\\
			&=(\gamma_0, \lambda\gamma_0, \lambda^2\gamma_0, \dots).
		\end{align*}
	\end{proof}
	
	It follows from the previous lemma and the universal property of $\VUCob$, see Section \ref{universal property}, that there exists a symmetric tensor functor
	\begin{align*}
		F_{\mathcal A}: \VUCob \to \Rep(S_t \wr \mathbb Z_2) \boxtimes \Rep(S_{t_+}) \boxtimes \Rep(S_{t_-}),
	\end{align*}
	mapping  the circle object of $\VUCob$ to  $\mathcal A,$ and the extended Frobenius algebra structure maps of the circle to those of $\mathcal A$.

	\begin{lemma}
		$F_{\mathcal A}: \operatorname{VUCob}_{\alpha,\beta,\gamma} \to \Rep(S_t \wr \mathbb Z_2) \boxtimes \Rep(S_{t_+}) \boxtimes \Rep(S_{t_-})$ annihilates the handle relation $x-\lambda \Id,$ and so it factors through $\SUCob$. 
	\end{lemma}

\begin{proof}
	We know by equation \eqref{computation of x in StC} that $m_{A_t}\Delta_{A_t}-\lambda \text{Id}_{A_t}=0$ in $\Rep(S_t\wr \mathbb Z_2)$. On the other hand, 
	\begin{align*}
		m_{A_{\pm}} \Delta_{A_{\pm}} =\lambda
		\begin{aligned}
			\resizebox{30pt}{!}{%
				\begin{tikzpicture}[block/.style={draw, rectangle, minimum height=0.5cm,fill=white}]
					\draw[fill=black] (-0.5,0) circle (0.05cm and 0.05cm);
					\draw[fill=black] (0.5,0) circle (0.05cm and 0.05cm);
					\draw[fill=black] (0,1) circle (0.05cm and 0.05cm);
					\draw  (0,1) to   (0.5,0);
					\draw  (0,1) to   (-0.5,0);
					\draw  (0.5,0) to   (-0.5,0);
			\end{tikzpicture}}
		\end{aligned} \circ \begin{aligned}
		\resizebox{30pt}{!}{%
			\begin{tikzpicture}[block/.style={draw, rectangle, minimum height=0.5cm,fill=white}]
				\draw[fill=black] (-0.5,1) circle (0.05cm and 0.05cm);
				\draw[fill=black] (0.5,1) circle (0.05cm and 0.05cm);
				\draw[fill=black] (0,0) circle (0.05cm and 0.05cm);
				\draw  (0,0) to   (0.5,1);
				\draw  (0,0) to   (-0.5,1);
				\draw  (0.5,1) to   (-0.5,1);
		\end{tikzpicture}}
	\end{aligned}= \lambda \begin{aligned}
	\resizebox{3pt}{!}{%
		\begin{tikzpicture}[block/.style={draw, rectangle, minimum height=0.5cm,fill=white}]
			\draw[fill=black] (0,0) circle (0.05cm and 0.05cm);
			\draw[fill=black] (0,1) circle (0.05cm and 0.05cm);
			\draw  (0,1) to   (0,0);
	\end{tikzpicture}}
\end{aligned},
	\end{align*}
in $\Rep(S_{t_{\pm}})$. It follows that $F_{\mathcal A}(x-\lambda\text{Id}_{\mathcal A})=0$.
\end{proof}

	\begin{lemma}\label{1st lemma}
		The functor $	F_{\mathcal A}: \operatorname{SUCob}_{\alpha, \beta, \gamma} \to \Rep(S_t \wr \mathbb Z_2) \boxtimes \Rep(S_{t_+}) \boxtimes \Rep(S_{t_-})$ satisfies that any indecomposable object of $\Rep(S_t \wr \mathbb Z_2) \boxtimes \Rep(S_{t_+}) \boxtimes \Rep(S_{t_-})$ is a direct sumand of $F(n)$ for some $n$. 
	\end{lemma}
	
	\begin{proof}
		The functor $F_{\mathcal A}$ maps $1\mapsto \mathcal A$. Any object in $\Rep(S_t\wr \mathbb Z_2)$, respectively, in $\Rep(S_{t_{\pm}}) $, is a direct summand of tensor powers of $A$, respectively $A_{\pm}$, and so the statement follows. 
	\end{proof}
	
	\begin{lemma}\label{2nd lemma}
		The unique extension
		\begin{align*}
			F_{\mathcal A}: \operatorname{UCob}_{\alpha, \beta, \gamma} \to \Rep(S_t \wr \mathbb Z_2) \boxtimes \Rep(S_{t_+}) \boxtimes \Rep(S_{t_-}),
		\end{align*}
		is surjective on Hom's.
	\end{lemma}
	
	\begin{proof} Let $\mathcal C:=\Rep(S_t \wr \mathbb Z_2) \boxtimes \Rep(S_{t_+}) \boxtimes \Rep(S_{t_-})$.
		To show that $F_{\mathcal A}$ is surjective on morphisms, it is enough to check that the maps 
		\begin{align*}
			\Hom_{\UCob}(n,m)  \xrightarrow{F_{\mathcal A}} \Hom_{\mathcal C}( \mathcal  A^{\otimes n},\mathcal A ^{\otimes m}),
		\end{align*}
		induced by $F_{\mathcal A}$ are surjective for all $n,m\geq 1$. Since $\mathcal A$ is self-dual, it is enough to check surjectivity of the maps 
		\begin{align*}
			\Hom_{\UCob}(0,m)  \xrightarrow{F_{\mathcal A}} \Hom_{\mathcal C}( \mathbb 1,\mathcal A^{\otimes m}),
		\end{align*}
		for all $m\geq 1$.  Since $	\mathcal A= (A_t \boxtimes \mathbb 1 \boxtimes \mathbb 1) \oplus (\mathbb 1 \boxtimes A_+ \boxtimes \mathbb 1) \oplus (\mathbb 1 \boxtimes \mathbb 1 \boxtimes A_-)$, it follows it is enough to check that the direct summands  $\Hom_{\mathcal C}(\mathbb 1, A_t^{\otimes m}\boxtimes \mathbb 1 \boxtimes \mathbb 1)$, $\Hom_{\mathcal C}(\mathbb 1, \mathbb 1 \boxtimes A_+^{\otimes m}\boxtimes \mathbb 1)$ and $\Hom_{\mathcal C}(\mathbb 1, \mathbb 1 \boxtimes \mathbb 1\boxtimes  A_-^{\otimes m})$ of $\Hom_{\mathcal C}( \mathbb 1,\mathcal A^{\otimes m})$ are in the image of $F_{\mathcal A},$ for all $m\geq 1$. Then, all remaining summands will be in the image by induction. 
		
		We show surjectivity on $\Hom_{\mathcal C}(\mathbb 1, \mathbb 1 \boxtimes A_+^{\otimes m}\boxtimes \mathbb 1)$ and $\Hom_{\mathcal C}(\mathbb 1, \mathbb 1 \boxtimes \mathbb 1\boxtimes  A_-^{\otimes m})$ for all $m\geq 1$ first. Note that partitions of the form 
		\begin{align*}
			\begin{aligned}
				\resizebox{60pt}{!}{%
					\begin{tikzpicture}[block/.style={draw, rectangle, minimum height=0.5cm,fill=white}]
						\draw[fill=black] (0,0) circle (0.05cm and 0.05cm);
						\draw[fill=black] (0.5,0) circle (0.05cm and 0.05cm);
						\node at (1,0) {$\dots$};
						\draw[fill=black] (1.5,0) circle (0.05cm and 0.05cm);
						\draw  (0,0) to   (0.7,0);
						\draw  (1.3,0) to   (1.5,0);
						\node at (0.8,-0.5)[color=black][font=\large]{$\underbrace{\ \ \ \ \ \ \ \ \ \ \ \ \ }_k$};
				\end{tikzpicture}}
			\end{aligned}\  \ \ \text{for } k\leq m,
		\end{align*}
		generate $\Hom_{\Rep(S_{t_{\pm}})}(0,m)$. That is, any morphism in $\Hom_{\Rep(S_{t_{\pm}})}(0,m)$ is linear combination of tensor products of partitions of this form. So it is enough to show that these partitions are in the image of $F_{\mathcal A}$. Recall that $F_{\mathcal A}$ maps the structure maps of the circle object in $\PsUCob$ to those of $\mathcal A$. So
		\vspace{-0.5cm}
		\begin{align*}
			&	\theta_m \mapsto 0\oplus \lambda^{m-1/2} 
			\begin{aligned}
				\resizebox{50pt}{!}{%
					\begin{tikzpicture}[block/.style={draw, rectangle, minimum height=0.5cm,fill=white}]
						\node at (1,1) {};
						\draw[fill=black] (0,0) circle (0.05cm and 0.05cm);
						\draw[fill=black] (0.5,0) circle (0.05cm and 0.05cm);
						\node at (1,0) {$\dots$};
						\draw[fill=black] (1.5,0) circle (0.05cm and 0.05cm);
						\draw  (0,0) to   (0.7,0);
						\draw  (1.3,0) to   (1.5,0);
						\node at (0.8,-0.6)[color=black][font=\large]{$\underbrace{\ \ \ \ \ \ \ \ \ \ \ \ \ }_m$};
				\end{tikzpicture}}
			\end{aligned}\oplus (- \lambda^{m-1/2}  )
			\begin{aligned}
				\resizebox{50pt}{!}{%
					\begin{tikzpicture}[block/.style={draw, rectangle, minimum height=0.5cm,fill=white}]
						\node at (1,1) {};
						\draw[fill=black] (0,0) circle (0.05cm and 0.05cm);
						\draw[fill=black] (0.5,0) circle (0.05cm and 0.05cm);
						\node at (1,0) {$\dots$};
						\draw[fill=black] (1.5,0) circle (0.05cm and 0.05cm);
						\draw  (0,0) to   (0.7,0);
						\draw  (1.3,0) to   (1.5,0);
						\node at (0.8,-0.6)[color=black][font=\large]{$\underbrace{\ \ \ \ \ \ \ \ \ \ \ \ \ }_m$};
				\end{tikzpicture}}
			\end{aligned}
		\end{align*}
		and
			\vspace{-0.5cm}
		\begin{align*}
			&	\theta_m[2] \mapsto 0\oplus  \lambda^m 
			\begin{aligned}
				\resizebox{50pt}{!}{%
					\begin{tikzpicture}[block/.style={draw, rectangle, minimum height=0.5cm,fill=white}]
						\node at (1,1) {};
						\draw[fill=black] (0,0) circle (0.05cm and 0.05cm);
						\draw[fill=black] (0.5,0) circle (0.05cm and 0.05cm);
						\node at (1,0) {$\dots$};
						\draw[fill=black] (1.5,0) circle (0.05cm and 0.05cm);
						\draw  (0,0) to   (0.7,0);
						\draw  (1.3,0) to   (1.5,0);
						\node at (0.8,-0.6)[color=black][font=\large]{$\underbrace{\ \ \ \ \ \ \ \ \ \ \ \ \ }_m$};
				\end{tikzpicture}}
			\end{aligned}\oplus \ \lambda^m 
			\begin{aligned}
				\resizebox{50pt}{!}{%
					\begin{tikzpicture}[block/.style={draw, rectangle, minimum height=0.5cm,fill=white}]
						\node at (1,1) {};
						\draw[fill=black] (0,0) circle (0.05cm and 0.05cm);
						\draw[fill=black] (0.5,0) circle (0.05cm and 0.05cm);
						\node at (1,0) {$\dots$};
						\draw[fill=black] (1.5,0) circle (0.05cm and 0.05cm);
						\draw  (0,0) to   (0.7,0);
						\draw  (1.3,0) to   (1.5,0);
						\node at (0.8,-0.6)[color=black][font=\large]{$\underbrace{\ \ \ \ \ \ \ \ \ \ \ \ \ }_m$};
				\end{tikzpicture}}
			\end{aligned},
		\end{align*}
		for all $m\geq 1$.

		Hence 
			\vspace{-0.5cm}
		\begin{align*}
				\lambda^{-(m-1/2)} \theta_m  + \lambda^{-m} \theta_m[2] \mapsto 
			0\oplus  
			\begin{aligned}
				\resizebox{50pt}{!}{%
					\begin{tikzpicture}[block/.style={draw, rectangle, minimum height=0.5cm,fill=white}]
						\node at (1,1) {};
						\draw[fill=black] (0,0) circle (0.05cm and 0.05cm);
						\draw[fill=black] (0.5,0) circle (0.05cm and 0.05cm);
						\node at (1,0) {$\dots$};
						\draw[fill=black] (1.5,0) circle (0.05cm and 0.05cm);
						\draw  (0,0) to   (0.7,0);
						\draw  (1.3,0) to   (1.5,0);
						\node at (0.8,-0.6)[color=black][font=\large]{$\underbrace{\ \ \ \ \ \ \ \ \ \ \ \ \ }_m$};
				\end{tikzpicture}}
			\end{aligned}\oplus 0,\ \ \ \text{and} 
			\end{align*}
			\vspace{-0.5cm}
				\begin{align*}
				-\lambda^{-(m-1/2)} \theta_m  + \lambda^{-m} \theta_m[2] \mapsto 0\oplus  0\oplus  
			\begin{aligned}
				\resizebox{50pt}{!}{%
					\begin{tikzpicture}[block/.style={draw, rectangle, minimum height=0.5cm,fill=white}]
						\node at (1,1) {};
						\draw[fill=black] (0,0) circle (0.05cm and 0.05cm);
						\draw[fill=black] (0.5,0) circle (0.05cm and 0.05cm);
						\node at (1,0) {$\dots$};
						\draw[fill=black] (1.5,0) circle (0.05cm and 0.05cm);
						\draw  (0,0) to   (0.7,0);
						\draw  (1.3,0) to   (1.5,0);
						\node at (0.8,-0.6)[color=black][font=\large]{$\underbrace{\ \ \ \ \ \ \ \ \ \ \ \ \ }_m$};
				\end{tikzpicture}}
			\end{aligned}.
		\end{align*}
		Thus  we have surjectivity on $\Hom_{\mathcal C}(\mathbb 1, \mathbb 1\boxtimes A_{+}^{\otimes m}\boxtimes \mathbb 1)$ and $\Hom_{\mathcal C}(\mathbb 1, \mathbb 1\boxtimes \mathbb 1 \boxtimes A_-^{\otimes m}),$  for all $m\geq 1$.
		
		It remains to show that $F_{\mathcal A}$ is surjective on $\Hom_{\mathcal C}(\mathbb 1, A_t^{\otimes m}\boxtimes \mathbb 1 \boxtimes \mathbb 1)$. Consider the set $\{\xi_{\overline{J}}^m\}_{\overline J \in \mathcal R_n}$ as  in Definition \ref{xi definition}. Then $F_{\mathcal A}$ maps 
			\vspace{-0.5cm}
		\begin{align*}
			\xi_{\overline{J}}^m \mapsto F_{\overline{J}}^m  \oplus \lambda^{m-1} 
			\begin{aligned}
				\resizebox{50pt}{!}{%
					\begin{tikzpicture}[block/.style={draw, rectangle, minimum height=0.5cm,fill=white}]
						\node at (1,1) {};
						\draw[fill=black] (0,0) circle (0.05cm and 0.05cm);
						\draw[fill=black] (0.5,0) circle (0.05cm and 0.05cm);
						\node at (1,0) {$\dots$};
						\draw[fill=black] (1.5,0) circle (0.05cm and 0.05cm);
						\draw  (0,0) to   (0.7,0);
						\draw  (1.3,0) to   (1.5,0);
						\node at (0.8,-0.6)[color=black][font=\large]{$\underbrace{\ \ \ \ \ \ \ \ \ \ \ \ \ }_m$};
				\end{tikzpicture}}
			\end{aligned}\oplus \lambda^{m-1} 
		\begin{aligned}
		\resizebox{50pt}{!}{%
			\begin{tikzpicture}[block/.style={draw, rectangle, minimum height=0.5cm,fill=white}]
				\node at (1,1) {};
				\draw[fill=black] (0,0) circle (0.05cm and 0.05cm);
				\draw[fill=black] (0.5,0) circle (0.05cm and 0.05cm);
				\node at (1,0) {$\dots$};
				\draw[fill=black] (1.5,0) circle (0.05cm and 0.05cm);
				\draw  (0,0) to   (0.7,0);
				\draw  (1.3,0) to   (1.5,0);
				\node at (0.8,-0.6)[color=black][font=\large]{$\underbrace{\ \ \ \ \ \ \ \ \ \ \ \ \ }_m$};
		\end{tikzpicture}}
	\end{aligned},
		\end{align*}
		where $\{ F_{\overline{J}}^m\}$ is a basis for the subspace of connected maps $\mathbb 1 \to \langle A\rangle_t^{\otimes m}$ in $\Rep(S_t\wr \mathbb Z_2)$, see Lemma  \ref{surjective on connected} and Equation \eqref{image of xi under F}. Since we already know $F_{\mathcal A}$ is surjective on $\Hom_{\mathcal C}(\mathbb 1, \mathbb 1\boxtimes  A_{+}\boxtimes \mathbb 1)$ and $\Hom_{\mathcal C}(\mathbb 1, \mathbb 1\boxtimes  \mathbb 1 \boxtimes A_{+})$, this implies that $F_{\overline J}^n$ is in the image of $F_{\mathcal A}$, and it follows that $F_{\mathcal A}$ is surjective on $\Hom_{\mathcal C}(\mathbb 1, A_t\boxtimes \mathbb 1\boxtimes  \mathbb 1)$, as desired. 
	\end{proof}

		\begin{proof}[Proof of Theorem \ref{maintheorem2}]
		Consider the symmetric tensor functor 
		$$	\underline{F_{\mathcal A}}: \operatorname{UCob}_{\alpha, \beta, \gamma} \xrightarrow{F_{\mathcal A}} \Rep(S_t \wr \mathbb Z_2) \boxtimes \Rep(S_{t_+}) \boxtimes \Rep(S_{t_-})\to \underline{\Rep(S_t \wr \mathbb Z_2)} \boxtimes \underline{\Rep(S_{t_+})} \boxtimes \underline{\Rep(S_{t_-})},$$
		where $F_{\mathcal A}$ is as defined previously,  followed by the semisimplification functor. By Lemma \ref{2nd lemma}, $\underline{F_{\mathcal A}}$ satisfies the conditions of Proposition \ref{proposition 2}. Moreover, by Lemma \ref{1st lemma} $\underline{F_{\mathcal A}}$  is essentially surjective.  Hence $\underline{F_{\mathcal A}}$ induces an equivalence
		\begin{align*}
		\underline{	\operatorname{UCob}_{\alpha, \beta, \gamma} }\cong  \underline{\Rep(S_t \wr \mathbb Z_2)} \boxtimes \underline{\Rep(S_{t_+})} \boxtimes \underline{\Rep(S_{t_-})},
		\end{align*}
		as desired. 
	\end{proof}

		\begin{proposition}\label{Corollary II}
		Let $\alpha, \beta$ and $\gamma$ be sequences as in \eqref{sequences}. Suppose $\lambda\alpha_0-\gamma_0$ and $\gamma_0\pm \sqrt{\lambda}\beta_0$ are not non-negative even  integers.  If Conjecture \ref{conjecture for product} holds, we get an equivalence of tensor categories
		\begin{align*}
			\operatorname{UCob}_{\alpha, \beta, \gamma} \cong  \Rep(S_t \wr \mathbb Z_2) \boxtimes \Rep(S_{t_+}) \boxtimes \Rep(S_{t_-}).
		\end{align*}
	\end{proposition}
	
	\begin{proof}
	 If Conjecture \ref{conjecture for product} holds, then for $\lambda, \alpha_0,\beta_0$ and $\gamma_0$ as in the statement there are no negligible morphisms in $\PsUCob$. 
		On the other hand, by \cite{M, D}  the categories $\Rep(S_t\wr \mathbb Z_2)$ and $\Rep(S_{t_{\pm}})$ are semisimple when $t=\frac{1}{2}(\lambda\alpha_0-\gamma_0)$ and $t_{\pm}= \frac{1}{2}(\gamma_0\pm \sqrt{\lambda}\beta_0)$ are not positive integers. Hence by Theorem \ref{maintheorem2} we conclude
		\begin{align*}
		\operatorname{UCob}_{\alpha, \beta, \gamma} \cong  \Rep(S_t \wr \mathbb Z_2) \boxtimes \Rep(S_{t_+}) \boxtimes \Rep(S_{t_-}).
	\end{align*}
	\end{proof}

	
	\bibliographystyle{plain}
	
\end{document}